\documentclass[final]{amsart} 
\usepackage{mathtools, amsthm, amssymb}
\usepackage[margin=1in]{geometry}
\usepackage{mathrsfs}
\usepackage{tikz-cd}

\usepackage{hyperref}

\usepackage{thmtools}
\makeatletter
\IfFormatAtLeastTF{2024-11-01}{
  \renewcommand*\@addtoreset[2]{%
    \bgroup
      \edef\aliasctr@@truelist{\aliasctr@follow{#2}}%
      \let\@elt\relax
      \expandafter\@cons\aliasctr@@truelist{{#1}}%
    \egroup
    \expandafter\xdef\csname theH#1\endcsname{%
      \expandafter\noexpand\csname theH#2\endcsname.%
      \noexpand\the\noexpand\value{#1}}%
  }
}{}
\makeatother

\usepackage[capitalise]{cleveref} % need capitalise option to get capitalisation for multiple references of different types

\crefname{enumi}{item}{items}

\usepackage{microtype}

\newcommand{\resp}{{\sfcode`\.1000 resp.}}
\newcommand{\ie}{{\sfcode`\.1000 i.e.}}
\newcommand{\eg}{{\sfcode`\.1000 e.g.}}
\newcommand{\cf}{{\sfcode`\.1000 cf.}}

\DeclareFontFamily{U}{min}{}
\DeclareFontShape{U}{min}{m}{n}{<-> udmj30}{}
\newcommand\yo{\!\text{\usefont{U}{min}{m}{n}\symbol{'210}}\!}

\DeclareMathOperator\Corr{Corr}
\DeclareMathOperator\Span{Span}
\newcommand\all{\mathrm{all}}
\DeclareMathOperator\sixFF{6FF}
\DeclareMathOperator\VsixFF{V6FF}

\newcommand\BM{\mathrm{BM}}

\newcommand\pt{\mathrm{pt}}

\DeclareMathOperator\SH{\mathbf{SH}}

\DeclareMathOperator\Fun{Fun}
\DeclareMathOperator\Hom{Hom}

\DeclareMathOperator\Spec{Spec}

\DeclareMathOperator\Psh{Psh}
\DeclareMathOperator\Shv{Shv}

\DeclareMathOperator\aff{\mathbb{A}}
\DeclareMathOperator\proj{\mathbb{P}}

\newcommand\LAd{\mathrm{LAd}}

\newtheorem{thm}{Theorem}[section]
\Crefname{thm}{Theorem}{Theorems}
\newtheorem{thmX}{Theorem}
\Crefname{thmX}{Theorem}{Theorems}
\newtheorem{prp}[thm]{Proposition}
\Crefname{prp}{Proposition}{Propositions}
\newtheorem{lem}[thm]{Lemma}
\Crefname{lem}{Lemma}{Lemmas}
\newtheorem{cor}[thm]{Corollary}
\Crefname{cor}{Corollary}{Corrolaries}
\theoremstyle{remark}
\newtheorem{rmk}[thm]{Remark}
\Crefname{rmk}{Remark}{Remarks}
\theoremstyle{definition}
\newtheorem{defn}[thm]{Definition}
\Crefname{defn}{Definition}{Definitions}
\newtheorem{exa}[thm]{Example}
\Crefname{exa}{Example}{Examples}
\Crefname{warn}{Warning}{Warnings}
\newtheorem{nota}[thm]{Notation}
\Crefname{nota}{Notation}{Notations}
\Crefname{cnstr}{Construction}{Constructions}
\newtheorem{setng}[thm]{Setting}
\Crefname{setng}{Setting}{Settings}
\newtheorem{ass}[thm]{Assumption}
\Crefname{ass}{Assumption}{Assumptions}
\newcommand\id{\mathrm{id}}

\newcommand\comps{\mathbb{C}}
\newcommand\reals{\mathbb{R}}

\newcommand\fin{\mathrm{fin}} %%

\newcommand\lred{\mathrm{lred}}
\newcommand\nice{\mathrm{nice}}

\newcommand\et{\mathrm{\acute{e}t}}

\newcommand\cdh{\mathrm{cdh}}

\DeclareMathOperator\B{\mathbf{B}\!}

\newcommand\lax{\mathrm{lax}}

\newcommand\Betti{\mathrm{Betti}}
\newcommand\an{\mathrm{an}}
\newcommand\hol{\mathrm{hol}}
\newcommand\alg{\mathrm{alg}}

\newcommand\mot{\mathrm{mot}}

\DeclareMathOperator\ho{ho}

\newcommand\coCart{\mathrm{coCart}}

\newcommand\op{\mathrm{op}}

\newcommand\red{\mathrm{red}}
\newcommand\clcl{\mathrm{cl}}

\DeclareMathOperator\QC{QCoh}

\newcommand\cstr{\mathrm{cstr}}
\DeclareMathOperator\PPF{PPF}
\DeclareMathOperator\PF{PF}

\newcommand\Sch{\mathrm{Sch}}

\newcommand\AlgSp{\mathrm{AlgSp}}
\newcommand\AlgStk{\mathrm{AlgStk}}
\newcommand\HolStk{\mathrm{HolStk}}
\newcommand\An{\mathrm{An}}

\newcommand\spaces{\mathcal{S}}

\newcommand\Cat{\mathbf{Cat}}

\newcommand\PrL{{\mathbf{Pr}^{\mathbf{L}}}}
\newcommand\PrR{{\mathbf{Pr}^{\mathbf{R}}}}

\newcommand\LMod{\mathrm{LMod}}
\newcommand\Mod{\mathrm{Mod}}
\newcommand\Alg{\mathrm{Alg}}
\newcommand\CAlg{\mathrm{CAlg}}

\DeclareMathOperator\Sp{Sp}

\DeclareMathOperator\Th{Th}

\DeclarePairedDelimiter\cls\lbrack\rbrack

\newcommand\qadm{\mathrm{qadm}}

\newcommand\stackslink[1]{\href{https://stacks.math.columbia.edu/tag/#1}{#1}}
\newcommand\stackscite[1]{\cite[Tag \stackslink{#1}]{stacks-project}}
\newcommand\kerodonlink[1]{\href{https://kerodon.net/tag/#1}{#1}}
\newcommand\kerodoncite[1]{\cite[Tag \kerodonlink{#1}]{kerodon}}

\title[Geometric Criteria for 6-Functor Formalisms]{Geometric Criteria for 6-Functor Formalisms in the Setting of Pullback Formalisms}
\author{Roy Magen}
\address{Institute of Mathematics and Informatics, Bulgarian Academy of Sciences \\ Bulgaria, Sofia 1113, Acad. G. Bonchev Str., Bl. 8}
\thanks{Supported by the Simons Foundation, grant SFI-MPS-T-Institutes-00007697, and the Ministry of Education and Science of the Republic of Bulgaria, grant DO1-239/10.12.2024}

\begin{document}

\begin{abstract}
	In this article, we study criteria for producing six-functor formalisms and morphisms between them.
	% placing earlier work on motivic six-functor formalisms of Voevodsky, Ayoub, Cisinski-D\'{e}glise, and Khan-Ravi, in the context of more recent work on abstract six-functor formalisms of Cnossen-Lenz-Linskens, Dauser-Kuijper, Mann, and Liu-Zheng.
	One notable application is that the motivic homotopy theory of algebraic stacks is the universal six-functor functor formalism in a strong sense: it is initial in some category whose objects are six-functor formalisms, and whose morphisms \emph{commute with all six operations}. As a further application, we produce an analytic realization to a complex analytic version of motivic homotopy theory that is compatible with the six operations, and extend Betti realization to a map from this complex analytic version that is also compatible with the six operations. The abstract nature of our results is suitable for applications to many geometric contexts, allowing us to prove a similar result for the motivic homotopy theory of complex analytic stacks as a six-functor formalism \emph{defined on complex analytic stacks}.

	% Our applications strengthen previous results of this form, which only showed that motivic homotopy theory is initial as a six-functor formalism on \emph{schemes}, or that it is a six-functor formalism which is initial in some category of functor formalisms on stacks, but not that the morphisms from it are fully compatible with the six operations.
	% % NOTE: the Dauser-Kuijper result is only for schemes
	% Furthermore, the abstract nature of our results is suitable for applications to many geometric contexts, allowing us to prove a similar result for the motivic homotopy theory of complex analytic stacks as a six-functor formalism \emph{defined on complex analytic stacks}.

	Our main general result is a generalized and enhanced version of Voevodsky's geometric criterion for six-functor formalisms, given in terms of localization and duality properties. Our version of Voevodsky's principle makes sense in very general geometric contexts, and provides criteria not only for showing when presheaves extend to six-functor formalism, and when a transformation between six-functor formalisms is compatible with the six operations, but also for when a transformation to an ordinary presheaf extends to a morphism of six-functor formalisms (and therefore establishing the six operations for the codomain).
\end{abstract}
\maketitle

\tableofcontents

\section{Introduction}

\subsection{Setting the stage}

\subsubsection{Cohomology theories and Grothendieck's six operations}

Cohomology theories are a powerful and ubiquitous tool for studying geometric objects of various types. Some common examples are given by singular cohomology for topological spaces, and sheaf cohomology for schemes. These assign algebraic invariants to our geometric objects depending on the ``coefficients'' we choose for our cohomology theory.

Indeed, for a given type of geometric object, the cohomology theories on these objects are usually organized into systems of coefficients. For example, one could take any abelian group as a coefficient for singular cohomology of topological spaces. More generally, for any topological space $X$, one can take singular cohomology with coefficients in any local system on $X$. Similarly, for a scheme $X$, one can compute sheaf cohomology of $X$ with coefficients in any coherent sheaf on $X$.

Thus, cohomology theories are often organized into ``systems of coefficients'': given a category $\mathcal C$ of geometric objects, we associate to each object $X \in \mathcal C$ a category $D(X)$ of coefficients for cohomology theories on $X$. Many properties of the resulting cohomology theories are then governed by the properties of the system of coefficients $X \mapsto D(X)$.

First considered in the setting of \'{e}tale cohomology, the yoga of Grothendieck's six operations is a recurring behaviour of these systems of coefficients that governs many important properties of cohomology theories, such as duality theorems and K\"{u}nneth formulae. These so-called 6-functor formalisms, which are systems of coefficients that have this behaviour, have been studied for a long time before they were finally given a general definition in Lucas Mann's thesis \cite{Mann6FF}, after which they have received even more attention in the literature.

\subsubsection{Motivic homotopy and Voevodsky's geometric criterion for 6-functor formalisms}

Apart from defining 6-functor formalisms, Mann also proves some important results about them, building on \cite{LZ6FF}. One crucial result is \cite[Proposition A.5.10]{Mann6FF}, which is the main result used to construct 6-functor formalisms, and gives criteria for a system of coefficients to extend to a 6-functor formalisms by showing that certain maps behave cohomologically like open immersions or proper maps. This principle was later refined in \cite{6FFUnique, CLL6FF}.

In contrast, and more than 20 years earlier, Voevodsky also outlined in \cite[1.2.1]{voe-crit} some geometric axioms for a system of coefficients on the category of schemes to admit the structure of a 6-functor formalism. This principle was proven in Ayoub's thesis \cite{Ayoub6I, Ayoub6II}, where it was then used to show that stable motivic homotopy theory admits the structure of a 6-functor formalism.

% eg CD19 2.4.50
This was later generalized in \cite{tri-cat-mixed-motives}, allowing the authors to drop some quasi-projectivity hypotheses. In \cite{sixopsequiv}, this principle was established for the stable motivic homotopy theory of \emph{equivariant} schemes. In \cite{SixAlgSp}, the principle was established for general systems of coefficients on algebraic spaces, and in \cite{SixAlgSt} this was done for algebraic stacks, allowing the authors to produce a 6-functor formalism of stable motivic homotopy on algebraic stacks.

% Voevodsky's principle is particularly well-suited to showing that motivic homotopy theory and variations of it admit the structure of a 6-functor formalism. We will give a more detailed account of the version of Voevodsky's principle for algebraic stacks that can be extracted from \cite{SixAlgSt} in \Cref{S:current SoA}.

One of the goals of the present work is to establish a version of Voevodsky's principle that works in the same level of generality as Mann's definition of 6-functor formalisms, so that it should be able to recover these previous accounts, as well as allow for applications to other geometric contexts beyond algebraic geometry.

\subsubsection{Betti realization and morphisms of 6-functor formalisms}

It is often fruitful to compare the different types of cohomology theories coming from different types of systems of coefficients. We have already mentioned the examples of sheaf cohomology for schemes coming from the system of coefficients of coherent sheaves, and the example singular cohomology coming from the system of coefficients of local systems. More generally, if we consider cohomology theories on topological spaces with coefficients in general sheaves of abelian groups (instead of just locally constant ones), then we can try to compare cohomology theories on a scheme $X$ with coefficients in coherent sheaves on $X$, to the cohomology theories on the underlying topological space of the analytification $X^\an$ of $X$, with coefficients given by sheaves of abelian groups on $X^\an$.

Motivic homotopy theory was developed, at least partially, as an algebro-geometric analog of the usual homotopy theory of CW-complexes.\footnotemark{} As mentioned in Joseph Ayoub's ICM address \cite{AyoubICM}, one of the goals of motivic homotopy is to provide a bridge between so-called ``transcendental'' invariants of an algebraic variety, such as its Betti cohomology, and its algebro-geometric invariants, such as its Chow groups and $K$-theory. Indeed, there is a Betti realization functor that can be seen as a morphism from the coefficient system of motivic spectra to the coefficient system of sheaves of abelian groups on analytic spaces. Ayoub showed in \cite{AyoubBetti} that this morphism is actually compatible with the structure of Grothendieck's six operations. This has proven to be a crucial tool not only for using topology to study algebraic geometry, but the reverse as well: Voevodsky used Betti realization in his proof of the Bloch-Kato conjecture in $K$-theory to compute the motivic Steenrod algebra in terms of the topological one, and in \cite{C2equiv}, Behrens and Shah show how to use Betti realization to compute $C_2$-equivariant homotopy groups in terms of motivic homotopy groups over $\reals$.
\footnotetext{This paragraph is mostly repeated from the introduction of \cite{Fundamentals}.}

% 2.4.53, 4.4.25
Thus, we are not only interested in producing 6-functor formalisms, but also morphisms between them. In the implementation of Voevodsky's principle given in \cite{tri-cat-mixed-motives}, the authors also enhance the principle by providing criteria for when transformations between 6-functor formalisms on schemes are actually compatible with the six operations, allowing them to prove some results about realizations in \cite[\S17]{tri-cat-mixed-motives}. On the other hand, the recent works \cite{6FFUnique, CLL6FF} are dedicated to enhancing the result given in \cite[Proposition A.5.10]{Mann6FF} about constructing \emph{abstract} 6-functor formalisms to statements about \emph{categories} of 6-functor formalisms, which can also be seen as giving criteria for transformations of 6-functor formalisms to be compatible with the six operations.

\subsection{An abstract version of Voevodsky's principle} \label{S:voe}

In view of the preceding discussion, we can identify two types of related results: on the one hand, we have the abstract categorical results of \cite{6FFUnique, CLL6FF, Mann6FF, LZ6FF} that allow us to construct abstract 6-functor formalisms and morphisms between them, and on the other hand, we have Voevodsky's principle and its refinements given in \cite{SixAlgSt, SixAlgSp, tri-cat-mixed-motives, Ayoub6I, Ayoub6II, voe-crit} that take advantage of a particular algebro-geometric setup to give more ``geometric'' criteria for producing 6-functor formalisms on schemes, and morphisms between them.

The primary objective of the present work is to find an optimal midpoint between these two types of results that allows us to give general ``geometric'' criteria for producing particularly well-behaved 6-functor formalisms and morphisms between them. This allows us to more easily extend the existing results of the second type to new settings, such as that of derived algebraic stacks (allowing us to improve some results of \cite{SixAlgSt}), or even beyond algebraic geometry, allowing us to prove results about 6-functor formalisms on complex analytic stacks. Furthermore, we will see that this abstract perspective allows us to make some improvements on past applications of Voevodsky's principle.

In order to formulate our results, we fix an $\infty$-category $\mathcal C$ of ``geometric objects''. In $\mathcal C$, we fix some collections of maps that we will refer to as \textbf{$\mathcal C$-smooth} maps and \textbf{$\mathcal C$-closed} maps. These maps are assumed to be closed under base change and composition, and to contain all equivalences. We also assume that every $\mathcal C$-closed map has $\mathcal C$-smooth complement in the sense of \Cref{defn:complements}.

The following definition is a slight weakening of the corresponding notion from \Cref{defn:PPF cat} -- see \Cref{prp:gluing implies loc,prp:loc for stab PF}.
\begin{defn} \label{defn:intro cstr PF}
	A \textbf{constructible pullback formalism} on $\mathcal C$ is a presheaf $D : \mathcal C^\op \to \CAlg(\PrL)$ of symmetric monoidal presentable $\infty$-categories satisfying axioms similar to some of the axioms from \cite[1.2.1]{voe-crit} or \cite[1.4.1]{Ayoub6I}:
	\begin{enumerate}

		\item \textbf{Pointed and reduced:}
			$D$ takes values in pointed presentable $\infty$-categories,\footnote{These are $\infty$-categories $\mathcal C$ that are presentable in the sense of \cite[\S5.5]{htt}, and which contain a zero object. A zero object of an $\infty$-category is an object that is both initial and terminal.} and $D(\emptyset) \simeq \pt$, where $\emptyset$ denotes any initial object of $\mathcal C$.

		\item \textbf{Localization:} \label{itm:intro/loc}
			For any $\mathcal C$-closed map $i : Z \to S$ with complement $j : U \to S$, the right adjoint $i_*$ of $i^*$ fits into a fibre sequence
			\[
				D(Z) \xrightarrow{i_*} D(S) \xrightarrow{j^*} D(U)
			\]
			of pointed $\infty$-categories.

		\item \textbf{Smooth projection formula:}
			For any $\mathcal C$-smooth map $f : X \to Y$, the functor $f^* \coloneqq D(f) : D(Y) \to D(X)$ admits a left adjoint $f_\sharp$, and for any $M \in D(X)$ and $N \in D(Y)$, the natural map
			\[
				f_\sharp(M \otimes f^* N) \to f_\sharp M \otimes N
			\]
			is an equivalence.

		\item \textbf{Smooth base change:}
			If
			\[
				\begin{tikzcd}
					X' \ar[d, "p"'] \ar[r, "f'"] & Y' \ar[d, "q" ] \\
					X \ar[r, "f"'] & Y
				\end{tikzcd}
			\]
			is a Cartesian square in $\mathcal C$, and $f$ is a $\mathcal C$-smooth map, then the natural map
			\[
				f'_\sharp p^* \to q^* f_\sharp
			\]
			is an equivalence.

	\end{enumerate}
\end{defn}

\begin{exa} \label{exa:intro AlgStk general}
	One might wonder what sorts of presheaves $D : \mathcal C^\op \to \CAlg(\PrL)$ are constructible pullback formalisms. Versions of motivic homotopy theory provide our main source of examples. Indeed, a good example to keep in mind is when $\mathcal C = \Sch$ is a suitable category of schemes, the $\Sch$-smooth maps are the smooth morphisms, and the $\Sch$-closed maps are the closed immersions. Then the presheaf $\SH : \Sch^\op \to \CAlg(\PrL)$ of Morel-Voevodsky's stable motivic homotopy is a constructible pullback formalism.

	In fact, we can do better: if $\mathcal C = \AlgStk$ is a suitable category of algebraic stacks as considered in \cite{SixAlgSt}, we can take the $\AlgStk$-smooth maps are the representable smooth morphisms, and the $\AlgStk$-closed maps are the closed immersions. Then the presheaf $\SH : \AlgStk^\op \to \CAlg(\PrL)$ constructed in \cite{SixAlgSt} is a constructible pullback formalism on $\AlgStk$.

	More generally, \cite{UnivFF} and \cite{Fundamentals} study constructions that produce these sorts of presheaves. We will briefly sketch the situation: for any $S \in \mathcal C$, we denote by $\mathcal C_S$ the full subcategory of $\mathcal C_{/S}$ consisting of $\mathcal C$-smooth maps to $S$. Since $\mathcal C$-smooth maps are stable under base change and composition, the presheaf $S \mapsto \Psh(\mathcal C_S)$ defines a presheaf satisfying the smooth projection formula and smooth base change. 

	The constructions studied in \cite[\S5 and \S6]{UnivFF} or \cite[\S3 and \S4]{Fundamentals} show that the operations of
	\begin{itemize}

		\item imposing ``smooth'' descent or invariance conditions on our presheaves,

		\item taking pointed objects,

		\item formally adjoining $\otimes$-inverses,

	\end{itemize}
	all preserve the smooth projection formula and smooth base change axioms. Of course, taking pointed objects allows us to produce presheaves $D$ that take values in pointed categories, and imposing mild descent conditions ensure that $D$ is a reduced presheaf ($D(\emptyset) = \pt$). 

	In general, the localization property can be quite difficult to show, but still, this has been done for various versions of ($\aff^1$-invariant) motivic homotopy theory. In \cite{Gluing} we study general tools for showing this property, and establish it in some special cases.
	% PERF: cite other works showing this property
\end{exa}

Next, we will need to consider a notion of duality for maps in $\mathcal C$. Before introducing this notion, it will be convenient to introduce the following notation: given a $\mathcal C$-smooth map $X \to S$, we denote by $\cls{X}$ the object of $D(S)$ given by $(X \to S)_\sharp$ of the monoidal unit of $D(X)$.
\begin{defn} \label{defn:intro ambidext}
	Let $D$ be a constructible pullback formalism on $\mathcal C$, and let $f : X \to Y$ be a $\mathcal C$-smooth map that has $\mathcal C$-closed diagonal. Then there is a natural map (see \Cref{defn:duality,prp:monoidal twists})
	\[
		\eth_f : f_\sharp \to f_*(- \otimes \omega_f)
	,\]
	where
	\[
		\omega_f \coloneqq \cls{X \times_Y X}/\cls{X \times_Y X \setminus X}
	\]
	(the map $X \to X \times_Y X$ is given by the diagonal). Then $\omega_f$ can be thought of as a \textbf{Thom object of the tangent bundle of $f$}.

	We say that $f$ is \textbf{stably $D$-ambidextrous} if $\omega_f$ is $\otimes$-invertible, and $\eth_f$ is an equivalence.
\end{defn}

This notion of ``stable ambidexterity'' can be seen as a twisted ambidexterity property since it identifies the left and right adjoints of $f^*$ up to a twisting by a ``line bundle'' ($\otimes$-invertible object). This corresponds to familiar notions of duality in topology and geometry. % PERF: write examples (Wirthmuller and Poincare/Atiyah duality? Serre?)

The final ingredient required for our results is a collection of maps in $\mathcal C$ called \textbf{$\mathcal C$-proper maps}. We now come to the key geometric input from the algebro-geometric setting:
% ANNOTATION: generating proper maps
\begin{rmk}[Generating proper maps in algebraic geometry] \label{rmk:gen prop in AG}
	In the case $D = \SH$ of stable motivic homotopy theory in the setting of algebraic geometry, it has been shown (see \cite[Lemma 6.9]{SixAlgSt} or \cite[Theorem 6.9]{sixopsequiv}) that if $\mathcal E$ is a finite locally free sheaf on a (suitable) algebraic stack $S$, the map $\proj(\mathcal E) \to S$ is stably $\SH$-ambidextrous. 

	Recalling \Cref{exa:intro AlgStk general}, we see that projective maps are given as composites of $\AlgStk$-closed maps, and stably $\SH$-ambidextrous maps. Furthermore, versions of Chow's Lemma (see \Cref{lem:source-locally qproj}) show that if $X \to Y$ is a representable proper map, then there is a projective cdh cover $X' \to X$ such that $X' \to Y$ is also projective. This shows that the representable proper maps are ``generated'' by the $\AlgStk$-closed maps and the stably $\SH$-ambidextrous maps in a suitable sense.
\end{rmk}

Thus, we are led to consider constructible pullback formalisms $D$ such that the $\mathcal C$-proper maps are ``generated'' in an appropriate but lenient sense, which may depend on $D$, by the $\mathcal C$-closed maps and certain stably $D$-ambidextrous maps. We then say that $D$ is a \textbf{strongly projective pullback formalism} if it also takes values in stable $\infty$-categories. See \Cref{defn:proj sat} for the precise definition, where the terms ``$\mathcal C$-smooth'', ``$\mathcal C$-closed'', and ``$\mathcal C$-proper'' are replaced by ``quasi-admissible'', ``exceptionally closed'', and ``exceptionally quasi-proper''.

\subsubsection{Results about strongly projective pullback formalisms: the statement of \Cref{thmX:PPF}} \label{S:intro PPF}

One of the main contributions of this paper is given by the properties of strongly projective pullback formalisms. A general discussion of the strategy for showing these results is given in \Cref{S:outline}. Here is a first result:
\begin{thm} \label{thm:intro strongly PPF is PPF}
	If $D$ is a strongly projective pullback formalism, and $f : X \to Y$ is a $\mathcal C$-proper map, then $f$ ``behaves cohomologically like a proper map'' in the sense that $D$ satisfies the \textbf{proper projection formula} for $f$, has \textbf{proper base change} and \textbf{smooth-proper base change} for $f$, and the right adjoint $f_*$ of $f^* \coloneqq D(f)$ has a further right adjoint. See \Cref{rmk:desc proper} for a precise formulation (where ``quasi-admissible'' means $\mathcal C$-smooth).
	% \begin{enumerate}
	%
	% 	\item \textbf{Additional adjoint:}
	% 		The right adjoint $f_*$ of $f^* \coloneqq D(f)$ admits a further right adjoint $f^\sharp$.
	%
	% 	\item \textbf{Proper projection formula:}
	% 		For any $M \in D(X)$, and $N,N' \in D(Y)$, the natural maps
	% 		\[
	% 			f_* M \otimes N \to f_*(M \otimes f^* N) \quad\text{and}\quad \underline\Hom(f^* N, f^\sharp N') \to f_\sharp \underline\Hom(N, N')
	% 		\]
	% 		are equivalences.
	%
	% 	\item \textbf{Proper base change:}
	% 		If
	% 		\[
	% 			\begin{tikzcd}
	% 				X' \ar[d, "p"'] \ar[r, "f'"] & Y' \ar[d, "q" ] \\
	% 				X \ar[r, "f"'] & Y
	% 			\end{tikzcd}
	% 		\]
	% 		is a Cartesian square in $\mathcal C$, then the natural maps
	% 		\[
	% 			q^* f_* \to f'_* p^* \quad\text{and}\quad p_* f'^\sharp \to f^\sharp q_*
	% 		\]
	% 		are equivalences.
	%
	%
	% 	\item \textbf{Smooth-proper base change:}
	% 		In the previous point, if $q$ is $\mathcal C$-smooth, then the natural maps
	% 		\[
	% 			p^* f^\sharp \to f'^\sharp q^* \quad\text{and}\quad q_\sharp f'_* \to f_* p_\sharp
	% 		\]
	% 		are equivalences.
	%
	%
	% \end{enumerate}
\end{thm}

Constructible pullback formalisms that satisfy the statement of \Cref{thm:intro strongly PPF is PPF} will turn out to be very important. Indeed, if $D$ is such a constructible pullback formalisms and $D$ takes values in stable categories, we will say that $D$ is a \textbf{projective pullback formalisms}. Also see \Cref{defn:proper,defn:PPF cat}.

The notion of $\mathcal C$-closed maps allows us to consider notions of excision, and when combined with the notions of $\mathcal C$-smooth and $\mathcal C$-proper maps, we can define a good notion of \textbf{cdh excision} for presheaves on $\mathcal C$. See \Cref{defn:cdh} for a precise definition, where the notions of $\mathcal C$-smooth, $\mathcal C$-proper, and $\mathcal C$-closed maps are replaced by quasi-admissible, exceptionally quasi-proper and exceptionally closed maps, as before. \Cref{lem:PPF cdh} shows the following (also see \Cref{prp:gluing implies exc}):
\begin{thm} \label{thm:intro cdh PPF}
	If $D$ is a projective pullback formalism, then $D$ has cdh excision.
\end{thm}

As an aside, we mention that in addition to the descent property given in \Cref{thm:intro cdh PPF}, we actually have that for any constructible pullback formalism $D$, $D$ has descent along any base change of a family of $\mathcal C$-smooth maps $\{X_i \to S\}_i$ if and only if the family of functors $\{(X_i \to S)^*\}_i$ is jointly conservative. This holds even without the localization property, and when $D$ is not pointed and reduced (see \cite[Theorem 2.4.3]{Fundamentals}, recalled in \Cref{thm:D-topology}).

The next result concerns morphisms out of strongly projective pullback formalisms. Note that in addition to establishing properties of transformations between strongly projective pullback formalisms, it also shows when the codomain of a transformation from a strongly projective pullback formalism is itself a strongly projective pullback formalism:
\begin{thm} \label{thm:intro PPF morphism}
	Let $D$ be a strongly projective pullback formalism, and let $\phi : D \to D'$ be a transformation of constructible pullback formalisms. Suppose that for every $\mathcal C$-smooth map $f$, the natural map
	\begin{equation} \label{eqn:smooth trans}
		f_\sharp \phi \to \phi f_\sharp
	\end{equation}
	is an equivalence. Then we also have that $D'$ is a strongly projective pullback formalism, and for any $\mathcal C$-proper map $f$, the natural map
	\begin{equation} \label{eqn:proper trans}
		\phi f_* \to f_* \phi
	\end{equation}
	is an equivalence.
\end{thm}

If we write $\PF^\cstr_\bullet(\mathcal C)$ for the $\infty$-category of constructible pullback formalisms and transformations $\phi : D \to D'$ between them such that \eqref{eqn:smooth trans} is an equivalence for all $\mathcal C$-smooth $f$, and $\PPF(\mathcal C)$ for the subcategory of $\PF^\cstr_\bullet(\mathcal C)$ consisting of projective pullback formalisms and morphisms $\phi : D \to D'$ between them such that \eqref{eqn:proper trans} is an equivalence for every $\mathcal C$-proper $f$, then we have the following main result:
\begin{thmX}[\Cref{thm:PPF}] \label{thmX:PPF}
	If $D$ is a strongly projective pullback formalism, then $D \in \PPF(\mathcal C)$, and the functor
	\[
		\PPF(\mathcal C)_{D/} \to \PF^\cstr_\bullet(\mathcal C)_{D/}
	\]
	is an equivalence. Furthermore, for any morphism $D \to D'$ in $\PF^\cstr_\bullet(\mathcal C)$, $D'$ is a strongly projective pullback formalism.
\end{thmX}

\subsubsection{6-functor formalisms} \label{S:intro 6FF}

Before coming to specific applications of \Cref{thmX:PPF} in \Cref{S:intro apps}, let us explain the relevance to 6-functor formalisms.

Let $I$ be a collection of $\mathcal C$-smooth maps (\eg{} open immersions or \'{e}tale maps), and let $P$ be a collection of $\mathcal C$-proper maps (\eg{} projective maps or proper maps). Let $E$ be some collection of maps containing $I \cup P$, and assume the following:
\begin{itemize}

	\item $\mathcal C$ admits finite products.

	\item $I,P,E$ contain all equivalences, and are stable under base change, composition, and taking diagonals.

	\item Every map in $I \cap P$ is truncated (this is automatic if $\mathcal C$ is an ordinary category).

\end{itemize}

When every map in $E$ is of the form $p \circ j$ for $p \in P$ and $j \in I$, combining \cite[Proposition A.5.10]{Mann6FF} with \Cref{thm:intro strongly PPF is PPF} immediately shows that any strongly projective pullback formalism extends to a 6-functor formalism on $(\mathcal C,E)$. In fact, \cite[Theorem B]{CLL6FF} produces a functor from $\PPF(\mathcal C)$ to the category of 6-functor formalisms on $(\mathcal C,E)$, and when every map in $E$ is truncated, \cite[Theorem 3.3]{6FFUnique} shows that this functor is the inclusion of a subcategory.

By a strategic application of our results, and using results from \Cref{S:6FF ext} that refine some of the extension results for 6-functor formalisms given in \cite[\S3.4]{HM6FF} and \cite[\S A.5]{Mann6FF}, we will be able to show stronger versions of these results under the following weaker assumptions:
\begin{enumerate}

	\item Every composite of maps in $I \cup P$ is of the form $p \circ j$ for $p \in P$ and $j \in I$. In fact, this only needs to hold locally with respect to $\mathcal C$-smooth cdh covers.

	\item Every map in $E$ is cdh locally on the source and target a composite of maps in $I \cup P$.

\end{enumerate}
A precise formulation is given in \Cref{setng:V6FF transfer}, where $E'$ is taken to be the collection of composites of maps in $I' \cup P'$.

\begin{exa}
	Consider the case that $\mathcal C \subseteq \AlgStk$ consists only of quasi-compact quasi-separated algebraic stacks. We can take $I$ to be the collection of (quasi-compact) open immersions, and $P$ to be the collection of projective morphisms, so that by \Cref{lem:qc imm} and \cite[Lemma A.0.3]{Fundamentals}, the composites of maps in $I \cup P$ are precisely the quasi-projective morphisms, which are always of the form $p \circ j$ for $p \in P$ and $j \in I$.

	If all objects of $\AlgStk$ have separated diagonals and nice\footnote{In the sense of \cite[Definition 2.1(i)]{SixAlgSt}: an fppf affine group scheme $G$ over an affine scheme $S$ is \emph{nice} if it is an extension of a finite \'{e}tale group scheme of order prime to the residue characteristic of $S$, by a group scheme of multiplicative type.} stabilizers, then by combining \cite[Theorem 2.12(ii), Theorem 2.14(i), Theorem 6.11, and Remark 7.8]{SixAlgSt}, we find that $E$ can be the collection of finite type representable morphisms (see \Cref{lem:source-locally qproj} and the proof of \Cref{thm:nice mot 6FF}).
\end{exa}

\begin{exa}
	Consider the case that $\mathcal C \subseteq \AlgStk$ consists of the quasi-compact quasi-separated quasi-Deligne-Mumford stacks with locally separated diagonals. By \cite[Theorem B]{qDM_compactification}, if we take $I$ to be the collection of open immersions, and $P$ to be the collection of proper representable morphisms, then the collection of all composites of maps in $I \cup P$ is the collection of finite type separated representable morphisms in $\mathcal C$.
\end{exa}

\Cref{thm:6FF} then implies the following:
\begin{thmX}[Criterion for 6-functor formalisms] \label{thmX:6FF}
	Any strongly projective pullback formalism $D$ extends to a 6-functor formalism on $(\mathcal C,E)$ satisfying the following properties:
	\begin{enumerate}

		\item $D$ and $D^!$ have cdh descent.

			%Furthermore, $D$ (\resp{} $D^!$) sends any constructible cover (\resp{} $I$-constructible cover) to a jointly conservative family of functors.

		\item If $p \in P$, then $p_* \simeq p_!$, and if $j \in I$, then $j^* \simeq j^!$.

		\item Every $\mathcal C$-smooth map is $D$-suave, and every $\mathcal C$-proper map is $D$-prim. (See \Cref{thm:suave prim consequences} for consequences.)

	\end{enumerate}
\end{thmX}

In fact, we even have the following criterion for morphisms of 6-functor formalisms, also given in \Cref{thm:6FF}:
\begin{thmX}[Criterion for morphisms of 6-functor formalisms] \label{thmX:6FF morphism}
	Let $D$ be a strongly projective pullback formalism, and let $\phi : D \to D'$ be a transformation of constructible pullback formalisms. Suppose that for any $\mathcal C$-smooth map $f$, the natural map
	\[
		f_\sharp \phi \to \phi f_\sharp
	\]
	is an equivalence (\ie{} $\phi$ is a morphism in $\PF^\cstr_\bullet(\mathcal C)$). Then $D'$ is also a strongly projective pullback formalism, and $\phi$ extends to a morphism of 6-functor formalisms on $(\mathcal C,E)$. For any map $f$, in addition to the canonical equivalences
	\[
		f^* \phi \simeq \phi f^*, \quad\text{and if $f \in E$,} \quad f_! \phi \simeq \phi f_!
	,\]
	we also have that
	\[
		\begin{array}{cl}
			\phi f_* \to f_* \phi & \text{is an equivalence if $f$ is $\mathcal C$-proper, and} \\
			\phi f^! \to f^! \phi & \text{is an equivalence if $f \in E$ is $\mathcal C$-smooth.} \\
		\end{array}
	\]
\end{thmX}

To obtain a statement on the level of categories of 6-functor formalisms, we define the following subcategory $\VsixFF(\mathcal C,E)$ of the category of 6-functor formalisms on $(\mathcal C,E)$, which is given in greater generality in \Cref{defn:V6FF}:
\begin{description}
	
	\item[Objects] are 6-functor formalisms $D$ on $(\mathcal C,E)$ taking values in stable presentable $\infty$-categories such that
		\begin{enumerate}

			\item every $\mathcal C$-proper map is $D$-prim and every $\mathcal C$-smooth map is $D$-suave,\footnote{In order to make sense of this for maps not in $E$, we use \Cref{defn:suave prim}.} 

			\item both of the associated presheaves $D^* : \mathcal C^\op \to \CAlg(\PrL)$ and $D^! : E^\op \to \PrR$ have cdh descent, and

			\item $D^*$ is a projective pullback formalism.

		\end{enumerate}

	\item[Morphisms] are morphisms of 6-functor formalisms $\phi : D \to D'$ such that for any map $f$ in $\mathcal C$, in addition to the canonical equivalences
		\[
			\phi f^* \simeq f^* \phi, \quad\text{and}\quad f_! \phi \simeq f_! \phi \quad\text{when $f \in E$}
		,\]
		we also have that
		\[
			\begin{array}{cl}
				\phi f_* \to f_* \phi & \text{is an equivalence if $f$ is $\mathcal C$-proper,} \\
				\phi f^! \to f^! \phi & \text{is an equivalence if $f \in E$ is $\mathcal C$-smooth,} \\
				f_\sharp \phi \to \phi f_\sharp & \text{is an equivalence if $f$ is $\mathcal C$-smooth, and}\footnotemark{} \\
				f^\flat \phi \to \phi f^\flat & \text{is an equivalence if $f \in E$ is $\mathcal C$-proper.}\footnotemark[\thefootnote]{}
			\end{array}
		\]
		\footnotetext{In general, $f_\sharp$ denotes a left adjoint of $f^*$ when it exists, and $f^\flat$ denotes a left adjoint of $f_!$ when it exists.}
		
\end{description}

\Cref{rmk:6FF ext functor} shows the following:
\begin{thmX} \label{thmX:V6FF PPF}
	The natural restriction functor
	\[
		\VsixFF(\mathcal C,E) \to \PPF(\mathcal C)
	\]
	admits a section, and if all maps in $I \cup P$ are truncated, then it is an equivalence.
\end{thmX}

Therefore, if $D^*$ is a strongly projective pullback formalism, and all maps in $I \cup P$ are truncated, the following result follows immediately by combining \Cref{thmX:V6FF PPF} with \Cref{thmX:PPF}.
\begin{thmX} \label{thmX:V6FF}
	The presheaf $D^*$ extends uniquely to a 6-functor formalism $D \in \VsixFF(\mathcal C,E)$, and the functor
	\[
		\VsixFF(\mathcal C,E)_{D/} \to \PF^\cstr_\bullet(\mathcal C)_{D^*/}
	\]
	is an equivalence.
\end{thmX}

\subsection{Applications to motivic 6-functor formalisms} \label{S:intro apps}

We will now present some applications to general versions of stable motivic homotopy theory that follow form our abstract results. We leave most of the discussion of the strategy for deducing these to \Cref{S:outline}.

In what follows, $\mathcal C$ is some geometric category of ``stacks''. We will present the case that $\mathcal C$ is either a suitable category $\mathcal C^\alg$ of algebraic stacks, or a suitable category $\mathcal C^\hol$ of complex analytic stacks. When we need to specialize to a particular case, we will write $\mathcal C^\alg$ for the algebraic case, and $\mathcal C^\hol$ for the complex analytic case.

All of the results stated in this section hold in the following cases, but see \Cref{S:applications} for more general versions:
\begin{itemize}

	\item We can take $\mathcal C^\alg$ to be the category of quasi-compact quasi-separated algebraic stacks with separated diagonals and nice stabilizers.

	\item We can take $\mathcal C^\hol$ to be the category $\HolStk^{\fin, \red}$ defined in \Cref{defn:HolStk}, which can be approximately described as the category of reduced complex analytic stacks with finite stabilizers.

\end{itemize}

We have the following natural classes of maps:
\begin{description}
	
	\item[$\mathcal C$-smooth maps] given by representable smooth morphisms in the algebraic case, and representable submersions in the complex analytic case.
		%Sometimes in the algebraic case we must instead take the \emph{quasi-projective} smooth morphisms -- see \Cref{ass:lred or nice}.

	\item[$\mathcal C$-closed maps] given by closed immersions in the algebraic case, and embeddings in the complex analytic case.

	\item[$\mathcal C$-proper maps] given by representable proper maps in the algebraic case. These are less straightforward to define in the complex analytic case, but are given in \Cref{defn:hol} in a way that can be seen as forcing Chow's Lemma to hold (\ie{} we ``generate'' them from embeddings and projective bundles).

\end{description}

In either case, we obtain a notion of ``motivic pullback formalisms'' (\Cref{defn:mot alg PF,defn:mot hol PF}) which can be thought of as presheaves $\mathcal C^\op \to \PrL$ that satisfy Voevodsky's axioms. Recall that for any $D \in \PF^\cstr_\bullet(\mathcal C)$, and $\mathcal C$-smooth map $X \to S$, we write $\cls{X} \in D(S)$ to denote $(X \to S)_\sharp(1)$.
\begin{defn} \label{defn:mot PF}
	The category $\PF^\mot(\mathcal C)$ of \textbf{motivic pullback formalisms} on $\mathcal C$ is the full subcategory of the category $\PF^\cstr_\bullet(\mathcal C)$ of pullback formalisms on $\mathcal C$ consisting of those constructible pullback formalisms $D : \mathcal C^\op \to \CAlg(\PrL)$ satisfying the following:
	\begin{description}
		
		\item[Thom stability] If $V$ is a vector bundle on $S \in \mathcal C$, then the ``Thom object'' $\cls{V}/\cls{V \setminus 0} \in D(S)$ is $\otimes$-invertible.

		\item[Homotopy invariance] For any $S \in \mathcal C$, we have that $\cls{\aff^1_S} \simeq \cls{S}$ in $D(S)$. When $\mathcal C$ consists of complex analytic stacks, $\aff^1_S$ denotes $S \times \comps$.
			% In fact, we often actually need to ask that for any vector bundle torsor $V$ on $S$, we have an equivalence $\cls{V} \simeq \cls{S}$, but this is automatic in many cases of interest (see \Cref{rmk:nice htpy}).

	\end{description}
\end{defn}

We also have a constructible pullback formalism $\SH$ of stable motivic homotopy theory on $\mathcal C$. When necessary, we will write $\SH^\alg$ to denote the algebraic version (on $\mathcal C^\alg$), and $\SH^\hol$ to denote the complex analytic version (on $\mathcal C^\hol$). These are constructed in \cite[Theorem 5.3.10]{Fundamentals} and \cite{SixAlgSt} or \cite[Theorem 5.1.11]{Fundamentals}.

\begin{rmk}[The role of Voevodsky's axioms in the definition of motivic pullback formalisms]
	The axioms given in \Cref{defn:mot PF} are quite common in the literature -- see especially \cite[\S5]{SixAlgSt}, \cite[\S2.4.d]{tri-cat-mixed-motives}, \cite[\S2.1]{SixAlgSp}, \cite[1.4.1]{Ayoub6I}, and \cite[1.2.1]{voe-crit}. In fact, the main motivation for considering these axioms is that they describe the notion of stable motivic homotopy theory we want to study. Indeed, the presheaf $\SH$ of stable motivic homotopy theory on $\mathcal C$ is constructed as a universal categorical invariant that has Thom stability and homotopy invariance -- see \cite{SixAlgSt, SixAlgSp, robalo, UnivFF}.

	Apart from an independent interest in categorical invariants that satisfy these conditions, we can also sketch the relevance of these axioms to strongly projective pullback formalisms on $\mathcal C$. Indeed, as we have mentioned in \Cref{rmk:gen prop in AG}, the $\mathcal C$-proper maps are ``generated'' under certain operations by the $\mathcal C$-closed maps and projective bundles, and we want $\SH$ to satisfy the localization axiom of \Cref{defn:intro cstr PF}(\ref{itm:intro/loc}) for $\mathcal C$-closed maps, and for projective bundles to be $\SH$-ambidextrous.

	The homotopy invariance axiom is then necessary, not only because we are interested in studying homotopy-invariant cohomology theories, but also in order to show that $\SH$ satisfies the localization axiom (while still allowing for $\mathcal C$-smooth maps to have positive relative dimension) -- see, for example, the introduction of \cite{locspalg}. Next, we note that even without the Thom stability condition, it is possible to show some unstable version of ambidexterity for projective bundles -- see \cite[Theorem 5.22]{sixopsequiv} -- and the Thom stability condition then gives us the stable version.
\end{rmk}

Next, we show that $\SH$ is the initial motivic pullback formalism, and is a strongly projective pullback formalism.
This allows us to obtain the following result, given in \Cref{thm:alg PPF,thm:hol PPF}:
\begin{thmX} \label{thmX:mot PPF}
	Every motivic pullback formalism $D \in \PF^\mot(\mathcal C)$ is a strongly projective pullback formalism. Furthermore, $\SH$ is the initial motivic pullback formalism, and the functor
	\[
		\PF^\cstr_\bullet(\mathcal C)_{{\SH}/} \to \PF^\cstr_\bullet(\mathcal C)
	\]
	is fully faithful with essential image given by $\PF^\mot(\mathcal C)$, so for any constructible pullback formalism $D$, the following are equivalent:
	\begin{itemize}

		\item $D$ is a motivic pullback formalism.

		\item There exists a morphism $\SH \to D$ in $\PF^\cstr_\bullet(\mathcal C)$.

		\item There exists a unique morphism $\SH \to D$ in $\PF^\cstr_\bullet(\mathcal C)$.

	\end{itemize}
\end{thmX}

Now we will present results about 6-functor formalisms. First we will show that $\SH$ is the universal 6-functor formalism. For this we must fix a geometric setup $(\mathcal C,E)$. In the algebraic case, let $E$ be the collection of finite type representable maps in $\mathcal C^\alg$, and in the complex analytic case, let $E$ be some collection of truncated $\mathcal C^\hol$-proper maps that is stable under taking diagonals. The following result is given in greater generality in \Cref{thm:nice mot 6FF,thm:hol 6FF}:
\begin{thmX}[$\SH$ is the initial motivic 6-functor formalism] \label{thmX:univ SH}
	$\SH$ extends uniquely to a 6-functor formalism on $(\mathcal C,E)$ so that $\SH \in \VsixFF(\mathcal C,E)$, and $\SH$ is initial in the full subcategory of $\VsixFF(\mathcal C,E)$ consisting of those 6-functor formalisms $D$ such that $D^*$ is a motivic pullback formalism. In fact, the functor
	\[
		\VsixFF(\mathcal C,E)_{{\SH}/} \to \PF^\cstr_\bullet(\mathcal C)
	\]
	is fully faithful with essential image given by the motivic pullback formalisms.
\end{thmX}

\begin{rmk} \label{rmk:explain univ SH}
	Let us spell out the consequences of \Cref{thmX:univ SH}. These say that every motivic pullback formalism $D^*$ extends uniquely to a 6-functor formalism $D$ on $(\mathcal C,E)$ such that
	\begin{enumerate}

		\item $D^*$ and $D^!$ have cdh descent,

		\item every $\mathcal C$-smooth map is $D$-suave and every $\mathcal C$-prim map is $D$-prim,

		\item \label{itm:mot morphism} for any morphism $\phi^* : D^* \to D'^*$ of \emph{constructible pullback formalisms}, $\phi^*$ extends uniquely to a morphism $\phi : D \to D'$ such that, in addition to the given equivalences
			\[
				\begin{array}{cl}
					f^* \phi \simeq \phi f^* & \text{for any map $f$,} \\
					f_! \phi \simeq \phi f_! & \text{for any map $f \in E$, and} \\
					f_\sharp \phi \simeq \phi f_\sharp & \text{for any $\mathcal C$-smooth map $f$},
				\end{array}
			\]
			we also have that
			\[
				\begin{array}{cl}
					\phi f_* \to f_* \phi & \text{is an equivalence for any $\mathcal C$-proper map $f$, and} \\
					\phi f^! \to f^! \phi & \text{is an equivalence for any $\mathcal C$-smooth map $f \in E$},
				\end{array}
			\]
			and

		\item the unique morphism $\SH \to D^*$ of motivic pullback formalisms extends uniquely to a morphism $\SH \to D$ of 6-functor formalisms satisfying the properties as above.

	\end{enumerate}
	
\end{rmk}

Finally, as an application of \Cref{thmX:univ SH} in the algebraic case (\Cref{thm:nice mot 6FF}), we are able to produce a stacky analytic realization in \Cref{thm:analytification}, and easily deduce a stacky Betti realization in \Cref{rmk:Betti}. Let us now present an equivariant version of these results, so let $G$ be a constant finite group, and let $\AlgSp_\comps^G$ be the category of finite type quasi-separated algebraic spaces over $\Spec \comps$ with $G$-action. By \Cref{rmk:global quot of ft algsp}, we have functors $\AlgSp_\comps^G \to \mathcal C^\alg$ and $\AlgSp_\comps^G \to \mathcal C^\hol$ given by sending $X \in \AlgSp_\comps^G$ to $[X/G]$ and $[X^\an_\red/G]$, where $X^\an_\red$ is the analytification of the underlying reduced algebraic space of $X$. This lets us view $\SH^\alg$ and $\SH^\hol$ as presheaves $\SH^\alg_G$ and $\SH^\hol_G$ on $\AlgSp_\comps^G$, and there is a canonical \textbf{analytic realization} map $\alpha : \SH^\alg_G \to \SH^\hol_G$.

Furthermore, there is a presheaf $\SH^\Betti_{G,\et}$ on $\AlgSp_\comps^G$ such that for any $X \in \AlgSp_\comps^G$,
$\SH^\Betti_{G,\et}(X)$ is naturally identified with the category of $G$-equivariant sheaves of spectra on $X^\an$:
\[
	\SH^\Betti_{G,\et}(X) \simeq \Shv_{\Sp}(X^\an)^G
.\]

In fact, we have a canonical \textbf{Betti realization} map $\beta_\et : \SH^\hol_G \to \SH^\Betti_{G,\et}$. By using \Cref{rmk:global quot of ft algsp} to restrict along $\AlgSp_\comps^G \to \mathcal C^\alg$, \Cref{thm:analytification,rmk:Betti} give the following result:
\begin{thmX}[Realizations] \label{thmX:realization}
	The transformations
	\[
		\SH^\alg_G \xrightarrow{\alpha} \SH^\hol_G \xrightarrow{\beta_\et} \SH^\Betti_{G,\et}
	\]
	extend to morphisms of 6-functor formalisms on $(\AlgSp_\comps^G, \{\text{all maps}\})$ such that for any map $f$ in $\AlgSp_\comps^G$, and $\phi \in \{\alpha,\beta_\et\}$, in addition to the canonical equivalences
	\[
		f^* \phi \simeq \phi f^* \quad\text{and}\quad f_! \phi \simeq \phi f_!
	,\]
	we also have the following natural equivalences:
	\[
		\begin{array}{cl}
			f_\sharp \phi \simeq \phi f_\sharp \quad\text{and}\quad \phi f^! \simeq f^! \phi & \text{if $f$ is smooth,} \\
			\phi f_* \simeq f_* \phi \quad\text{and}\quad f^\flat \phi \simeq \phi f^\flat & \text{if $f$ is proper,}
		\end{array}
	\]
	where $f_\sharp \dashv f^*$ if $f$ is smooth, and $f^\flat \dashv f_!$ if $f$ is proper.

	Furthermore, for $D \in \{\SH^\alg_G, \SH^\hol_G, \SH^\Betti_{G,\et}\}$, $D^*$ and $D^!$ have cdh descent, every smooth map is $D$-suave, and every proper map is $D$-prim.
\end{thmX}

\subsection{Comparison with previous work on this topic}

As mentioned previously, the work of \cite{Mann6FF, LZ6FF} gives abstract categorical conditions for constructing 6-functor formalisms. In fact, given a geometric setup $(\mathcal C,E)$, the work of \cite{CLL6FF, 6FFUnique} even compares categories of 6-functor formalisms on $(\mathcal C,E)$ to certain categories of presheaves on $\mathcal C$. On the other hand, the work of \cite{voe-crit, Ayoub6I, Ayoub6II, tri-cat-mixed-motives, SixAlgSp, SixAlgSt} gives more geometric criteria for constructing 6-functor formalisms and morphisms between them in the setting of schemes. In light of the aforementioned more recent works on abstract 6-functor formalisms, the latter works can be thought of as producing the hypotheses necessary to apply to the abstract theory of 6-functor formalisms (although these works also prove other results). The main general result of the present work, given in \Cref{thmX:PPF}, can be seen as giving an abstract setting in which the ``strongest possible'' result of this form can be proven -- it shows as much as possible without using the exceptional adjunctions. \Cref{thmX:V6FF}, or \Cref{thmX:6FF,thmX:6FF morphism}, then apply this to produce results about 6-functor formalisms that generalize and enhance some of the results of \cite{voe-crit, Ayoub6I, Ayoub6II, tri-cat-mixed-motives, SixAlgSp, SixAlgSt} to do with producing 6-functor formalisms and morphisms between them.

To the author's knowledge, there are no results similar to \Cref{thmX:PPF,thmX:V6FF} currently appearing in the literature. \Cref{thmX:V6FF} is easily given by combining \Cref{thmX:PPF} with \Cref{thmX:V6FF PPF}, and the latter result does have a strong relationship with pre-existing results, most notably the results of \cite{6FFUnique, CLL6FF} which are actually used in its proof. \Cref{thmX:V6FF PPF} enhances the relevant results of \cite{6FFUnique, CLL6FF} by also showing cdh descent, establishing the exceptional adjunctions for a much larger class of maps (instead of just the compactifiable ones), and showing the suaveness and primness of many maps (but it requires stronger hypotheses). The strategy for proving \Cref{thmX:V6FF PPF} is explained in \Cref{S:outline}, and relies heavily on \cite[Theorem B]{CLL6FF} and \cite[Theorem 3.3]{6FFUnique}, as well as the arguments for \cite[Lemma A.5.11 and Proposition A.5.16]{Mann6FF} and \cite[Proposition 3.4.8]{HM6FF}.

One finds many more parallels in the literature when considering our applications from \Cref{S:intro apps} to the setting of algebraic geometry.

To the author's knowledge, the most general result about giving ``geometric'' criteria for producing algebro-geometric 6-functor formalisms currently appearing in the literature is given in \cite[Theorem 7.1]{SixAlgSt}. When combined with \cite[Corollary 7.15, Theorem 6.1, Theorem 7.10, Remark 7.11]{SixAlgSt}, this shows that if $D$ is a motivic pullback formalisms,\footnote{In the language of \cite{SixAlgSt}, this means that $D$ is a $(*,\sharp,\otimes)$-formalisms that satisfy Voevodsky's axioms} then $D$ extends to 6-functor formalism on $(\mathcal C^\alg, \{\text{finite type representable maps}\})$ such that $D^*$ and $D^!$ have cdh descent, the proper representable maps behave like $D$-prim maps, and many smooth representable maps are $D$-suave. In fact, \cite[Theorem 6.1]{SixAlgSt} can be seen as establishing \Cref{thm:intro strongly PPF is PPF,thm:intro cdh PPF} for $D = \SH$. These results generalize and enhance previous results from \cite{Ayoub6I, Ayoub6II, tri-cat-mixed-motives, sixopsequiv, SixAlgSp}, and recover part of \Cref{thmX:univ SH} (see \Cref{rmk:explain univ SH}), but only for a particular $\infty$-category of algebraic stacks, and without establishing any results about morphisms (or categories) of 6-functor formalisms.

To the author's knowledge, the most general results giving criteria for the compatibility of morphisms with the six operations are given in \cite[Proposition 2.3.11, Proposition 2.4.53, and Theorem 4.4.25]{tri-cat-mixed-motives}, but these are only given in the context of categories of schemes. We give more general results in \Cref{rmk:explain univ SH}(\ref{itm:mot morphism}), \Cref{thm:intro PPF morphism}, and \Cref{thmX:6FF morphism} (as well as the categorical versions given in \Cref{thmX:univ SH,thmX:PPF,thmX:V6FF}), although we note that in the case of schemes of finite type over a field, \cite[Theorem 4.4.25]{tri-cat-mixed-motives} establishes slightly stronger compatibilities than those given by our more general results. One enhancement given by our results is the additional compatibilities with operations of the form $f^!$ for $\mathcal C$-smooth $f$. Another important enhancement is the fact that we do not need to assume the codomain of our transformation satisfies the same criteria as the domain: it suffices for it to be a constructible pullback formalism, and for the transformation to commute with $f_\sharp$ for $\mathcal C$-smooth $f$. 

For the remainder of this section, it will be convenient to introduce the following notation whenever we have a notion of $\mathcal C$-smooth, $\mathcal C$-proper, and $\mathcal C$-closed maps, and a notion of motivic pullback formalisms on $\mathcal C$: we write $\VsixFF^\mot(\mathcal C,E)$ to denote the full subcategory of $\VsixFF(\mathcal C,E)$ consisting of those $D$ such that the associated presheaf $D^*$ is a motivic pullback formalism.

In the setting of a suitable category of schemes $\Sch$, if we define the $\mathcal C$-smooth maps to be the open immersions, and the $\mathcal C$-proper maps to be the proper maps, and the $\mathcal C$-closed maps to be the equivalences, the result about abstract 6-functor formalisms given in \cite[Theorem 3.3]{6FFUnique} already gives the identification
\[
	\VsixFF(\Sch, \{\text{finite type separated maps}\}) \xrightarrow{\sim} \PPF(\Sch)
.\]
If we instead take the $\mathcal C$-smooth maps to be the smooth morphisms, and the $\mathcal C$-closed maps to be the closed immersions, \cite[Corollary 3.6]{6FFUnique} combines \cite[Theorem 3.3]{6FFUnique} with the results of \cite{UnivFF} and \cite{tri-cat-mixed-motives} (by way of \cite{MotivicHodge}) to show that $\SH$ defines an initial object of $\VsixFF^\mot(\Sch, \{\text{finite type separated maps}\})$.

The most obvious enhancement of \cite[Corollary 3.6]{6FFUnique} afforded by our applications is that they hold for many $\infty$-categories of algebraic stacks, and possibly non-separated maps, since we actually show that $\SH$ is an initial object of $\VsixFF^\mot(\text{reasonable algebraic stacks}, \{\text{finite type representable maps}\})$. Once again, another enhancement is given by the fact that, as in \Cref{rmk:explain univ SH}, we do not need to fully construct a 6-functor formalism in order to get a morphism of 6-functor formalisms from $\SH$: if $D^* \in \PF^\cstr_\bullet(\mathcal C)$, then in order to extend $D^*$ to a 6-functor formalism along with a morphism of 6-functor formalisms $\SH \to D$, we can either show that $D^*$ is a motivic pullback formalism, or produce a morphism of constructible pullback formalisms $\SH \to D^*$.

Our results actually show that $\VsixFF^\mot(\mathcal C,E)$ is a ``\emph{full cosieve}'' of $\PF^\cstr_\bullet(\mathcal C)$: the functor $D \mapsto D^*$ defines a fully faithful functor $\VsixFF^\mot(\mathcal C,E) \to \PF^\cstr_\bullet(\mathcal C)$, and for any $D \in \VsixFF^\mot(\mathcal C,E)$, and map $D^* \to D'^*$ in $\PF^\cstr_\bullet(\mathcal C)$, we have that $D'^*$ is in the essential image of this functor, and the map lifts (uniquely) to a map $D \to D'$ in $\VsixFF^\mot(\mathcal C,E)$.

Finally, to the author's knowledge, results similar to \Cref{thmX:realization} that factorize Betti realization are usually given in terms of Hodge realizations, as in \cite{MotivicHodge,TubachHodge,HodgeStacks}. Our result is interesting because, unlike the case of Hodge realization, our analytic realization is fine enough that it does not factor through \'{e}tale sheafification, and in particular, it cannot be obtained simply by extending from the case of schemes using \'{e}tale descent. It will be interesting to compare our analytic realization with Hodge realization in future works, and describe a Hodge realization map from $\SH^\hol$. We also hope to refine our Betti realization to a ``genuine'' version in future work, so that we get a functor $\SH^\hol_G(\pt) \to \Sp_G$, where $\Sp_G$ is the category of \emph{genuine} $G$-spectra, instead of $\SH^\hol_G(\pt) \to \SH^\Betti_{G,\et}(\pt) \simeq \Sp^G$, where $\Sp^G$ is simply the category of spectra equipped with $G$-actions.

\subsection{Outline} \label{S:outline}

After going over some preliminary notions to do with pullback formalisms in \Cref{S:prelims}, we introduce a notion of cohomological properness of maps in \Cref{S:prop}, and study some of its closure properties. This notion of cohomological properness corresponds to the properties enjoyed by $\mathcal C$-proper maps in \Cref{thm:intro strongly PPF is PPF}, and will be essential to our study of 6-functor formalisms.

In \Cref{S:gluing,S:duality}, we study the notions of gluing/localization (as in \Cref{defn:intro cstr PF}(\ref{itm:intro/loc})) and duality/ambidexterity (as in \Cref{defn:intro ambidext}) for pullback formalisms, which can be seen as identifying two important types of proper maps: gluing/localization corresponds to closed immersions, and duality/ambidexterity corresponds to projective bundles (or smooth proper maps more generally). ``Closed immersions'' allow us to consider notions of excision, and indeed, we establish some excision and descent properties in \Cref{S:excision} -- see \Cref{prp:elementary cdh descent,prp:gluing implies exc}.

The key results from \Cref{S:gluing,S:duality} are that the maps behaving like closed immersions and projective bundles automatically have the cohomological properness studied in \Cref{S:prop}, and also guarantee automatic compatibility with morphisms of pullback formalisms. See \Cref{thm:closed,thm:duality}.

In \Cref{S:6FF}, we consider the abstract geometric setting of \Cref{S:voe} in which we are given an $\infty$-category $\mathcal C$ along with collections of $\mathcal C$-smooth maps, $\mathcal C$-proper maps, and $\mathcal C$-closed maps. We introduce the notion of strongly projective pullback formalisms in \Cref{defn:proj sat}, and give the proof of \Cref{thmX:PPF} in \Cref{thm:PPF}, which is shown by combining the results from \Cref{S:gluing,S:duality} about ``closed immersions'' and ``projective bundles'' with the closure properties of cohomologically proper maps shown in \Cref{S:prop}.

In \Cref{S:V6FF}, we introduce the category $\VsixFF(\mathcal C,E)$ of Voevodsky-6-functor formalisms, and show how this relates to projective pullback formalisms. In particular, we establish \Cref{thmX:V6FF PPF} using the following strategy:
\begin{enumerate}

	\item we use \cite[Theorem B]{CLL6FF} and \cite[Theorem 3.3]{6FFUnique} to obtain 6-functor formalisms,

	\item we use \Cref{lem:PPF cdh} (\Cref{thm:intro cdh PPF}) and \Cref{lem:desc V6FF} to get cdh descent for our 6-functor formalisms,

	\item we extend our 6-functor formalisms using cdh descent and results from \Cref{S:6FF ext} that are refined versions of \cite[Lemma A.5.11 and Proposition A.5.16]{Mann6FF} and \cite[Proposition 3.4.8]{HM6FF}, and

	\item we use \Cref{lem:suave prim source-local,lem:suave prim target-local}, and \Cref{prp:nat sq adj 6FF}(\ref{itm:nat sq adj/suave prim}), to show that many maps are suave and prim.

\end{enumerate}
The results of \Cref{S:6FF} can be seen as giving ``geometric'' criteria for extending pullback formalisms to 6-functor formalisms, and for extending morphisms of pullback formalisms to morphisms of 6-functor formalisms -- recall \Cref{thmX:6FF,thmX:6FF morphism}, which are given by \Cref{thm:6FF}, in turn proven using \Cref{thmX:V6FF PPF,thmX:PPF}.

In \Cref{S:applications}, we apply the results of \Cref{S:6FF} to the study of motivic homotopy theory, and prove the results presented in \Cref{S:intro apps}. For this, we fix a more specialized geometric context $\mathcal C$ where we can formulate the Thom stability and homotopy invariance properties of constructible pullback formalisms, leading to the notion of motivic pullback formalisms on $\mathcal C$. We then establish the key result given by \Cref{thmX:mot PPF}, which shows that every motivic pullback formalism is strongly projective, and that the presheaf $\SH$ of stable motivic homotopy theory is the initial motivic pullback formalism. The strategy for showing this result is as follows:
\begin{enumerate}

	\item The fact that $\SH$ is the initial motivic pullback formalism is relatively easy to show by its construction, and using results already appearing in the literature (such as \cite[Theorems 5.1.11 and 5.3.10]{Fundamentals} and \cite[Proposition 5.13]{SixAlgSt}). In the complex analytic case, we use \cite{Gluing} to get that $\SH$ is a constructible pullback formalism (which is necessary to show that it is a motivic pullback formalism).

	\item Therefore, using \Cref{thmX:PPF}, it suffices to show that $\SH$ is a strongly projective pullback formalism, as this will imply that every motivic pullback formalism is also strongly projective, and that morphisms between motivic pullback formalisms are morphisms of projective pullback formalisms.

	\item In the algebraic case, to show that $\SH$ is a strongly projective pullback formalism, we use the fact that, as in \Cref{rmk:gen prop in AG}, the $\mathcal C$-proper maps are ``generated'' by closed immersions and projective bundles. Since closed immersions are the $\mathcal C$-closed maps, it suffices to show that projective bundles are stably $\SH$-ambidextrous in the sense given by \Cref{defn:duality}. The same reduction holds in the complex analytic case by our choice of $\mathcal C$-proper maps.

	\item The fact that projective bundles $\proj(V) \to S$ are stably $\SH$-ambidextrous in the algebraic case is given by \cite[Theorem 6.9]{sixopsequiv} or \cite[Lemma 6.9]{SixAlgSt}. In the complex analytic case, this is shown in \Cref{lem:hol ambidext} using the following strategy. By trivialize $V$, we see that locally on $S$, $\proj(V) \to S$ is actually a base change of $\proj^n \to \pt$, and this is the analytification of $\proj^n_\comps \to \Spec \comps$. Using the fact that there is a morphism $\SH^\alg \to \SH^\hol((-)^\an)$, we use \Cref{thm:duality} to deduce that $\proj(V) \to S$ is stably $\SH^\hol$-ambidextrous from the fact that $\proj^n_\comps \to \Spec \comps$ is stably $\SH^\alg$-ambidextrous.

\end{enumerate}

Once we have established \Cref{thmX:mot PPF}, it is relatively easy to combine it with \Cref{thmX:V6FF PPF} to obtain \Cref{thmX:univ SH}, but in the algebraic case we need to make some additional arguments to show that we can choose our geometric setup $(\mathcal C^\alg, E)$ to be such that $E$ consists of the finite type representable maps. We do this by showing that locally on the source and target, finite type representable maps reduce to quasi-projective maps (relying on arguments from \cite{SixAlgSt}).

This argument uses results about quasi-projective maps which are only known to us in the case of classical algebraic stacks, so we make a separate argument (\Cref{lem:reduce to clcl}) allowing us to reduce the case of derived algebraic stacks to classical algebraic stacks. We use this reduction to show results about 6-functor formalisms for derived algebraic stacks in \Cref{thm:nice mot 6FF,thm:alg mot 6FF}. Alternatively, in the case of derived Deligne-Mumford stacks, \cite{qDM_compactification} shows that we have a good theory of compactifications for finite type separated representable maps, so are able to easily prove \Cref{thm:qDM mot 6FF} using \Cref{thmX:mot PPF,thmX:V6FF PPF}.

In \Cref{S:realization}, we apply the algebraic version of \Cref{thmX:univ SH} (\Cref{thm:nice mot 6FF}) in order to show \Cref{thmX:realization}. 

Finally, we mention that some new general results about 6-functor formalisms are shown in \Cref{S:app 6FF}. Perhaps the most interesting ones are those given in \Cref{S:suave prim} for suave and prim maps, most of which are conveniently summarized in \Cref{thm:suave prim consequences}. In \Cref{S:6FF ext}, we also give categorical refinements of some of the extension results given in \cite[\S A.5]{Mann6FF} and \cite[Proposition 3.4.8]{HM6FF}.

\subsection{Acknowledgements}

The author would like to thank Andrew Blumberg and Johan de Jong for their support as PhD advisors, Elden Elmanto for his encouragement and advice, and Bastiaan Cnossen for his comments and suggestions. 

\subsection{Notations and Conventions}

Throughout this article, we will make heavy use of the machinery of $\infty$-categories as developed in \cite{htt} and \cite{ha}. Therefore, all of our language will be implicitly $\infty$-categorical:
\begin{enumerate}

	\item We say ``category'' to mean ``$\infty$-category''. Note that then functors, adjoints, and (co)limits must all be understood in the context of $\infty$-categories.

	\item Following \cite[Remark 3.0.0.5]{htt}, we will write $\Cat$ to denote the category of small categories, and $\widehat{\Cat}$ to denote the category of all categories.

	\item We write $\spaces$ for the category of small spaces/$\infty$-groupoids/anima (see \cite[\S1.2.16]{htt}), and $\Sp$ for the category of spectra (see \cite[\S1.4.3]{ha}). Write $\widehat{\spaces}$ for the category of all spaces (not necessarily small).

	\item Unless otherwise specified, presheaves and sheaves are always implicitly assumed to take values in $\spaces$. Given a category $\mathcal C$, we write $\Psh(\mathcal C)$ to denote that category of presheaves on $\mathcal C$, and if $\mathcal C$ is equipped with a Grothendieck topology that is understood from context, we write $\Shv(\mathcal C)$ to denote the category of sheaves on $\mathcal C$.

	\item Given a category $\mathcal C$, we write $\mathcal C(-,-)$ for the hom functor $\mathcal C^\op \times \mathcal C \to \widehat{\spaces}$. $\mathcal C$ is locally small if this functor takes values in $\spaces$.

	\item The very large categories $\PrL, \PrR$ of presentable categories are defined in \cite[Definition 5.5.3.1]{htt}. These are the categories of presentable categories and left adjoint functors or right adjoint functors respectively. Note that $\PrL$ is equipped with the structure of a symmetric monoidal category as in \cite[Proposition 4.8.1.15]{ha}.

	\item For any symmetric monoidal category $\mathcal C$, we write $\CAlg(\mathcal C)$ for the category of commutative algebra objects in $\mathcal C$ -- see \cite[Definition 2.1.3.1]{ha}. In particular, $\CAlg(\Cat)$ is the category of symmetric monoidal categories, and $\CAlg(\PrL)$ is the category of symmetric monoidal presentable categories where the monoidal product preserves small colimits in each variable.

	\item A ``zero object'' of a category $\mathcal C$ is an object that is both initial and terminal. We will say that a category is pointed if it has a zero object. See \cite[Definition 1.1.1.1]{ha}.

\end{enumerate}

As this article is a sequel to \cite{Fundamentals}, we will take for granted the basic notions of pullback formalisms set down in \cite[\S1.2]{Fundamentals}, although we will review the necessary language in \Cref{S:prelims}. \cite[\S1.1,1.4,2.3]{Fundamentals} may also be helpful.

We will also use the following notations and conventions:
\begin{enumerate}

	\item All quotients by group actions are assumed to be stacky. Often these quotients are denoted by $[X/G]$, but we will instead simply write $X/G$.

	\item The abbreviation ``qcqs'' will be used to mean ``quasi-compact and quasi-separated'' in any context where these adjectives make sense.

	\item Whenever we say "limit-preserving" or "colimit-preserving", we are referring only to \emph{small} limits and colimits.

	\item If $f : X \to Y$ and $g : X' \to Y$ are maps in a category $\mathcal C$, and the fibred product $X \times_Y X'$ exists, we will sometimes write $f^{-1}(g) : X \times_Y X' \to X$ for the base change of $g$ along $f$ in $\mathcal C$.

	\item Given some implicitly specified ambient category, we will write $\pt$ for a terminal object of that category.

	\item Given some implicitly specified ambient monoidal category, we will write $1$ for a monoidal unit of that category.

	\item Given a locally small category $\mathcal C$, we will write $\yo : \mathcal C \to \Psh(\mathcal C)$ for the Yoneda embedding of $\mathcal C$. We generally do not include $\mathcal C$ in the notation as it is often clear from context which category's Yoneda embedding we are considering.

	\item Following \cite[\S1.2.8]{htt}, we will use the symbol $\star$ to denote the join of simplicial sets.

	\item Following \cite[Notation 1.2.8.4]{htt}, for any simplicial set $K$, we write $K^\triangleleft$ and $K^\triangleright$ for the simplicial sets obtained by adjoining an initial or terminal cone point respectively to $K$. We will also write $-\infty, \infty$ to denote the cone points of $K^\triangleleft, K^\triangleright$ respectively, so we can write $K^\triangleleft = \{-\infty\} \star K$ and $K^\triangleright = K \star \{\infty\}$.

\end{enumerate}

\section{Preliminaries} \label{S:prelims}

In this section, we will introduce some notions that will be fundamental to the rest of the article. Many of these are recalled from \cite{Fundamentals}.

\subsection{General presheaves of categories and adjointability}

\begin{nota} \label{nota:psh}
	Let $\mathcal C$ be a category, and let $D : \mathcal C^\op \to \widehat{\Cat}$ be a presheaf.
	\begin{itemize}

		\item If $f : X \to Y$ is a map in $\mathcal C$, and $D$ is clear from context, we will often write $f^* \coloneqq D(f) : D(Y) \to D(X)$.

		\item If $F : \mathcal C' \to \mathcal C$ is a functor, we will often write $F^* D \coloneqq D \circ F^\op : (\mathcal C')^\op \to \widehat{\Cat}$.

		\item
			If $D$ is actually a presheaf taking values in $E_0$-algebras, then for $S \in \mathcal C$, $X \in \mathcal C_{/S}$, and $M \in D(S)$, we write $D(X;M) \coloneqq D(X)(1, (X \to S)^* M)$, where $1$ is the monoidal unit of $D(X)$, is the structure map of $X \in \mathcal C_{/S}$. This defines a functor $D(-;-) : \mathcal C_{/S}^\op \times D(S) \to \widehat{\spaces}$. More generally, if $\xi \in D(X)$, we also define the ``\emph{$\xi$-twisted cohomology space of $X$ with coefficients in $M$}''.
			\[
				D(X;M)[\xi] \coloneqq D(X)(1, \xi \otimes (X \to S)^* M)
			.\]
	\end{itemize}
	
\end{nota}

\begin{defn}
	\cf{} \cite[\S2.3]{Fundamentals}.

	If $\mathcal C$ is a category, and $D : \mathcal C^\op \to \widehat{\Cat}$ is a presheaf, then we say that a map $f$ in $\mathcal C$ is \emph{$D$-acyclic} if $f^* = D(f)$ is fully faithful. More generally, a diagram $X : K^\triangleright \to \mathcal C$ is $D$-acyclic if
	\[
		D(X(\infty)) \to \varprojlim_{a \in K} D(X(a))
	\]
	is fully faithful.
\end{defn}

We will now recall the notions of adjointable squares and mates from \cite[Remark 4.5]{laxtrans} and \cite[\S2.1]{CLL6FF}. Note that \cite[\S F]{TwAmb} and the material following \cite[Definition 4.7.4.13]{ha} are also useful references on adjointable squares.
\begin{defn}
	Let $\mathscr V$ be a 2-category, and let
	\[
		\begin{tikzcd}
			B \ar[d, "b"'] \ar[r, "r"] & \ar[dl, Rightarrow, "\phi"] A \ar[d, "a"] \\
			B' \ar[r, "r'"'] & A'
		\end{tikzcd}
	\]
	be a lax square, so $\phi$ is a 2-morphism $a \circ r \to r' \circ b$. If $r$ and $r'$ have left adjoint $l$ and $l'$, then the \emph{left mate} of this square is the 2-morphism $\psi : l' a \to b l$ given by the following composite:
	\[
		l' a \xrightarrow{l' a (\id \to rl)} l' a r l \xrightarrow{l' \phi l} l' r' b l \xrightarrow{(l' r' \to \id)bl} b l
	.\]
	We refer to the resulting colax square
	\[
		\begin{tikzcd}
			A \ar[d, "a"'] \ar[r, "l"] & B \ar[d, "b"] \\
			A' \ar[ur, Rightarrow, "\psi"]\ar[r, "l'"'] & B'
		\end{tikzcd}
	\]
	as the colax \emph{left mate square}.

	There are evident analogous definitions of right mates for colax squares. We can also define the right mate of a lax square by taking the right mate of the colax square given by the transpose, and similarly the left mate of a colax square.

	Now, when the left mate of a lax square is invertible, we may view the colax square given by the left mate as a lax square. In this case, the resulting right mate is called the \emph{left-right mate} of the original lax square. There is an analogous definition for \emph{right-left mates} given by colax squares.

	When the left (\resp{} right) mate is an equivalence, we say the square is \emph{left (\resp{} right) adjointable}. Similarly, when the left-right (\resp{} right-left) mate is an equivalence, we say that the square is \emph{left-right (\resp{} right-left) adjointable}.

	When considering commutative squares, it may not be clear if we are viewing the square as a lax square or a colax square, so we instead refer to \emph{horizontal or vertical left, right, right-left, or left-right, mates} and \emph{horizontally or vertically left, right, right-left, or left-right, adjointable squares}.
\end{defn}

Now we recall some notions from \cite[\S D]{Fundamentals} and introduce the notion of right-left base change.
\begin{defn}
	Let $\mathcal C$ be a category, let $f : X \to Y$ be a map in $\mathcal C$.
	\begin{enumerate}

		\item If $\mathscr V$ is a 2-category, and $\phi : D \to D'$ is a transformation of presheaves $D,D' : \mathcal C^\op \to \mathscr V$, Say $\phi$ \emph{is left (right) adjointable at} $f$ if the square
			\[
				\begin{tikzcd}
					D(Y) \ar[d] \ar[r] & D(X) \ar[d] \\
					D'(Y) \ar[r] & D'(X)
				\end{tikzcd}
			\]
			is horizontally left (right) adjointable. \cf{} \cite[\S D.1]{Fundamentals}.

		\item If $D : \mathcal C^\op \to \CAlg(\widehat{\Cat})$ is a $\CAlg(\widehat{\Cat})$-valued presheaf, then say $D$ \emph{has the left (right) projection formula for} $f$ if for all $a \in D(Y)$, the square
			\[
				\begin{tikzcd}
					D(Y) \ar[d, "\otimes a"'] \ar[r] & D(X) \ar[d, "\otimes f^* a"] \\
					D(Y) \ar[r] & D(X)
				\end{tikzcd}
			\]
			is left (right) adjointable. Alternatively, we say that $f^*$ \emph{has a linear left (right) adjoint}. \cf{} \cite[\S D.2]{Fundamentals}

		\item Given a 2-category $\mathscr V$, say $D : \mathcal C^\op \to \mathscr V$ \emph{has left (right) base change for $f$ against a map $Y' \to Y$} if the pullback $X \times_Y Y'$ exists, and the square
			\[
				\begin{tikzcd}
					D(Y) \ar[d] \ar[r] & D(X) \ar[d] \\
					D(Y') \ar[r] & D(X \times_Y Y')
				\end{tikzcd}
			\]
			is horizontally left (right) adjointable. We say $D$ \emph{has left (right) base change for $f$} if it has left (right) base change for $f$ against all maps to $Y$. \cf{} \cite[\S D.3]{Fundamentals}.

			In this case, we also say $D$ \emph{has left (right) exchange or left-right (right-left) base change for $f$ against $Y' \to Y$} if the horizontal left (right) mate
			\[
				\begin{tikzcd}
					D(X) \ar[d] \ar[r] & D(Y) \ar[d] \\
					D(X \times_Y Y') \ar[r] & D(Y')
				\end{tikzcd}
			\]
			is vertically right (left) adjointable.

	\end{enumerate}
	
\end{defn}

\begin{thm} \label{thm:bc descent}
	Let $\Cat^L$ be the subcategory of $\widehat{\Cat}$ consisting of categories that admit all small colimits, and functors between them that admit right adjoints. Let $\mathcal C$ be a locally small category, let $D : \mathcal C^\op \to \Cat^L$ be a presheaf, and let $Q$ be a collection of maps in $\mathcal C$ that is stable under base change, and such that $D$ has left base change for maps in $Q$.

	Note that we can view $D$ as a limit-preserving presheaf $\Psh(\mathcal C)^\op \to \widehat{\Cat}$ (whose restriction to $\mathcal C$ lands in $\Cat^L$). Let $Y \in \mathcal C$, and let $\mathcal U \to \yo(Y)$ be a sieve generated by a small family of maps $\{X_i \to Y\}_i$ in $Q$. Then
	\begin{enumerate}

		\item The functor $D(Y) \to D(\mathcal U)$ admits a fully faithful left adjoint.

		\item For any morphism $\phi : D' \to D$ in $\Fun(\mathcal C^\op, \Cat^L)$, if $\phi$ is left adjointable at every map in $Q$, then $\phi$ is left adjointable at $\mathcal U \to \yo(Y)$. If $D'(\mathcal U) \to D'(X_i)$ admits a left adjoint, then $\phi$ is also left adjointable at $\yo(X_i) \to \mathcal U$.

		\item $D$ has left base change for the maps $\mathcal U \to \yo(Y)$, and $\{\yo(X_i) \to \mathcal U\}_i$.

	\end{enumerate}
	\begin{proof}
		This follows from \cite[Theorem 2.4.1]{Fundamentals}, where the quasi-admissibility structure is given by $Q$. In fact, although $Q$ is not assumed to be a quasi-admissibility structure, the argument still holds. Alternatively, we can replace $Q$ with the collection of morphisms that are composites of equivalences and maps in $Q$ in order to obtain a quasi-admissibility structure.
	\end{proof}
\end{thm}

The following notion will be useful:
\begin{defn} \label{defn:pseudocover}
	Given a category $\mathcal C$, and a presheaf $D : \mathcal C^\op \to \widehat{\Cat}$, a \emph{(small) $D$-pseudocover of} $S \in \mathcal C$ is a (small) family of maps $\{X_i \to S\}_i$ such that the functors $\{D(S) \to D(X_i)\}_i$ are jointly conservative.
\end{defn}

\begin{lem}[Locality on the target for adjointability] \label{lem:target-locality of radj}
	Let $\mathcal C$ be a category, and let $\phi : D \to D'$ be a transformation of presheaves $D,D' : \mathcal C^\op \to \widehat{\Cat}$.
	Let $f : X \to Y$ be a map in $\mathcal C$, and let $\{Y_i \to Y\}_i$ be a $D'$-pseudocover of $Y$ such that for each $i$, $D$ and $D'$ have right base change for $f$ against $Y_i \to Y$, and $\phi$ is right adjointable at the base change $f_i : X_i \to Y_i$ of $f$ along $Y_i \to Y$. Then $\phi$ is right adjointable at $f$.
	\begin{proof}
		For any index $i$, the top and bottom squares of
		\[
			\begin{tikzcd}
				D(Y) \ar[d] \ar[r] & D(X) \ar[d] \\
				D(Y_i) \ar[d] \ar[r] & D(X_i) \ar[d] \\
				D'(Y_i) \ar[r] & D'(X_i)
			\end{tikzcd}
		\]
		are horizontally right adjointable. Therefore the outer rectangle is horizontally right adjointable by \cite[Lemma F.6(2)]{TwAmb}. This outer rectangle is equivalent to the outer rectangle of
		\[
			\begin{tikzcd}
				D(Y) \ar[d] \ar[r] & D(X) \ar[d] \\
				D'(Y) \ar[d] \ar[r] & D'(X) \ar[d] \\
				D'(Y_i) \ar[r] & D'(X_i)
			\end{tikzcd}
		,\]
		and we have assumed that the bottom square in this diagram is horizontally right adjointable.

		It follows from \cite[Lemma F.6(2)]{TwAmb} that the functor $D'(Y) \to D'(Y_i)$ sends the horizontal right mate of the top square to an equivalence, so since $\{Y_i \to Y\}_i$ is a $D'$-pseudocover, the top square must be horizontally right adjointable.
	\end{proof}
\end{lem}

\subsection{Pullback contexts and pullback formalisms}

Now we recall the definition of pullback contexts and pullback formalisms from \cite{Fundamentals}.
\begin{defn}[Pullback contexts]
	A \emph{pullback context} is a category $\mathcal C$ equipped with a \emph{quasi-admissibility structure}, which is a collection of maps that contains all equivalences, and is closed under base change and composition. We call the elements of this collection \emph{quasi-admissible maps}.

	For any $S \in \mathcal C$, we write $\mathcal C_S$ for the full subcategory of $\mathcal C_{/S}$ consisting of quasi-admissible maps to $S$.

	A \emph{morphism of pullback contexts} is a functor between pullback contexts that preserves quasi-admissible morphisms, and base changes along quasi-admissible morphisms.

	An \emph{anodyne morphism of pullback contexts} is a morphism of pullback contexts $F : \mathcal C \to \mathcal D$ such that for all $S \in \mathcal C$, the induced functor $\mathcal C_S \to \mathcal D_{/F(S)}$ is an equivalence.
\end{defn}

% \begin{defn}
% 	If $\mathcal C$ is a pullback context, then a sieve $\mathcal U$ on an object $S \in \mathcal C$ is said to be \emph{quasi-admissible} if it is generated by quasi-admissible maps to $S$. \cf{} \cite[Lemma 2.4.5]{Fundamentals}.
% \end{defn}

\begin{defn}[Pullback formalisms] \label{defn:PF}
	Given a pullback context $\mathcal C$, we make the following definitions:
	\begin{enumerate}

		\item Say a presheaf $D : \mathcal C^\op \to \widehat{\Cat}$ \emph{respects quasi-admissibility} if it sends quasi-admissible maps to right adjoint functors. In this case, for a quasi-admissible map $f$, we will often denote the left adjoint of $f^* = D(f)$ by $f_\sharp$.

			When $D$ actually takes values in monoidal categories, for any quasi-admissible map $X \to S$, we also write $\cls{X} \in D(S)$ for the object of $D(S)$ given by $(X \to S)_\sharp$ of the monoidal unit of $D(X)$.

		\item Say a presheaf $D : \mathcal C^\op \to \widehat{\Cat}$ \emph{has quasi-admissible base change} if it has left base change for all quasi-admissible maps. In this case, we say that $D$ has \emph{quasi-admissible exchange for a map $f$} if it has right-left base change for $f$ against every quasi-admissible map.

		\item Say a transformation $D \to D'$ between presheaves $D,D' : \mathcal C^\op \to \widehat{\Cat}$ \emph{respects quasi-admissibility} if it is left adjointable at every quasi-admissible map.

		\item Say a presheaf $D : \mathcal C^\op \to \CAlg(\widehat{\Cat})$ \emph{satisfies the quasi-admissible projection formula} if it has the left projection formula for every quasi-admissible map.

	\end{enumerate}
	The \emph{category of pullback formalisms on $\mathcal C$} is the subcategory $\PF(\mathcal C)$ of $\Psh_{\CAlg(\PrL)}(\mathcal C)$ consisting of presheaves that have quasi-admissible base change and satisfy the quasi-admissible projection formula, and transformations that respect quasi-admissibility.
\end{defn}

\begin{rmk}[Comparison with other notions of pullback formalisms]
	Given a pullback context $\mathcal C$, our definition of pullback formalisms and the category $\PF(\mathcal C)$ agrees with the notion of pullback formalisms considered in \cite[\S1.2]{Fundamentals}. This also coincides with the notion of \emph{presentable} pullback formalisms given in \cite[Definition 4.5]{UnivFF}, except that we do not assume that $\mathcal C$ admits all finite limits. We have chosen to assume presentability in order to simplify the terminology of some results, but we will often formulate results about presheaves that only satisfy some of the properties mentioned in \Cref{defn:PF}.
\end{rmk}

\begin{lem} \label{lem:qadm mono loc}
	Let $\mathcal C$ be a pullback context, and let $D : \mathcal C^\op \to \widehat{\Cat}$ be a presheaf with quasi-admissible base change. Then for any quasi-admissible monomorphism $j$ in $\mathcal C$, $j_\sharp$ is fully faithful.
	\begin{proof}
		Since $j$ is a monomorphism, the base change of $j$ along $j$ is an equivalence, so since $D$ has left base change for $j$ against $j$, we have that $j^* j_\sharp$ is an equivalence, whence $j_\sharp$ is fully faithful by \cite[Lemma 3.3.1]{CARMELI2021107763}.
	\end{proof}
\end{lem}

We recall the following important result \cite[Theorem 2.4.3]{Fundamentals} about descent for pullback formalisms:
\begin{thm} \label{thm:D-topology}
	Let $D$ be a pullback formalism on a pullback context $\mathcal C$, and let $\mathcal R$ be a $D$-pseudocover of an object $X \in \mathcal C$. If all maps in $\mathcal R$ are quasi-admissible, then $D$ has descent along any base change of $\mathcal R$, and the same is true of any pullback formalism that receives a morphism from $D$.
\end{thm}

% Finally, we make a remark about describing ``realization'' maps between pullback formalisms:
% % PERF: this should instead directly cite some remark from Fundamentals (that I still need to add)
% \begin{rmk} \label{rmk:realization}
% 	Let $F : \mathcal C \to \mathcal C'$ be a morphism of pullback contexts, let $D' \in \PF(\mathcal C')$, and let $\phi : D \to F^* D'$ be a morphism in $\PF(\mathcal C)$. Then for any $S \in \mathcal C'$, we have a commutative diagram of symmetric monoidal functors
% 	\[
% 		\begin{tikzcd}
% 			\mathcal C_S \ar[d] \ar[r, "F"] & \mathcal C'_{F(S)} \ar[d] \\
% 			D(S) \ar[r, "\phi_S"'] & D'(F(S))
% 		\end{tikzcd}
% 	,\]
% 	where the vertical arrows are given by $X \mapsto \cls{X}$. This is because, by \cite[Theorem 1.4.3]{Fundamentals}, we have initial objects $H^\slice_{\mathcal C} \in \PF(\mathcal C)$ and $H^\slice_{\mathcal C'} \in \PF(\mathcal C')$, where for any pullback context $\mathcal D$, $H^\slice_{\mathcal D}$ is the pullback formalism given by $S \mapsto \mathcal D_S$.
% \end{rmk}

\subsection{Reduction to morphisms} \label{S:red to morph}

Fix a pullback context $\mathcal C$. In this section we establish some strategies that allow us to relate properties of pullback formalisms on $\mathcal C$ to properties of \emph{morphisms} of pullback formalisms. Thus, the general pattern for many of our results will be to first prove results for morphisms of pullback formalisms, and then deduce corresponding statements for pullback formalisms.

Here is an important application of this strategy:
\begin{prp} \label{prp:target-locality of dual bc}
	Let $f : X \to Y$ be a map in a pullback context $\mathcal C$, let $\mathcal R = \{Y_i \to Y\}_i$ be a small family of quasi-admissible maps to $Y$, and for each $i$, write $f_i : X_i \to Y_i$ for the base change of $f$ along $Y_i \to Y$.

	Let $D : \mathcal C^\op \to \widehat{\Cat}$ be a presheaf that has quasi-admissible base change, and which sends every base change of $f$ to a left adjoint functor.

	Then we have the following:
	\begin{enumerate}

		\item Suppose that $\phi : D \to D'$ is a transformation of presheaves that have quasi-admissible base change, $\mathcal R$ is a $D'$-pseudocover, and for every $i$, $\phi$ is right adjointable at $f_i$. Then $\phi$ is right adjointable at $f$.

		\item Suppose that $D$ lifts to a presheaf $\mathcal C^\op \to \CAlg(\widehat{\Cat})$, and that $\mathcal R$ is a $D$-pseudocover. If $D$ has the right projection formula for $f_i$ for all $i$, then $D$ has the right projection formula for $f$.

		\item Let
			\[
				\begin{tikzcd}
					X' \ar[d] \ar[r, "f'"] & Y' \ar[d, "q"] \ar[d] \\
					X \ar[r, "f"'] & Y
				\end{tikzcd}
			\]
			be a Cartesian square in $\mathcal C$, for each $i$, let $q_i : Y'_i \to Y_i$ be the base change of $q$ along $Y_i \to Y$.

			\begin{enumerate}

				\item Suppose that the base change $\mathcal R'$ of $\mathcal R$ along $q$ is a $D$-pseudocover. If $D$ has right base change for $f_i$ against $q_i$ for all $i$, then $D$ has right base change for $f$ against $q$.

				\item Suppose that $\mathcal R$ is a $D$-pseudocover, that $q$ is quasi-admissible, and that $D$ has quasi-admissible base change. If $D$ has right-left base change for $f_i$ against $q_i$ for all $i$, then $D$ has right-left base change for $f$ against $q$.

			\end{enumerate}

	\end{enumerate}
	
\end{prp}

We will prove \Cref{prp:target-locality of dual bc} at the end of the section, using \Cref{exa:tensor morph,exa:base change morph,exa:sharp morph}.

\begin{exa} \label{exa:tensor morph}
	Let $Y \in \mathcal C$ be an object. Note that $\pi : \mathcal C_{/Y} \to \mathcal C$ is an anodyne morphism of pullback contexts.

	Let $D : \mathcal C^\op \to \CAlg(\widehat{\Cat})$ be a presheaf of symmetric monoidal categories. Note then that
	\[
		\pi^* D \coloneqq D \circ \pi^\op
	\]
	can be seen as a presheaf of $D(Y)$-modules
	\[
		\pi^* D : (\mathcal C_{/Y})^\op \to \Mod_{D(Y)} \widehat{\Cat}
	.\]

	For any $N \in D(Y)$, $\otimes N$ defines a transformation $\pi^* D \to \pi^* D$, and if $D$ satisfies the quasi-admissible projection formula, then this transformation respects quasi-admissibility.
\end{exa}

\begin{exa} \label{exa:base change morph}
	Let $y : Y' \to Y$ be a map in $\mathcal C$. Let $\mathcal C' \subseteq \mathcal C_{/Y}$ be the full subcategory of maps $X \to Y$ such that $X \times_Y Y'$ exists in $\mathcal C$. Then the map
	\[
		\pi : \mathcal C' \to \mathcal C_{/Y} \to \mathcal C
	\]
	is an anodyne morphism of pullback contexts.

	Furthermore, by the construction of $\mathcal C'$, base change along $y$ defines a functor $\mathcal C' \to \mathcal C_{/Y'}$, and the composite
	\[
		\pi' : \mathcal C' \to \mathcal C_{/Y'} \to \mathcal C
	\]
	is also a morphism of pullback contexts.

	There is a transformation $\pi' \to \pi$ given by $- \times_Y Y' \to -$. For any presheaf $D$ on $\mathcal C$, the map
	\[
		D \circ (\pi^\op \to (\pi')^\op)
	,\]
	which we will denote by $\phi : \pi^* D \to (\pi')^* D$, evaluates to
	\[
		D(W \times_Y Y' \to W) : D(W) \to D(W \times_Y Y')
	\]
	at any $W \in \mathcal C'$. If $D : \mathcal C^\op \to \widehat{\Cat}$ has quasi-admissible base change, it follows that $\pi^* D \to (\pi')^* D$ respects quasi-admissibility.
\end{exa}

\begin{exa} \label{exa:sharp morph}
	In \Cref{exa:base change morph}, assume $y$ is quasi-admissible. By viewing $\phi$ as a functor
	\[
		(\mathcal C')^\op \to \Fun(\Delta^1, \widehat{\Cat})
	,\]
	using the fact that $D$ has quasi-admissible base change, we find that this functor actually lands in the subcategory $\Fun^\LAd(\Delta^1, \widehat{\Cat})$, so by \cite[Corollary 4.7.4.18(3)]{ha} taking left adjoints defines a transformation $\psi : (\pi')^* D \to \pi^* D$.

	Note then that for $W \in \mathcal C'$, evaluating $\psi$ at $W$ gives the functor
	\[
		(W \times_Y Y' \to W)_\sharp : D(W \times_Y Y') \to D(W)
	,\]
	and we easily see that $\psi$ respects quasi-admissibility.
\end{exa}

% \begin{exa} \label{exa:star morph}
% 	In \Cref{exa:base change morph}, if $D$ has quasi-admissible base change, and sends all maps to functors admitting right adjoints, then if we write $\mathcal C''$ for the wide subcategory of $\mathcal C'$ containing only quasi-admissible morphisms, $\phi$ sends $\mathcal C''$ to $\Fun^\RAd(\Delta^1, \widehat{\Cat})$, so by \cite[Corollary 4.7.4.18(3)]{ha}, taking right adjoints gives us a transformation $\chi : (\pi')^* D|_{\mathcal C''} \to \pi^* D|_{\mathcal C''}$ such that evaluating $\chi$ at $W \in \mathcal C'$ is
% 	\[
% 		(W \times_Y Y' \to W)_* : D(W \times_Y Y') \to D(W)
% 	.\]
%
% 	If $D$ has quasi-admissible exchange for all base changes of $y : Y' \to Y$, then $\chi$ respects quasi-admissibility.
% \end{exa}

\begin{proof}[Proof of \Cref{prp:target-locality of dual bc}]
	First note that since $D$ has quasi-admissible base change, and sends every base change of $f$ to a left adjoint functor, we have that for any map $q : Y' \to Y$, if the base change $f' : X' \to Y'$ of $f$ along $q$ exists, then $D$ has right base change for $f'$ against $Y' \times_Y Y_i \to Y'$ for all $i$, since the latter map is quasi-admissible.
	\begin{enumerate}

		\item Since $D$ and $D'$ have quasi-admissible base change, they have right base change for $f$ against $Y_i \to Y$ for all $i$, so this follows immediately from \Cref{lem:target-locality of radj}.

		\item For any $N \in D(Y)$, \Cref{exa:tensor morph} gives a transformation $\otimes N : \pi^* D \to \pi^* D$ of presheaves on $\mathcal C_{/Y}$ such that for any $W \in \mathcal C_{/Y}$, the square
			\[
				\begin{tikzcd}
					\pi^* D(W) \ar[d, "\otimes N"'] \ar[r] & \pi^* D(W \times_Y X) \ar[d, "\otimes N"] \\
					\pi^* D(W) \ar[r] & \pi^* D(W \times_Y X)
				\end{tikzcd}
			\]
			is equivalent to the square
			\[
				\begin{tikzcd}
					D(W) \ar[d, "\otimes N"'] \ar[r] & D(W \times_Y X) \ar[d, "\otimes N"] \\
					D(W) \ar[r] & D(W \times_Y X)
				\end{tikzcd}
			.\]
			Thus, if $D$ has the right projection formula for $f_i$ for all $i$, then $\otimes N$ is right adjointable at $f_i$ for each $i$ and all $N \in D(Y)$, so by the first statement, $\otimes N$ is right adjointable at $f$ for all $N \in D(Y)$, so $D$ has the right projection formula for $f$.

		\item
			\begin{enumerate}

				\item Consider the morphism of pullback contexts $\pi, \pi' : \mathcal C' \to \mathcal C$ of \Cref{exa:base change morph} applied to the map $q : Y' \to Y$, so $\mathcal C'$ is the full subcategory of $\mathcal C_{/Y}$ consisting of maps $W \to Y$ such that $W \times_Y Y'$ exists, and we have a transformation $\phi : \pi^* D \to (\pi')^* D$ such that for each $W \in \mathcal C'$, $\phi : \pi^* D(W) \to (\pi')^* D(W)$ is $D(W) \to D(W \times_Y Y')$, and the square
					\[
						\begin{tikzcd}
							\pi^* D(W) \ar[d] \ar[r] & \pi^* D(W \times_Y X) \ar[d] \\
							(\pi')^* D(W) \ar[r] & (\pi')^* D(W \times_Y X)
						\end{tikzcd}
					\]
					is equivalent to the square
					\[
						\begin{tikzcd}
							D(W) \ar[d] \ar[r] & D(W \times_Y X) \ar[d] \\
							D(W \times_Y Y') \ar[r] & D(W \times_Y X')
						\end{tikzcd}
					\]
					given by applying $D(W \times_Y -)$ to the Cartesian square
					\[
						\begin{tikzcd}
							X' \ar[d] \ar[r] & Y' \ar[d, "q"] \\
							X \ar[r, "f"'] & Y
						\end{tikzcd}
					\]
					in $\mathcal C'$.

					If $D$ has right base change for $f_i$ against $q_i$ for each $i$, we have that for each $i$, $\phi$ is right adjointable at $f_i : X_i \to Y_i$, so since $\mathcal R'$ is a $D'$-pseudocover, we conclude by the first statement.

				\item Suppose $q$ is quasi-admissible. As in \Cref{exa:sharp morph}, if we use the same $\pi,\pi'$ as in the previous point, there is a transformation $\psi : (\pi')^* D \to \pi^* D$ such that for any $W \in \mathcal C'$, the square
					\[
						\begin{tikzcd}
							(\pi')^* D(W) \ar[d] \ar[r] & (\pi')^* D(W \times_Y X) \ar[d] \\
							\pi^* D(W) \ar[r] & \pi^* D(W \times_Y X)
						\end{tikzcd}
					\]
					is equivalent to the square
					\[
						\begin{tikzcd}
							D(W \times_Y Y') \ar[d, "(W \times_Y Y' \to W)_\sharp"'] \ar[r] & D(W \times_Y X') \ar[d, "(W \times Y X' \to W \times_Y X)_\sharp"] \\
							D(W) \ar[r] & D(W \times_Y X)
						\end{tikzcd}
					.\]

					Thus, if $D$ has right-left base change for $f_i$ against $q_i$ for each $i$, we have that for each $i$, $\psi$ is right adjointable at $f_i : X_i \to Y_i$, so by the first statement, since $\mathcal R$ is a $D$-pseudocover, it follows that $\psi$ is right adjointable at $f$, whence $D$ has right-left base change for $f$ against $q$. 

			\end{enumerate}

	\end{enumerate}

\end{proof}

\section{Cohomological Properness} \label{S:prop}

Fix a pullback context $\mathcal C$. In this section we introduce a notion (\Cref{defn:proper}) of ``properness'' for maps in $\mathcal C$. In fact, a quasi-admissibility structure is not enough to give us a good notion of properness, so instead we will give a definition that depends on a choice of system of coefficients on $\mathcal C$, which we will often assume is a pullback formalism. This can be seen as a cohomological characterization of properness.

\begin{defn} \label{defn:proper}
	Given a presheaf $D : \mathcal C^\op \to \CAlg(\widehat{\Cat})$, say a map $f : X \to Y$ is \emph{$D$-quasi-proper} if $D$ has the right projection formula for $f$, right base change for $f$, quasi-admissible exchange for $f$, and furthermore, the functor $f_* : D(X) \to D(Y)$ is colimit-preserving. If $f_*$ has a right adjoint, we will denote it by $f^\sharp$.
\end{defn}

Let us spell out \Cref{defn:proper} in the case of pullback formalisms.
\begin{rmk} \label{rmk:desc proper}
	If $D \in \PF(\mathcal C)$, then a map $f : X \to Y$ in $\mathcal C$ is $D$-quasi-proper if and only if the following hold:
	\begin{enumerate}

		\item \textbf{Additional adjoint:}
			The right adjoint $f_*$ of $f^* \coloneqq D(f)$ admits a further right adjoint $f^\sharp$.

		\item \textbf{Proper projection formula:}
			For any $M \in D(X)$, and $N,N' \in D(Y)$, the natural maps
			\[
				f_* M \otimes N \to f_*(M \otimes f^* N) \quad\text{and}\quad \underline\Hom(f^* N, f^\sharp N') \to f_\sharp \underline\Hom(N, N')
			\]
			are equivalences.

		\item \textbf{Proper base change:}
			If
			\[
				\begin{tikzcd}
					X' \ar[d, "p"'] \ar[r, "f'"] & Y' \ar[d, "q" ] \\
					X \ar[r, "f"'] & Y
				\end{tikzcd}
			\]
			is a Cartesian square in $\mathcal C$, then the natural maps
			\[
				q^* f_* \to f'_* p^* \quad\text{and}\quad p_* f'^\sharp \to f^\sharp q_*
			\]
			are equivalences.

		\item \textbf{Smooth-proper base change:}
			In the previous point, if $q$ is quasi-admissible, then the natural maps
			\[
				p^* f^\sharp \to f'^\sharp q^* \quad\text{and}\quad q_\sharp f'_* \to f_* p_\sharp
			\]
			are equivalences.

	\end{enumerate}
\end{rmk}

The relevance of this property to 6-functor formalisms is made explicit in the following \lcnamecref{rmk:6FF by decomp}.
\begin{rmk} \label{rmk:6FF by decomp}
	Let $I,P$ be collections of maps in $\mathcal C$, and assume that
	\begin{enumerate}

		\item $I$ and $P$ contain all equivalences, and are stable under base change, composition, and taking diagonals.

		\item For any $j \in I$ and $p \in P$, the composite $j \circ p$ is equivalent to $p' \circ j'$, where $j' \in \bar I$ and $p' \in \bar P$.

		\item Every morphism in $I \cap P$ is truncated.

	\end{enumerate}
	Write $E$ for the collection of maps that are composites of maps in $I \cup P$. Then
	% by \Cref{lem:comp of maps with diagonal},
	$(\mathcal C,E)$ is a geometric setup in the sense of \cite[Convention 2.1.3]{HM6FF}.

	Suppose $\mathcal C$ is a pullback context such that every map in $I$ is quasi-admissible, and let $D$ be a pullback formalism on $\mathcal C$ such that every map in $P$ is $D$-quasi-proper. It follows immediately from \cite[Proposition A.5.10]{Mann6FF} that $D$ extends to a 6-functor formalism on $(\mathcal C,E)$ such that for any $f \in E$,
	\[
		\begin{array}{cl}
			f^! \simeq f^* & \text{if $f \in I$} \\
			f_! \simeq f_* & \text{if $f \in P$}
		\end{array}
	.\]
	Using this description, as well as \cite[Lemma F.6(2,3)]{TwAmb}, we should be able to show that all quasi-admissible maps are $D$-suave and all $D$-quasi-proper maps are $D$-prim (in the sense of \Cref{defn:suave prim}). Indeed, this argument can be made precise if instead of using \cite[Proposition A.5.10]{Mann6FF} to get a lax symmetric monoidal functor $\Span(\mathcal C,E) \to \PrL$, we use \cite[Theorem B]{CLL6FF} to get a lax symmetric monoidal 2-functor $\Span_2(\mathcal C,E)_{P,I} \to \PrL$, and then apply \Cref{prp:nat sq adj 6FF}(\ref{itm:nat sq adj/suave prim}).
\end{rmk}

Recall from \Cref{nota:psh} that for any presheaf $D : \mathcal C^\op \to \Alg_{E_0} \widehat{\Cat}$, we have a notion of ``cohomology spaces'' for $D$: if $S \in \mathcal C$, and $X \in \mathcal C_{/S}$, for any $M \in D(S)$, we can define $D(X;M)$ to be the mapping space $D(X)(1, (X \to S)^* M)$. In the setting of \Cref{rmk:6FF by decomp}, we also have a notion of (twisted) Borel-Moore homology given in \Cref{defn:BM} such that
\[
	\begin{array}{rcll}
		D^\BM(X;M) &\simeq& D(S)((X \to S)_* 1, M) & \text{if $f \in P$, and} \\
		D^\BM(X;M) &\simeq& D(X;M) & \text{if $f \in I$,}
	\end{array}
,\]
and \cite[Corollary 4.5.11]{HM6FF} says that
\[ \tag{PD6FF} \label{eqn:duality}
	\begin{array}{rcll}
		D^\BM(X;M)[\delta_f] &\simeq& D(S)((X \to S)_* 1, M) & \text{if $f$ is $D$-quasi-proper, and} \\
		D^\BM(X;M) &\simeq& D(X;M)[\omega_f] & \text{if $f$ is quasi-admissible.}
	\end{array}
\]

\begin{defn} \label{defn:BM}
	If $D$ is a 3-functor formalism on a geometric setup $(\mathcal C, E)$, $S \in \mathcal C$, and $X \in \mathcal C_{/S}$ is such that the structure map $X \to S$ is in $E$, we define the \emph{Borel-Moore homology} of $X$ with coefficients in $M \in D(S)$ to be the mapping space
	\[
		D^\BM(X;M) \coloneqq D(S)((X \to S)_! 1, M)
	.\]
	More generally, if $\xi \in D(X)$, we define the \emph{$\xi$-twisted} Borel-Moore homology to be
	\[
		D^\BM(X;M)[\xi] \coloneqq D(S)((X \to S)_! \xi, M)
	.\]
	If $D$ is a 6-functor formalism, then this is equivalent to $D(X)(\xi, (X \to S)^! M)$ (\cf{} \cite[Definition 9.1]{SixAlgSt}).
\end{defn}

We will see how to produce special relationships between Borel-Moore homology and usual cohomology in \Cref{rmk:BM closed,rmk:BM duality}.

\subsection{Closure properties of proper maps} \label{S:proper stable}

\begin{lem} \label{lem:composites of proper maps}
	Let $D$ be a pullback formalism on $\mathcal C$. Then the collection of $D$-quasi-proper maps is closed under composition.
	\begin{proof}
		This follows immediately from \cite[Lemma F.6]{TwAmb} and \cite[Lemma F.13]{TwAmb}.
	\end{proof}
\end{lem}

\begin{lem} \label{lem:bc prop map by qadm mono}
	If $D$ is a pullback formalism on $\mathcal C$, then the collection of $D$-quasi-proper maps is stable under base change along quasi-admissible monomorphisms. In fact, we have the following more general result: let
	\[
		\begin{tikzcd}
			U' \ar[d, "j'"'] \ar[r, "f_U"] & U \ar[d, "j"] \\
			X' \ar[r, "f"'] & X
		\end{tikzcd}
	\]
	be a Cartesian square in a pullback context $\mathcal C$, where $j$ is a quasi-admissible monomorphism. If $D : \mathcal C^\op \to \widehat{\Cat}$ is a presheaf with quasi-admissible base change, then we have the following:
	\begin{enumerate}

		\item If $f^*$ has a right adjoint, then so does $f_U^*$.

		\item If $D$ has right base change for $f$, then it has right base change for $f_U$.

		\item If $D$ has quasi-admissible exchange for $f$, then it also has quasi-admissible exchange for $f_U$.

		\item If $D$ has right-left base change for $f$ against $j$, and $f_*$ is colimit-preserving, then $(f_U)_*$ is colimit-preserving.

		\item If $D$ lifts to a $\CAlg(\widehat{\Cat})$-valued presheaf that has the right projection formula for $f$, then it also has the right projection formula for $f_U$.

	\end{enumerate}

	\begin{proof}
		We first note that by \Cref{lem:qadm mono loc}, $j_\sharp$ and $j'_\sharp$ are conservative, and $j^*$ and $j'^*$ are essentially surjective.
		\begin{enumerate}

			\item Since $D$ has quasi-admissible base change, there is a commutative square
				\[
					\begin{tikzcd}
						D(U) \ar[d, "j_\sharp"'] \ar[r, "f_U^*"] & D(U') \ar[d, "j'_\sharp"] \\
						D(X) \ar[r, "f^*"] & D(X')
					\end{tikzcd}
				.\]
				Since the columns are conservative and colimit-preserving, if $f^*$ is colimit-preserving, then so is $f_U^*$.

			\item If $D$ has right base change for $f$, then $D$ has right base change for $f_U$ by \cite[Lemma F.6(2)]{TwAmb}, the fact that $D$ has right base change for $f$, and the fact that $j'^*$ is essentially surjective.

			\item Similarly, if $D$ has quasi-admissible exchange for $f$, then $D$ has quasi-admissible exchange for $f_U$ by \cite[Lemma F.13]{TwAmb}, the fact that $D$ has quasi-admissible exchange for $f$, and the fact that $j'_\sharp$ is conservative.

			\item Since, $D$ has right-left base change for $f$ against $j$, we have a commutative square
				\[
					\begin{tikzcd}
						D(U') \ar[d, "j'_\sharp"'] \ar[r, "(f_U)_*"] & D(U) \ar[d, "j_\sharp"] \\
						D(X') \ar[r, "f_*"] & D(X)
					\end{tikzcd}
				.\]
				Since the columns are conservative and colimit-preserving, we have that $(f_U)_*$ is colimit-preserving if $f_*$ is colimit-preserving.

			\item For any $M \in D(X)$ and $M' \in D(X')$, by \cite[Lemma F.19(2)]{TwAmb}, and since $D$ has left base change for $j$ against $f$, we have that
				\[
					(f_U)_* j'^* M' \otimes j^* M \to (f_U)_*(j'^* M' \otimes f_U^* j^* M)
				\]
				is equivalent to $j^*$ of
				\[
					f_* M' \otimes M \to f_*(M' \otimes f^* M)
				,\]
				so it is an equivalence if $D$ has the right projection formula for $f$. Since $j^*$ and $j'^*$ are essentially surjective, it follows that $D$ has the right projection formula for $f_U$.

		\end{enumerate}
	\end{proof}
\end{lem}

\subsubsection{Locality on the target}

\begin{prp}[Locality on the target for proper maps] \label{prp:proper local on target}
	Let $D$ be a pullback formalism on $\mathcal C$, let $f : X \to Y$ be a map in $\mathcal C$, and let $\{Y_i \to Y\}_i$ be a small quasi-admissible $D$-pseudocover of $Y$.

	If $\phi : D \to D'$ is a morphism of pullback formalisms, then $\phi$ is right adjointable at $f$ if it is right adjointable at the base change $f_i : X_i \coloneqq X \times_Y Y_i \to Y_i$ for each $i$.

	Furthermore, if all base changes of $f$ exist, and $f_i$ is $D$-quasi-proper for all $i$, then $f$ is $D$-quasi-proper.
	\begin{proof}
		By \Cref{thm:D-topology}, we have that every base change of $\{Y_i \to Y\}_i$ is a $D$-pseudocover, and a $D'$-pseudocover.

		Hence, the first statement follows from \Cref{prp:target-locality of dual bc}, which also shows that if $f_i$ is $D$-quasi-proper for all $i$, then $D$ satisfies the right projection formula for $f$, and if $f$ admits all base changes, then $D$ has right base change and quasi-admissible exchange for $f$.

		Thus, to show the second statement, it only remains to show that $f_*$ has a right adjoint. By \cite[Lemma F.0.6]{Fundamentals}, it suffices to show that for every $i$, the composite
		\[
			D(X) \xrightarrow{f_*} D(Y) \xrightarrow{(Y_i \to Y)^*} D(Y_i)
		\]
		has a right adjoint. Indeed, since $D$ has right base change for $f$, we have that this composite is equivalent to the composite
		\[
			D(X) \xrightarrow{(X \times_Y Y_i \to X)^*} D(X_i) \xrightarrow{(f_i)_*} D(Y_i)
		.\]
		The first functor has a right adjoint since $D$ takes values in $\PrL$, and the second has a right adjoint since $f_i$ is $D$-quasi-proper, so the composite has a right adjoint as desired.
	\end{proof}
\end{prp}

% PERF: generalize to pseudocover
\begin{lem} \label{lem:radj proper-local on target}
	Let
	\[
		\left\{\begin{tikzcd}
			\bar X_i \ar[d, "p_i"'] \ar[r, "\bar f_i"] & \bar Y \ar[d, "q_i"] \\
			X \ar[r, "f"'] & Y
		\end{tikzcd}\right\}_i
	\]
	be a family of commutative squares in $\mathcal C$.

	Let $\phi : D \to D'$ be a transformation of presheaves $D,D' : \mathcal C^\op \to \widehat{\Cat}$. Suppose that for each $i$, $\phi$ is right adjointable at $p_i, q_i, \bar f_i$. If $D(f)$ and $D'(f)$ admit right adjoints (both denoted by $f_*$), and the images of $\{(p_i)_*\}_i$ generate $D(X)$ under colimits that are preserved by $\phi f_*$ and $f_* \phi$, then $\phi$ is right adjointable $f$. 
	\begin{proof}
		For each $i$, we have a commutative diagram
		\[
			\begin{tikzcd}
				D(Y) \ar[r, "q_i^*"] \ar[d] & D(\bar Y_i) \ar[r, "\bar f_i^*"] \ar[d] & D(\bar X_i) \ar[d] \\
				D'(Y) \ar[r, "q_i^*"] & D'(\bar Y_i) \ar[r, "\bar f^*"] & D'(\bar X_i)
			\end{tikzcd}
		.\]
		Since $\phi$ is right adjointable at $q_i$ and $\bar f_i$, \cite[Lemma F.6(4)]{TwAmb} shows that the outer rectangle is horizontally right adjointable. This is equivalent to the outer rectangle in the following commutative diagram:
		\[
			\begin{tikzcd}
				D(Y) \ar[r, "f^*"] \ar[d] & D(X) \ar[r, "p_i^*"] \ar[d] & D(\bar X_i) \ar[d] \\
				D'(Y) \ar[r, "f^*"] & D'(X) \ar[r, "p_i^*"] & D'(\bar X_i)
			\end{tikzcd}
		.\]
		Since $\phi$ is right adjointable at $p_i$, we find that along with the outer rectangle, the right square is also horizontally right adjointable.

		Thus, by \cite[Lemma F.6(4)]{TwAmb}, the horizontal right mate of the left square is an equivalence at objects in the essential image of $(p_i)_*$ for each $i$. We conclude by our assumption on the generation of $D(X)$ by the images of $\{(p_i)_*\}_i$.
	\end{proof}
\end{lem}

\begin{prp} \label{prp:proper proper-local on target}
	Let
	\[
		\begin{tikzcd}
			\bar X \ar[d, "p"'] \ar[r, "\bar f"] & \bar Y \ar[d, "q"] \\
			X \ar[r, "f"'] & Y
		\end{tikzcd}
	\]
	be a Cartesian square in $\mathcal C$, and let $D$ be a pullback formalism on $\mathcal C$ such that $q$ is $D$-quasi-proper and $D$-acyclic, $p$ is $D$-quasi-proper, and $D$ has the right projection formula for every quasi-admissible base change of $p$.

	If $\bar f$ is $D$-quasi-proper, and all base changes of $f$ exist, then $f$ is $D$-quasi-proper. Furthermore, if $\phi : D \to D'$ is a transformation that is right adjointable at $p,q,\bar f$, then $\phi$ is right adjointable at $f$.
	\begin{proof}
		Since $q^*$ is fully faithful, we have that $1 \simeq q_* 1$. Since $D$ has right base change for $q$, we find that for any base change $q'$ of $q$, we also have that $1 \simeq q'_* 1$. Thus, by the dual of \cite[Lemma D.2.3]{Fundamentals}, we know that every base change of $q$ for which $D$ has the right projection formula is $D$-acyclic. In particular, if $p'$ is a quasi-admissible base change of $p$, then $p'_*$ is essentially surjective.

		% PERF: more details
		By applying \Cref{lem:radj proper-local on target}, we see that $\phi$ is right adjointable at $f$, and by taking $\phi$ to be the morphisms given by \Cref{exa:base change morph}, \Cref{exa:sharp morph}, or \Cref{exa:tensor morph}, we see that $D$ has right base change for $f$, quasi-admissible exchange for $f$, and the right projection formula for $f$.

		Finally, to see that $f_*$ is colimit-preserving, we note that since $D$ has right base change for $f$ against $q$, we have an equivalence
		\[
			q^* f_* \simeq \bar f_* p^*
		.\]
		Since $\bar f_* p^*$ is colimit-preserving, and $q^*$ is colimit-preserving and conservative, it follows that $f_*$ is colimit-preserving.
	\end{proof}
\end{prp}

\subsubsection{Locality on the source}

\begin{lem} \label{lem:pre source locality of proper}
	Let $D : \mathcal C^\op \to \PrL$ be a presheaf on $\mathcal C$, let $f : X \to Y$ be a map in $\mathcal C$ such that the right adjoint $f_*$ of $f^* = D(f)$ is colimit-preserving, and let $\{f_i : X_i \to X\}_i$ be a small family of maps such that for each $i$, the functor $(f_i)_*$ admits a right adjoint $f_i^\sharp$, and the functors $\{f_i^\sharp : D(X) \to D(X_i)\}_i$ are jointly conservative.

	If $D' : \mathcal C^\op \to \widehat{\Cat}$ is a presheaf such that $D'(f)$ has a colimit-preserving right adjoint, and $\phi : D \to D'$ is a transformation that is right adjointable at $f_i$ and $f \circ f_i$ for all $i$, then $\phi$ is right adjointable at $f$.
	\begin{proof}
		Note that by \Cref{lem:radj cons iff ladj gens}, since $D$ takes values in presentable categories, we have that the category $D(X)$ is generated under small colimits by the union of the images of the functors $\{(f_i)_* : D(X_i) \to D(X)\}_i$. Thus, this statement follows immediately from \cite[Lemma D.1.13]{Fundamentals}, since $f_* : D(X) \to D(Y)$ preserves small colimits.
	\end{proof}
\end{lem}

\begin{prp}[Locality on the source for proper maps] \label{prp:proper local on source}
	Let $D$ be a pullback formalism on $\mathcal C$, write $\mathcal C^\sharp$ for the wide subcategory of $\mathcal C$ consisting of $D$-quasi-proper maps, and let $D^\sharp : (\mathcal C^\sharp)^\op \to \PrR$ for the presheaf given by sending a $D$-quasi-proper map $g$ to the right adjoint $g^\sharp$ of $g_*$.

	Let $K$ be a simplicial set, and let $X : K^\triangleright \to \mathcal C^\sharp$ be a $D^\sharp$-acyclic diagram, \ie, the functor
	\[
		D(X(\infty)) \to \varprojlim_K D^\sharp X|_K^\op
	\]
	is fully faithful.

	If $f : X(\infty) \to Y$ is a map in $\mathcal C$ such that all base changes of $f$ exist, and $f_*$ preserves $K$-indexed colimits, then $f$ is $D$-quasi-proper if and only if for each $a \in K$, the composite $X(a) \to Y$ is $D$-quasi-proper.
	\begin{proof}
		It follows from \Cref{lem:source-locality of colim-pres} that $f_*$ is colimit-preserving. Thus, we may use \Cref{lem:pre source locality of proper} to see that $f^*$ has a linear right adjoint (using \Cref{exa:tensor morph}), and that $D$ has right base change and quasi-admissible exchange for $f$ (using \Cref{exa:base change morph,exa:sharp morph}), so $f$ is $D$-quasi-proper since $f_*$ admits a right adjoint.
	\end{proof}
\end{prp}

\section{Gluing and Localization} \label{S:gluing}

% PERF: don't need complements for all of this?

First we will need the following definition:
\begin{defn} \label{defn:complements}
	Given a category $\mathcal C$, say objects $Z,U \in \mathcal C$ are \emph{disjoint}, or that $U$ is \emph{disjoint from} $Z$, if any object $X \in \mathcal C$ admitting maps to both $Z$ and $U$ must be initial.

	Note that the objects of $\mathcal C$ that are disjoint from $Z$ form a sieve in $\mathcal C$, which we may denote by $Z^\complement$. Say that an object $U$ of $\mathcal C$ is a \emph{complement} of $Z$ if it is terminal in $Z^\complement$. Equivalently, $U$ is a complement of $Z$ if it is $(-1)$-truncated, and every object of $\mathcal C$ that is disjoint from $Z$ admits a map to $U$.

	Note that for any object $S \in \mathcal C$, we get corresponding notions for maps to $S$ by applying the above definitions in $\mathcal C_{/S}$. Sometimes, if a complement of $Z \to S$ exists, we denote it by $S \setminus Z$. 
\end{defn}

The gluing property can be thought of as a cohomological property that is characteristic of closed immersions. Indeed, we only consider it for maps that are ``closed'' in the following sense:
\begin{defn} \label{defn:closed}
	If $\mathcal C$ is a pullback context, then a map $i : Z \to S$ is closed if it has a quasi-admissible complement $j : U \to S$ in the sense of \Cref{defn:complements}.
\end{defn}

Later (\Cref{cor:closed immersions are proper}), we will see that, just as closed immersions are proper, under mild additional hypotheses, the gluing property implies the cohomological properness of \Cref{defn:proper}. In fact, if $D$ is a pullback formalism taking values in stable categories and that sends initial objects to the zero category, and $i_*$ is fully faithful, then for $D$ to have gluing for $i : Z \to S$ is equivalent to each of the following conditions (see \Cref{prp:loc for stab PF,rmk:ptd gluing}), where $j : U \to S$ is a complement of $i$:
\begin{enumerate}

	\item $i^*, j^*$ are jointly conservative.

	\item $D(Z) \to D(S) \to D(U)$ is a fibre sequence.

	\item $j_\sharp j^* \to \id \to i_* i^*$ is an exact sequence.

\end{enumerate}

We now give the general definition of the gluing property:
\begin{defn} \label{defn:gluing}
	Let $\mathcal C$ be a pullback context, and let $D : \mathcal C^\op \to \widehat{\Cat}$ be a presheaf that respects quasi-admissibility. Given a closed map $i : Z \to S$, if $i^*$ admits a right adjoint $i_*$, and $j$ is a quasi-admissible complement of $i$, then the counit of $j_\sharp \dashv j^*$ and the unit of $i^* \dashv i_*$ define a square
	\[
		\begin{tikzcd}
			j_\sharp j^* \ar[d] \ar[r] & \id \ar[d] \\
			j_\sharp j^* i_* i^* \ar[r] & i_* i^*
		\end{tikzcd}
	,\]
	which we denote $\square_i$, or $\square_i^D$ if $D$ is not clear from context.
	
	Say $D$ \emph{has gluing for $i$ (at $F$)} if $\square_i$ ($\square_i(F)$) is coCartesian.
\end{defn}

\begin{lem} \label{lem:gluing for acyclic}
	In the setting of \Cref{defn:gluing}, if $i$ is $D$-acyclic, or $i$ has $D$-acyclic complement, then $D$ has gluing for $i$.
	\begin{proof}
		If $i$ is $D$-acyclic, the unit $\id \to i_* i^*$ is invertible, so the columns of $\square_i$ are invertible, and similarly, if the complement of $i$ is $D$-acyclic, then the rows of $\square_i$ are equivalences. In either case it follows that $\square_i$ is coCartesian.
	\end{proof}
\end{lem}

Generally speaking, we will mostly be interested in studying gluing for \emph{reduced} presheaves with quasi-admissible base change. Recall that a presheaf is reduced if it preserves terminal objects. In particular, any reduced presheaf on a category $\mathcal C$ can be viewed as a limit-preserving presheaf on $\Psh^\emptyset(\mathcal C)$, the category of reduced presheaves $\mathcal C^\op \to \spaces$.

For the remainder of the section we will make the following assumption:
\begin{ass} \label{ass:gluing}
	$\mathcal C$ is a pullback context, $D : \mathcal C^\op \to \widehat{\Cat}$ is a reduced presheaf with quasi-admissible base change, and $i : Z \to S$ is a closed map in $\mathcal C$ such that $i^*$ has a right adjoint $i_*$, and $j : U \to S$ is a quasi-admissible complement of $i$.
\end{ass}

The next result is useful for identifying the square $\square_i$:
\begin{lem} \label{lem:desc reduced gluing}
	For any functor $M : I \to D(S)$, let
	\[
		\begin{tikzcd}
			j_\sharp j^* M \ar[d] \ar[r] & M \ar[d] \\
			j_\sharp \pt \ar[r] & i_* i^* M
		\end{tikzcd}
	\]
	be a not-necessarily-commutative square in $\Fun(I, D(S))$. If the top arrow is given by the counit of $j_\sharp \dashv j^*$, and the right arrow is given by the unit of $i^* \dashv i_*$, then this square commutes in an essentially unique way, and the resulting commutative square is equivalent to $\square_i M$.

	% and if $D$ lifts to a $\CAlg(\widehat{\Cat})$-valued presheaf that satisfies the quasi-admissible projection formula, then this square is also equivalent to any square of the form
	% \[
	% 	\begin{tikzcd}
	% 		M \otimes \cls{U} \ar[d] \ar[r] & M \ar[d] \\
	% 		\pt \otimes \cls{U} \ar[r] & i_* i^*
	% 	\end{tikzcd}
	% ,\]
	% where the top arrow is given by $M \otimes (\cls{U} \to 1)$, and right square is given by the unit of $i$.
	\begin{proof}
		Note that by quasi-admissible base change, the functor $j^* i_*$ factors through a right adjoint functor from the category
		\[
			D(Z \times_S U) = D(\emptyset) = \pt
		,\]
		so it must be the constant functor with value $\pt$ (a terminal object). It follows that
		\[
			j^*(i_* i^* \to \pt) \simeq j^* i_* (i^* \to \pt)
		\]
		is an equivalence, so
		\[
			j_\sharp j^* (i_* i^* \to \pt)
		\]
		is an equivalence
		\[
			j_\sharp j^* i_* i^* \to j_\sharp \pt %\simeq \pt \otimes \cls{U}
		,\]
		where $j^* \pt \simeq \pt$ since $j^*$ is a right adjoint.

		Note that for any $A \in D(U)$ and $B \in D(Z)$, the space $D(S)(j_\sharp(A), i_*(B))$ is equivalent to a mapping space in the category
		\[
			D(Z \times_S U) = D(\emptyset) = \pt
		,\]
		so it must be contractible. In particular, the space of maps $j_\sharp j^* M \to i_* i^* M$ is contractible, so the square commutes in an essentially unique way. Furthermore, the space of maps $j_\sharp j^* i_* i^* M \to i_* i^* M$ is contractible, so any such map must be equivalent to the map induced by the counit of $j_\sharp \dashv j^*$.

		Now, since $j$ is a quasi-admissible monomorphism, \Cref{lem:qadm mono loc} shows that $j^*$ is a colocalization, so it defines a colocalization $\Fun(I, D(S)) \to \Fun(I, D(U))$. Thus, $j_\sharp j^* M$ is a colocal object for this colocalization, so there is a contractible space of lifts of the essentially unique map $j_\sharp j^* M \to i_* i^* M$ along the counit $j_\sharp j^* i_* i^* M \to i_* i^* M$, which is equivalent to any map $j_\sharp \pt \to i_* i^* M$. Thus, we see that any commutative square whose boundary has the same top and right arrows that we specified, is equivalent to this one. In particular, $\square_i M$ is equivalent to this commutative square.

		% Finally, when $j_\sharp$ is a linear left adjoint of $j^*$, we have that the counit $j_\sharp j^* \to \id$ is equivalent to $- \otimes (\cls{U} \to 1)$, so we are done.
	\end{proof}
\end{lem}

\begin{rmk} \label{rmk:gluing base change}
	If $\phi : D' \to D$ is a transformation that is right adjointable at $i$ and left adjointable at $j$, then we may use \Cref{lem:desc reduced gluing} to show that $\phi \square_i \simeq \square_i \phi$.

	In particular, if $D$ has right base change for $i$ against a map $\sigma : S' \to S$, then since $D$ also has left base change for $j$ against $\sigma$, as in \Cref{exa:base change morph}, we find that $\sigma^* \square_i \simeq \square_{i'} \sigma^*$, where $i'$ is the base change of $i$ along $\sigma$.

	Thus, if $\{X_k \to S\}_k$ is a quasi-admissible $D$-pseudocover of $S$, then $D$ has gluing for $i$ if for all $k$, $D$ has gluing for the base change of $i$ along $X_k \to S$.
\end{rmk}

\begin{rmk} \label{rmk:ptd gluing}
	Suppose that $D(S)$ has a zero object, and $j_\sharp j^*$ preserves zero objects. It follows that $j_\sharp \pt \simeq j_\sharp j^* \pt$ is a zero object. By \Cref{lem:desc reduced gluing}, $D$ has gluing for $i$ at $F \in D(S)$ if and only if a square of the form
	\[
		\begin{tikzcd}
			j_\sharp j^* F \ar[d] \ar[r] & F \ar[d] \\
			j_\sharp \pt \ar[r] & i_* i^* F
		\end{tikzcd}
	\]
	is coCartesian (where the top arrow is given by the counit, and the right arrow is given by the unit), so since $j_\sharp \pt$ is a zero object, we have that $D$ has gluing for $i$ at $F \in D(S)$ if and only if
	\[
		j_\sharp j^* F \to F \to i_* i^* F
	\]
	is a cofibre sequence in $D(S)$.

	Note that $j_\sharp j^*$ preserves zero objects under any one of the following conditions:
	\begin{enumerate}

		\item $D(U)$ has a zero object. In this case, the terminal object $\pt \in D(U)$ is an initial object, so since $j_\sharp$ preserves initial objects, we have that $j_\sharp \pt \simeq j_\sharp j^* \pt$ is a zero object.

		\item $D$ takes values in $\PrL$. In this case, $j^*$ admits a right adjoint, so it preserves initial objects, so since it also preserves terminal objects, it preserves zero objects. Therefore $D(U)$ has a zero object, and we may reduce to the previous case.

		\item $j^*$ lifts to a symmetric monoidal functor that admits a linear left adjoint, and the monoidal product on $D(S)$ preserves empty colimits in each variable. In this case, $j_\sharp j^* \simeq \cls{U} \otimes -$, which we have assumed preserves initial objects.

			Furthermore, in this case, since $j_\sharp j^* F \to F$ is equivalent to $(\cls{U} \to 1) \otimes F$, we have that $D$ has gluing for $i$ at $F$ if and only if
			\[
				\cls{U} \otimes F \to F \to i_* i^* F
			\]
			is a cofibre sequence in $D(S)$.

	\end{enumerate}
\end{rmk}

\begin{rmk} \label{rmk:gluing on fundamental}
	If $D$ is a reduced pullback formalism, then by \Cref{lem:desc reduced gluing}, we have that for any $X \in \mathcal C_S$, any square
	\[
		\begin{tikzcd}
			\cls{X} \otimes \cls{U} \ar[d] \ar[r] & \cls{X} \ar[d] \\
			\pt \otimes \cls{U} \ar[r] & i_* \cls{i^{-1}(X)}
		\end{tikzcd}
	\]
	where the right arrow is the unit and the top arrow is $\cls{X \times_S U \to X}$, commutes in an essentially unique way and is equivalent to $\square_i^D(\cls{X})$.

	In fact, this holds more generally if $D : \mathcal C^\op \to \CAlg(\widehat{\Cat})$ is a reduced presheaf that has quasi-admissible base change and the quasi-admissible projection formula.
\end{rmk}

\subsection{Preliminary consequences of gluing} \label{sec:gluing consequences}
We continue to assume \Cref{ass:gluing}.

% PERF: make some comments about what results we will prove
% {prp:gluing implies loc}

\begin{rmk} \label{rmk:BM closed}
	If $D$ is a pointed reduced pullback formalism on $\mathcal C$, and $D$ has gluing for $i$, then by \Cref{rmk:ptd gluing}, we have a cofibre sequence
	\[
		j_\sharp 1 \to 1 \to i_* 1
	\]
	in $D(S)$.

	We often have that $D$ extends to a 3-functor formalism on a geometric setup $(\mathcal C,E)$ where $i \in E$, and $i_* \simeq i_!$, so, recalling the definition of Borel-Moore homology from \Cref{defn:BM}, this cofibre sequence induces a fibre sequence of spaces
	\[
		D^\BM(Z;M) \to D(S;M) \to D(U;M)
	\]
	for all $M \in D(S)$. Thus, we can view $D^\BM(Z;M)$ as cohomology on $S$ with supports in $Z$. We often also have that $j \in E$ and $j_\sharp \simeq j_!$, in which case we actually get a fibre sequence
	\[
		D^\BM(Z;M) \to D^\BM(S;M) \to D^\BM(U;M)
	.\]
\end{rmk}

\begin{lem} \label{lem:gluing at kernel}
	Let $F \in D(S)$ such that $j^* F$ is terminal. Then $\square_i(F)$ is coCartesian if and only if the unit $F \to i_* i^* F$ is invertible.
	\begin{proof}
		Note that since $j^* F$ is terminal, the map $j^* F \to \pt$ is invertible, so by \Cref{lem:desc reduced gluing}, the left map of $\square_i(F)$ is invertible. Thus, $\square_i(F)$ is coCartesian if and only if the right map, which is the unit $F \to i_* i^* F$, is invertible.
	\end{proof}
\end{lem}

\begin{defn}[Localization] \label{defn:localization}
	Say $D$ has localization for $i$ if
	\[
		D(Z) \xrightarrow{i_*} D(S) \xrightarrow{j^*} D(U)
	\]
	is a fibre sequence of categories, \ie{} $i_*$ is fully faithful with essential image given by those $F \in D(S)$ such that $j^* F$ is terminal.
\end{defn}

\begin{prp} \label{prp:gluing implies loc}
	If $D$ has gluing for $i$, then $F \in D(S)$ is in the essential image of $i_*$ if and only if $j^* F$ is terminal. Furthermore, if $i_*$ is conservative, then it is fully faithful, so
	\[
		D(Z) \xrightarrow{i_*} D(S) \xrightarrow{j^*} D(U)
	\]
	is a fibre sequence of categories, and if $j^*$ preserves weakly contractible colimits, then so does $i_*$.
	\begin{proof}
		Since $\square_i$ is coCartesian, \Cref{lem:gluing at kernel} tells us that if $j^* F$ is terminal, then $F \to i_* i^* F$ is invertible, whence $F$ is in the essential image of $i_*$. The converse follows from the fact that $D$ is reduced, and has quasi-admissible base change, so $j^* i_*$ factors through a right adjoint functor from $D(\emptyset) = \pt$.

		Next we will consider the property that $i_*$ is fully faithful. Note that since $j^* i_*$ is terminal, \Cref{lem:gluing at kernel} says that since $\square_i \circ i_*$ is coCartesian, the map $(\id \to i_* i^*)i_*$ is invertible, so by the triangle identities for $i^* \dashv i_*$, we have that $i_* (i^* i_* \to \id)$ is invertible. Thus, if $i_*$ is conservative, then the counit of $i^* \dashv i_*$ is an equivalence, so the right adjoint $i_*$ is fully faithful.

		The sequence of categories is a fibre sequence since the first map is fully faithful with essential image given by the kernel of the second map, \ie{} those $F \in D(S)$ such that $j^* F$ is terminal.

		Finally, since $i_*$ is fully faithful, to show that it preserves weakly contractible colimits, we just need to show that its essential image is closed under weakly contractible colimits in $D(S)$. By our description of the essential image as the kernel of $j^*$, we must show that if $K$ is a weakly contractible simplicial set, and $F : K \to D(S)$ is a functor such that for each $p \in K$, $j^* F(p)$ is terminal, then $j^* \varinjlim F$ is terminal.

		Indeed, if $j^*$ preserves weakly contractible colimits, it suffices to show that $\varinjlim_{p \in K} j^* F(p)$ is terminal, but this is a weakly contractible colimit of terminal objects, so it is terminal.
	\end{proof}
\end{prp}

In view of \Cref{prp:gluing implies loc}, we make the following definition of cohomological closedness:
\begin{defn}
	Say a map $i : Z \to S$ in $\mathcal C$ is $D$-closed if it is closed, $i_* : D(Z) \to D(S)$ is conservative, and $D$ has gluing for $i$.
\end{defn}

The following result can be seen as a generalization of nil-invariance results, \cf{} \cite[Theorem A and Corollary 3.2.6]{locspalg}, \cite[Lemma 2.13]{SixAlgSp}, and \cite[Theorem 3.15]{SixAlgSt}.
\begin{rmk}[Invariance for ``surjective closed immersions''] \label{rmk:nil invar}
	Let $i : Z \to S$ be a $D$-closed map whose complement is initial. Then the map $i^* : D(S) \to D(Z)$ is an equivalence.
	\begin{proof}
		Since $i$ is $D$-closed, and $j : \emptyset \to S$ is a complement of $i$, \Cref{prp:gluing implies loc} says that $i_* : D(Z) \to D(S)$ is fully faithful with essential image given by those $F \in D(S)$ such that $j^*(F)$ is terminal. Since $D$ is reduced, we have that $j^*(F)$ is terminal for all $F \in D(S)$, so $i_*$ is essentially surjective, so it is an equivalence. Therefore, its left adjoint $i^*$ is also an equivalence.
	\end{proof}
\end{rmk}

We have the following locality result for $D$-closedness:
\begin{lem} \label{lem:locality of closedness}
	Suppose that $D$ lifts to a pullback formalism, and let $\{X_k \to S\}_k$ be a quasi-admissible $D$-pseudocover of $S$ such that for each $k$, the base change of $i$ along $X_k \to S$ is $D$-closed. Then $i$ is $D$-closed.
	\begin{proof}
		Since $D$ has quasi-admissible base change, we have a commutative square
		\[
			\begin{tikzcd}
				D(Z) \ar[d] \ar[r, "i_*"] & D(S) \ar[d] \\
				\prod_k D(Z \times_S X_k) \ar[r] & \prod_k D(X_k)
			\end{tikzcd}
		,\]
		and we have assumed that the bottom arrow and the right arrow are conservative. By \Cref{thm:D-topology}, we know that $\{Z \times_S X_k \to Z\}_k$ is a $D$-pseudocover, so the left arrow is also conservative, whence $i_*$ is conservative.

		Finally, we know that $D$ has gluing for $i$ by \Cref{rmk:gluing base change}.
	\end{proof}
\end{lem}

\begin{prp} \label{prp:loc for stab PF}
	% NOTE: if not assuming that $j^*$ has right adjoint, not sure if can deduce that $D(U)$ is stable from fact that $D(S)$ is stable
	If $D(Z), D(U), D(S)$ are stable categories, then the following are equivalent:
	\begin{enumerate}

		\item $i^*, j^*$ are jointly conservative and $i_*$ is fully faithful.

		\item $i$ is $D$-closed.

		\item $D$ has localization for $i$.

	\end{enumerate}
	In this case we also have an exact triangle of endofunctors of $D(S)$:
	\[
		j_\sharp j^* \to \id \to i_* i^*
	.\]
	If $D(Z)$ is presentable, $D(S)$ is locally small, and $j^*$ has a right adjoint $j_*$, then $i_*$ has a right adjoint $i^\sharp$, there is an exact triangle
	\[
		i_* i^\sharp \to \id \to j_* j^*
	,\]
	and $i^\sharp, j^*$ are jointly conservative.
	\begin{proof}
		Quasi-admissible base change implies that both $i^* j_\sharp$ and $j^* i_*$ factor through $D(\emptyset)$ by exact functors. Since $D$ is reduced, $D(\emptyset) = 0$, whence we have that $i^* j_\sharp$ and $j^* i_*$ are zero functors.

		The equivalence then follows from \Cref{rmk:ptd gluing,lem:char ex seq of st cats}, where we use \Cref{lem:qadm mono loc} to see that the left adjoint of $j^*$ is fully faithful. \Cref{rmk:ptd gluing} also gives the first exact triangle.

		To see that $i_*$ has a right adjoint, note that by \Cref{prp:gluing implies loc}, $i_*$ preserves weakly contractible colimits, but it also preserves finite colimits since it is an exact functor of stable categories, so it preserves all small colimits. Thus, \cite[Corollary 5.5.2.9 and Remark 5.5.2.10]{htt} show that $i_*$ has a right adjoint. We obtain the second exact triangle by taking right adjoints.
		% PERF: taking adjoints preserves exactness
		
		Finally, since $i_* i^\sharp \to \id \to j_* j^*$ is exact, it follows that $i^\sharp,j^*$ are jointly conservative by the 5-lemma (\eg{} \stackscite{014A}).
	\end{proof}
\end{prp}

\begin{rmk}[Constructible separation] \label{rmk:constructible sep}
	Inspired by \Cref{prp:loc for stab PF}, we can define the $D$-constructible topology on $\mathcal C$ to be the smallest topology such that the empty sieve covers the initial object, and if $i$ is a $D$-closed map with complement $j$, then $\{i,j\}$ is a covering family. If $D$ takes values in stable categories, then any $D$-constructible cover is a $D$-pseudocover.
\end{rmk}

\subsection{Properties of cohomologically closed maps} \label{S:coh cl}

We continue to assume \Cref{ass:gluing}.

The main result of this section is
\begin{thm} \label{thm:closed}
	Let $\widehat{\Cat}_{\text{pointed}} \subseteq \widehat{\Cat}$ be the subcategory consisting of pointed categories and functors between them that preserves zero objects. Suppose $D$ takes values in $\widehat{\Cat}_{\text{pointed}}$, and $i$ is $D$-closed.
	\begin{enumerate}

		\item\label{itm:closed/morphism}
			Suppose $\phi : D \to D'$ is a transformation of reduced $\widehat{\Cat}_{\text{pointed}}$-valued presheaves with quasi-admissible base change, and that $\phi$ respects quasi-admissibility. If $D'$ has gluing for $i$, then $\phi$ is right adjointable at $i$, and the converse holds if $D'(S)$ is generated by the essential image of $\phi : D(S) \to D'(S)$ under colimits that are preserved by $i_* : D'(Z) \to D'(S)$.

		\item If $D$ lifts to a presheaf $\mathcal C^\op \to \CAlg(\widehat{\Cat}_{\text{pointed}})$\footnotemark{} that satisfies the quasi-admissible projection formula, then $D$ has the right projection formula for $i$.
			\footnotetext{The category $\CAlg(\widehat{\Cat}_{\text{pointed}})$ is the category of symmetric monoidal categories $\mathcal C$ such that $\mathcal C$ has a zero object $0$, and $0 \otimes -$ is the constant functor with value $0$.}

		\item If all quasi-admissible base changes of $i$ are $D$-closed, then $D$ has quasi-admissible exchange for $i$.

		\item If all base changes of $i$ are $D$-closed, then $D$ has right base change for $i$.

	\end{enumerate}
	\begin{proof}
		Follows immediately from \Cref{prp:gluing morphisms,prp:closed bc}.
	\end{proof}
\end{thm}

Before turning our attention to the proofs of \Cref{prp:closed bc,prp:gluing morphisms} needed to establish this \lcnamecref{thm:closed}, we consider the following important consequence:

\begin{cor} \label{cor:closed immersions are proper}
	If $D$ is a pointed reduced pullback formalism, and $i$ is $D$-closed, then the functor $i_*$ has a right adjoint $i^\sharp$, and if all base changes of $i$ are also $D$-closed, then $i$ is $D$-quasi-proper.
	\begin{proof}
		Since $i$ is a $D$-closed map, the existence of a right adjoint of $i_*$ follows from \Cref{prp:gluing implies loc} since $D$ takes values in pointed presentable categories. When all base changes of $i$ are $D$-closed, we conclude by \Cref{thm:closed}.
	\end{proof}
\end{cor}

Now we address the proof of \Cref{thm:closed}.

\begin{prp} \label{prp:gluing morphisms}
	Let $\phi : D \to D'$ be a morphism of reduced presheaves with quasi-admissible base change on $\mathcal C$, and assume that $i$ is $D$-closed, and $\phi : D(U) \to D'(U)$ preserves terminal objects. If $D'$ has gluing for $i$, then $\phi$ is right adjointable at $i$.

	Conversely, if $\phi$ is right adjointable at $i$, then $D'$ has gluing for $i$ at any object in the essential image of $\phi : D(S) \to D'(S)$.
	\begin{proof}
		The square
		\[
			\begin{tikzcd}
				D(S) \ar[d, "\phi"'] \ar[r, "i^*"] & D(Z) \ar[d, "\phi"] \\
				D'(S) \ar[r, "i^*"'] & D'(Z)
			\end{tikzcd}
		\]
		is horizontally right adjointable if and only if the composite
		\[
			\phi i_* \to i_* i^* \phi i_* \simeq i_* \phi i^* i_* \to i_* \phi
		\]
		is invertible. Since $i$ is $D$-closed, $i_*$ is fully faithful, so the map $i^* i_* \to \id$ is invertible, so this is equivalent to the condition that
		\[
			\phi i_* \to i_* i^* \phi i_*
		\]
		is invertible.

		Let $j : U \to S$ be a quasi-admissible complement of $i$. Then
		\[
			j^* \phi i_* \simeq \phi j^* i_* \simeq \phi \pt \simeq \pt
		,\]
		where the last equivalence follows from our assumption that $\phi : D(U) \to D'(U)$ preserves terminal objects.

		Thus, by \Cref{lem:gluing at kernel}, if $D'$ has gluing for $i$, then the unit $\phi i_* \to i_* i^* \phi i_*$ is invertible, so that the square is horizontally right adjointable.

		Conversely, if the square is horizontally right adjointable, then $\phi \circ \square^D_i \simeq \square^{D'}_i \circ \phi$ by \Cref{lem:desc reduced gluing}, so since $\phi$ preserves colimits, and $\square^D_i$ is coCartesian, so is $\square^{D'}_i \circ \phi$.
	\end{proof}
\end{prp}

\begin{prp}[Closed base change] \label{prp:closed bc}
	Let
	\[
		\begin{tikzcd}
			Z' \ar[d, "\bar \sigma"'] \ar[r, "i'"] & S' \ar[d, "\sigma"] & \ar[l, "j'"'] \ar[d, "\sigma_U"] U' \\
			Z \ar[r, "i"'] & S & \ar[l, "j"] U
		\end{tikzcd}
	\]
	be a Cartesian square in $\mathcal C$, and assume $i$ is $D$-closed.
	\begin{enumerate}

		\item Suppose $\sigma_U^*$ preserves terminal objects. If $D$ has gluing for $i'$, then $D$ has right base change for $i$ against $\sigma$, and the converse holds if $\sigma^*$ is essentially surjective.

		\item If $D$ has gluing for $i'$, $\sigma$ is quasi-admissible, and $(\sigma_U)_\sharp$ preserves terminal objects, then $D$ has right-left base change for $i$ against $\sigma$.

		\item Suppose that $D$ takes values in $\CAlg(\widehat{\Cat})$, and satisfies the quasi-admissible projection formula. If $j^* M \otimes \pt \simeq \pt$ for any $M \in D(S)$, then $D$ has the right projection formula for $i$.

	\end{enumerate}

	\begin{proof}\hfill
		\begin{enumerate}

			\item This follows from \Cref{prp:gluing morphisms,exa:base change morph}.

			\item This follows from \Cref{prp:gluing morphisms,exa:sharp morph}.

			\item This follows from \Cref{prp:gluing morphisms,exa:tensor morph}.

		\end{enumerate}

	\end{proof}
\end{prp}

\subsection{Excision and Descent} \label{S:excision}

The notion of $D$-closedness allows us to consider ``excision'' for $D$, as well as some versions of ``cdh'' covers. We will now show some  results about cdh descent and excision that follow from $D$-closedness.

We continue to assume \Cref{ass:gluing}.

It will be useful to fix a wide subcategory $\mathcal C^\sharp$ of $\mathcal C$ such that for any morphism $g$ in $\mathcal C^\sharp$, the functor $g^*$ admits a right adjoint $g_*$ that admits a further right adjoint $g^\sharp$. We write $D^\sharp : (\mathcal C^\sharp)^\op \to \widehat{\Cat}$ for the presheaf that sends any map $g$ to the functor $g^\sharp$. Note that by \Cref{prp:gluing implies loc}, if $D(Z)$ and $D(S)$ has zero objects, and $i$ is $D$-closed, then $i_*$ admits a right adjoint, so we can choose $\mathcal C^\sharp$ to contain $i$.

\begin{prp} \label{prp:elementary cdh descent}
	Assume that $\mathcal C$ is locally small, that $D(Z)$, $D(S)$, and $D(U)$ are stable categories, that $i$ is $D$-closed, and let $\{p_k : X_k \to S\}$ be a small family of maps in $\mathcal C$. For each $k$, write $p_k^Z : X_k \times_S Z \to Z$ for the base change of $p_k$ along $i : Z \to S$, and write $p_k^U : X_k \times_S U \to U$ for the base change of $p_k$ along $j : U \to S$.
	\begin{enumerate}

		\item If $p_k$ is quasi-admissible for each $k$, and $\{p_k \times_S Z : X_k \times_S Z \to Z\}_k$ is a $D$-pseudocover, then $\{p_k\}_k \cup \{j\}$ is a $D$-pseudocover. If $D$ takes values in categories that admit small colimits, then $D$ has descent along the family $\{p_k\}_k \cup \{j\}$.

		\item If $\{p_k \times_S U : X_k \times_S U \to U\}_k$ is a $D$-pseudocover, then $\{p_k\}_k \cup \{i\}$ is a $D$-pseudocover. If $D$ takes values in categories that admit small limits, and $D$ has right base change for all base changes of $p_k$ for all $k$, and all base changes of $i$, then $D$ descends along $\{p_k\}_k \cup \{i\}$.

		\item Assume that $D$ takes values in presentable categories, $i \in \mathcal C^\sharp$, and that for all $k$, the map $p_k$ is in $\mathcal C^\sharp$, and $D$ has quasi-admissible exchange for $p_k$. If $\{p_k \times_S U : X_k \times_S U \to U\}_k$ is a $D^\sharp$-pseudocover, then $\{i\} \cup \{p_k\}_k$ is a $D^\sharp$-pseudocover.

			Furthermore, if $D$ has right base change for all maps in $\mathcal C^\sharp$, $\mathcal C^\sharp$ admits pullbacks, and the inclusion into $\mathcal C$ preserves pullbacks, then $D^\sharp$ descends along $\{i\} \cup \{p_k\}_k$.

	\end{enumerate}
	
	\begin{proof}
		For each $k$, we have a Cartesian square
		\[
			\begin{tikzcd}
				Z_k \ar[d, "\bar p_k"'] \ar[r, "i_k"] & X_k \ar[d, "p_k"] & \ar[l, "j_k"'] U_k \ar[d, "\mathring p_k"] \\
				Z \ar[r, "i"'] & S & \ar[l, "j"] U
			\end{tikzcd}
		.\]
		Note that by \Cref{prp:loc for stab PF}, both $\{i^*,j^*\}$ and $\{i^\sharp,j^*\}$ are jointly conservative.
		\begin{enumerate}

			\item Since $\{\bar p_k^*\}_k$ is jointly conservative, and for each $k$, $\bar p_k^* i^* \simeq i_k^* p_k^*$, it follows that a map $f \in D(S)$ satisfies that $i^* f$ is invertible if $p_k^* f$ is invertible for all $k$. Thus, $\{p_k^*\}_k \cup \{j^*\}$ is jointly conservative since $\{i^*,j^*\}$ is jointly conservative, so it follows from \Cref{thm:bc descent} that since $D$ has quasi-admissible base change, it descends along the family $\{p_k\}_k \cup \{j\}$.

			\item For each $k$, since $j_k^* p_k^* \simeq \mathring p_k^* j^*$, and $\{\mathring p_k^*\}_k$ is jointly conservative, it follows that a map $f$ in $D(X)$ satisfies that $j^* f$ is an equivalence if $p_k^* f$ is an equivalence for all $k$. Thus, $\{p_k^*\}_k \cup \{i^*\}$ is jointly conservative since $\{j^*, i^*\}$ is jointly conservative. By the dual of \Cref{thm:bc descent} (where the collection $Q$ consists of all base changes of maps of the form $p_k$ or $i$), we find that $D$ descends along $\{p_k\}_k \cup \{i\}$.

			\item Note that by \Cref{lem:bc prop map by qadm mono}, we have that the right adjoint $(\mathring p_k)_*$ of $\mathring p_k^*$ exists, and admits a further right adjoint $\mathring p_k^\sharp$.

				Since $D$ has quasi-admissible exchange for $p_k$, we know that $j_k^* p_k^\sharp \simeq \mathring p_k^\sharp j^*$. Since $\{\mathring p_k^\sharp\}_k$ is jointly conservative, it follows that a map $f$ in $D(X)$ satisfies that $j^* f$ is an equivalence if $p_k^\sharp f$ is an equivalence for all $k$. Thus, $\{p_k^\sharp\}_k \cup \{i^\sharp\}$ is jointly conservative since $\{j^*, i^\sharp\}$ is jointly conservative.

				If $D$ has right base change for all maps in $\mathcal C^\sharp$, and $\mathcal C^\sharp \to \mathcal C$ preserves pullbacks, it follows that $D^\sharp$ has left base change for all maps. Thus, by \Cref{thm:bc descent}, it follows that $D^\sharp$ has descent along $\{p_k\}_k \cup \{i\}$.

		\end{enumerate}
		
	\end{proof}
\end{prp}

\begin{prp} \label{prp:gluing implies exc}
	Assume $D$ takes values in stable categories, and let
	\[
		\begin{tikzcd}
			X_Z \ar[d, "\bar p"'] \ar[r, "i_X"] & X \ar[d, "p"] & \ar[l, "j_X"'] X_U \ar[d, "\mathring p"] \\
			Z  \ar[r, "i"'] & S & \ar[l, "j"] U
		\end{tikzcd}
	\]
	be Cartesian squares in $\mathcal C$, where $i, i_X$ are $D$-closed, and $j$ is a complement of $i$.

	\begin{enumerate}

		\item If $p$ is quasi-admissible, and $\bar p^*$ is an equivalence, then $D$ sends the right square to a Cartesian square.

		\item If $D$ has right base change and quasi-admissible exchange for $p$, and $\mathring p^*$ is an equivalence, then $D$ sends the left square to a Cartesian square. If the left square is in $\mathcal C^\sharp$, then $D^\sharp$ sends it to a Cartesian square.

	\end{enumerate}
	\begin{proof}
		We follow the strategy of the proof of \cite[Proposition 6.24]{sixopsequiv}.

		By \Cref{prp:loc for stab PF}, we know that $i^*, j^*$ are jointly conservative. If $\bar p^*$ is an equivalence, it follows that $p^*, j^*$ are jointly conservative, and if $\mathring p^*$ is an equivalence, it follows that $i^*, p^*$ are jointly conservative.

		\Cref{prp:loc for stab PF} also shows that $i^\sharp, j^*$ are jointly conservative. Note that if $\mathring p^*$ is an equivalence, then $\mathring p_*$ is an equivalence, so it admits a right adjoint $\mathring p^\sharp$. Since $D$ has right-left base change for $p$ against $j$, it follows that $i^\sharp, p^\sharp$ are jointly conservative.

		Thus, by \cite[Proposition 5.2.2.36]{ha}, it suffices to show that if $E_X \in D(X)$, then
		\begin{enumerate}

			\item in the first case: for any $E_U \in D(U)$ and equivalence $\mathring p^* E_U \simeq j_X^* E_X$, the natural maps
				\begin{align}
					\label{eqn:p qadm exc}
					p^*(j_* E_U \bigtimes_{j_* \mathring p_* \mathring p^* E_U} p_* E_X) &\to E_X \\
					\label{eqn:j qadm exc}
					j^*(j_* E_U \bigtimes_{j_* \mathring p_* \mathring p^* E_U} p_* E_X) &\to E_U
				\end{align}
				are equivalences;

			\item in the first statement of the second case: for any $E_Z \in D(Z)$ and equivalence $p_Z^* E_Z \simeq i_X^* E_X$, the natural maps
				\begin{align}
					\label{eqn:p proper exc}
					p^*(i_* E_Z \bigtimes_{i_* \bar p_* \bar p^* E_Z} p_* E_X) &\to E_X \\
					\label{eqn:i proper exc}
					i^*(i_* E_Z \bigtimes_{i_* \bar p_* \bar p^* E_Z} p_* E_X) &\to E_Z
				\end{align}
				are equivalences.

			\item in the second statement of the second case: for any $E_Z \in D(Z)$ and equivalence $\bar p^\sharp E_Z \simeq i_X^\sharp E_X$, the natural maps
				\begin{align}
					\label{eqn:sharp p proper exc}
					E_X &\to p^\sharp(i_* E_Z \coprod_{i_* \bar p_* \bar p^\sharp E_Z} p_* E_X) \\
					\label{eqn:sharp i proper exc}
					E_Z &\to i^\sharp(i_* E_Z \coprod_{i_* \bar p_* \bar p^\sharp E_Z} p_* E_X)
				\end{align}
				are equivalences.

		\end{enumerate}

		The arguments are analogous, so we will only consider the first case, and the second statement of the second case.

		\textbf{First case:}
		To show \eqref{eqn:j qadm exc} is invertible, note that by \Cref{lem:qadm mono loc}, we have that $j^* j_* \to \id$ is an equivalence, so we just need to show that
		\[
			E_U \times_{\mathring p_* \mathring p^* E_U} j^* p_* E_X \to E_U
		\]
		is an equivalence, but since $D$ has left base change for $j$ against $p$, $j^* p_* E_X \to \mathring p_* \mathring p^* E_U$ is equivalent to $\mathring p_* j_X^* E_X \to \mathring p_* \mathring p^* E_U)$, which is $\mathring p_*$ of the equivalence $j_X^* E_X \to \mathring p^* E_U)$.

		To show \eqref{eqn:p qadm exc} is an equivalence, it suffices to show that it is an equivalence after applying $j_X^*$ and $i_X^\sharp$, since by \Cref{prp:loc for stab PF}, $j_X^*, i_X^\sharp$ are jointly conservative. Indeed, $j_X^*$ of \eqref{eqn:p qadm exc} is $\mathring p^*$ of \eqref{eqn:j qadm exc}, which we already showed was an equivalence. To consider the case of $i_X^\sharp$, first note that $i_X^\sharp$ preserves pullbacks since it is an exact functor, and by \Cref{prp:closed bc}, we have that $\bar p^* i^\sharp \to i_X^\sharp p^*$ is an equivalence and $0 \simeq (\emptyset \to X_Z)_* (\emptyset \to U)^\sharp \to i^\sharp j_*$ is an equivalence (where we use \Cref{lem:gluing for acyclic} to see that $\emptyset \to U$ is $D$-closed), so we have equivalences
		\[
			i_X^\sharp p^* j_* \simeq \bar p^* i^\sharp j_* \simeq \bar p^* (\emptyset \to X_Z)_* (\emptyset \to U)^\sharp \simeq \bar p^* 0 \simeq 0
		.\]
		Thus, $i_X^\sharp$ of \eqref{eqn:p qadm exc} is
		\[
			i_X^\sharp p^* p_* E_X \to i_X^\sharp E_X
		,\]
		which is equivalent to
		\[
			\bar p^* \bar p_* i_X^\sharp E_X \to i_X^\sharp E_X
		\]
		by \Cref{prp:closed bc}. Since $\bar p^*$ is an equivalence, the counit $\bar p^* \bar p_* \to \id$ is an equivalence, so this map is an equivalence.

		\textbf{Second statement of the second case:}
		To show \eqref{eqn:sharp i proper exc} is invertible, note that since $i_*$ is fully faithful, we have that $\id \to i^\sharp i_*$ is an equivalence, so -- since $i^\sharp$ preserves pushouts as an exact functor -- we just need to show that
		\[
			E_Z \to E_Z \coprod_{\bar p_* \bar p^\sharp E_Z} i^\sharp p_* E_X
		\]
		is an equivalence, but since $D$ has right base change for $i$ against $p$ by \Cref{prp:closed bc}, $\bar p_* \bar p^\sharp E_Z \to i^\sharp p_* E_X$ is equivalent to $\bar p_* \bar p^\sharp E_Z \to \bar p_* i_X^\sharp E_X$, which is $\bar p_*$ of the equivalence $\bar p^\sharp E_Z \to p_X^\sharp E_X)$.

		To show \eqref{eqn:sharp p proper exc} is an equivalence, it suffices to show that it is an equivalence after applying $j_X^*$ and $i_X^\sharp$, since by \Cref{prp:loc for stab PF}, $j_X^*, i_X^\sharp$ are jointly conservative. Indeed, $i_X^\sharp$ of \eqref{eqn:sharp p proper exc} is $\bar p^\sharp$ of \eqref{eqn:sharp i proper exc}, which we already showed was an equivalence. To consider the case of $j_X^*$, first note that $j_X^*$ preserves coproducts since it is an exact functor, and since $D$ has right-left base change for $p$ against $j$, we have that $j_X^* p^\sharp \to \mathring p^\sharp j^*$ is an equivalence, so we have equivalences
		\[
			0 \simeq \mathring p^\sharp 0 \simeq \mathring p^\sharp (\emptyset \to U)_* (\emptyset \to Z)^* \simeq \mathring p^\sharp j^* i_* \simeq j_X^* p^\sharp i_*
		.\]
		Thus, $j_X^*$ of the \eqref{eqn:sharp p proper exc} is
		\[
			j_X^* E_X \to j_X^* p^\sharp p_* E_X
		,\]
		which is equivalent to
		\[
			j_X^* E_X \to \mathring p^\sharp \mathring p_* j_X^* E_X
		\]
		since $D$ has right-left base change and right base change for $p$ against $j$, where the fact about right base change follows from the fact that $j$ is quasi-admissible and $D$ has quasi-admissible base change. Since $\mathring p^*$ is an equivalence, the unit $\id \to \mathring p^\sharp \mathring p_*$ is an equivalence, so this map is an equivalence.
	\end{proof}
\end{prp}

\section{Duality and Ambidexterity} \label{S:duality}

% PERF: talk more about suave maps?
In Mann's framework for 6-functor formalisms, we are given a category $\mathcal C$, along with a certain class of morphisms $E$ in $\mathcal C$ for which the exceptional operations should be defined. One notes (\cite[Remark 3.2.4]{HM6FF}) that this is not enough structure to express Poincar\'{e} duality, for which it would be necessary to have a notion of ``smooth map''. Nevertheless, for any 6-functor formalism, it is possible to define a notion of cohomological smoothness for maps such that the cohomologically smooth maps satisfy Poincar\'{e} duality. This is captured by the notion of $D$-suave maps for a 6-functor formalism $D$ -- see \Cref{S:suave prim}.

% PERF: reference to equation looks bad
The approach to categorical invariants given by pullback formalisms instead takes the ``smooth'' maps as the fundamental geometric input (given by the quasi-admissibility structure). Nevertheless, pullback formalisms do not have enough structure to define Borel-Moore homology, which is necessary to express Poincar\'{e} duality in full generality (as in \eqref{eqn:duality}). It turns that analogously to the case of 6-functor formalisms, we can still define a notion of Atiyah duality or ambidexterity for maps (\Cref{defn:duality}) and we can show (\Cref{thm:duality}) that under some mild conditions, the quasi-admissible maps that have Atiyah duality are precisely the ones that have good cohomological properness properties.

Note that duality properties for cohomology theories are often given by a form of twisted ambidexterity: they say that the right and left adjoints of some functor agree up to a twist. We will see how to interpret this as a version of Poincar\'{e}/Serre duality in \Cref{rmk:BM duality}.
% PERF: Explain; give examples of Poincar\'{e} duality and Serre duality

\subsection{Thom twists}
Throughout this section, $\mathcal C$ will denote a pullback context, and $D : \mathcal C^\op \to \widehat{\Cat}$ is a presheaf that respects quasi-admissibility.

We will begin by defining the relevant notion of ``twists''.

\begin{defn} \label{defn:Thom twist}
	If $S \in \mathcal C$, and $X$ is a pointed object of $\mathcal C_S$ such that the pointing $0 : S \to X$ satisfies that $0^*$ has a right adjoint $0_*$, define the \emph{Thom twist} of $X$ to be the endomorphism $\Sigma^X$ of $D(S)$ given by $(X \to S)_\sharp 0_*$.

	If $D$ takes values in $\Alg_{E_0}(\widehat{\Cat})$, define the Thom class of $X$ to be
	\[
		\Th_S^D X = \Th_S X = \Th X \coloneqq (X \to S)_\sharp 0_* 1
	.\]

	We write $\Sigma^{-X}$ to denote a right adjoint of $\Sigma^X$, if it exists. Say $X$ is \emph{$D$-stable} if $\Sigma^X$ is invertible.
\end{defn}

\begin{rmk} \label{rmk:morph twist}
	Let $\phi : D \to D'$ be a transformation of presheaves that respects quasi-admissibility, let $S \in \mathcal C$, and let $X$ be a pointed object of $\mathcal C_S$ such that the pointing $0 : S \to X$ satisfies that $0^*$ has a right adjoint $0_*$ for both $D$ and $D'$. Then there is a canonical map 
	\begin{equation} \label{eqn:morph twist}
		\phi \Sigma^X \to \Sigma^X \phi
	\end{equation}
	given as the composite
	\[
		\phi (X \to S)_\sharp 0_* \xleftarrow{\sim} (X \to S)_\sharp \phi 0_* \to (X \to S)_\sharp 0_* \phi
	,\]
	where the first map is induced by the left mate $(X \to S)_\sharp \phi \to \phi (X \to S)_\sharp$, and the last map is induced by the right mate $\phi 0_* \to 0_* \phi$.

	If $\phi$ lifts to a transformation of presheaves $\mathcal C^\op \to \Alg_{E_0}(\widehat{\Cat})$, this defines a map
	\[
		\phi \Th_S^D X \to \Th_S^{D'} X
	.\]
\end{rmk}

\begin{lem}[Basic properties of Thom twists] \label{lem:Thom twist properties}
	If $D$ has quasi-admissible base change, $S \in \mathcal C$, and $X$ is a pointed object of $\mathcal C_S$ with pointing given by $0 : S \to X$, we have the following properties of Thom twists:
	\begin{enumerate}

		\item Additivity: If $Y$ is also a pointed object of $\mathcal C_S$, then if $D$ has quasi-admissible exchange for $0$,
			\[
				\Sigma^X \Sigma^Y \simeq \Sigma^{X \times_S Y}
			,\]
			where the pointing of $X \times_S Y$ is induced by the pointing of $X,Y$. In particular,
			\[
				\Sigma^X \Sigma^Y \simeq \Sigma^Y \Sigma^X
			.\]

		\item Compatibility with equivalences: if $X$ and $Y$ are equivalent pointed objects of $\mathcal C_S$, then there is an equivalence $\Sigma^X \simeq \Sigma^Y$.

		\item Suppose that $D$ takes values in $\PrL$. If $0$ is a $D$-closed map, then $\Sigma^X$ preserves weakly contractible colimits.

		\item\label{itm:morph twist}
			If $\phi  : D \to D'$ is a transformation that respects quasi-admissibility, there is a natural map
			\[
				\phi \Sigma^X \to \Sigma^X \phi
			,\]
			which is invertible if $\phi$ is right adjointable at the pointing of $X$.

		\item Let $f : T \to S$ be a map in $\mathcal C$. We write $f^{-1}(X)$ for the pointed object of $\mathcal C_T$ given by base change along $f$.
			\begin{enumerate}

				\item There is a natural map
					\[
						f^* \Sigma^X \to \Sigma^{f^{-1}(X)} f^*
					\]
					which is invertible if $D$ has right base change for $0 : S \to X$ against $X \times_S f$.
					% In particular, this holds if $f$ is quasi-admissible or $0$ is universally $D$-closed and $D$ is reduced and takes values in pointed categories (by \Cref{thm:closed}).

				\item If $f$ is a quasi-admissible map, then there is a natural map
					\[
						f_\sharp \Sigma^{f^{-1}(X)} \to \Sigma^X f_\sharp
					,\]
					which is invertible if $D$ has right-left base change for the pointing $0 : S \to X$ against $f$.

				% \item There is a natural map
				% 	\[
				% 		\Sigma^X f_* \to f_* \Sigma^{f^{-1}(X)}
				% 	,\]
				% 	which is invertible if $D$ has right-left base change for $f$ against $X \to S$.

			\end{enumerate}
			
	\end{enumerate}
	\begin{proof}\hfill
		\begin{enumerate}

			\item Write $0$ for the pointing map of any pointed object of $\mathcal C_S$. Write $p,q,r$ for the structure maps to $S$ of $X,Y, X \times_S Y$ respectively. Then
				\[
					r \simeq p \circ (X \times_S q) \quad\text{and}\quad 0 = (0 \times_S Y) \circ 0
				,\]
				and if we assume without loss of generality that $D$ has quasi-admissible exchange for the pointing $0 : S \to X$ of $X$, then since
				\[
					\begin{tikzcd}
						Y \ar[d, "q"'] \ar[r, "0 \times_S Y"] & X \times_S Y \ar[d, "X \times_S q"] \\
						S \ar[r, "0"'] & X
					\end{tikzcd}
				\]
				is Cartesian,
				\[
					(X \times_S q)_\sharp (0 \times_S Y)_* \simeq 0_* q_\sharp
				\]
				so
				\[
					\Sigma^{X \times_S Y} = r_\sharp 0_* \simeq p_\sharp (X \times_S q)_\sharp (0 \times_S Y)_* 0_* \simeq p_\sharp 0_* q_\sharp 0_* = \Sigma^X \Sigma^Y
				,\]
				as desired.

			\item We write $0$ for the pointing of any object in $\mathcal C_S$, and pick some equivalence $\alpha : X \to Y$.

				Note that since $\alpha$ is an equivalence, it is quasi-admissible, and since $0 \simeq \alpha \circ 0$, the square
				\[
					\begin{tikzcd}
						S \ar[d, "0"'] \ar[r, "\id"] & S \ar[d, "0"] \\
						X \ar[r, "\alpha"'] & Y
					\end{tikzcd}
				\]
				commutes. Since $\alpha$ is an equivalence, the square is Cartesian, so we have
				\[
					\alpha^* 0_* \simeq 0_*
				,\]
				whence the counit $\alpha_\sharp \alpha^* \to \id$ induces
				\[
					\Sigma^X = p_\sharp 0_* \simeq q_\sharp \alpha_\sharp \alpha^* 0_* \to q_\sharp 0_* = \Sigma^Y
				,\]
				but since $\alpha^*$ is an equivalence, the counit is invertible, so this composite is an equivalence.

			\item This follows from \Cref{prp:gluing implies loc}, since $0_*$ preserves weakly contractible colimits.

			\item This follows immediately from \Cref{rmk:morph twist}.

			\item This follows easily from \cref{itm:morph twist} and \Cref{exa:tensor morph,exa:base change morph}.
				% Alternatively, we may argue as follows:
				%
				% Given Cartesian squares
				% \[
				% 	\begin{tikzcd}
				% 		T \ar[d, "0"'] \ar[r, "f"] & S \ar[d, "0"] \\
				% 		Y \ar[d, "q"'] \ar[r, "\tilde f"] & X \ar[d, "p"] \\
				% 		T \ar[r, "f"'] & S
				% 	\end{tikzcd}
				% ,\]
				% the natural maps (where we assume $f$ is quasi-admissible in the last case)
				% \[
				% 	\begin{aligned}
				% 		p_\sharp \tilde f_* &\to f_* q_\sharp \\
				% 		\tilde f^* 0_* &\to 0_* f^* \\
				% 		\tilde f_\sharp 0_* &\to 0_* f_\sharp
				% 	\end{aligned}
				% 	\quad\text{give}\quad
				% 	\begin{aligned}
				% 		\Sigma^X f_* \simeq p_\sharp 0_* f_* \simeq p_\sharp \tilde f_* 0_* &\to f_* q_\sharp 0_* \simeq f_* \Sigma^Y \\
				% 		f^* \Sigma^X \simeq f^* p_\sharp 0_* \simeq q_\sharp \tilde f^* 0_* &\to q_\sharp 0_* f^* \simeq \Sigma^Y f^* \\
				% 		f_\sharp \Sigma^Y \simeq f_\sharp q_\sharp 0_* \simeq p_\sharp \tilde f_\sharp 0_* &\to p_\sharp 0_* f_\sharp \simeq \Sigma^X f_\sharp.
				% 	\end{aligned}
				% \]
				% So
				% \begin{enumerate}
				%
				%% 	\item if $D$ has right-left base change for $f$ against $p$, then $\Sigma^X f_* \to f_* \Sigma^Y$ is an equivalence,
				%
				% 	\item if $D$ has right base change against $\tilde f$ for $0 : S \to X$, then $f^* \Sigma^X \to \Sigma^Y f^*$ is an equivalence, and
				%
				% 	\item if $D$ has right-left base change for $0 : S \to X$ against $f$, then $f_\sharp \Sigma^Y \to \Sigma^X f_\sharp$ is an equivalence.
				%
				% \end{enumerate}

		\end{enumerate}
	\end{proof}
\end{lem}

\subsubsection{Monoidal Thom twists}
We will now study Thom twists in the presence of monoidal structures. Our main application is the following.
\begin{prp} \label{prp:monoidal twists}
	Let $D : \mathcal C^\op \to \CAlg(\widehat{\Cat})$ be a pointed reduced presheaf that has quasi-admissible base change and satisfies the quasi-admissible projection formula, let $S \in \mathcal C$, and let $X$ be a pointed object of $\mathcal C_S$, where the pointing $0 : S \to X$ is $D$-closed.
	\begin{enumerate}

		\item\label{itm:monoidal twists/desc} We have an equivalence $\Sigma^X \simeq - \otimes \cls{X}/\cls{X \setminus S}$, where $X \setminus S$ denotes the complement of $0$.

			% NOTE: it's okay that this statement is complicated -- simpler statement for pullback formalisms given in outline
			% PERF: actually I removed it. Should I put it back? {prp:twists}
		\item\label{itm:monoidal twists/morph}
			Let $\phi : D \to D'$ be a transformation that respects quasi-admissibility and preserves initial objects pointwise, where $D'$ is also reduced and pointed, has quasi-admissible base change, and satisfies the quasi-admissible projection formula. If $0$ is $D'$-closed, then there is an equivalence $\phi \Sigma^X \simeq \Sigma^X \phi$, and if $X$ is $D$-stable, then $X$ is $D'$-stable. If $D(S)$ is closed monoidal, $\phi : D(S) \to D'(S)$ is conservative and admits a linear left adjoint, and $X$ is $D'$-stable, then $X$ is $D$-stable.

		\item\label{itm:monoidal twists/bc} Suppose that for all quasi-admissible maps $S' \to S$, the base change of $0$ along $X \times_S (S' \to S)$ is $D$-closed. Then any quasi-admissible base change of $X$ is $D$-stable if $X$ is $D$-stable, and conversely, if $D(S)$ is closed monoidal, and $\{S_i \to S\}_i$ is a quasi-admissible $D$-pseudocover of $X$ such that the base change of $X$ along $S_i \to S$ is $D$-stable for all $i$, then $X$ is $D$-stable.

	\end{enumerate}
\end{prp}

Before turning to the proof of \Cref{prp:monoidal twists}, we will consider when invertibility of Thom twists implies stability in the sense of \cite[Definition 1.1.1.9]{ha}.

\begin{lem} \label{lem:auto stable}
	Suppose $D$ is a pointed reduced pullback formalism. Let $S \in \mathcal C$, and $X$ a $D$-stable pointed object of $\mathcal C_S$ such that the map $S \to X$ is $D$-closed. If $X \to S$ is $D$-acyclic, and $X \setminus S \to S$ admits a section, then $D(S)$ is a stable category.
	\begin{proof}
		By \Cref{thm:closed}, we can apply \Cref{lem:Thom twist descriptions} to see that $\Sigma^X \simeq \otimes \cls{X}/\cls{X \setminus S}$. Since $X \to S$ is $D$-acyclic, we have that $\cls{X} \simeq 1$, and since $X$ is $D$-stable as a pointed object of $\mathcal C_S$, we find that $1/\cls{X \setminus S}$ is $\otimes$-invertible. Thus, we conclude by \cite[Lemma 4.2.4]{Fundamentals}.
	\end{proof}
\end{lem}

The next result allows us to understand Thom twists as twists in the sense that they are given by tensoring by some object.
\begin{lem} \label{lem:Thom twist descriptions}
	In the situation of \Cref{defn:Thom twist}, assume $D : \mathcal C^\op \to \CAlg(\widehat{\Cat})$ satisfies the quasi-admissible projection formula.
	\begin{enumerate}

		\item If $D$ has the right projection formula for $0$, then $\Sigma^X \simeq \Th X \otimes -$.

		\item If $D$ is reduced and takes values in pointed categories, and $0$ is a closed map for which $D$ has gluing, then $\Th X \simeq \cls{X}/\cls{X \setminus S}$.

	\end{enumerate}
	\begin{proof}\hfill
		\begin{enumerate}

			\item We have
				\[
					p_\sharp 0_* 1 \otimes - \xleftarrow{\sim} p_\sharp(0_* 1 \otimes p^*(-)) \xrightarrow{\sim} p_\sharp(0_*(1 \otimes 0^* p^*(-))) \simeq p_\sharp 0_* (-)
				,\]
				where the first equivalence is from the left projection formula for $p$, the second from the right projection formula for $0$, and the third is from $1 \otimes - \simeq \id$ and $0^* p^* \simeq (p0)^* \simeq \id$.

			\item By \Cref{rmk:ptd gluing}, we have a natural equivalence
				\[
					0_* 0^* 1 \simeq 1/(1 \otimes \cls{X \setminus S}) = 1/\cls{X \setminus S}
				.\]
				In particular, since $p_\sharp$ preserves colimits,
				\[
					p_\sharp 0_* 1 \simeq p_\sharp 0_* 0^* 1 \simeq p_\sharp(1/\cls{X \setminus S}) \simeq \cls{X}/\cls{X \setminus S}
				.\]

		\end{enumerate}
	\end{proof}
\end{lem}

\begin{lem} \label{lem:propagate stability}
	Let $\phi : D \to D'$ be a transformation that respects quasi-admissibility between presheaves $\mathcal C^\op \to \CAlg(\widehat{\Cat})$ that satisfy the quasi-admissible projection formula. Let $S \in \mathcal C$, and let $X$ be a pointed object of $\mathcal C_S$ such that $\phi$ is right adjointable at the pointing $0 : S \to X$, and $D,D'$ have the right projection formula for $0$.

	If $X$ is $D$-stable, then it is $D'$-stable, and the converse holds if $D(S)$ is closed monoidal, and $\phi : D(S) \to D'(S)$ is conservative and has a linear left adjoint.

	In particular, we have the following:
	\begin{enumerate}

		\item Let $X'$ be the base change of $X$ along a quasi-admissible map $S' \to S$, and assume $D$ has the right projection formula for the base change $S' \to X'$ of $0$. If $X$ is $D$-stable, so is $X'$.

		\item Let $\{S_i \to S\}_i$ is a quasi-admissible $D$-pseudocover, and for each $i$, write $X_i$ for the base change of $X$ along $S_i \to S$, and $0_i : S_i \to X_i$ for the pointing of $X_i$. Suppose that for each $i$, $D$ has the right projection formula for $0_i$, and $X_i$ is $D$-stable. If $D(S)$ is closed monoidal, then $X$ is $D$-stable.

	\end{enumerate}
	\begin{proof}
		It follows from \Cref{lem:Thom twist descriptions} applied to both $D$ and $D'$, that $\Sigma^X$ is invertible if and only if $\Th X$ is $\otimes$-invertible. Since $\phi$ is right adjointable at $0$, by \Cref{rmk:morph twist} we have that $\phi \Th^D_S X \simeq \Th^{D'}_S X$, so if $\Th^D_S X$ is $\otimes$-invertible, then so is $\Th^{D'}_S X$, and the converse holds if $D(S)$ is closed monoidal and $\phi : D(S) \to D'(S)$ is conservative and has a linear left adjoint by \cite[Lemma D.2.5]{Fundamentals}.
		
		The last statements follows by applying this to the case that $\phi$ is of the form given in \Cref{exa:base change morph}, and using \cite[Lemma D.2.5]{Fundamentals}.
	\end{proof}
\end{lem}

\begin{proof}[Proof of \Cref{prp:monoidal twists}]\hfill
	\begin{enumerate}

		\item By \Cref{thm:closed}, we know that $D$ has the right projection formula for $0$, so this statement follows from \Cref{lem:Thom twist descriptions}.

		\item By \Cref{thm:closed}, we know that $\phi$ is right adjointable at $0$, so $\phi \Sigma^X \simeq \Sigma^X \phi$ by \Cref{rmk:morph twist}. \Cref{thm:closed} also shows that $D,D'$ have the right projection formula for $0$, so by \Cref{lem:propagate stability}, we have that if $X$ is $D$-stable, it is also $D'$-stable, and the converse holds if $D(S)$ is closed monoidal, and $\phi : D(S) \to D'(S)$ is conservative and admits a linear left adjoint.

		\item This statement follows from the previous point and \Cref{exa:base change morph}.

	\end{enumerate}
	
\end{proof}

\subsection{Tangential Thom twists and Atiyah duality}

Throughout this section, $\mathcal C$ will denote a pullback context, and $D,D'$ will denote presheaves $\mathcal C^\op \to \widehat{\Cat}$ that have quasi-admissible base change, and that send every quasi-admissible maps the diagonal of any quasi-admissible map to a left adjoint functor.

According to \Cref{prp:monoidal twists}(\ref{itm:monoidal twists/desc}), when $D$ is a pullback formalism on $\mathcal C$, and $S \to X$ is a $D$-closed section of a quasi-admissible map $X \to S$, the Thom twist $\Sigma^X$ is controlled by the object $\cls{X}/\cls{X \setminus S}$ of $D(S)$. This can roughly be seen as encoding cohomological information information about the ``normal bundle'' of $S$ in $X$. For a general map $X \to Y$, the ``relative tangent bundle'' of $X \to Y$ should be given by the ``normal bundle'' of the diagonal $X \to X \times_Y X$, which is a section of either projection $X \times_Y X \to X$. With this in mind, we now come to the definition of the type of twist that will be used to study duality (\cf{} \cite[2.4.20]{tri-cat-mixed-motives}):
\begin{defn}[Tangential Thom twist] \label{defn:tangential Thom twist}
	Let $f : X \to Y$ be a quasi-admissible map in $\mathcal C$. Since $f$ is quasi-admissible, $X \times_Y X$ exists, and we have a diagonal map $\Delta : X \to X \times_Y X$, which is a section of both quasi-admissible projections maps $\pi_1, \pi_2 : X \times_Y X \to X$.

	Note that the transposition of $X \times_Y X$ defines an equivalence between the objects $A_1,A_2$ of $\mathcal C_X$ corresponding to $\pi_1, \pi_2$ with pointing given by $\Delta$, so by \Cref{lem:Thom twist properties}, we have $\Sigma^{A_1} \simeq \Sigma^{A_2}$. We will denote either one by $\Sigma^f$, and call it the \emph{tangential Thom twist of $f$}.

	If $\Sigma^f$ admits a right adjoint, we will write $\Sigma^{-f}$ for a right adjoint of $\Sigma^f$.

	Say $f$ is tangentially $D$-stable if $\Sigma^f$ is invertible.
\end{defn}

\begin{defn} \label{defn:tangentially twisted cohomology}
	Suppose $D$ takes values in $\Alg_{E_0}(\widehat{\Cat})$. If $f : X \to S$ is a quasi-admissible map in $\mathcal C$, and $M \in D(S)$, then we can define the \emph{$f$-tangentially twisted cohomology with coefficients in $M$} to be 
	% If $f : X \to S$ is a map in $\mathcal C$, and $M \in D(S)$, we have defined in \Cref{defn:twisted cohomology} the twisted cohomology
	% \[
	% 	D(X;M)[V] = D(X)(1, \Sigma^V f^* M)
	% \]
	% for any pointed object $V$ of $\mathcal C_X$.
	%
	% Now, if $f$ is a quasi-admissible map, then we can define the \emph{$f$-tangentially twisted cohomology with coefficients in $M$} to be
	\[
		D(X;M)[f] \coloneqq D(X)(1, \Sigma^f f^* M)
	.\]
\end{defn}

\begin{rmk}[Quasi-admissible monomorphisms are ``\'{e}tale''] \label{rmk:mono is etale}
	Let $j : U \to S$ be a quasi-admissible monomorphism in $\mathcal C$. In this case, we expect $j$ to be ``\'{e}tale'' in the sense that its ``relative tangent bundle'' vanishes. Indeed, we can see that the tangential Thom twist $\Sigma^j$ is equivalent to the identity.

	Since $j$ is a monomorphism, the diagonal $\Delta : U \to U \times_S U$ is an equivalence, and if $\pi : U \times_S U$ is either projection, then $\pi \circ \Delta \simeq \id_U$. Thus, $\pi^*, \Delta^*$ are inverse equivalences. In particular, $\pi_\sharp \simeq \pi_*$ since they are both inverses of the equivalence $\pi^*$, and
	\[
		\Sigma^j \simeq \pi_\sharp \Delta_* \simeq \pi_* \Delta_* \simeq (\pi \Delta)_* \simeq \id_{D(U)}
	.\]
\end{rmk}

Now we come to the key definition for studying Atiyah duality. Compare with \cite[2.4.20]{tri-cat-mixed-motives}, \cite[Construction 6.1.5]{TwAmb}, \cite[Construction 2.27]{SixAlgSp}, and \cite[Construction 6.7]{SixAlgSt}, where we have identical constructions. We will see in \Cref{lem:equivalent description of norm} that this also agrees with the transformation described at the beginning of \cite[\S5.4]{mot-norms}.
\begin{defn}[Duality map] \label{defn:duality}
	Let $f : X \to Y$ be a quasi-admissible map in $\mathcal C$. The Cartesian square
	\[
		\begin{tikzcd}
			X \times_Y X \ar[d, "\pi_2"'] \ar[r, "\pi_1"] & X \ar[d, "f"] \\
			X \ar[r, "f"'] & Y
		\end{tikzcd}
	\]
	leads to a right-left mate
	\[
		f_\sharp (\pi_2)_* \to f_* f^* f_\sharp (\pi_2)_*  \simeq f_* (\pi_1)_\sharp \pi_2^* (\pi_2)_* \to f_* (\pi_1)_\sharp
	,\]
	which gives the \emph{duality map}
	\[
		\eth_f : f_\sharp \simeq f_\sharp (\pi_2)_* \Delta_* \to f_* (\pi_1)_\sharp \Delta_* \simeq f_* \Sigma^f
	.\]

	% When $f_*, \Sigma^f$ admit right adjoints $f^\sharp, \Sigma^{-f}$, we write $\eth_f^\vee$ for the induced map on right adjoints, which is
	% \[
	% 	\eth_f^\vee : \Sigma^{-f} f^\sharp \simeq \Delta^\sharp \pi_1^* f^\sharp \to \Delta^\sharp \pi_2^\sharp f^* \simeq f^*
	% \]
	% when $\Delta_*, (\pi_1)_*, (\pi_2)_*$ also admit right adjoints $\Delta^\sharp, \pi_1^\sharp, \pi_2^\sharp$.

	When the ambient presheaf $D$ is not clear from context, we will write $\eth_f^D$ for $\eth_f$.

	Say $D$ \emph{satisfies Atiyah duality} for a quasi-admissible map $f$ if the duality map $\eth_f$ is an equivalence. If $\Sigma^f$ is also invertible, say that $f$ is \emph{stably $D$-ambidextrous}.
\end{defn}

\begin{rmk} \label{rmk:duality and stability give shriek}
	If $f$ is a quasi-admissible stably $D$-ambidextrous map, then $f_* \simeq f_\sharp \Sigma^{-f}$ has a right adjoint $f^\sharp$ given as
	\[
		f^\sharp \simeq \Sigma^f f^*
	.\]
	More generally, if $D$ satisfies Atiyah duality for $f$, $f_*$ has a right adjoint $f^\sharp$, and $\Sigma^f$ has a right adjoint $\Sigma^{-f}$, we have
	\[
		\Sigma^{-f} f^\sharp \simeq f^*
	.\]
\end{rmk}

\begin{rmk} \label{rmk:BM duality}
	If $f : X \to Y$ is a quasi-admissible $D$-quasi-proper map, it follows from \Cref{defn:duality} that since $f$ has quasi-admissible exchange, $D$ satisfies Atiyah duality for $f$, so if $f$ is tangentially $D$-stable, we have an adjunction
	\[
		f_* \dashv f^\sharp \simeq \Sigma^f f^*
	\]
	as in \Cref{rmk:duality and stability give shriek}.

	Recall the notion of Borel-Moore homology from \Cref{defn:BM}. If $D$ extends to a 6-functor formalism on a geometric setup $(\mathcal C,E)$ where $f \in E$ and $f_* \simeq f_!$, then for $M \in D(Y)$,
	\[
		D^\BM(X;M) \simeq D(X)(1, f^\sharp M) \simeq D(X)(1, \Sigma^f f^* M) = D(X;M)[f]
	,\]
	using \Cref{defn:tangentially twisted cohomology}. This can be seen as a version of Poincar\'{e} duality, since it identifies the Borel-Moore homology with a twisted version of cohomology.
\end{rmk}

The following \lcnamecref{lem:equivalent description of norm} reconciles the definition of $\eth_f$ with the map considered in \cite[\S5.4]{mot-norms}.
\begin{lem} \label{lem:equivalent description of norm}
	Let $f : X \to Y$ be a quasi-admissible map in a $\mathcal C$. The map $\eth_f$ is adjunct (with respect to $f^* \dashv f_*$) to the composite
	\[
		f^* f_\sharp \simeq (\pi_2)_\sharp \pi_1^* \to (\pi_2)_\sharp \Delta_* \Delta^* \pi_1^* \simeq (\pi_2)_\sharp \Delta_* \simeq \Sigma^f
	,\]
	where $\pi_1, \pi_2 : X \times_Y X \to X$ are the projections, and $\Delta : X \to X \times_Y X$ is the diagonal.
	\begin{proof}
		Consider the following diagram
		\[
			\begin{tikzcd}
				f_\sharp \ar[d, "\sim"'] \ar[r] & f_* f^* f_\sharp \ar[d, "\sim"] \ar[r] & f_* (\pi_2)_\sharp \pi_1^* \ar[d, "\sim"] \ar[r] & f_* (\pi_2)_\sharp \Delta_* \Delta^* \pi_1^* \ar[d, "\sim"] \\
				f_\sharp (\pi_1)_* \Delta_* \ar[r] & f_* f^* f_\sharp (\pi_1)_* \Delta_* \ar[r] & f_* (\pi_2)_\sharp \pi_1^* (\pi_1)_* \Delta_* \ar[r] & f_* (\pi_2)_\sharp \Delta_*
			\end{tikzcd}
		,\]
		where
		\begin{enumerate}

			\item all vertical arrows are induced by $\id \to (\pi_1)_* \Delta_*$, except for the rightmost arrow, which is induced by its adjunct $\Delta^* \pi_1^* \to \id$,

			\item the horizontal arrows of the leftmost square are induced by the unit of $f^* \dashv f_*$,

			\item the horizontal arrows of the middle square are induced by the inverse of the horizontal left mate $(\pi_2)_\sharp \pi_1^* \to f^* f_\sharp$,

			\item the horizontal arrows of the rightmost square are induced by the unit and counit of $\Delta^* \dashv \Delta_*$ and $\pi_1^* \dashv (\pi_1)_*$.

		\end{enumerate}
		It follows immediately that the left and middle squares commute, and the rightmost square commutes by \Cref{lem:ladj ret to radg sec}.

		By observing that the composite of the bottom arrows is given by the horizontal right-left mate $f_\sharp (\pi_1)_* \to f_* (\pi_2)_\sharp$, we find that traversing the outer rectangle one way gives the map $\eth_f$ of \Cref{defn:duality}, and traversing the other way gives the map adjunct to the composite 
		\[
			f^* f_\sharp \simeq (\pi_2)_\sharp \pi_1^* \to (\pi_2)_\sharp \Delta_* \Delta^* \pi_1^* \simeq (\pi_2)_\sharp \Delta_*
		,\]
		as desired.
	\end{proof}
\end{lem}

It will be necessary to characterize Atiyah duality in terms of an adjunction between certain functors. More specifically, we will need that a certain natural transformation defines a counit of an adjunction:
\begin{lem} \label{lem:duality by adjunction}
	Let $f : X \to Y$ be a tangentially $D$-stable quasi-admissible map in $\mathcal C$. Then $D$ satisfies Atiyah duality for $f$ if and only if the composite
	\[
		f^* f_\sharp \Sigma^{-f} \xrightarrow{f^* \eth_f \Sigma^{-f}} f^* f_* \to \id
	\]
	is the counit of an adjunction $f^* \dashv f_\sharp \Sigma^{-f}$, and in this case, the unit of the adjunction is given by
	\[
		\id \to f_* f^* \xrightarrow{\eth_f^{-1} \Sigma^{-f} f^*} f_\sharp \Sigma^{-f} f^*
	.\]
	% that is, for any $M \in D(X)$ and $N \in D(Y)$, the composite
	% \[
	% 	D(Y)(N, f_\sharp \Sigma^{-f} M) \to D(X)(f^* N, f^* f_\sharp \Sigma^{-f} M) \to D(X)(f^* N, f^* f_* M) \to D(X)(f^* N, M)
	% \]
	% is an equivalence.
	\begin{proof}
		For any $M \in D(X)$, and $N \in D(Y)$, consider the commutative diagram
		\[
			\begin{tikzcd}
				D(Y)(N, f_\sharp M) \ar[d] \ar[r] & D(X)(f^* N, f^* f_\sharp M) \ar[d] & \\
				D(Y)(N, f_* \Sigma^f M) \ar[r] & D(X)(f^* N, f^* f_* \Sigma^f M) \ar[r] & D(X)(f^* N, \Sigma^f M)
			\end{tikzcd}
		.\]
		Since $f^* \dashv f_*$, the composite of the bottom two arrows is always an equivalence. Thus, the other composite
		\[
			D(Y)(N, f_\sharp M) \to D(X)(f^* N, \Sigma^f M)
		\]
		is an equivalence if and only if the leftmost vertical arrow is an equivalence. Since $\Sigma^f$ is invertible, this means that the composite
		\[
			f^* f_\sharp \Sigma^{-f} \xrightarrow{f^* \eth_f \Sigma^{-f}} f^* f_* \to \id
		\]
		is the counit of an adjunction $f^* \dashv f_\sharp \Sigma^{-f}$ if and only if $\eth_f \Sigma^{-f}$ is an equivalence, which is equivalent to $\eth_f$ being an equivalence.

		By taking $M = \Sigma^{-f} f^* N$ in the above diagram, we see that the leftmost vertical map,
		\[
			D(Y)(N, f_\sharp \Sigma^{-f} f^* N) \to D(Y)(N, f_* f^* N)
		,\]
		sends the unit of $f^* \dashv f_\sharp \Sigma^{-f}$ to the unit of $f^* \dashv f_*$.
	\end{proof}
\end{lem}

The following \lcnamecref{lem:norm morphisms} will be an important technical ingredient in many of our arguments. Given a morphism of pullback formalisms $\phi : D \to D'$ on a pullback context $\mathcal C$, and a quasi-admissible map $f$ in $\mathcal C$, it expresses a relationship between the following maps:
\begin{itemize}

	\item the duality map $\eth_f$ of $f$ for $D$,

	\item the duality map $\eth_f$ of $f$ for $D'$,

	\item the right mate $\phi f_* \to f_* \phi$, and
		
	\item the right mate $\phi \Delta_* \to \Delta_* \phi$, where $\Delta$ is the diagonal of $f$.

\end{itemize}
This will have important consequences for $D$ when applied to the morphisms of \Cref{exa:tensor morph,exa:base change morph,exa:sharp morph}.

\begin{lem} \label{lem:norm morphisms}
	Let $f : X \to Y$ be a quasi-admissible map in $\mathcal C$. Write $\Delta : X \to X \times_Y X$ for the diagonal of $f$, write $\pi : X \times_Y X \to X$ for one of the projections, and assume that the right adjoints $\Delta_*, \pi_*$ of $\Delta^*, \pi^*$ exist. If $\phi : D \to D'$ is a transformation that respects quasi-admissibility, there is a commutative diagram
	\[
		\begin{tikzcd}
			f_\sharp \phi \ar[dd, "\sim"'] \ar[r, "\eth_f \phi"] & f_* \Sigma^f \phi \\
			& \ar[u] f_* \phi \Sigma^f \\
			\phi f_\sharp \ar[r, "\phi \eth_f"'] & \ar[u] \phi f_* \Sigma^f
		\end{tikzcd}
	,\]
	where the left vertical arrow is the left mate $f_\sharp \phi \to \phi f_\sharp$, the bottom right vertical arrow is induced by the right base transformation $\phi f_* \to f_* \phi$, and the top right vertical arrow is $f_*$ of the composite
	\[
		\pi_\sharp \Delta_* \phi \gets \pi_\sharp \phi \Delta_* \xrightarrow{\sim} \phi \pi_\sharp \Delta_*
	.\]
	In particular, if $\phi$ is right adjointable at $\Delta$, then the top right vertical arrow is an equivalence.

	% Furthermore, if $\phi$ is right adjointable at $f,\pi,\Delta$, and for both $D$ and $D'$, $f_*, \pi_*, \Delta_*$ admit right adjoints $f^\sharp, \pi^\sharp, \Delta^\sharp$, then we also have a commutative diagram
	% \[
	% 	\begin{tikzcd}
	% 		\Sigma^{-f} f^\sharp \phi \ar[r, "\eth_f^\vee \phi"] & f^* \phi \ar[dd, "\sim"] \\
	% 		\ar[u] \Sigma^{-f} \phi f^\sharp & \\
	% 		\ar[u] \phi \Sigma^{-f} f^\sharp \ar[r, "\phi \eth_f^\vee"'] & \phi f^*
	% 	\end{tikzcd}
	% ,\]
	% where the right vertical arrow is the equivalence $f^* \phi \to \phi f^*$, the top left vertical arrow is induced by the right base transformation $\phi f^\sharp \to f^\sharp \phi$, and the bottom left vertical arrow is given by precomposing the composite
	% \[
	% 	\Delta^\sharp \pi^* \phi \xrightarrow{\sim} \Delta^\sharp \phi \pi^* \gets \phi \Delta^\sharp \pi^*
	% \]
	% by $f^\sharp$.
	\begin{proof}
		By applying \cite[Proposition F.14]{TwAmb} to the commutative cube
		\[
			\begin{tikzcd}
				D(Y) \ar[dr, "\phi"description] \ar[ddd, "f^*"'] \ar[rrr, "f^*"] &&& D(X) \ar[dr, "\phi"] \ar[ddd, "\pi_2^*"] & \\
																	   & D'(Y) \ar[ddd, "f^*"'] \ar[rrr, "f^*"] &&& D'(X) \ar[ddd, "\pi_2^*"] \\
				\\
				D(X) \ar[dr, "\phi"'] \ar[rrr, "\pi_1^*"] &&& D(X \times_Y X) \ar[dr, "\phi"] & \\
														 & D'(X) \ar[rrr, "\pi_1^*"'] &&& D'(X \times_Y X) 
			\end{tikzcd}
		,\]
		we obtain a commutative diagram
		\[
			\begin{tikzcd}
				f_\sharp \phi (\pi_2)_* \ar[d, "\sim"'] \ar[r] & f_\sharp (\pi_2)_* \phi \ar[r] & f_* (\pi_1)_\sharp \phi \ar[d, "\sim"]  \\
				\phi f_\sharp (\pi_2)_* \ar[r] & \phi f_* (\pi_1)_\sharp \ar[r] & f_* \phi (\pi_1)_\sharp
			\end{tikzcd}
		,\]
		where all the arrows are given by the appropriate left, right, or right-left mates. By precomposing this diagram with $\Delta_*$, we obtain the bottom right rectangle in the following commutative diagram:
		% PERF: I guess I could explain why the rest of the diagram commutes
		% in particular the rectangles with the long arrows, since they involve commuting some base change and exchange maps
		\[
			\begin{tikzcd}
				f_\sharp \phi \ar[d, equals] \ar[r, no head, "\sim"] \ar[rrr, bend left=20, "\eth_f \phi"] & f_\sharp (\pi_2)_* \Delta_* \phi \ar[rr] & & f_* (\pi_1)_\sharp \Delta_* \phi \\
				f_\sharp \phi \ar[d, "\sim"'] \ar[r, no head, "\sim"] & \ar[u] f_\sharp \phi (\pi_2)_* \Delta_* \ar[d, "\sim"'] \ar[r] & f_\sharp (\pi_2)_* \phi \Delta_* \ar[r] & \ar[u] f_* (\pi_1)_\sharp \phi \Delta_* \ar[d, "\sim"]  \\
				\phi f_\sharp \ar[r, no head, "\sim"] \ar[rr, bend right=20, "\phi \eth_f"'] & \phi f_\sharp (\pi_2)_* \Delta_* \ar[r] & \phi f_* (\pi_1)_\sharp \Delta_* \ar[r] & f_* \phi (\pi_1)_\sharp \Delta_*
			\end{tikzcd}
		.\]
		By \Cref{defn:duality}, the outermost rectangle is the desired commutative diagram.

	\end{proof}
\end{lem}

% \begin{cor} \label{cor:compatibility of norm with tensor}
% 	Let $f : X \to Y$ be a quasi-admissible map in $\mathcal C$, and suppose $D$ lifts to a presheaf $\mathcal C^\op \to \CAlg(\widehat{\Cat})$ that satisfies the quasi-admissible projection formula. If $D$ has the right projection formula for the diagonal of $f$, then for any $N \in D(Y)$, there is a commutative diagram
% 	\[
% 		\begin{tikzcd}[column sep=huge]
% 			f_\sharp(- \otimes f^* N) \ar[dd, "\sim"'] \ar[r, "\eth_f (- \otimes f^* N)"] & f_* \Sigma^f (- \otimes f^* N) \\
% 			& \ar[u, "\sim"'] f_* (\Sigma^f - \otimes f^* N) \\
% 			f_\sharp(-) \otimes N \ar[r, "\eth_f \otimes N"'] & \ar[u] f_* \Sigma^f(-) \otimes N
% 		\end{tikzcd}
% 	,\]
% 	where the bottom right vertical map is induced by the right projection formula map
% 	\[
% 		f_*(-) \otimes - \to f_*(- \otimes f^*(-))
% 	.\]
% 	\begin{proof}
% 		By \Cref{exa:tensor morph}, if we write $\pi : \mathcal C_{/Y} \to \mathcal C$ for the slice projection, then for any $N \in D(Y)$, there is a transformation $\phi : \pi^* D \to \pi^* D$ that evaluates to $D(W) \xrightarrow{\otimes w^* N} D(W)$ at any $w : W \to Y$ in $\mathcal C_{/Y}$, and $\phi$ respects quasi-admissibility since $D$ satisfies the quasi-admissible projection formula.
%
% 		Thus, we conclude by \Cref{lem:norm morphisms}.
% 	\end{proof}
% \end{cor}

Since $D$ has quasi-admissible base change, the morphisms of \Cref{exa:base change morph,exa:sharp morph} respect quasi-admissibility, and we invite the reader to work out what \Cref{lem:norm morphisms} says when applied to these morphisms. We will later prove stronger results when we have access to monoidal structures -- see \Cref{prp:duality base change,prp:duality qadm base change}. When $D$ is actually a presheaf $\mathcal C^\op \to \CAlg(\widehat{\Cat})$ that satisfies the quasi-admissible projection formula, we can also apply \Cref{lem:norm morphisms} to \Cref{exa:tensor morph}.

\subsection{Duality for pullback formalisms} \label{S:PF duality}

This section will explore the good behaviour of Atiyah duality when, in addition to quasi-admissible base change, we also have the quasi-admissible projection formula.

Throughout this section, $\mathcal C$ will denote a pullback context, and $D,D' : \mathcal C^\op \to \CAlg(\widehat{\Cat})$ will denote presheaves that have quasi-admissible base change and the quasi-admissible projection formula, and send every map to a left adjoint functor. In particular, this holds if $D,D'$ are pullback formalisms. We will also write $\phi : D \to D'$ for a transformation that respects quasi-admissibility.

% NOTE: sometimes it can be convenient to assume we are pullback formalisms
% since then closed monoidality is automatic and can apply {cor:propagate stability}
% but only really get to use this in qadm case anyway, since otherwise you also have to assume that you have linear left adjoints
% Anyway I can just separately propagate the stability without making it be part of thm statement

The main application of this section is
\begin{thm} \label{thm:duality}
	Let $f : X \to Y$ be a quasi-admissible map in $\mathcal C$, and assume $D,D'$ are reduced and take values in pointed closed monoidal categories, and that $\phi$ preserves zero objects pointwise.
	\begin{enumerate}

		\item \label{itm:duality/morph}
			Suppose the diagonal of $f$ is $D$-closed and $D'$-closed. If $f$ is stably $D$-ambidextrous, then $\phi$ is right adjointable at $f$, $f$ is stably $D'$-ambidextrous, and
			\[
				f^\sharp \phi \simeq \phi f^\sharp
			,\]
			where $f^\sharp$ is a right adjoint adjoint of $f_*$.

			Furthermore, if $f$ is stably $D'$-ambidextrous and tangentially $D$-stable, $\phi$ is right adjointable at $f$, and $\phi : D(Y) \to D'(Y)$ is conservative, then $f$ is also stably $D$-ambidextrous.

		\item\label{itm:duality/locality}
			If the diagonal of any quasi-admissible base change of $f$ is $D$-closed, then the following are equivalent:
			\begin{enumerate}

				\item $f$ is tangentially $D$-stable, and $D$ has quasi-admissible exchange for every quasi-admissible base change of $f$.

				\item Every quasi-admissible base change of $f$ is stably $D$-ambidextrous.

				\item There is a $D$-pseudocover of $Y$ by quasi-admissible maps $Y' \to Y$ such that the base change of $f$ to $Y'$ is stably $D$-ambidextrous.

			\end{enumerate}

		\item \label{itm:duality/proper}
			If the diagonal of every base change of $f$ is $D$-closed, then $f$ is stably $D$-ambidextrous if and only if $f$ is $D$-quasi-proper and tangentially $D$-stable. In this case, the same properties hold for all base changes of $f$, and for any Cartesian square
			\[
				\begin{tikzcd}
					X' \ar[d, "p"'] \ar[r, "f'"] & Y' \ar[d, "q"] \\
					X \ar[r, "f"'] & Y
				\end{tikzcd}
			\]
			in $\mathcal C$, there is an equivalence
			\[
				p^* f^\sharp \simeq f'^\sharp q^*
			,\]
			where $f^\sharp$ denotes a right adjoint of $f_*$.

	\end{enumerate}
\end{thm}

First we will need to introduce the following notation:
\begin{nota}
	If $f$ is a quasi-admissible map in $\mathcal C$, then we write $\Th(f) \coloneqq \Sigma^f(1)$, and if $\Sigma^f$ admits a right adjoint $\Sigma^{-f}$, we write $\Th(-f) \coloneqq \Sigma^{-f}(1)$.
\end{nota}

We begin with the following result that gives a useful criterion for checking Atiyah duality.
\begin{lem} \label{lem:duality when diag is linear}
	Let $f : X \to Y$ be a tangentially $D$-stable quasi-admissible map in $\mathcal C$. If $D$ satisfies Atiyah duality for $f$ then there is a map $u : 1 \to f_\sharp \Th(-f)$ such that the composite
	\[
		1 \xrightarrow{f^* u} f^* f_\sharp \Th(-f) \to 1
	\]
	is equivalent to the identity, where the map $f^* f_\sharp \Th(-f) \to 1$ is given by evaluating the transformation
	\[
		f^* f_\sharp \simeq (\pi_2)_\sharp \pi_1^* \to (\pi_2)_\sharp \Delta_* \Delta^* \pi_1^* \simeq (\pi_2)_\sharp \Delta_* \simeq \Sigma^f
	\]
	at $\Th(-f)$, and $\Delta : X \to X \times_Y X$ is the diagonal of $f$, and $\pi_1, \pi_2 : X \times_Y X \to X$ are the projections. 

	Furthermore, the converse holds if $D$ has the right projection formula for $\Delta$, and in this case, $\eth_f$ is $D(Y)$-linear, and $D$ has the right projection formula for $f$.
	\begin{proof}
		By \Cref{lem:duality by adjunction}, $D$ has Atiyah duality for $f$ if and only if the map $f^* f_\sharp \Sigma^{-f} \to \id$ adjunct to $\eth_f \Sigma^{-f}$ is the counit of an adjunction $f^* \dashv f_\sharp \Sigma^{-f}$. By \Cref{lem:equivalent description of norm}, we have that this map $f^* f_\sharp \Sigma^{-f} \to \id$ is induced by the map $f^* f_\sharp \to \Sigma^f$ given in the statement of that result. Thus, \Cref{lem:crit for linear adj by unit} says that if $D$ has Atiyah duality for $f$, then the map $u$ satisfying the required property exists, and the converse holds if $f_\sharp \Sigma^{-f}$ and the map $f^* f_\sharp \Sigma^{-f} \to \id$ are $D(Y)$-linear.

		If $D$ has the right projection formula for $\Delta$, the description of $\Sigma^f$ given in \Cref{lem:Thom twist descriptions} shows that $\Sigma^f$ is $D(X)$-linear, and since it is invertible, it has a $D(X)$-linear inverse. In particular, $\Sigma^{-f}$ is $D(Y)$-linear, so $f_\sharp \Sigma^{-f}$ is $D(Y)$-linear since it is a composite of $D(Y)$-linear functors. In fact, since $\Delta_*$ is a linear right adjoint of $\Delta^*$, we find that the map $f^* f_\sharp \Sigma^{-f} \to \id$ is $D(Y)$-linear, which concludes the proof of the converse.

		Now, in this case, when $\eth_f$ is invertible, by applying \Cref{lem:norm morphisms} to \Cref{exa:tensor morph}, we find that for any $N \in D(Y)$, the map $f_* \Sigma^f(-) \otimes N \to f_* (\Sigma^f - \otimes f^* N)$ is invertible, so since $\Sigma^f$ is an equivalence, we have that $f_*$ is a linear right adjoint of $f^*$.

		% Now, as we have shown that we can choose $D(Y)$-linear units and counits of the adjunction $f^* \dashv f_*$, it follows that the composite
		% \[
		% 	f_* M \otimes - \to f_* f^*(f_* M \otimes -) \simeq f_*( f^* f_* M \otimes f^* -) \to f_* (M \otimes f^* -)
		% \]
		% is $D(Y)$-linear. In particular, it is equivalent to
		% \[
		% 	(f_* M \simeq f_* M \otimes 1 \to f_*(M \otimes f^* 1) \simeq f_* M) \otimes -
		% ,\]
		% so since the map $f_* M \to f_* M$ is $f_* M \to f_* f^* f_* M \to f_* M$ is equivalent to the identity by the triangle identities of $f^* \dashv f_*$, we have that
		% \[
		% 	f_* M \otimes N \to f_*(M \otimes f^* N)
		% \]
		% is invertible.

		Finally, to see that $\eth_f$ is $D(Y)$-linear, we note that by \Cref{lem:equivalent description of norm}, it is adjunct to the composite
		\[
			\varepsilon : f^* f_\sharp \simeq (\pi_2)_\sharp \pi_1^* \to (\pi_2)_\sharp \Delta_* \Delta^* \pi_1^* \simeq (\pi_2)_\sharp \Delta_* \simeq \Sigma^f
		,\]
		where $\pi_1, \pi_2 : X \times_Y X \to X$ are the projections, and this composite is $D(Y)$-linear since $D$ has the right projection formula for $\Delta$. Thus, since $f_*$ is a linear right adjoint of $f^*$, we have that the composite
		\[
			f_\sharp \to f_* f^* f_\sharp \xrightarrow{f_* \varepsilon} f_* \Sigma^f
		\]
		is also $D(Y)$-linear, as desired.
	\end{proof}
\end{lem}

Using \Cref{lem:duality when diag is linear,lem:norm morphisms}, we can now prove our key result about how Atiyah duality interacts with morphisms of pullback formalisms:
\begin{prp} \label{prp:duality propagation}
	Let $f : X \to Y$ be a quasi-admissible map in $\mathcal C$ that is tangentially $D$-stable, and such that $\phi$ is right adjointable at the diagonal $\Delta : X \to X \times_Y X$ of $f$, and $D,D'$ have the right projection formula for $\Delta$.

	If $f$ is stably $D$-ambidextrous, then
	\begin{enumerate}

		\item $f$ is stably $D'$-ambidextrous, and the converse holds if $\pi_0 D(Y; f_\sharp \Th(-f)) \to \pi_0 D'(Y; f_\sharp \Th(-f))$ is surjective, and $\pi_0 D(X;1) \to \pi_0 D'(X;1)$ is injective.

		\item $\phi$ is right adjointable at $f$, and the converse holds if $D'$ has Atiyah duality for $f$, and $D(Y) \to D'(Y)$ is conservative.

			% TODO: consider removing this statement
			% this is repeated in several statements,
			% but always just follows from {duality when diag is linear}?
		\item $D$ and $D'$ have the right projection formula for $f$.

		\item There is an equivalence $\phi f^\sharp \simeq f^\sharp \phi$, where $f^\sharp$ denotes a right adjoint of $f_*$.

	\end{enumerate}
	\begin{proof}
		By \Cref{lem:propagate stability}, we have that $f$ is tangentially $D'$-stable, and by \Cref{rmk:morph twist}, since $\phi$ is right adjointable at $\Delta$, we have that $\phi \Th(-f) \simeq \Th(-f)$.

		By \cite[Lemma F.5]{TwAmb}, since $\phi$ is right adjointable at $\Delta$, we have that $\phi$ sends the map $f^* f_\sharp \Th(-f) \to 1$ of \Cref{lem:duality when diag is linear} to the same map for $D'$. Thus, \Cref{lem:duality when diag is linear} shows that if $D$ has Atiyah duality for $f$, then so does $D'$, and the converse holds if $\pi_0 D(Y)(1, f_\sharp \Th(-f)) \to \pi_0 D'(Y)(1, f_\sharp \Th(-f))$ is surjective, and the connected component of the identity in $D(X)(1,1)$ is the only one sent to the connected component of the identity in $D'(X)(1,1)$. \Cref{lem:duality when diag is linear} also shows that $D$ and $D'$ have the right projection formula for $f$.

		Since $\eth_f^D$ and $\eth_f^{D'}$ are both equivalences, and $\phi$ is right adjointable at $\Delta$, \Cref{lem:norm morphisms} shows that the natural map $(\phi f_* \to f_* \phi) \Sigma^f$ is an equivalence, but since $\Sigma^f$ is an equivalence, this means that $\phi f_* \to f_* \phi$ is an equivalence, \ie{} $\phi$ is right adjointable at $f$.
		% Now, by \Cref{lem:norm morphisms}, there is a commutative diagram
		% \[
		% 	\begin{tikzcd}
		% 		f_\sharp \phi \ar[dd, "\sim"'] \ar[r, "\eth_f \phi"] & f_* \Sigma^f \phi \\
		% 		& \ar[u] f_* \phi \Sigma^f \\
		% 		\phi f_\sharp \ar[r, "\phi \eth_f"'] & \ar[u] \phi f_* \Sigma^f
		% 	\end{tikzcd}
		% ,\]
		% where the top right vertical map is invertible since $\phi$ is right adjointable at $\Delta$, and the top and bottom horizontal maps are invertible since $D,D'$ have Atiyah duality for $f$. It follows that the bottom right vertical map is an equivalence, but since $\Sigma^f$ is invertible, this means that $\phi f_* \to f_* \phi$ is invertible, that is, $\phi$ is right adjointable at $f$.

		For the converse, when $\phi$ is right adjointable at $f$, $D'$ has Atiyah duality for $f$, \Cref{lem:norm morphisms} shows that $\phi \eth_f$ is an equivalence, so if $\phi : D(Y) \to D'(Y)$ is conservative, then $D$ has Atiyah duality for $f$.

		Finally, it follows from \Cref{rmk:duality and stability give shriek} that for both $D$ and $D'$, $f^\sharp \coloneqq \Sigma^f f^*$ is a right adjoint of $f_*$. Since $\phi$ is right adjointable at the diagonal of $f$, it follows that $\phi$ commutes with $\Sigma^f$, so
		\[
			\phi f^\sharp \simeq \phi \Sigma^f f^* \simeq \Sigma^f \phi f^* \simeq \Sigma^f f^* \phi \simeq f^\sharp \phi
		.\]
	\end{proof}
\end{prp}

As usual, we can now use \Cref{prp:duality propagation} to obtain results about how Atiyah duality interacts with base change:
\begin{prp} \label{prp:duality base change}
	Let
	\[
		\begin{tikzcd}
			X' \ar[d, "p"'] \ar[r, "f'"] & Y' \ar[d, "q"] \\
			X \ar[r, "f"'] & Y
		\end{tikzcd}
	\]
	be a Cartesian square in $\mathcal C$. Suppose that $f$ is a tangentially $D$-stable quasi-admissible map in $\mathcal C$, and suppose $D$ has the right projection formula for the diagonals of $f$ and $f'$, and right base change for the diagonal of $f$ against $p \times_q p : X' \times_{Y'} X' \to X \times_Y X$.

	If $f$ is stably $D$-ambidextrous, then
	\begin{enumerate}

		\item $D$ has the right projection formula for $f'$.

		\item $f'$ is stably $D$-ambidextrous, and the converse holds if $\pi_0 D(Y; f_\sharp \Th(-f)) \to \pi_0 D(Y'; f'_\sharp \Th(-f'))$ is surjective, and $\pi_0 D(X;1) \to D(X';1)$ is injective.

		\item $D$ has right base change for $f$ against $q$, and the converse holds if $D$ has Atiyah duality for $f'$, and $q^*$ is conservative.

		\item There is an equivalence $p^* f^\sharp \simeq (f')^\sharp q^*$, where $f^\sharp$ denotes a right adjoint of $f_*$.

	\end{enumerate}
	\begin{proof}
		Note that
		\[
			p \times_q p \simeq X \times_Y q \times_q p \simeq X \times_Y p \simeq X \times_Y X \times_Y q
		,\]
		so the result follows from \Cref{prp:duality propagation,exa:base change morph}.
	\end{proof}
\end{prp}

The case of quasi-admissible base change is particularly nice.
\begin{prp} \label{prp:duality qadm base change}
	Let
	\[
		\begin{tikzcd}
			X' \ar[d, "p"'] \ar[r, "f'"] & Y' \ar[d, "q"] \\
			X \ar[r, "f"'] & Y
		\end{tikzcd}
	\]
	be a Cartesian square in $\mathcal C$. Suppose that $f,q$ are quasi-admissible, and suppose that $D$ has the right projection formula for the diagonals of $f$ and $f'$.

	If $f$ is stably $D$-ambidextrous, then
	\begin{enumerate}

		% PERF: should state the converse in terms of quasi-admissible pseudocovers
		\item $f'$ is stably $D$-ambidextrous, and the converse holds if $q^*$ is conservative, and $f$ is tangentially $D$-stable or $D(X)$ is closed monoidal.

		\item If $D$ has right-left base change for the diagonal of $f$ against $p \times_q p$, then $D$ has right-left base change for $f$ against $q$.

		% \item The commutative square
		% 	\[
		% 		\begin{tikzcd}
		% 			D(Y) \ar[d] \ar[r, "f^*"] & D(X) \ar[d] \\
		% 			D(Y') \ar[r, "f^*"'] & D(X')
		% 		\end{tikzcd}
		% 		\parbox{5\textwidth/12}{\center is horizontally right adjointable,\\and the resulting commutative square}
		% 		\begin{tikzcd}
		% 			D(X) \ar[d] \ar[r, "f_*"] & D(Y) \ar[d] \\
		% 			D(X') \ar[r, "f_*"'] & D(Y')
		% 		\end{tikzcd}
		% 	\]
		% 	\hfill is also horizontally right adjointable.

	\end{enumerate}
	\begin{proof}
		Recall from the proof of \Cref{prp:duality base change} that $p \times_q p \simeq X \times_Y X \times_Y q$.
		\begin{enumerate}

			\item Since $p \times_q p$ is a base change of $q$, it is quasi-admissible. Therefore, $D$ has right base change for the diagonal of $f$ against $p \times_q p$, so \Cref{prp:duality base change} says that if $f$ is stably $D$-ambidextrous, then so is $f'$, and it also says that if $f$ is tangentially $D$-stable, $f$ has right base change against $q$, and $q^*$ is conservative, then the converse holds (\ie{} $f$ is stably $D$-ambidextrous).

				% NOTE: uses that $D$ sends $f$ to left adjoint functor
				Since $q$ is quasi-admissible, we have that $D$ has right base change for $f$ against $q$, and \Cref{lem:propagate stability} shows that if $D(X)$ is closed monoidal, and $f'$ is tangentially $D$-stable, then so is $f$. Thus, our assumptions show that if $f'$ is stably $D$-ambidextrous, and $q^*$ is conservative, then $f$ is stably $D$-ambidextrous.

			\item Since $\eth_f^D$ and $\eth_f^{D'}$ are equivalences, and $f$ is tangentially $D$-stable, this follows easily from \Cref{exa:sharp morph,lem:norm morphisms}.
				% Using \Cref{exa:sharp morph,lem:norm morphisms}, we have a commutative diagram
				% \[
				% 	\begin{tikzcd}
				% 		f_\sharp p_\sharp \ar[dd, "\sim"'] \ar[r, "\eth_f p_\sharp"] & f_* \Sigma^f p_\sharp \\
				% 		& \ar[u] f_* p_\sharp \Sigma^{f'} \\
				% 		q_\sharp f'_\sharp \ar[r, "q_\sharp \eth_{f'}"'] & \ar[u] q_\sharp f'_* \Sigma^{f'}
				% 	\end{tikzcd}
				% ,\]
				% where the top horizontal arrow is invertible since we know that $D$ has Atiyah duality for $f$, and the bottom horizontal arrow is invertible since we have shown that $D$ has Atiyah duality for $f'$.
				%
				% Thus, if $D$ has right-left base change for the diagonal of $f$ against $p \times_q p$, then the top right vertical arrow is invertible, so the bottom right vertical arrow is invertible, but this implies that $D$ has right-left base change for $f$ against $q$, since $f'$ is tangentially $D$-stable.

			% \item Follows immediately from \cite[Remark 4.7.4.14]{ha} and the fact that $D$ has right-left base change for $f$ against $q$.

		\end{enumerate}
		
	\end{proof}
\end{prp}

Finally we come to the proof of our main application:
\begin{proof}[Proof of \Cref{thm:duality}]
	Write $\Delta : X \to X \times_Y X$ for the diagonal of $f$, and note that since $D$ takes values in closed monoidal categories that have zero objects, and sends every map to an adjoint functor, it takes values in the category $\CAlg(\widehat{\Cat}_{\text{pointed}})$ of \Cref{thm:closed}, and $\phi$ is a transformation of $\CAlg(\widehat{\Cat}_{\text{pointed}})$-valued presheaves.
	\begin{enumerate}

		\item By \Cref{thm:closed}, we have that $\phi$ is right adjointable at $\Delta$, and that $D,D'$ have the right projection formula for $\Delta$. Thus \Cref{prp:duality propagation} says that if $f$ is stably $D$-ambidextrous, then $f$ is stably $D'$-ambidextrous, $\phi$ is right adjointable at $f$, and $\phi f^\sharp \simeq f^\sharp \phi$, and that conversely, if $\phi$ is right adjointable at $f$, $f$ is tangentially $D$-stable, and $f$ is stably $D'$-ambidextrous, then $f$ is stably $D$-ambidextrous.

		\item It is clear that each condition implies the one below it, so we just need to show that the last condition implies the first.

			For any quasi-admissible map $q : Y' \to Y$, if we write $\Delta' : X' \to X' \times_{Y'} X'$ for the diagonal of the base change $f' : X' \to Y'$ of $f$ along $q$, then we have a commutative diagram
			\begin{equation} \label{eqn:diag bc}
				\begin{tikzcd}
					X' \ar[d, "p"'] \ar[r, "\Delta'"] & X' \times_{Y'} X' \ar[d, "p \times_q p"] \ar[r] & Y' \ar[d, "q"] \\
					X \ar[r, "\Delta"'] & X \times_Y X \ar[r] & Y
				\end{tikzcd}
			,\end{equation}
			where the right square and outer rectangle are Cartesian, so the left square is Cartesian. By our assumption, since $f'$ is a quasi-admissible base change of $f$, its diagonal $\Delta'$ is $D$-closed. Thus \Cref{thm:closed} says that $D$ has the right projection formula and quasi-admissible exchange for for $\Delta$ and $\Delta'$.

			Since $D$ has the right projection formula for the diagonals of $f$ and $f'$, \Cref{prp:duality qadm base change} says that if $f$ is stably $D$-ambidextrous, then so is $f'$, and that the converse holds if $q^*$ is conservative, and $D(Y')$ is closed monoidal. Therefore, the last condition implies that $f$ and all of its quasi-admissible base changes are stably $D$-ambidextrous. On the other hand, since $D$ has right-left base change for $\Delta$ against $p \times_q p$, \Cref{prp:duality qadm base change} shows that $D$ has right-left base change for $f$ against $q$, so we are done, since this holds for all quasi-admissible base changes of $f$, and all quasi-admissible maps $q$ to $Y$.

		\item If $f$ is $D$-quasi-proper, then $D$ has quasi-admissible exchange for $f$, so if $f$ is tangentially $D$-stable, it is stably $D$-ambidextrous.

			For the converse, consider the diagram \eqref{eqn:diag bc} from the previous point, where this time we do not assume that $q$ is quasi-admissible. Since this time we have assumed that the diagonal of \emph{any} base change of $f$ is $D$-closed, we still have that $\Delta'$ is $D$-closed, so by \Cref{thm:closed}, we have that $D$ has the right projection formula for $\Delta$ and $\Delta'$, and right base change for $\Delta$ against $p \times_q p$.

			It follows from \Cref{prp:duality base change} that if $f$ is stably $D$-ambidextrous, then $D$ has the right projection formula and right base change for $f$, and $p^* f^\sharp \simeq f^\sharp q^*$. The previous point shows that $D$ has quasi-admissible exchange for $f$, and \Cref{rmk:duality and stability give shriek} shows that $f_*$ admits a right adjoint, so that $f$ is $D$-quasi-proper, which completes the proof of the converse.

			To get the same statement for all base changes of $f$, note that \Cref{prp:duality base change} actually shows that the base change of $f$ along $q$ is also stably $D$-ambidextrous.

	\end{enumerate}
\end{proof}

\section{Pullback Formalisms and 6-Functor Formalisms} \label{S:6FF}

As we have seen in \Cref{rmk:6FF by decomp}, if $D$ is a pullback formalism, then $D$-quasi-properness is useful for enhancing $D$ to a 6-functor formalism. We have proven key results on how to produce $D$-quasi-proper maps and show that maps from $D$ are compatible with them in \Cref{S:coh cl,S:PF duality}, and we have also produced various results about closure properties of $D$-quasi-proper maps in \Cref{S:proper stable}. These results are combined in \Cref{thm:PPF}, which gives criteria for a large class of maps to be $D$-quasi-proper, and which also guarantee that a large class of morphisms from $D$ are right adjointable at these maps. This will later allow us not only to prove results about enhancing $D$ to a 6-functor formalism, but also tell us how to give compatibilities with the six operations and morphisms from $D$. We also note \Cref{lem:PPF cdh,lem:Atiyah duality PPF}, which use the results of \Cref{S:gluing,S:duality} to give cdh descent and Atiyah duality statements for $D$.

We then turn our attention to 6-functor formalisms in \Cref{S:V6FF}, which can be seen as a major refinement of the construction given in \Cref{rmk:6FF by decomp} about using $D$-quasi-proper maps to extend $D$ to a particularly well-behaved 6-functor formalism. \Cref{thm:6FF} gives a preview of the sort of result that can be obtained by combining \Cref{thm:PPF} with the results of \Cref{S:V6FF}.\\

We fix a pullback context $\mathcal C$, along with the data of two collections of maps in $\mathcal C$ called \emph{exceptionally quasi-proper maps} and \emph{exceptionally closed maps}, and make the following assumptions:
\begin{ass} \label{ass:exc}
	\hfill
	\begin{enumerate}

		\item $\mathcal C$ is locally small.

		\item Every exceptionally closed map is closed (\ie{} it has quasi-admissible complement).

		% \item Every equivalence is exceptionally closed and exceptionally quasi-proper.

			% NOTE: really need that closeds are stable under base change
			% base change of quasi-propers is mostly for convenience
				% anyway can always just add all base changes
				% nice to assume this so that know that the base changes exist for the purpose of base change results
				% also nice to have this for the purpose of defining cdh topology
				% also nice to have this when showing exc qprop cdh excision for $D^\sharp$
		\item The collections of exceptionally closed and exceptionally quasi-proper maps are stable under base change.

	\end{enumerate}
\end{ass}

We often also assume that every exceptionally closed map is also exceptionally quasi-proper, but this is not necessary for our results. Nevertheless, this would lead to no loss of generality, as we can always produce a new collection of exceptionally quasi-proper maps that contains the exceptionally closed ones by considering the collection of composites of exceptionally closed and exceptionally quasi-proper maps.

Next, we use the exceptionally closed maps to produce some Grothendieck topologies in the spirit of cdh excision. See \Cref{rmk:compare cdh} for a comparison with usual notions of cdh excision.
\begin{defn}[Cdh topologies] \label{defn:cdh}
	We may define the following notions of \emph{elementary cdh covers} given a map $i : Z \to S$ in $\mathcal C$:
	\begin{itemize}

		\item An \emph{elementary cdh cover of $i$} is a sieve $\mathcal R$ on $S$ such that the base change of $\mathcal R$ along $i$ contains $\id_Z$ -- \ie{} $i^* \mathcal R$ is the trivial sieve on $Z$.

		\item An \emph{elementary cdh cover away from $i$} is a sieve $\mathcal R$ on $S$ such that for any $j : U \to S$ complementary to $i$, the base change of $\mathcal R$ along $j$ contains $\id_U$ -- \ie{} $j^* \mathcal R$ is the trivial sieve on $U$.

	\end{itemize}

	We define the following Grothendieck topologies on $\mathcal C$:
	\begin{itemize}

		% \item Define the \emph{constructible topology} on $\mathcal C$ to be the Grothendieck topology generated by declaring that the empty sieve covers the initial object, and $\{i,j\}$ is a covering family for any exceptionally closed map $i$ with quasi-admissible complement $j$.

		\item Define the \emph{exceptionally quasi-proper cdh topology} on $\mathcal C$ to be the Grothendieck topology generated by declaring that the empty sieve covers any initial object, and any family $\{p_k\}_k \cup \{i\}$ is covering, where $i$ is exceptionally closed, $\{p_k\}_k$ generates an elementary cdh cover away from $i$, and for each $k$, $p_k$ is exceptionally quasi-proper.

		% \item Given a family $\mathcal E$ of maps in $\mathcal C$, define the \emph{quasi-admissible $\mathcal E$-cdh topology} on $\mathcal C$ to be the Grothendieck topology generated by declaring that the empty sieve covers any initial object, and $\{j,p\}$ is a covering family for any elementary quasi-admissible cdh cover $(j,p)$ with $j,p \in \mathcal E$. When $\mathcal E$ contains all maps in $\mathcal C$, we may instead simply refer to the \emph{quasi-admissible cdh topology}.

		\item Define the \emph{cdh topology} to be the Grothendieck topology on $\mathcal C$ determined by declaring that every exceptionally quasi-proper cdh covering family is covering, and also that $\{p_k\}_k \cup \{j\}$ is covering if $j$ is a complement of an exceptionally closed map $i$, and $\{p_k\}_k$ generates an elementary cdh cover of $i$, and for each $k$, $p_k$ is a quasi-admissible map.

			% When $\mathcal E$ contains all quasi-admissible maps, and $\mathcal P$ contain all exceptionally quasi-proper maps, we may simply refer to the \emph{cdh topology}.

	\end{itemize}
\end{defn}

% FIX: delete?
\begin{rmk}[Comparison with usual notions of cdh excision and descent for cd-structures] \label{rmk:compare cdh}
	Note that \Cref{defn:cdh} produces topologies that might be a bit finer than usual cdh topologies. For example, given an exceptionally closed map $i : Z \to S$, to produce an elementary cdh cover of $i$, we need to give a map $X \to S$ which only admits a section after base change along $i$, whereas normally one asks for it to instead become invertible after base change along $i$. Furthermore, our general assumptions do not guarantee that the usual results about topologies associated to cd-structures hold (such as \cite[Theorem 2.2.7]{locspalg}, \cite[Theorem 3.2.5]{affrepI}, and \cite[Corollary 5.10]{voe-cd}) -- for example we do not require the exceptionally closed maps to be monomorphisms, nor that the quasi-admissible or exceptionally quasi-proper maps are truncated, and we do not require that the collections of quasi-admissible or exceptionally quasi-proper maps are stable under taking diagonals. Nevertheless, \Cref{prp:elementary cdh descent} still allows us to prove descent for these topologies in \Cref{lem:PPF cdh,lem:desc V6FF}.

	See \Cref{rmk:compare alg cdh} for a comparison with notions of cdh descent for algebraic stacks.
\end{rmk}

Before coming to our main result about 6-functor formalisms, we will need to make the following definition, which should be seen as an abstraction of the situation in algebraic geometry, where Chow's Lemma shows that proper maps are, in some sense, generated by closed immersions and projective bundles (see \Cref{rmk:gen prop in AG}):
\begin{defn} \label{defn:proj sat}
	For any $D \in \PF(\mathcal C)$, a collection $P$ of maps in $\mathcal C$ is \emph{projectively $D$-saturated} if it satisfies the following:
	\begin{enumerate}

		% NOTE: may seem that exceptionally closed maps just provide some extra $D$-quasi-proper maps that are compatible with morphisms
		% but they also provide the cdh descent, which gives the source-locality here
		% NOTE: need diagonal to be exceptionally quasi-proper so that know is $D$-qprop and stuff for the purposes of showing $D'$-sat in proof of {thm:PPF}
		\item $P$ contains every exceptionally closed map, and every quasi-admissible map that is stably $D$-ambidextrous and whose diagonal is an exceptionally quasi-proper map $P$.

		\item $P$ is stable under composition and base change.

		\item Let $f : X \to Y$ be a map such all base changes of $f$ exist.
			\begin{enumerate}

				\item If $Y$ admits a $D$-pseudocover by quasi-admissible maps $Y' \to Y$ such that the base change $X \times_Y Y' \to Y'$ is in $P$, then $f \in P$.

					% NOTE: why need exc qprop and in $P$
					% exc qprop is used to show that is also $D'$-acyclic for any $D \to D'$ (once we show that all exc qprop maps have the desired properties) for showing $D'$-sat in proof of {thm:PPF}
					% in $P$ is to get that is $D$-qprop (and stuff) in the first place (before we've shown that all exc qprop maps are $D$-qprop and stuff)
				\item If there is a $D$-acyclic exceptionally quasi-proper map $\bar Y \to Y$ in $P$ such that $f \times_Y \bar Y$ is in $P$, then $f \in P$.

				\item Suppose that $i : Z \to X$ is an exceptionally closed map, and that $p : \bar X \to X$ is a map in $P$ that is invertible away from $i$.\footnotemark{} If $f \circ i$ and $f \circ p$ are in $P$, then $f \in P$.
					\footnotetext{This means that if $j$ is the complement of $i$, then the base change of $p$ along $j$ is an equivalence.}

			\end{enumerate}

	\end{enumerate}
	Say that $D$ is a \emph{strongly projective pullback formalism} if every projectively $D$-saturated collection contains all exceptionally quasi-proper maps, every exceptionally closed map is $D$-closed, $D$ is reduced, and $D$ takes values in stable categories.
\end{defn}

We are ready to present one of our main applications to 6-functor formalisms. This result roughly says that if $(\mathcal C,E)$ is a geometric setup such that every map in $E$ is cdh locally on the source and target ``compactifiable'' or ``quasi-projective'', then every strongly projective pullback formalism $D^*$ on $\mathcal C$ extends to a 6-functor formalism $D^*$ on $(\mathcal C, E)$ with very good behaviour, and for any morphism of pointed reduced pullback formalisms $\phi^* : D^* \to D'^*$, if all exceptionally closed maps are $D'^*$-closed, then $\phi^*$ extends to a morphism of 6-functor formalisms $\phi : D \to D'$ that also behaves especially well.
\begin{thm} \label{thm:6FF}
	Let $\mathcal C' \subseteq \mathcal C$ be a full subcategory such that
	\begin{itemize}
		
		\item $\mathcal C$ and $\mathcal C'$ admit finite products, and the inclusion $\mathcal C' \to \mathcal C$ preserves finite products,

		\item if $X \to Y$ is quasi-admissible, exceptionally quasi-proper, or exceptionally closed, and $Y \in \mathcal C'$, then $X \in \mathcal C'$, and
			% NOTE: inclusion C' -> C automatically preserves base changes along maps in $P \circ I$

		\item every object of $\mathcal C$ admits a small cdh cover by quasi-admissible maps from objects of $\mathcal C'$.

	\end{itemize}
	Let $I$ and $P$ be collections of quasi-admissible and exceptionally quasi-proper maps respectively in $\mathcal C'$, such that
	\begin{itemize}

		\item the collections $I,P$ each contain all equivalences in $\mathcal C'$, and are stable under base change, composition, and taking diagonals,\footnote{Our assumptions guarantee that the base changes and diagonals are the same when computed in $\mathcal C'$ and in $\mathcal C$, since any map in $I \cup P$ to an object of $\mathcal C'$ is in $\mathcal C'$.}

		\item every map in $I \cap P$ is truncated, and

		\item if $f$ is a composite of maps in $I \cup P$, then $f \simeq p \circ j$ for some $p \in P$ and $j \in I$.
		
	\end{itemize}
	Let $(\mathcal C,E)$ be a geometric setup such that $E$ is stable under taking diagonals, $I,P \subseteq E$, and for any map $X \to Y$ in $E$, there is a small cdh cover of $Y$ consisting of maps $Y' \to Y$ such that $Y', X \times_Y Y' \in \mathcal C'$ and there is a small cdh covering family of $X \times_Y Y'$ consisting of maps $X' \to X \times_Y Y'$ such that both $X' \to X \times_Y Y'$ and $X' \to Y'$ are composites of maps in $I \cup P$.

	Then any strongly projective pullback formalism $D^*$ on $\mathcal C$ extends to a 6-functor formalism $D$ on $(\mathcal C, E)$ satisfying the following properties:
	\begin{enumerate}

		\item\label{itm:6FF/descent}
			$D^*$ and $D^!$ have cdh descent.

			%Furthermore, $D$ (\resp{} $D^!$) sends any constructible cover (\resp{} $I$-constructible cover) to a jointly conservative family of functors.

		\item\label{itm:6FF/suave prim} Every quasi-admissible map is $D$-suave, and every exceptionally quasi-proper map is $D$-prim in the sense of \Cref{defn:suave prim}. (See \Cref{thm:suave prim consequences} for consequences.)

		\item\label{itm:6FF/describe shriek}
			Let $f \in E$ be a map in $\mathcal C'$ that is a composite of maps in $P \cup I$, and for each integer $n \geq 0$, write $\Delta^n_f$ for the $n$-fold diagonal of $f$.\footnote{So $\Delta^0_f = f$, and $\Delta^{n + 1}_f$ is the diagonal of $\Delta^n_f$.}
			If $\Delta^n_f$ is quasi-admissible (\resp{} exceptionally quasi-proper) for $n \geq 0$, and $\Delta^n_f \in I$ (\resp{} $P$) for $n \gg 0$, then $f_! \simeq f_\sharp$ (\resp{} $f_! \simeq f_*$).

		\item\label{itm:6FF/duality}
			Let $f : X \to Y$ be a quasi-admissible map in $E$ such that the diagonal $\Delta$ of $f$ satisfies that $\Delta_! \simeq \Delta_*$. Then we have an equivalence
			\[
				f^! \simeq \Sigma^f f^*
			,\]
			so $D^*$ satisfies the following Poincar\'{e} duality: there is an equivalence of functors $D(Y) \to \spaces$
			\[
				D^\BM(X;-) \simeq D(X;-)[f]\footnotemark
			.\]
			\footnotetext{See \Cref{defn:tangentially twisted cohomology} for this notation.}

		\item\label{itm:6FF/morph}
			For any morphism $\phi^* : D^* \to D'^*$ of pointed reduced pullback formalisms, if all exceptionally closed maps are $D'^*$-closed, then $D'^*$ is also a strongly projective pullback formalism, and $\phi^*$ extends to a morphism $\phi : D \to D'$ of 6-functor formalisms on $(\mathcal C,E)$, \ie, a morphism in the category of lax symmetric monoidal functors $\Span(\mathcal C,E) \to \PrL$. In particular, for any map $f : X \to Y$ in $\mathcal C$, we have canonical equivalences
			\[
				f^* \phi \simeq \phi f^*, \quad\text{and if $f \in E$,} \quad f_! \phi \simeq \phi f_!
			,\]
			which induce right mates
			\[
				\phi f_* \to f_* \phi, \quad\text{and if $f \in E$,} \quad \phi f^! \to f^! \phi
			.\]
			If $f$ is exceptionally quasi-proper, the first transformation is an equivalence, and if $f \in E$ is quasi-admissible, the second transformation is an equivalence.
			
	\end{enumerate}
\end{thm}

In fact, we will prove more refined versions of \Cref{thm:6FF} which compare certain categories of 6-functor formalisms and pullback formalisms, namely \Cref{thm:PPF,rmk:6FF ext functor}. This is also given in \Cref{thmX:V6FF}. The proof of \Cref{thm:6FF} is given at the end of \Cref{S:V6FF}.

\subsection{Projective pullback formalisms} \label{S:PPF}
In this section, we will combine our results about localization and ambidexterity by showing that for any strongly projective pullback formalism $D$, every exceptionally quasi-proper map is $D$-quasi-proper. This is a consequence of a stronger statement that will be given in \Cref{thm:PPF}.

We will first need to define the following categories of pullback formalisms:
\begin{defn} \label{defn:PPF cat}
	We define the following subcategories of $\PF(\mathcal C)$:
	\begin{description}

		\item[Constructible pullback formalisms] The category $\PF^\cstr_\bullet(\mathcal C)$ of \emph{constructible pullback formalisms} on $\mathcal C$ is the full subcategory of $\PF(\mathcal C)$ consisting of pointed reduced pullback formalisms $D$ on $\mathcal C$ such that every exceptionally closed map in $\mathcal C$ is $D$-closed.

		\item[Projective pullback formalisms] The category $\PPF(\mathcal C)$ of \emph{projective pullback formalisms} on $\mathcal C$ is the subcategory of $\PF^\cstr_\bullet(\mathcal C)$ where
			\begin{enumerate}

				\item the objects are those constructible pullback formalisms $D$ on $\mathcal C$ that take values in stable categories, and such that every exceptionally quasi-proper map is $D$-quasi-proper (in addition to every exceptionally closed map being $D$-closed), and

				\item the maps are those morphisms $D \to D'$ in $\PF^\cstr_\bullet(\mathcal C)$ that are right adjointable at exceptionally quasi-proper maps (in addition to being left adjointable at quasi-admissible maps).

			\end{enumerate}
			
	\end{description}
\end{defn}

Our main result implies, that every strongly projective pullback formalism is a projective pullback formalism. In fact, \Cref{thm:PPF} shows that the strongly projective pullback formalisms form a ``cosieve'' in $\PF^\cstr_\bullet(\mathcal C)$. Before coming to this result, we will make some remarks about the general behaviour of constructible and projective pullback formalisms.

% \begin{rmk}[Constructible separation]
% 	If $D \in \PF^\cstr_\bullet(\mathcal C)$, then $D$ is separated for the constructible topology on $\mathcal C$ \ie, the functor $D : \mathcal C^\op \to \widehat{\Cat}$ sends any constructible covering family to a family of functors that is jointly conservative.
% \end{rmk}

\begin{lem}[Atiyah duality for projective pullback formalims] \label{lem:Atiyah duality PPF}
	If $D \in \PPF(\mathcal C)$, then for any quasi-admissible exceptionally quasi-proper map $f$, we have that $D$ satisfies Atiyah duality for $f$. Furthermore, if the diagonal of $f$ is exceptionally quasi-proper, then we have an equivalence
	\[
		\Sigma^{-f} f^\sharp \simeq f^*
	\]
	(where $\Sigma^{-f}$ is a right adjoint of $\Sigma^f$).
	and if $\phi : D \to D'$ is a map in $\PPF(\mathcal C)$, then we have an equivalence
	\[
		\phi \Sigma^f \simeq \Sigma^f \phi
	.\]
	We also have this equivalence if $\phi$ is a map in $\PF^\cstr_\bullet(\mathcal C)$ and the diagonal of $f$ is exceptionally closed.

	In this case, if $f$ is actually tangentially $D$-stable, then $f$ is also stably $D'$-ambidextrous, the transformation
	\[
		\phi f_* \to f_* \phi
	\]
	is an equivalence, and there is an equivalence
	\[
		f^\sharp \phi \simeq \phi f^\sharp
	.\]
	\begin{proof}
		Since $D \in \PPF(\mathcal C)$, $f$ is $D$-quasi-proper, so $D$ has quasi-admissible exchange for $f$, whence $D$ satisfies Atiyah duality for $f$. If the diagonal $\Delta$ of $f$ is exceptionally quasi-proper, then it is $D$-quasi-proper, so we have that $\Delta_*$ admits a right adjoint, so $\Sigma^f$ admits a right adjoint $\Sigma^{-f}$, and the equivalence $\Sigma^{-f} f^\sharp \simeq f^*$ follows from \Cref{rmk:duality and stability give shriek}. 

		The equivalence $\phi \Sigma^f \simeq \Sigma^f \phi$ follows from \Cref{rmk:morph twist} (after applying \Cref{thm:closed}(\ref{itm:closed/morphism}) in the case of exceptionally closed diagonal).

		Finally, \Cref{prp:duality propagation} shows that (after applying \Cref{thm:closed} in the case of exceptionally closed diagonal) if $f$ is tangentially $D$-stable (so $f$ is stably $D$-ambidextrous), then $f$ is also stably $D$-ambidextrous, $\phi$ is right adjointable at $f$, and there is an equivalence $f^\sharp \phi \simeq \phi f^\sharp$.
	\end{proof}
\end{lem}

We also have a result about cdh descent for projective pullback formalisms. In order to formulate this result, it will be convenient to define $\mathcal C^\sharp$ to be the wide subcategory of $\mathcal C$ consisting of the composites of exceptionally quasi-proper maps and exceptionally closed maps, and write $D^\sharp$ for the presheaf $(\mathcal C^\sharp)^\op \to \PrR$ that sends $g$ to $g^\sharp$.
\begin{lem}[Cdh descent for projective pullback formalisms] \label{lem:PPF cdh}
	If $D \in \PPF(\mathcal C)$, then $D$ has cdh descent. Furthermore, every exceptionally quasi-proper cdh cover is a $D^\sharp$-pseudocover, and if the exceptionally closed and exceptionally quasi-proper maps are stable under taking diagonals, then $D^\sharp$ has descent for the exceptionally quasi-proper cdh topology.
	\begin{proof}
		For any exceptionally closed map $i : Z \to S$, and sieve $\mathcal U$ on $S$, if $\mathcal U$ is an elementary cdh cover of $i$, then $i^* \mathcal U$ is a $D$-pseudocover, and if $\mathcal U$ is an elementary cdh cover away from $i$, then $j^* \mathcal U$ is a $D$-pseudocover, where $j$ is a complement of $i$. So since $D$ is a reduced pullback formalism that takes values in stable categories, it follows from \Cref{prp:elementary cdh descent} that $D$ has cdh descent (since exceptionally closed maps and exceptionally quasi-proper maps are stable under base change). In fact, if $\{p_k : X_k \to S\}_k$ is a family of exceptionally quasi-proper maps that generates an elementary cdh cover away from $i$, then there is some $k$ such that $p_k \times_S U : X_k \times_S U \to U$ admits a section, so $\{p_k \times_S U\}_k$ is a $D^\sharp$-pseudocover, whence \Cref{prp:elementary cdh descent} shows that $\{p_k\} \cup \{i\}$ is a $D^\sharp$-pseudocover.

		The descent statement for $D^\sharp$ also follows from \Cref{prp:elementary cdh descent} by using \cite[Lemma 2.1.5]{HM6FF} when the exceptionally closed and exceptionally quasi-proper maps are stable under taking diagonals.
	\end{proof}
\end{lem}

\begin{rmk}
	When every exceptionally quasi-proper map is an equivalence, we have that $\PPF(\mathcal C)$ is the full subcategory of $\PF^\cstr_\bullet(\mathcal C)$ consisting of those constructible pullback formalisms taking values in stable categories.
\end{rmk}

Our main result is the following.
\begin{thm} \label{thm:PPF}
	Let $D \in \PF^\cstr_\bullet(\mathcal C)$ be a strongly projective pullback formalism. Then $D$ is a projective pullback formalism, the natural functor
	\[
		\PPF(\mathcal C)_{D/} \to \PF^\cstr_\bullet(\mathcal C)_{D/}
	\]
	is an equivalence, and in fact, for any morphism $D \to D'$ of constructible pullback formalisms, we have that every projectively $D'$-saturated collection is projectively $D$-saturated, so $D'$ is also a strongly projective pullback formalism.
\end{thm}

We will present the proof of \Cref{thm:PPF} at the end of the section. For now, we will consider the process of restricting and extending projective pullback formalisms to and from subcategories of $\mathcal C$. We first make the following observation:
\begin{rmk} \label{rmk:exc res}
	For any category $\mathcal C'$ equipped with notions of quasi-admissible maps, exceptionally closed maps, and exceptionally quasi-proper maps, we can define categories $\PF(\mathcal C'), \PF^\cstr_\bullet(\mathcal C'), \PPF(\mathcal C')$ without making any assumptions on these collection of maps. If base changes of quasi-admissible maps exist, and $F : \mathcal C' \to \mathcal C$ is a functor that preserves quasi-admissible maps and base changes along quasi-admissible maps, then precomposition by $F$ defines a functor
	\[
		\PF(\mathcal C) \to \PF(\mathcal C')
	.\]
	If exceptionally closed maps in $\mathcal C'$ have quasi-admissible complements, and $F$ also preserves exceptionally closed maps and their complements, then this further restricts to a functor
	\[
		\PF^\cstr_\bullet(\mathcal C) \to \PF^\cstr_\bullet(\mathcal C')
	.\]

	Finally, if $F$ preserves exceptionally quasi-proper maps and base changes along them, then this further restricts to a functor
	\[
		\PPF(\mathcal C) \to \PPF(\mathcal C')
	.\]
\end{rmk}

% \begin{defn} \label{defn:proj anodyne}
% 	Say a map $f$ in $\mathcal C$ is \emph{projectively anodyne} if $f^* = D(f)$ is an equivalence for every $D \in \PPF(\mathcal C)$.
% \end{defn}
%
% \begin{exa}
% 	If $i$ is exceptionally closed and has initial complement, then it is projectively anodyne.
% \end{exa}

\begin{prp} \label{prp:PPF res ext}
	% Let $\mathcal C' \subseteq \mathcal C$ be a full anodyne pullback subcontext, and suppose that exceptionally quasi-proper and exceptionally closed maps in $\mathcal C'$ are stable under base change in $\mathcal C'$, and that for any exceptionally quasi-proper or exceptionally closed map $X \to Y$ in $\mathcal C'$, all base changes of $X \to Y$ in $\mathcal C'$ exist and are respectively either exceptionally quasi-proper or exceptionally closed, and for any map $Z \to Y$ in $\mathcal C'$, the map from the pullback in $\mathcal C'$ to the pullback in $\mathcal C$ is projectively anodyne.
	%
	% Then the exceptionally quasi-proper and exceptionally closed maps in $\mathcal C'$ satisfy the conditions of \Cref{ass:exc} in $\mathcal C'$, and restriction along $\mathcal C' \subseteq \mathcal C$ defines a functor
	% \[
	% 	\PPF(\mathcal C) \to \PPF(\mathcal C')
	% .\]
	% If $\mathcal C'$ is small, and every object of $\mathcal C$ has a small cdh cover consisting of quasi-admissible maps from objects of $\mathcal C'$, then this functor is an equivalence.
	Let $\mathcal C' \subseteq \mathcal C$ be a full subcategory such that for any $Y \in \mathcal C'$, and map $X \to Y$ in $\mathcal C$ that is quasi-admissible, exceptionally closed, or exceptionally proper, we have that $X \in \mathcal C'$. Then $\mathcal C' \subseteq \mathcal C$ is a full anodyne pullback subcontext, the collections of exceptionally closed and exceptionally quasi-proper maps in $\mathcal C'$ satisfy the conditions of \Cref{ass:exc} in $\mathcal C'$, and restriction along $\mathcal C' \subseteq \mathcal C$ defines a functor
	\[
		\PPF(\mathcal C) \to \PPF(\mathcal C')
	.\]
	Furthermore, if $\mathcal C'$ is small, and every object of $\mathcal C$ has a small cdh cover consisting of quasi-admissible maps from objects of $\mathcal C'$, then this functor is an equivalence, and there is a Cartesian square
	% NOTE: instead of asking for qadm cdh covers, could have taken slices under $D$, and asked for qadm $D$-pseudocovers
	% NOTE: need *qadm* pseudocovers to apply {prp:proper local on target} about target-locality of qprop maps
	\[
		\begin{tikzcd}
			\PPF(\mathcal C) \ar[d] \ar[r] & \PPF(\mathcal C') \ar[d] \\
			\Shv^{\qadm \cap \cdh}_{\CAlg(\PrL)}(\mathcal C) \ar[r] & \Psh_{\CAlg(\PrL)}(\mathcal C')
		\end{tikzcd}
	,\]
	where $\qadm \cap \cdh$ is the Grothendieck topology of (small) quasi-admissible cdh covers, the vertical arrows are induced by $\PPF \subseteq \Psh_{\CAlg(\PrL)}$, and the horizontal arrows are given by restriction along $\mathcal C' \subseteq \mathcal C$.
	\begin{proof}
		% The fact that the condition of \Cref{ass:exc} hold in $\mathcal C'$ is immediate from the fact that they hold in $\mathcal C$, the fact that every equivalence between objects of $\mathcal C'$ is in $\mathcal C'$, and the fact that the collections of exceptionally closed and exceptionally quasi-proper maps in $\mathcal C'$ are stable under base change in $\mathcal C'$.
		%
		% Now, since $\mathcal C' \to \mathcal C$ is a morphism of pullback contexts, and preserves base changes of exceptionally closed and exceptionally quasi-proper maps up to projectively anodyne maps, it follows that restriction defines a functor
		% \[
		% 	\PPF(\mathcal C) \to \PPF(\mathcal C')
		% .\]
		%
		% When $\mathcal C'$ is small, and every object of $\mathcal C$ admits a small cdh cover by quasi-admissible maps from objects of $\mathcal C'$, we can apply \cite[Proposition 2.1.14]{Fundamentals}, and when combined with \cite[Proposition 2.2.6]{Fundamentals}, it follows that the restriction functor
		% \[
		% 	\PF(\mathcal C) \to \PF(\mathcal C')
		% \]
		% restricts to an equivalence on the full subcategories consisting of pullback formalisms that have descent along small quasi-admissible cdh covers.
		%
		% Thus, we conclude by \Cref{rmk:PPF cdh}, \Cref{prp:proper local on target}, and the fact that stable presentable categories are closed under small limits.

		The fact that $\mathcal C'$ is a full anodyne pullback subcontext of $\mathcal C$ follows from \cite[Remark 1.2.7]{Fundamentals}, and the same argument shows that the collections of exceptionally closed maps and exceptionally quasi-proper maps satisfy the conditions of \Cref{ass:exc} in $\mathcal C'$, and that the inclusion $\mathcal C' \to \mathcal C$ preserves and reflects base changes along these maps. Furthermore, if $i : Z \to S$ is an exceptionally closed map in $\mathcal C'$, then it is also exceptionally closed in $\mathcal C$, so since every quasi-admissible map to $S$ is in $\mathcal C'$, we have that the complement of $i$ in $\mathcal C$ is in $\mathcal C'$, so the inclusion also preserves and reflects complements of exceptionally closed maps. Thus, \Cref{rmk:exc res} shows that restriction defines a functor
		\[
			\PPF(\mathcal C) \to \PPF(\mathcal C')
		.\]

		Since every object of $\mathcal C$ has a small quasi-admissible cdh cover by objects of $\mathcal C'$, we may apply \cite[Proposition 2.2.6]{Fundamentals} and \Cref{lem:locality of closedness} to obtain a Cartesian square
		\[
			\begin{tikzcd}
				\PF^\cstr_\bullet(\mathcal C; \qadm \cap \cdh) \ar[d] \ar[r] & \PF^\cstr_\bullet(\mathcal C'; \qadm \cap \cdh) \ar[d] \\
				\Shv^{\qadm \cap \cdh}_{\CAlg(\PrL)}(\mathcal C) \ar[r] & \Shv^{\qadm \cap \cdh}_{\CAlg(\PrL)}(\mathcal C')
			\end{tikzcd}
		,\]
		where the top arrow is the restriction of $\PF^\cstr_\bullet(\mathcal C) \to \PF^\cstr_\bullet(\mathcal C')$ to the full subcategories consisting of constructible pullback formalisms that have descent along small quasi-admissible cdh covers, and the vertical arrows are the evident inclusions. When $\mathcal C$ is small, we can apply \cite[Proposition 2.1.14]{Fundamentals} to see that the bottom map is an equivalence, so the top one is too.

		Since exceptionally quasi-proper maps in $\mathcal C$ are stable under base change along quasi-admissible maps from objects of $\mathcal C'$, we find that the top arrow restricts to an equivalence
		\[
			\PPF(\mathcal C) \to \PPF(\mathcal C')
		\]
		by \Cref{lem:PPF cdh}, \Cref{prp:proper local on target}, and the fact that stable presentable categories are closed under small limits. We conclude since $\Shv^{\qadm \cap \cdh}_{\CAlg(\PrL)}(\mathcal C') \to \Psh_{\CAlg(\PrL)}(\mathcal C')$ is a monomorphism of categories.
	\end{proof}
\end{prp}

\begin{proof}[Proof of \Cref{thm:PPF}]
	It will suffice to show that $D$ is a projective pullback formalism, and that if $\phi : D \to D'$ is a morphism of constructible pullback formalisms, then $D'$ takes values in stable categories, every projectively $D'$-saturated collection is projectively $D$-saturated, and $\phi$ is right adjointable at exceptionally quasi-proper maps. Indeed, this will show that $D'$ satisfies the same hypotheses, so it is also a projective pullback formalism, $\phi$ is a map in $\PPF(\mathcal C)$, and any map in $\PF^\cstr_\bullet(\mathcal C)$ from $D'$ is also a map in $\PPF(\mathcal C)$.

	Note that by \cite[Lemma E.0.1]{Fundamentals}, if $\mathcal A \in \CAlg(\PrL)$ has a zero object, then $\mathcal A$ is stable if and only if the suspension of the monoidal unit is $\otimes$-invertible. Thus, it follows that since $D,D'$ take values in categories that have zero objects, since $D(S)$ is stable and $\phi : D(S) \to D'(S)$ is symmetric monoidal and preserves finite colimits for all $S \in \mathcal C$, we must have that $D'(S)$ is also stable. Thus, $D'$ takes values in stable categories.

	Let $P$ be the collection of maps $f$ in $\mathcal C$ such that every base change of $f$ exists and is $D$-quasi-proper and $D'$-quasi-proper, and $\phi$ is right adjointable at every base change of $f$. We will show that $P$ is projectively $D$-saturated:
	% NOTE: only need to ask for $D'$-quasi-proper so that have that $D'(f)$ has colimit-preserving right adjoint, to be used in last step.
	\begin{enumerate}

		\item By \Cref{cor:closed immersions are proper} and \Cref{thm:closed}(\ref{itm:closed/morphism}), we have that every exceptionally closed map is in $P$. Now, let $f$ be a quasi-admissible stably $D$-ambidextrous map whose diagonal is in $P$. By \Cref{prp:duality propagation}, we have that $f$ is also stably $D'$-ambidextrous, that $D$ and $D'$ have the right projection formula for $f$, and that $\phi$ is right adjointable at $f$. Next, we see that by \Cref{prp:duality qadm base change,prp:duality base change}, $D$ and $D'$ have quasi-admissible exchange and right base change for $f$. Using the fact that $P$ is stable under base change (by definition), we also have that every base change of $f$ has diagonal in $P$ and is stably $D$-ambidextrous by \Cref{prp:duality base change}. We conclude that $f \in P$ by \Cref{rmk:duality and stability give shriek}.

		\item $P$ is stable under base change by definition, and it is stable under composition by \Cref{lem:composites of proper maps} and \cite[Lemma F.6(2)]{TwAmb}.

		\item Let $f : X \to Y$ be a map such that all base changes of $f$ exist.
			\begin{enumerate}

				\item By \Cref{prp:proper local on target}, if $Y$ admits a $D$-pseudocover by quasi-admissible maps $Y' \to Y$ such that the base change $X \times_Y Y' \to Y'$ is in $P$, then $f$ is $D$-quasi-proper, and $\phi$ is right adjointable at $f$. Since every quasi-admissible $D$-pseudocover is also a (quasi-admissible) $D'$-pseudocover by \Cref{thm:D-topology}, we also have that $f$ is $D'$-quasi-proper. Since $P$ is stable under base change, we find that by taking base changes of the quasi-admissible $D$-pseudocovers, we may apply the same argument to every base change of $f$, so if all base changes of $f$ exist, then $f \in P$.

				\item Suppose $q : \bar Y \to Y$ is a $D$-acyclic map in $P$. Using the fact that $\phi$ is right adjointable at every base change of $q$, and that both $D$ and $D'$ have the right projection formula for every base change of $q$, the dual of \cite[Lemma D.2.3]{Fundamentals} shows that all base changes of $q$ are both $D$-acyclic and $D'$-acyclic. Thus, if the base change $X \times_Y \bar Y \to \bar Y$ of $f$ along $q$ is in $P$, then by \Cref{prp:proper proper-local on target}, we find that for any base change $f'$ of $f$, we have that $\phi$ is right adjointable at $f'$, and $f'$ is $D$-quasi-proper and $D'$-quasi-proper. Hence $f \in P$. 

				\item Let $i : Z \to X$ be an exceptionally closed map with complement $j : U \to X$, and let $p : \bar X \to X$ be a map in $P$ such that the base change of $p$ along $j$ is invertible. 

					Since $D$ and $D'$ take values in stable categories, by \Cref{prp:gluing implies exc}, we have that $D^\sharp$ and $(D')^\sharp$ send the Cartesian square
					\[
						\begin{tikzcd}
							\bar Z \ar[d] \ar[r] & \bar X \ar[d] \\
							Z \ar[r] & X
						\end{tikzcd}
					\]
					to a Cartesian square. Once again, using the fact that $D$ and $D'$ take values in stable categories, and $\bullet \to \bullet \gets \bullet$ is a finite simplicial set, we have that by \Cref{prp:proper local on source} (and where $i$ is $D$-quasi-proper and $D'$-quasi-proper by \Cref{cor:closed immersions are proper}), if $f \circ i$ and $f \circ p$ are $D$-quasi-proper and $D'$-quasi-proper, then $f$ is $D$-quasi-proper and $D'$-quasi-proper. In particular $D'(f)$ has a colimit-preserving right adjoint, so by \Cref{lem:pre source locality of proper}, if we also know that $\phi$ is right adjointable at $f \circ i$ and $f \circ p$, then $\phi$ is right adjointable at $f$.

					Since $P$ and the collection of exceptionally closed maps are stable under base change, we can apply the same argument to every base change of $f$, so that $f \in P$, as desired.

			\end{enumerate}

	\end{enumerate}
	Thus, since every exceptionally quasi-proper map is contained in every projectively $D$-saturated collection, it follows that every exceptionally quasi-proper map $f$ is in $P$, so $f$ is $D$-quasi-proper and $D'$-quasi-proper, and $\phi$ is right adjointable at $f$.

	Finally, we will show that every projectively $D'$-saturated collection $P'$ is projectively $D$-saturated. \Cref{thm:D-topology} shows that every $D$-pseudocover consisting of quasi-admissible maps is also a $D'$-pseudocover, and the dual of \cite[Lemma D.2.3]{Fundamentals} shows that if $q$ is a $D$-acyclic exceptionally quasi-proper map, then $q$ is also $D'$-acyclic (since we have shown that $\phi$ is right adjointable at $q$). It only remains to show that if $f$ is a quasi-admissible stably $D$-ambidextrous map whose diagonal is exceptionally quasi-proper and in $P'$, then $f \in P'$. Indeed, we have already shown that every exceptionally quasi-proper map is $D$-quasi-proper and $D'$-quasi-proper, and that $\phi$ is right adjointable at every exceptionally quasi-proper map. Therefore, $f$ is also stably $D'$-ambidextrous by \Cref{prp:duality propagation}, so it is in $P'$ since its diagonal is exceptionally quasi-proper and in $P'$.
\end{proof}

\subsection{Voevodsky 6-functor formalisms} \label{S:V6FF}

We now turn our attention to results about 6-functor formalisms, all of which will require the following condition that we assume for the rest of the section:
\begin{ass} \label{ass:fin prod}
	% NOTE: need finite products to apply CLL
	The category $\mathcal C$ admits finite products.
\end{ass}

We now define a category of highly structured 6-functor formalisms:
\begin{defn} \label{defn:V6FF}
	% NOTE: see CLL 4.12
	% need $E$ to also be stable under diagonals for the corresponding wide subcategory to be closed under pullbacks
	Let $I,P,E$ be collections of maps in $\mathcal C$ that are stable under base change and taking diagonals, contain all identities, and such that $I,P \subseteq E$. We can use \cite[\S4]{CLL6FF} to construct the symmetric monoidal 2-category $\Span_2(\mathcal C,E)_{P,I}$, along with the symmetric monoidal 2-functor $\mathcal C^\op \to \Span_2(\mathcal C,E)_{P,I}$.

	We define a category $\VsixFF(\mathcal C, E)_{P,I}$ of \emph{Voevodsky-6-functor formalisms} on $(\mathcal C, E, I, P)$ to be the subcategory of the category $\Alg_{\Span_2(\mathcal C, E)_{P,I}}(\PrL)$\footnotemark of lax symmetric monoidal 2-functors $\Span_2(\mathcal C, E)_{P,I} \to \PrL$ where
	\footnotetext{If $\mathcal A, \mathcal B$ are symmetric monoidal 2-categories, we write $\Alg_{\mathcal A}(\mathcal B)$ to denote the category of lax symmetric monoidal 2-functors $\mathcal A \to \mathcal B$. When $\mathcal A$ is just a symmetric monoidal category, this is the same as the category of lax symmetric monoidal functors $\mathcal A \to \mathcal B$, as in \cite[Definition 2.1.2.7]{ha}.}
	\begin{description}

		\item[Objects] are those lax symmetric monoidal 2-functors $D : \Span_2(\mathcal C, E)_{P,I} \to \PrL$ satisfying the following:
			\begin{itemize}

				\item $D$ takes values in \emph{stable} categories.

				\item After restricting $D$ along $\mathcal C^\op \to \Span_2(\mathcal C, E)_{P, I}$ and using \cite[Theorem 2.4.3.18]{ha} to obtain a functor $D^* : \mathcal C^\op \to \CAlg(\PrL)$, $D$ has left base change and the left projection formula for quasi-admissible maps, right base change and the right projection formula for exceptionally quasi-proper maps, and quasi-admissible exchange for exceptionally quasi-proper maps,

				\item Every quasi-admissible map is $D|_{\Span(\mathcal C,E)}$-suave, and every exceptionally quasi-proper map is $D|_{\Span(\mathcal C,E)}$-prim (see \Cref{defn:suave prim}).

				\item For any exceptionally closed map $i : Z \to S$, with quasi-admissible complement $j : U \to S$,
					\[
						D(Z) \xrightarrow{i_*} D(S) \xrightarrow{j^*} D(U)
					\]
					is a fibre sequence, and
					\[
					% TODO: explain that the maps are counit and unit
						j_\sharp j^* \to \id \to i_* i^*
					\]
					is an exact triangle of endofunctors of $D(S)$, where we use that $j$ is $D$-suave to get the existence of a left adjoint $j_\sharp$ of $j^*$.

				\item $D^*$ and $D^!$ have cdh descent.

			\end{itemize}

		\item[Morphisms] are those transformations $\phi : D \to D'$ such that
			\begin{itemize}

				\item the transformation $\phi^* : D^* \to D'^*$ is left adjointable at quasi-admissible maps, and right adjointable at exceptionally quasi-proper maps, and

				\item the transformation $\phi_! : D_! \to D'_!$ is right adjointable at quasi-admissible maps, and left adjointable at exceptionally quasi-proper maps.

			\end{itemize}
			
	\end{description}

	We also write $\VsixFF(\mathcal C, E)$ to denote $\VsixFF(\mathcal C,E)_{\text{equivalences}, \text{equivalences}}$.
\end{defn}

\begin{lem} \label{lem:desc V6FF}
	Let $I,P,E$ be as in \Cref{defn:V6FF}. The subcategory $\VsixFF(\mathcal C, E)_{P,I}$ of $\Alg_{\Span_2(\mathcal C, E)_{P,I}}(\PrL)$ is equal to the following apparently larger subcategory of $\Alg_{\Span_2(\mathcal C,E)_{P,I}}(\PrL)$.
	\begin{description}

		\item[Objects] are those $D : \Span_2(\mathcal C, E)_{P,I} \to \PrL$ satisfying the following:
			\begin{itemize}

				\item The associated $D^* : \mathcal C^\op \to \CAlg(\PrL)$ is a projective pullback formalism.

				\item Every quasi-admissible map is $D|_{\Span(\mathcal C,E)}$-suave, and every exceptionally quasi-proper map is $D|_{\Span(\mathcal C,E)}$-prim (see \Cref{defn:suave prim}).

			\end{itemize}

		\item[Morphisms] are those transformations $\phi : D \to D'$ such that the associated transformation $\phi^* : D^* \to D'^*$ in $\Fun(\mathcal C^\op, \CAlg(\PrL))$ is a morphism of projective pullback formalisms.
			
	\end{description}
	\begin{proof}
		The fact that this subcategory has the correct morphisms follows from \Cref{prp:suave prim morph}(\ref{itm:suave prim morph}), so it only remains to show that it has the correct objects. Indeed, by the definition of projective pullback formalisms, it only remains to show that $D^*$ and $D^!$ have cdh descent.

		The fact that $D^*$ has cdh descent follows from \Cref{lem:PPF cdh}, which also shows that exceptionally quasi-proper cdh covers are $D^\sharp$-pseudocovers. Thus, we may use \Cref{lem:suave prim descent} to conclude that $D^!$ also has cdh descent.
	\end{proof}
\end{lem}

\begin{rmk}
	If $I,P,E$ all denote the collection of equivalences in $\mathcal C$, then using \Cref{lem:desc V6FF}, it is clear that $D \mapsto D^*$ defines an equivalence
	\[
		\VsixFF(\mathcal C,E)_{P,I} \to \PPF(\mathcal C)
	.\]
	We will consider more general cases when this holds in \Cref{lem:V6FF PPF,rmk:6FF ext functor}.
\end{rmk}

\begin{rmk}[Poincar\'{e} duality for Voevodsky-6-functor formalisms] \label{rmk:V6FF duality}
	Suppose that $(\mathcal C,E)$ is a geometric setup in the sense of \cite[Convention 2.1.3]{HM6FF}, and $D \in \VsixFF(\mathcal C,E)$. Let $f : X \to Y$ be a quasi-admissible map in $E$ such that the diagonal $\Delta : X \to X \times_Y X$ of $f$ satisfies that $\Delta_* \simeq \Delta_!$. Then we have an equivalence
	\[
		f^! \simeq \Sigma^f f^*
	,\]
	so $D^*$ satisfies the following Poincar\'{e} duality: there is an equivalence of functors $D(Y) \to \spaces$
	\[
		D^\BM(X;-) \simeq D(X;-)[f]\footnotemark
	.\]
	\footnotetext{See \Cref{defn:tangentially twisted cohomology} for this notation.}
	\begin{proof}
		Since $f$ is quasi-admissible, it is $D$-suave, so since it is in $E$, by \Cref{lem:suave prim} and \cite[Lemma 4.5.6]{HM6FF}, we have that $\omega_f \simeq \pi_\sharp \Delta_! 1$, where $\pi : X \times_Y X \to X$ is one of the projections. Since $\Delta, \pi \in E$, we may use \cite[Proposition A.5.8(iv)]{Mann6FF} or \cite[Proposition 3.1.8(iv)]{Mann6FF} to see that
		\[
			\omega_f \otimes - \simeq \pi_\sharp \Delta_! 1 \otimes - \simeq \pi_\sharp(\Delta_! 1 \otimes \pi^*) \simeq \pi_\sharp \Delta_!(1 \otimes \Delta^* \pi^*) \simeq \pi_\sharp \Delta_!
		.\]
		Since $\Delta_! \simeq \Delta_*$, we find that
		\[
			\omega_f \otimes - \simeq \pi_\sharp \Delta_* = \Sigma^f
		.\]
		Using \Cref{lem:suave prim} again, we may apply \cite[Corollary 4.5.11(i)]{HM6FF} to find that
		\[
			f^! \simeq \omega_f \otimes f^* \simeq \Sigma^f f^*
		.\]

		Thus,
		\[
			D^\BM(X;-) = D(X)(1, f^!) \simeq D(X)(1, \Sigma^f f^*) = D(X;-)[f]
		.\]
	\end{proof}
\end{rmk}

We now consider the following setting:
\begin{setng} \label{setng:V6FF transfer}
	Let $I,P,E$ be collections of maps in $\mathcal C$ as in \Cref{defn:V6FF}. Let $\mathcal C' \subseteq \mathcal C$ be a full subcategory, and let $I',P',E'$ also be collections of maps in $\mathcal C'$ as in \Cref{defn:V6FF}. Assume the following:
	\begin{itemize}

		\item $I' \subseteq I$, $P' \subseteq P$, and $E' \subseteq E$.

		\item If $X \to Y$ is a map that is quasi-admissible, exceptionally quasi-proper, or exceptionally closed, and $Y \in \mathcal C'$, then $X \in \mathcal C'$.

		\item The category $\mathcal C'$ admits finite products and base changes along maps in $E'$, and the inclusion $\mathcal C' \to \mathcal C$ preserves these.

		\item Every object of $\mathcal C$ admits a small cdh cover consisting of quasi-admissible maps from objects of $\mathcal C'$.
			% NOTE: instead of asking for qadm cdh covers, could have taken slices under $D$, and asked for qadm $D$-pseudocovers

		\item For any $X \to Y$ in $E$, there is a small cdh cover of $Y$ consisting of maps $Y' \to Y$ such that $Y', X \times_Y Y' \in \mathcal C'$, and there is a small cdh cover of $X \times_Y Y'$ by maps $X' \to X \times_Y Y'$ in $E'$ such that $X' \to Y'$ is also in $E'$.

		% \item every object of $\mathcal C$ has a small cdh cover consisting of quasi-admissible maps from objects of $\mathcal C'$, and that
		%
		% \item for any $Y' \in \mathcal C'$, if $X \to Y'$ is in $E$, then $X \in \mathcal C'$, and there is a small cdh cover of $X$ by maps $X' \to X$ in $E'$ such that $X' \to Y'$ is also in $E'$.

	\end{itemize}
\end{setng}

\begin{lem} \label{lem:V6FF res}
	In \Cref{setng:V6FF transfer}, restriction along $\Span_2(\mathcal C', E')_{P',I'} \to \Span_2(\mathcal C, E)_{P, I}$ induces a functor
	\[
		\VsixFF(\mathcal C,E)_{P,I} \to \VsixFF(\mathcal C',E')_{P',I'}
	.\]
	Furthermore, if $\mathcal C' \to \mathcal C$ is an equivalence, this fits into a Cartesian square
	\[
		\begin{tikzcd}
			\VsixFF(\mathcal C, E)_{P, I} \ar[d] \ar[r] & \VsixFF(\mathcal C', E')_{P',I'} \ar[d] \\
			\Alg_{\Span_2(\mathcal C, E)_{P, I}}(\PrL) \ar[r] & \Alg_{\Span_2(\mathcal C', E')_{P',I'}}(\PrL)
		\end{tikzcd}
	\]
	where the vertical maps are the usual inclusions, and the horizontal maps are given by restriction along $\Span_2(\mathcal C',E')_{P',I'} \to \Span_2(\mathcal C, E)_{P, I}$.

	In general, we still get a Cartesian square if the bottom map is replaced by its restriction to the full subcategory consisting of those $D : \Span_2(\mathcal C,E)_{P,I} \to \PrL$ such that $D^*$ has descent along small quasi-admissible cdh covers.

	In particular, for any $D,D' \in \VsixFF(\mathcal C,E)_{P,I}$, any transformation $D \to D'$ is a morphism in $\VsixFF(\mathcal C,E)_{P,I}$ if and only if it restricts to a morphism in $\VsixFF(\mathcal C',E')_{P',I'}$.
	% admits a fully faithful section whose essential image consists of those $D \in \VsixFF(\mathcal C,E)$ such that $D^*$ is right Kan extended from $\mathcal C'^\op \subseteq \mathcal C$. Furthermore, this section is a left adjoint if $\mathcal C' \to \mathcal C$ is an equivalence, and is a right adjoint if every map in $E$ to an object of $\mathcal C'$ is in $E'$.
	
	% Let $I,P,E$ and $\tilde I, \tilde P, \tilde E$ be collections of maps in $\mathcal C$ as in \Cref{defn:V6FF}, and suppose that
	% \[
	% 	E \subseteq \tilde E \quad\text{and}\quad I \subseteq \tilde I \quad\text{and}\quad P \subseteq \tilde P
	% .\]
	% Also assume that for every map $X \to Y$ in $\tilde E$, there is a small cdh cover of $X$ consisting of maps $X' \to X$ in $E$ such that $X' \to Y$ is also in $E$. Then the following square is Cartesian
	% \[
	% 	\begin{tikzcd}
	% 		\VsixFF(\mathcal C, \tilde E)_{\tilde P, \tilde I} \ar[d] \ar[r] & \VsixFF(\mathcal C, E)_{P,I} \ar[d] \\
	% 		\Alg_{\Span_2(\mathcal C, \tilde E)_{\tilde P, \tilde I}}(\PrL) \ar[r] & \Alg_{\Span_2(\mathcal C, E)_{P,I}}(\PrL)
	% 	\end{tikzcd}
	% ,\]
	% where the vertical arrows are the usual inclusions, and the bottom horizontal arrow is induced by restriction along the inclusion $\Span_2(\mathcal C, E)_{P,I} \to \Span_2(\mathcal C, \tilde E)_{\tilde P, \tilde I}$.
	%
	% Furthermore, in the case that $\tilde I, \tilde P, I, P$ all consist only of equivalences, we actually have that the restriction functor
	% \[
	% 	\VsixFF(\mathcal C, \tilde E) \to \VsixFF(\mathcal C, E)
	% \]
	% admits a fully faithful left adjoint.
	\begin{proof}
		It is helpful to recall \cite[\S1.2.11]{htt} for the notion of subcategories.

		When $\mathcal C' \to \mathcal C$ is an equivalence, any $D : \Span_2(\mathcal C,E)_{P,I} \to \PrL$ such that $D|_{\Span_2(\mathcal C',E')_{P',I'}}$ is a Voevodsky-6-functor formalism satisfies that $D^*$ has descent along small quasi-admissible cdh covers, so the first statement follows from the second.

		To show the second statement, it suffices to show that a lax symmetric monoidal 2-functor $D : \Span_2(\mathcal C,E)_{P, I} \to \PrL$ is in$\VsixFF(\mathcal C,E)_{P,I}$ if and only if $D^*$ has descent for small quasi-admissible cdh covers and $D|_{\Span_2(\mathcal C',E')_{P',I'}} \in \VsixFF(\mathcal C',E')_{P',I'}$, and that a transformation $D \to D'$ between Voevodsky-6-functor formalisms is a morphism in $\VsixFF(\mathcal C, E)_{P, I}$ if and only if it restricts to a morphism in $\VsixFF(\mathcal C', E')_{P',I'}$.

		Since we have a commutative square of restriction functors
		\[
			\begin{tikzcd}
				\Alg_{\Span_2(\mathcal C,E)_{P,I}} \PrL \ar[d] \ar[r] & \Alg_{\Span_2(\mathcal C',E')_{P',I'}} \PrL \ar[d] \\
				\Psh_{\CAlg(\PrL)}(\mathcal C) \ar[r] & \Psh_{\CAlg(\PrL)}(\mathcal C')
			\end{tikzcd}
		,\]
		we may apply \Cref{lem:desc V6FF,prp:PPF res ext} to reduce to checking that if $D_{\Span_2(\mathcal C',E')_{P',I'}} \in \VsixFF(\mathcal C',E')_{P',I'}$ and $D^*$ has descent for small quasi-admissible cdh covers, then every quasi-admissible map is $D$-suave, and every exceptionally quasi-proper map is $D$-prim.

		Using the fact that $\mathcal C' \to \mathcal C$ preserves all base changes along maps in $E$, we find that every quasi-admissible (\resp{} exceptionally quasi-proper) map in $\mathcal C'$ is $D$-suave (\resp{} prim) against maps in $E'$. Since $D^*$ has descent for small quasi-admissible cdh covers, and $D^*|_{\mathcal C'^\op}$ is a projective pullback formalism, \Cref{prp:PPF res ext} shows that $D^*$ is a projective pullback formalism. Therefore we may apply \Cref{lem:suave prim target-local} to reduce to showing that for every quasi-admissible (\resp{} exceptionally quasi-proper map) $f : X \to Y$, if $Y \in \mathcal C'$, then $f$ is $D$-suave (\resp{} prim) against maps $Y' \to Y$ in $E$ where $Y' \in \mathcal C'$.

		By our assumptions, we have that there is a small cdh cover of $Y'$ consisting of maps $Y'' \to Y'$ in $E'$ such that $Y'' \to Y$ is also in $E'$, so we conclude by \Cref{lem:suave prim source-local}, which we can apply since $E'$ is right-cancellative by \cite[Lemma 2.1.5]{HM6FF}). In the exceptionally quasi-proper case we also need to note that for any base change $f'$ of $f$, $f'_*$ admits a right adjoint since $f'$ is exceptionally quasi-proper and $D^*$ is a projective pullback formalism.

	\end{proof}
\end{lem}

\begin{lem} \label{lem:V6FF ext}
	In \Cref{setng:V6FF transfer}, the restriction functor
	\[
		\VsixFF(\mathcal C,E) \to \VsixFF(\mathcal C', E')
	\]
	is an equivalence.
	\begin{proof}
		Note that our hypotheses guarantee that every object of $\mathcal C$ admits a small cdh cover consisting of quasi-admissible maps from objects of $\mathcal C'$. Thus, by \cite[Proposition 2.1.14]{Fundamentals} (or \cite[Lemma C.3]{quadratic-refinement-GLV-trace}), any presheaf on $\mathcal C$ has descent for small quasi-admissible cdh covers if and only if it is right Kan extended from a presheaf on $\mathcal C'$ that has descent for small quasi-admissible cdh covers.
		
		Since every Voevodsky-6-functor formalism $D$ satisfies that $D^*$ and $D^!$ have cdh descent, \Cref{prp:ext 6FF,lem:V6FF res} show that
		\[
			\VsixFF(\mathcal C,E) \to \VsixFF(\mathcal C', E')
		\]
		admits a fully faithful section whose essential image is given by those $D \in \VsixFF(\mathcal C,E)$ such that $D^*$ is a right Kan extension of $D^*|_{\mathcal C'^\op}$. Since every $D \in \VsixFF(\mathcal C, E)$ satisfies that $D^*$ has cdh descent, we have that $D^*$ is right Kan extended from $D^*|_{\mathcal C'^\op}$, which shows that the above fully faithful section is essentially surjective.
	\end{proof}
\end{lem}

\begin{rmk} \label{rmk:2-cat ext 6FF}
	The extension result \Cref{prp:ext 6FF} is only given for 6-functor formalisms that are lax symmetric monoidal functors $\Span(\mathcal C,E) \to \PrL$. We expect that it should be possible to prove a version of this result for 2-categorical 6-functor formalisms given by lax symmetric monoidal 2-functors $\Span_2(\mathcal C,E)_{P,I} \to \PrL$. In this case, it would not be necessary to assume that $I,P$ consist only of equivalences in \Cref{lem:V6FF ext}, which would allow for enhancements of many of our results, as well simplifications of some of our arguments, such as \Cref{rmk:6FF ext functor}.
\end{rmk}

For the remainder of the section, we will fix collections $I, P, P \circ I$ of maps in $\mathcal C$, and make the following assumptions:
\begin{ass} \label{ass:IPE}
	\hfill
	\begin{enumerate}

		\item The collections $I,P$ are stable under composition, base change, and taking diagonals.

		\item The collection $P \circ I$ consists of all composites of maps in $I \cup P$.

		\item Every map in $I$ is quasi-admissible, and every map in $P$ is exceptionally quasi-proper.

		\item Every map in $I \cap P$ is truncated.

		\item Every equivalence is in $P \circ I$.

	\end{enumerate}
\end{ass}

It follows that $P \circ I$ is also stable under composition, base change, and taking diagonals, and contains all equivalences, so that $(\mathcal C, P \circ I)$ is a geometric setup in the sense of \cite[Convention 2.1.3]{HM6FF}.

\begin{lem} \label{lem:V6FF PPF}
	Suppose that every map in $P \circ I$ is of the form $p \circ j$ for $p \in P$ and $j \in I$. Then restriction along $\mathcal C^\op \to \Span_2(\mathcal C, P \circ I)_{P, I}$ induces an equivalence\footnote{Recall \cite[Theorem 2.4.3.18]{ha} for the identification of lax symmetric monoidal functors $\mathcal C^\op \to \PrL$ with presheaves $\mathcal C^\op \to \CAlg(\PrL)$.}
	\[
		\VsixFF(\mathcal C, P \circ I)_{P,I} \to \PPF(\mathcal C)
	.\]
	\begin{proof}
		By \cite[Theorem B and Example 4.34]{CLL6FF}, for any symmetric monoidal 2-category $\mathcal V$, restriction along $\mathcal C^\op \to \Span_2(\mathcal C, P \circ I)_{P,I}$ induces an equivalence from the category of lax symmetric monoidal 2-functors $\Span_2(\mathcal C, P \circ I)_{P,I} \to \mathcal V$ to the subcategory of $\Fun(\mathcal C^\op, \CAlg(\mathcal V))$ whose morphisms are those transformations that are left adjointable at maps in $I$ and right adjointable at maps in $P$, and whose objects are those $D : \mathcal C^\op \to \CAlg(\mathcal V))$ that have the left projection formula and left base change for maps in $I$, and the right projection formula and right base change for maps in $P$, and have right-left base change for maps in $P$ against maps in $I$.

		Let $\mathcal V$ be the 2-category of stable presentable categories with symmetric monoidal structure given in \cite[Proposition 4.8.2.18]{ha}. By \Cref{lem:desc V6FF}, it therefore follows that restriction along $\mathcal C^\op \to \Span_2(\mathcal C,P \circ I)_{P,I}$ induces a fully faithful functor $\VsixFF(\mathcal C,P \circ I)_{P,I} \to \PPF(\mathcal C)$, so it only remains to show that it is essentially surjective.

		Any projective pullback formalism on $\mathcal C$ is in the subcategory of $\Fun(\mathcal C^\op, \CAlg(\mathcal V))$ mentioned above, so that \cite[Theorem B and Example 4.34]{CLL6FF} shows that it extends to a lax symmetric monoidal 2-functor $D : \Span_2(\mathcal C, P \circ I)_{P,I} \to \mathcal V$, and it only remains to show that $D$ is a Voevodsky-6-functor formalism. Using \Cref{lem:desc V6FF} again, we reduce to showing that every quasi-admissible map is $D$-suave,  and every exceptionally quasi-proper map is $D$-prim. This follows immediately from \Cref{prp:nat sq adj 6FF}(\ref{itm:nat sq adj/suave prim}).

		% Since $D^!$ is automatically a reduced presheaf, to show that $D^!$ has cdh descent, it suffices to show that for any exceptionally closed map $i : Z \to S$ in $P \circ I$, if $\{p_k : X_k \to S\}_k$ is a family of maps in $P \circ I$, then we have the following:
		% \begin{enumerate}
		%
		% 	\item $D^!$ has descent along $\{i\} \cup \{p_k\}_k$ if $p_k$ is exceptionally quasi-proper for all $k$, and $\{p_k\}_k$ generates an elementary cdh cover away from $i$, and 
		%
		% 	\item $D^!$ has descent along $\{j\} \cup \{p_k\}_k$ if $j$ is a complement of $i$ that is in $P \circ I$, $p_k$ is quasi-admissible for all $k$, and $\{p_k\}_k$ is an elementary cdh cover of $i$.
		%
		% \end{enumerate}
		% Indeed, the first follows from 
		% \begin{enumerate}
		%
		% 	\item If $\{X_i \to S\}_i$ is a family of quasi-admissible maps in $P \circ I$ that is a $D^*$-pseudocover, then 
		%
		% \end{enumerate}
		
	\end{proof}
\end{lem}

\begin{lem} \label{lem:V6FF truncated}
	Suppose that every map in $I \cup P$ is truncated. Restriction along $\Span(\mathcal C, P \circ I) \to \Span_2(\mathcal C, P \circ I)_{P,I}$ induces an equivalence
	\[
		\VsixFF(\mathcal C, P \circ I)_{P,I} \to \VsixFF(\mathcal C, P \circ I)
	.\]

	In particular, restriction along $\mathcal C^\op \to \Span(\mathcal C, P \circ I)$ induces an equivalence
	\[
		\VsixFF(\mathcal C,P \circ I) \to \PPF(\mathcal C)
	.\]
	\begin{proof}
		By \Cref{lem:V6FF res}, the following square is Cartesian:
		\[
			\begin{tikzcd}
				\VsixFF(\mathcal C, P \circ I)_{P,I} \ar[d] \ar[r] & \VsixFF(\mathcal C, P \circ I) \ar[d] \\
				\Alg_{\Span_2(\mathcal C, P \circ I)_{P,I}}(\PrL) \ar[r] & \Alg_{\Span(\mathcal C, P \circ I)}(\PrL)
			\end{tikzcd}
		.\]
		By \Cref{prp:2-cat 6FF trunc}, the bottom arrow is fully faithful with essential image given by those lax symmetric monoidal functors $D : \Span(\mathcal C, P \circ I) \to \PrL$ such that every map in $I$ is $D$-suave, and every map in $P$ is $D$-prim. Since every map in $I$ is quasi-admissible, and every map in $P$ is exceptionally quasi-proper, this contains the subcategory $\VsixFF(\mathcal C, P \circ I)$, so the top arrow in the diagram is an equivalence.

		Thus, $\VsixFF(\mathcal C, P \circ I) \to \PPF(\mathcal C)$ is an equivalence since \Cref{lem:V6FF PPF} says that $\VsixFF(\mathcal C,P \circ I)_{P,I} \to \PPF(\mathcal C)$ is an equivalence.
	\end{proof}
\end{lem}

We now turn our attention to the proof of \Cref{thm:6FF}. First we will see how to combine our results about Voevodsky-6-functor formalisms:
% KEY!!!
\begin{rmk} \label{rmk:6FF ext functor}
	In \Cref{setng:V6FF transfer}, restriction along
	\[
		\begin{tikzcd}
			\Span_2(\mathcal C', E')_{P',I'} & \ar[l] \Span(\mathcal C', E') \ar[r] & \Span(\mathcal C,E) \\
			\mathcal C'^\op \ar[u] & \ar[l, equals] \mathcal C'^\op \ar[u] \ar[r] & \mathcal C^\op \ar[u]
		\end{tikzcd}
		% \begin{tikzcd}[column sep=small]
		% 	\PPF(\mathcal C) \ar[d] \ar[r] & \PPF(\mathcal C') \ar[d] & \ar[l, "\sim"'] \VsixFF(\mathcal C', P \circ I \cap \mathcal C')_{P \cap \mathcal C', I \cap \mathcal C'} \ar[d] \ar[r] & \VsixFF(\mathcal C', P \circ I \cap \mathcal C') \ar[d] & \ar[l] \VsixFF(\mathcal C, E) \ar[d] \\
		% 	\Psh_\PrL(\mathcal C) \ar[r] & \Psh_\PrL(\mathcal C') & \ar[l, equals] \Psh_\PrL(\mathcal C') \ar[r, equals] & \Psh_\PrL(\mathcal C') & \ar[l] \Psh_\PrL(\mathcal C)
		% \end{tikzcd}
	\]
	induces the following commutative diagram
	\[
		% \PPF(\mathcal C) \to \PPF(\mathcal C') \xleftarrow{\sim} \VsixFF(\mathcal C', P \circ I \cap \mathcal C')_{P \cap \mathcal C', I \cap \mathcal C'} \to \VsixFF(\mathcal C', P \circ I \cap \mathcal C') \gets \VsixFF(\mathcal C, E)
		\begin{tikzcd}
			\VsixFF(\mathcal C', E')_{P', I'} \ar[d] \ar[r] & \VsixFF(\mathcal C', E') \ar[d] & \ar[l, "\sim"'] \VsixFF(\mathcal C, E) \ar[d] \\
			\PPF(\mathcal C') \ar[r, equals] & \PPF(\mathcal C') & \ar[l, "\sim"'] \PPF(\mathcal C)
		\end{tikzcd}
	.\]
	The rightmost horizontal arrows are equivalences by \Cref{lem:V6FF ext,prp:PPF res ext}. If every map in $E'$ is equivalent to $p \circ j$ for some $p \in P'$ and $j \in I'$, \Cref{lem:V6FF PPF} shows that the leftmost vertical functor is an equivalence. In particular, we get a section of the restriction
	\[
		\VsixFF(\mathcal C,E) \to \PPF(\mathcal C)
	.\]

	Furthermore, if every map in $E'$ is truncated, then \Cref{lem:V6FF truncated} shows that the top left horizontal map is an equivalence, so that this restriction map is actually an equivalence.
\end{rmk}

\begin{proof}[Proof of \Cref{thm:6FF}]
	By \Cref{rmk:6FF ext functor}, we have that the functor $\VsixFF(\mathcal C,E) \to \PPF(\mathcal C)$ admits a section. By \Cref{thm:PPF}, we have that $D^*$ is a projective pullback formalism, so it extends to a 6-functor formalism $D \in \VsixFF(\mathcal C,E)$, and the functor
	\[
		\VsixFF(\mathcal C,E)_{D/} \to \PF^\cstr_\bullet(\mathcal C)_{D^*/}
	\]
	admits a section. This immediately implies \cref{itm:6FF/descent,itm:6FF/suave prim,itm:6FF/morph}, and \cref{itm:6FF/duality} is given by \Cref{rmk:V6FF duality}.

	In fact, \Cref{rmk:6FF ext functor} shows that if we write $P \circ I$ for the collection of maps in $\mathcal C'$ that are composites of maps in $I \cup P$, then $D|_{\Span(\mathcal C', P \circ I)}$ extends to a 2-functor $\mathring D : \Span_2(\mathcal C', P \circ I)_{P,I} \to \PrL$, and it suffices to show \cref{itm:6FF/describe shriek} for $\mathring D$. This follows by the same inductive argument as \cite[Lemma 4.6.4]{HM6FF}, where instead of the base case being the case that $f$ is an equivalence, it is when $X \to Y$ is in $I$ or $P$. In this case, the result follows from the fact that (as in the proof of \cite[Proposition 4.14]{CLL6FF}) in $\Span_2(\mathcal C', P \circ I)_{P,I}$, the morphism $X = X \to Y$ is a left or right adjoint of $Y \gets X = X$ depending on if the map $X \to Y$ is in $I$ or $P$.
\end{proof}

\section{Applications} \label{S:applications}

We will now present some applications of our general results to 6-functor formalisms as they relate to motivic homotopy theory.

\subsection{Universality of motivic homotopy theory as a 6-functor formalism} \label{sec:univ alg}

% PERF: add intro

In order to state our results, we must first review some notions from \cite[\S5.1]{Fundamentals}:

Let $\AlgStk$ be the category of derived algebraic stacks.
\begin{nota}
	Given a derived algebraic stack $X$, we write $X_\clcl$ for its classical truncation. In fact, $(-)_\clcl$ is the right adjoint of the inclusion of the category $\AlgStk^\clcl$ of (classical) algebraic stacks into the category $\AlgStk$ of all possibly derived algebraic stacks.
\end{nota}

\begin{defn}
	A map $X \to Y$ in $\AlgStk^\clcl$ is \emph{(quasi-)projective} if there is a finite type $\mathcal F \in \QC(Y)$, and a closed (\resp{} quasi-compact) immersion over $Y$ from $X$ to $\proj_Y(\mathcal F)$.
\end{defn}

We note the following subcategories of $\AlgStk$:
\begin{enumerate}

	% PERF: recall defn of Nisnevich
	\item $\AlgStk^\lred \subseteq \AlgStk^\clcl$ is the full subcategory consisting of qcqs algebraic stacks that admit quasi-projective Nisnevich covers by global quotients of the form $X/G$, where $G$ is a linearly reductive group scheme over an affine scheme, and $X/G \to \B G$ is quasi-projective.

	\item $\AlgStk^\nice \subseteq \AlgStk$ is the full subcategory consisting of qcqs derived algebraic stacks with separated diagonal whose stabilizers are nice groups in the sense of \cite[Definition 2.1(i)]{SixAlgSt}, which we recall here for convenience: an fppf affine group scheme $G$ over an affine scheme $S$ is \emph{nice} if it is an extension of a finite \'{e}tale group scheme of order prime to the residue characteristic of $S$, by a group scheme of multiplicative type.

\end{enumerate}

Note that if $X$ is a qcqs derived algebraic stack with separated diagonal, and $X$ is a tame derived Deligne-Mumford stack, or a tame algebraic stack in the sense of \cite[\S3]{tame-stacks}, then $X \in \AlgStk^\nice$ by \cite[Examples 2.15 and 2.16]{SixAlgSt}. By \cite[Theorem 2.14]{SixAlgSt}, we also have that $\AlgStk^\nice$ contains all quotients of qcqs derived algebraic spaces by nice group schemes over affine schemes.

Furthermore, by \cite[Theorem 2.12]{SixAlgSt}, any object of $\AlgStk^\nice$ that has affine diagonal is also in $\AlgStk^\lred$.

Fix a full subcategory $\mathcal C^\alg \subseteq \AlgStk$ with a quasi-admissibility structure satisfying the following:
\begin{ass} \label{ass:alg}
	\hfill
	\begin{enumerate}

		\item $\mathcal C^\alg$ admits finite products.

		\item Every open immersion in $\mathcal C^\alg$ is quasi-admissible.

		\item The collection of closed immersions in $\mathcal C^\alg$ is stable under base change in $\mathcal C^\alg$.
			% NOTE: don't want to ask for the inclusion into $\AlgStk$ to preserve these since might want $\mathcal C^\alg$ to only have classical stacks
				% could consider asking for this up to classical truncations, but I think this is not necessary

	\end{enumerate}
\end{ass}

We put ourselves in the setting of \Cref{S:6FF} by defining the exceptionally closed maps in $\mathcal C^\alg$ to be the closed immersions, and the exceptionally quasi-proper maps to be the representable proper maps.

\begin{rmk}[Comparison with usual cdh topologies] \label{rmk:compare alg cdh}
	In this setting, the exceptionally quasi-proper cdh topology on $\mathcal C^\alg$ refines the (representable) proper cdh topology of \cite[Definition 6.2(i)]{SixAlgSt}, and if every smooth representable morphism is quasi-admissible, then the cdh topology on $\mathcal C^\alg$ refines the (representable) cdh topology of \cite[Definition 6.2(ii)]{SixAlgSt}.

	In fact, if $\mathcal C^\alg \subseteq \AlgStk^\clcl$, then by \cite[Remark 6.3]{SixAlgSt}, the exceptionally quasi-proper cdh topology on $\mathcal C^\alg$ refines the topology of projective cdh excision, so if every smooth quasi-projective morphism is quasi-admissible, then by \cite[Theorem 8.6(ii)]{Rydh_2015}, the cdh topology on $\mathcal C^\alg$ refines the topology of \emph{quasi-projective cdh excision}.
\end{rmk}

We are ready to present our first general result about 6-functor formalisms on $\mathcal C^\alg$:
\begin{thm} \label{thm:qDM V6FF}
	Suppose that
	\begin{itemize}

		\item every object of $\mathcal C^\alg$ is a qcqs derived algebraic stack whose diagonal (taken in $\AlgStk$) is locally quasi-finite and locally separated,
			% NOTE: 02NH + 0DSN => qDM + affine lft stabilizers = finite stabilizers
				% in particular, if have embeddable stabilizers, then qDM = finite stabilizers
				% I think that stabilizers of $X/G$ are always subgroups of $G$, so if $G$ is embeddable...
				% recall that every finite group is embeddable since symmetric groups are embeddable

		\item if $X \to Y$ is a quasi-compact open immersion or a proper representable morphism in $\AlgStk$, and $Y \in \mathcal C^\alg$, then $X \in \mathcal C^\alg$,
			% NOTE: 050E and 050L
			% NOTE: don't bother trying to make a weaker assumption that would still work for classical stacks, since should just have separate result that encompasses everything we can say when $\mathcal C^\alg$ only contains classical stacks (in this case, the statement should involve just 1-categorical 6FFs anyway)

		\item $E$ is the collection of finite type separated representable morphisms in $\mathcal C^\alg$,

		\item $I$ is the collection of (quasi-compact) open immersions in $\mathcal C^\alg$, and

		\item $P$ is the collection of representable proper morphisms in $\mathcal C^\alg$.

	\end{itemize}
	Then the collections $I,P,E$ are stable under base change, composition, and taking diagonals, so that we can consider the category $\VsixFF(\mathcal C^\alg, E)_{P,I}$ from \Cref{defn:V6FF}. The functor
	\[
		\VsixFF(\mathcal C^\alg, E)_{P,I} \to \PPF(\mathcal C^\alg)
	\]
	is an equivalence.
	% NOTE: not fixing which version of qadm maps we take
	% this way, get the result for either version of motivic pullback formalisms 
	\begin{proof}
		As in \cite[Example 7.6]{SixAlgSt}, we have that every map in $E$ is of the form $p \circ j$ for $p \in P$ and $j \in I$. More specifically, this follows from \cite[Theorem B]{qDM_compactification} and \cite[Remark 7.4]{SixAlgSt}, since every object of $\mathcal C^\alg$ is a qcqs derived algebraic stack that has locally separated locally quasi-finite diagonal, and every proper representable map to an object of $\mathcal C^\alg$ is in $\mathcal C^\alg$.

		It is clear that the collections $I,P,E$ are stable under composition. Since every map in $E$ is a composite of maps in $I \cup P$, we just need to show that $I$ and $P$ are stable under base change and taking diagonals.\footnote{Note that the inclusion $\mathcal C^\alg \to \AlgStk$ is not assumed to preserve diagonals of maps in $E$, even up to classical truncation, so the fact that $E$ is closed under diagonals does not immediately follow from the fact that it only consists of separated maps.}
		% by \Cref{lem:comp of maps with diagonal}.

		Since any quasi-compact open immersion or proper representable map to an object of $\mathcal C^\alg$ is in $\mathcal C^\alg$, we have that the inclusion $\mathcal C^\alg \to \AlgStk$ preserves base changes along these maps, so $I$ and $P$ are stable under base change and taking diagonals because open immersions and proper representable maps in $\AlgStk$ are stable under base change and taking diagonals.

		Thus, since every map in $I$ is truncated and quasi-admissible, and every map in $P$ is exceptionally quasi-proper, the result follows from \Cref{lem:V6FF PPF}.
	\end{proof}
\end{thm}

We will make the following mild assumptions for the remainder of the section:
\begin{ass} \label{ass:alg clcl}
	\hfill
	\begin{enumerate}
		
		\item If $Y \in \mathcal C^\alg$, then any algebraic stack admitting a quasi-projective map to $Y_\clcl$ is in $\mathcal C^\alg$.
			% NOTE: needed to show that can argue by Chow's lemma locally in $\mathcal C^\alg$
			% NOTE: this already includes
				% opens, vector bundle torsors, projective bundles
			% NOTE: if have projective maps and smooth quasi-projectives, then have quasi-projectives

		\item Every quasi-admissible map $f$ satisfies that $f_\clcl$ is also quasi-admissible.

	\end{enumerate}
\end{ass}

\begin{rmk} \label{rmk:alg nil invariance}
	For any $X \in \mathcal C^\alg$, since $\id_{X_\clcl}$ is a quasi-projective map to $X_\clcl$, we have that $X_\clcl \in \mathcal C^\alg$, and since $X_{\clcl, \red} \to X_\clcl$ is quasi-projective, we have that $X_{\clcl, \red} \in \mathcal C^\alg$.

	In fact, the maps
	\[
		X_{\clcl, \red} \to X_\clcl \to X
	\]
	are closed immersions that have empty complement, so for any $D \in \PF(\mathcal C^\alg)$, if $D$ is a reduced pullback formalism such that closed immersions are $D$-closed, we have that the functors
	\[
		D(X_{\clcl, \red}) \to D(X_\clcl) \to D(X)
	\]
	are equivalences by \Cref{rmk:nil invar}.
\end{rmk}

\begin{rmk} \label{rmk:red to clcl}
	As in \Cref{rmk:alg nil invariance}, the functor $(-)_\clcl$ restricts to a right adjoint of the fully faithful inclusion
	\[
		\mathcal C^{\alg,\clcl} \coloneqq \AlgStk^\clcl \cap \mathcal C^\alg \to \mathcal C^\alg
	\]
	of classical stacks in $\mathcal C^\alg$. It is easy to see that $\mathcal C^{\alg,\clcl}$ satisfies the same conditions that we have assumed for $\mathcal C^\alg$, given in \Cref{ass:alg,ass:alg clcl}.

	Let $(\mathcal C^\alg, E), (\mathcal C^{\alg,\clcl}, E^\clcl)$ geometric setups such that $(-)_\clcl$ sends $E$ to $E^\clcl$. Then $(-)_\clcl$ defines a morphism of geometric setups $(\mathcal C^\alg, E) \to (\mathcal C^{\alg,\clcl}, E^\clcl)$, which induces a commutative diagram
	\[
		\begin{tikzcd}
			(\mathcal C^\alg)^\op \ar[d] \ar[r] & \mathcal (C^{\alg, \clcl})^\op \ar[d] \\
			\Span(\mathcal C^\alg, E) \ar[r] & \Span(\mathcal C^{\alg, \clcl}, E^\clcl)
		\end{tikzcd}
	\]
	of (lax) symmetric monoidal functors, which then induces the following commutative diagram:
	\begin{equation} \label{eqn:clcl trunc 6FF}
		\begin{tikzcd}
			\Alg_{\Span(\mathcal C^{\alg, \clcl}, E^\clcl)} \PrL \ar[d] \ar[r] & 
			\Alg_{\Span(\mathcal C^\alg, E)} \PrL \ar[d] \\
			\Psh_{\CAlg(\PrL)}(\mathcal C^{\alg,\clcl}) \ar[r] &
			\Psh_{\CAlg(\PrL)}(\mathcal C^\alg)
		\end{tikzcd}.
	\end{equation}
\end{rmk}

The following result will be useful for reducing to the case of classical algebraic stacks.
\begin{lem} \label{lem:reduce to clcl}
	The commutative square \eqref{eqn:clcl trunc 6FF} of \Cref{rmk:red to clcl} restricts to a commutative square
	\[
		\begin{tikzcd}
			\VsixFF(\mathcal C^{\alg, \clcl}, E^\clcl) \ar[d] \ar[r] & 
			\VsixFF(\mathcal C^\alg, E) \ar[d] \\
			\PPF(\mathcal C^{\alg,\clcl}) \ar[r, "\sim"] &
			\PPF(\mathcal C^\alg)
		\end{tikzcd}
	.\]
	\begin{proof}
		Since $(-)_\clcl$ is the right adjoint of a fully faithful functor, it follows from \cite[Proposition 5.2.7.12]{htt} that the bottom functor of \eqref{eqn:clcl trunc 6FF} is fully faithful with essential image given by those presheaves $D^* : (\mathcal C^\alg)^\op \to \CAlg(\PrL)$ such that for any map $X \to Y$ in $\mathcal C^\alg$, if $X_\clcl \to Y_\clcl$ is an equivalence, then $D^*(X \to Y)$ is an equivalence. This is equivalent to the condition that for any $X \in \mathcal C^\alg$, the functor $D^*(X_\clcl \to X)$ is an equivalence.

		In fact, since $(-)_\clcl$ preserves quasi-admissible maps, proper representable maps, closed immersions, all base changes, and all complements of closed immersions, \Cref{rmk:exc res} shows that the bottom functor of \eqref{eqn:clcl trunc 6FF} restricts to a functor
		\[
			\PPF(\mathcal C^{\alg,\clcl}) \to \PPF(\mathcal C^\alg)
		.\]
		It is easy to see that this functor is still fully faithful. Since its essential image consists of those $D^*$ such that $D^*(X_\clcl \to X)$ is an equivalence for all $X \in \mathcal C^\alg$, \Cref{rmk:alg nil invariance} shows that this functor is also essentially surjective, so it is an equivalence.

		Next, for any Cartesian square
		\[
			\begin{tikzcd}
				X' \ar[d, "p"'] \ar[r, "f'"] & Y' \ar[d, "q"] \\
				X \ar[r, "f"'] & Y
			\end{tikzcd}
		\]
		in $\mathcal C^\alg$, if $q \in E$, it is easy to see that for any $D : \Span(\mathcal C^\alg, E) \to \PrL$, the square
		\[
			\begin{tikzcd}
				D(Y') \ar[d, "q_!"'] \ar[r, "f'^*"] & D(X') \ar[d, "p_!"] \\
				D(Y) \ar[r, "f^*"] & D(X)
			\end{tikzcd}
		\]
		is equivalent to the square
		\[
			\begin{tikzcd}
				D(Y'_\clcl) \ar[d, "(q_\clcl)_!"'] \ar[r, "f_\clcl'^*"] & D(X'_\clcl) \ar[d, "(p_\clcl)_!"] \\
				D(Y_\clcl) \ar[r, "f_\clcl^*"] & D(X_\clcl)
			\end{tikzcd}
		,\]
		so by \Cref{lem:desc V6FF} (and recalling \Cref{defn:suave prim}), we find that \eqref{eqn:clcl trunc 6FF} restricts to the desired commutative square.
	\end{proof}
\end{lem}

For the remainder of the section, we will need the following assumption about $\mathcal C^\alg$:
\begin{ass} \label{ass:lred or nice}
	Assume one of the following two conditions:
	\begin{enumerate}

		\item \label{itm:SHalg lred}
			The quasi-admissible maps in $\mathcal C^\alg$ are the quasi-projective smooth morphisms, and $\mathcal C^\alg \subseteq \AlgStk^\lred$. Furthermore, for any quasi-projective smooth map $X \to Y$ in $\AlgStk^\lred$, if $Y \in \mathcal C^\alg$, then $X \in \mathcal C^\alg$.

		\item \label{itm:SHalg nice}
			The quasi-admissible maps in $\mathcal C^\alg$ are the representable smooth morphisms, and $\mathcal C^\alg \subseteq \AlgStk^\nice$ (\ie{} stacks in $\mathcal C^\alg$ have nice stabilizers). Furthermore, for any qcqs smooth representable map $X \to Y$ in $\AlgStk^\nice$, if $Y \in \mathcal C^\alg$, then $X \in \mathcal C^\alg$.

	\end{enumerate}
	% NOTE: actually need that is anodyne subctx in order to get result about SHalg?
	% also need that inclusion into AlgStk preserves base changes along qadm in proof of alg PPF thm
\end{ass}
Note that \Cref{ass:lred or nice} implies that the inclusion $\mathcal C^\alg \to \AlgStk$ preserves base changes along quasi-admissible maps.

Now we come to our next result about 6-functor formalisms on $\mathcal C^\alg$:
\begin{thm} \label{thm:alg V6FF}
	% NOTE: need that all objects of Calg are of global type in order to get that ft sep rep maps are qproj locally on the source
	Let $(\mathcal C^\alg, E)$ be a geometric setup in the sense of \cite[Convention 2.1.3]{HM6FF}. Assume that for any map $f : X \to Y$ in $E$, $f_\clcl \in E$, and there is a small cdh cover of $Y$ by maps $Y' \to Y$ such that there is a small cdh cover of $X \times_Y Y'$ by finite type separated representable maps $X' \to X \times_Y Y'$ where $X' \to Y'$ is also of finite type, separated, and representable. Then the functor
	\[
		\VsixFF(\mathcal C^\alg, E) \to \PPF(\mathcal C^\alg)
	\]
	given by $D \mapsto D^*$ has a section, and is an equivalence if $\mathcal C^\alg \subseteq \AlgStk^\clcl$.
	% PERF: also describe shriek
\end{thm}
We will postpone the proof of \Cref{thm:alg V6FF} until the end of the section, but we make the following remark about the argument:
\begin{rmk}
	\Cref{thm:alg V6FF} is proven by considering the case that $\mathcal C^\alg \subseteq \AlgStk^\clcl$, and then deducing the general case using \Cref{lem:reduce to clcl}. The case that $\mathcal C^\alg \subseteq \AlgStk^\clcl$ is shown by reducing to the case that all maps in $E$ are quasi-projective using \Cref{lem:source-locally qproj,rmk:6FF ext functor}. It is necessary to reduce to the case of classical stacks both in order to leverage pre-existing results about quasi-projective maps between classical stacks, and in order to be able to apply \Cref{lem:V6FF truncated} in \Cref{rmk:6FF ext functor}.

	If we had a good theory of quasi-projective maps between derived stacks, and extension results (\Cref{prp:ext 6FF}) for ``2-categorical'' 6-functor formalisms defined on categories $\Span_2(\mathcal C,E)_{P,I}$ instead of just $\Span(\mathcal C,E)$, then it would be possible to improve this result in the case that some objects of $\mathcal C^\alg$ have nontrivial derived structures.

	Instead of reducing to quasi-projective maps, one might want to reduce to \emph{compactifiable} maps as was done in \Cref{thm:qDM V6FF}. This could also lead to more refined results, but has the disadvantage that we need to make sure that compactifiable maps in $\mathcal C^\alg$ are well-behaved -- in particular, we need that they are stable under composition.
\end{rmk}

\Cref{thm:qDM V6FF,thm:alg V6FF} can be thought of as giving categorical criteria for 6-functor formalisms. We now come to the following definition that should be thought of as giving more geometric criteria for producing 6-functor formalisms:
\begin{defn} \label{defn:mot alg PF}
	Define the category $\PF^\mot(\mathcal C^\alg)$ of \emph{motivic pullback formalisms} on $\mathcal C^\alg$ to be the full subcategory of $\PF(\mathcal C^\alg)$ consisting of those pullback formalisms $D$ satisfying the following:
	\begin{description}
		
		\item[Pointed and reduced] $D$ is a pointed reduced pullback formalism.

		\item[Localization] For any closed immersion $i : Z \to S$ in $\mathcal C^\alg$ with complement $j : U \to S$,
			\[
				D(Z) \xrightarrow{i_*} D(S) \xrightarrow{j^*} D(U)
			\]
			is a fibre sequence of pointed categories.

		\item[Thom stability] For any linearly reductive group scheme $G$ over an affine scheme, $G$-scheme $S$ such that $S/G \in \mathcal C^\alg$, and $G$-equivariant vector bundle $V \to S$, the object $\cls{V/G}/\cls{V/G \setminus 0} \in D(S)$ is $\otimes$-invertible.

		\item[Homotopy invariance] For any $S \in \mathcal C^\alg$, and vector bundle torsor $V \to S$, we have that $\cls{V} \simeq \cls{S}$ in $D(S)$. Also see \Cref{rmk:nice htpy}.

	\end{description}
\end{defn}

\begin{rmk} \label{rmk:nice htpy}
	When $\mathcal C^\alg \subseteq \AlgStk^\nice$, and representable smooth maps in $\mathcal C^\alg$ are quasi-admissible, \cite[Lemma 5.1.10]{Fundamentals} shows that the homotopy invariance condition of \Cref{defn:mot alg PF} can be replaced by the condition that for any $S \in \mathcal C^\alg$, $\cls{\aff^1_S} \simeq \cls{S}$.
\end{rmk}

Note that by \Cref{ass:lred or nice}, \cite[Theorem 5.1.11]{Fundamentals} gives us a pullback formalism $\SH^\alg$ on $\mathcal C^\alg$ of \emph{stable motivic homotopy theory}. This coincides with the constructions of \cite{SixAlgSt, sixopsequiv} by \cite[Remark 5.1.12]{Fundamentals}.

We have the following key result about motivic pullback formalisms:
\begin{thm} \label{thm:alg PPF}
	Every object $D \in \PF^\mot(\mathcal C^\alg)$ is a strongly projective pullback formalism, and all separated quasi-admissible maps are tangentially $D$-stable. Furthermore, $\SH^\alg$ is a projective pullback formalism, so we can consider the functors
	\[
		\PPF(\mathcal C^\alg)_{{\SH^\alg}/} \to \PF^\cstr_\bullet(\mathcal C^\alg)_{{\SH^\alg}/} \to \PF(\mathcal C^\alg)
	.\]
	The first functor is an equivalence, and the second is fully faithful with essential image given by $\PF^\mot(\mathcal C^\alg)$. In particular, $\SH^\alg$ is initial in $\PF^\mot(\mathcal C^\alg)$.
\end{thm}

We will present the proof of \Cref{thm:alg PPF} at the end of the section. The importance of this result for the study of 6-functor formalisms is made clear by the following \lcnamecref{rmk:SHalg init}.

\begin{rmk} \label{rmk:SHalg init}
	Suppose that $I,P,E$ are collections of maps in $\mathcal C^\alg$ such that
	\[
		\VsixFF(\mathcal C^\alg,E)_{P,I} \to \PPF(\mathcal C^\alg)
	\]
	is an equivalence. It follows that for any $D \in \PPF(\mathcal C)$, there is a unique way to extend $D$ to a lax symmetric monoidal 2-functor $D : \Span_2(\mathcal C^\alg,E)_{P,I} \to \PrL$ in $\VsixFF(\mathcal C^\alg, E)_{P,I}$, and the functor
	\[
		(\VsixFF(\mathcal C^\alg,E)_{P,I})_{D/} \to \PPF(\mathcal C^\alg)_{D/}
	\]
	is an equivalence which is actually a base change of the first one.

	Thus, using \Cref{thm:alg PPF}, there is a unique way to extend $\SH^\alg$ to a lax symmetric monoidal 2-functor $\SH^\alg : \Span_2(\mathcal C,E)_{P,I} \to \PrL$ in $\VsixFF(\mathcal C^\alg,E)_{P,I}$, and since $\PPF(\mathcal C^\alg) \to \PF(\mathcal C^\alg)$ is a monomorphism of categories (see \kerodoncite{04W5}), it follows that the functors
	\[
		(\VsixFF(\mathcal C^\alg, E)_{P,I})_{{\SH^\alg}/} \to \VsixFF(\mathcal C^\alg,E)_{P,I} \times_{\PF(\mathcal C^\alg)} \PF^\mot(\mathcal C^\alg) \to \PF^\mot(\mathcal C^\alg)
	\]
	are equivalences. In particular $\SH^\alg$ is an initial object of the full subcategory of $\VsixFF(\mathcal C^\alg,E)_{P,I}$ consisting of those $D$ such that $D^*$ is a motivic pullback formalism.
\end{rmk}

\begin{thm} \label{thm:qDM mot 6FF}
	In the setting of \Cref{thm:qDM V6FF}, the pullback formalism $\SH^\alg$ extends uniquely to a lax symmetric monoidal 2-functor $\Span_2(\mathcal C^\alg,E)_{P,I} \to \PrL$ in $\VsixFF(\mathcal C^\alg, E)_{P,I}$, and the functors
	\[
		(\VsixFF(\mathcal C^\alg,E)_{P,I})_{{\SH^\alg}/} \to \VsixFF(\mathcal C^\alg,E)_{P,I} \times_{\PF(\mathcal C^\alg)} \PF^\mot(\mathcal C^\alg) \to \PF^\mot(\mathcal C^\alg)
	\]
	are equivalences. In particular $\SH^\alg$ is an initial object of the full subcategory of $\VsixFF(\mathcal C^\alg,E)_{P,I}$ consisting of those $D$ such that $D^*$ is a motivic pullback formalism.

	% NOTE: not fixing which version of qadm maps we take
	% this way, get the result for either version of motivic pullback formalisms 
	\begin{proof}
		This follows immediately from \Cref{thm:qDM V6FF,rmk:SHalg init}.
	\end{proof}
\end{thm}

\begin{thm} \label{thm:alg mot 6FF}
	In the setting of \Cref{thm:alg V6FF}, if $\mathcal C^\alg \subseteq \AlgStk^\clcl$, then $\SH^\alg$ extends uniquely to a lax symmetric monoidal functor $\Span(\mathcal C^\alg,E) \to \PrL$ in $\VsixFF(\mathcal C^\alg,E)$, and the functors
	\[
		\VsixFF(\mathcal C^\alg, E)_{{\SH^\alg}/} \to \VsixFF(\mathcal C^\alg, E) \times_{\PF(\mathcal C^\alg)} \PF^\mot(\mathcal C^\alg) \to \PF^\mot(\mathcal C^\alg)
	\]
	are equivalences. In particular, $\SH^\alg$ is an initial object in the full subcategory of $\VsixFF(\mathcal C^\alg,E)$ consisting of those $D$ such that $D^*$ is a motivic pullback formalism.

	In general, $\SH^\alg$ still extends to a lax symmetric monoidal functor $\Span(\mathcal C^\alg,E) \to \PrL$ in $\VsixFF(\mathcal C^\alg,E)$ such that the functor
	\[
		\VsixFF(\mathcal C^\alg, E)_{{\SH^\alg}/} \to \PF(\mathcal C^\alg)
	\]
	lands in $\PF^\mot(\mathcal C^\alg)$ and admits a section
	\[
		\PF^\mot(\mathcal C^\alg) \to \VsixFF(\mathcal C^\alg, E)_{{\SH^\alg}/}
	\]
	that preserves initial objects.
	% PERF: also describe shriek
	\begin{proof}
		This follows immediately from \Cref{thm:alg V6FF} and either \Cref{rmk:SHalg init} or \Cref{thm:alg PPF}.
	\end{proof}
\end{thm}

\begin{thm} \label{thm:nice mot 6FF}
	Suppose we are in case \ref{itm:SHalg nice} of \Cref{ass:lred or nice}, and that the collection $E$ of finite type representable morphism in $\mathcal C^\alg$ is stable under base changes and diagonals (taken in $\mathcal C^\alg$). It follows that $(\mathcal C^\alg, E)$ is a geometric setup in the sense of \cite[Convention 2.1.3]{HM6FF}.

	If $\mathcal C^\alg \subseteq \AlgStk^\clcl$, then $\SH^\alg$ extends uniquely to a lax symmetric monoidal functor $\SH^\alg : \Span(\mathcal C^\alg,E) \to \PrL$ which is initial in the full subcategory of $\VsixFF(\mathcal C^\alg, E)$ consisting of those $D$ such that $D^*$ is a motivic pullback formalism.
	% If $\mathcal C^\alg \subseteq \AlgStk^\clcl$, and all representable smooth or proper maps are in $E$, then $\SH^\alg$ extends uniquely to a lax symmetric monoidal functor $\SH^\alg : \Span(\mathcal C^\alg,E) \to \PrL$ which is initial in the full subcategory of $\Alg_{\Span(\mathcal C^\alg, E)} \PrL$ consisting of those lax symmetric monoidal $D$ such that $D^*$ is a motivic pullback formalism, and all smooth representable maps are $D$-suave, and all proper representable maps are $D$-prim. Furthermore every morphism $\phi : D \to D'$ in this category satisfies that $\phi^*$ is left adjointable at smooth representable maps and right adjointable at proper representable maps, and $\phi_!$ is right adjointable at smooth representable maps, and left adjointable at proper representable maps.

	In general, $\SH^\alg$ still extends to a lax symmetric monoidal functor $\Span(\mathcal C^\alg,E) \to \PrL$ in $\VsixFF(\mathcal C^\alg,E)$ such that the functor
	\[
		\VsixFF(\mathcal C^\alg, E)_{{\SH^\alg}/} \to \PF(\mathcal C^\alg)
	\]
	lands in $\PF^\mot(\mathcal C^\alg)$ and admits a section
	\[
		\PF^\mot(\mathcal C^\alg) \to \VsixFF(\mathcal C^\alg, E)_{{\SH^\alg}/}
	\]
	that preserves initial objects, and is an equivalence if $\mathcal C^\alg \subseteq \AlgStk^\clcl$.
	% PERF: also describe shriek and duality
	% can get description for all representable etale or proper maps
	% implies duality statement for smooth separated representable maps
	\begin{proof}
		It is clear that $E$ is stable under composition, so that $(\mathcal C^\alg, E)$ is a geometric setup in the sense of \cite[Convention 2.1.3]{HM6FF}. Furthermore, for any finite type representable map $f$, we have that $f_\clcl \in E$.

		% NOTE: actually remark 7.8 shows that
		% any rep map is Nis-locally on source and target affine
		% and easy to see that any finite type affine map is quasi-projective
		Since $\mathcal C^\alg \subseteq \AlgStk^\nice$, and the representable smooth morphisms in $\mathcal C^\alg$ are quasi-admissible, the argument of \cite[Remark 7.8]{SixAlgSt} implies that every map in $E$ is cdh locally on the source and target of finite type, separated, and representable. Indeed, for any finite type representable morphism $f : X \to Y$, we may use \cite[Theorem 2.12(ii)]{SixAlgSt} to obtain a representable Nisnevich cover $\bar Y/G \to Y$, where $G$ is a nice embeddable group scheme over an affine scheme $S$, and $\bar Y$ is an affine derived $S$-scheme with $G$-action. Since $f$ is of finite type and representable, we have that $f \times_Y \bar Y/G$ is of the form $\bar f/G$, where $\bar f : \bar X \to \bar Y$ is of finite type, and $\bar X$ is a qcqs derived algebraic space. Hence, \cite[Theorem 2.14(i)]{SixAlgSt} lets us find a representable Nisnevich cover $(U \to \bar X)/G$ of $\bar X/G$ such that $U$ is a derived affine over $S$. In particular, $U \to \bar Y$ is affine and of finite type, and since $S$ is affine, $U$ is a derived affine scheme, so it is separated. Since $\bar X$ is a derived algebraic space, its diagonal is separated by \stackscite{02X4}, so $U \to \bar X$ is separated. Therefore, $U/G \to \bar X/G$ is a separated representable Nisnevich cover such that $U/G \to \bar Y/G$ is affine and of finite type. Thus, $U/G \to \bar X/G$ is a finite type separated representable map that is a cdh cover, and  $U/G \to \bar Y/G$ is also of finite type, separated, and representable.

		Hence, we conclude by \Cref{thm:alg mot 6FF}.

		% When every representable smooth or proper map is in $E$, we find that every quasi-admissible or exceptionally proper map is in $E$ so the inclusion $\VsixFF(\mathcal C^\alg, E) \to \Alg_{\Span(\mathcal C^\alg, E)} \PrL$ is fully faithful by \Cref{prp:suave prim morph}. Finally, by \Cref{lem:desc V6FF,thm:alg PPF}, we find that the full subcategory of $\VsixFF(\mathcal C^\alg, E)$ consisting of those $D$ such that $D^*$ is a motivic pullback formalism is the full subcategory of $\Alg_{\Span(\mathcal C^\alg,E)} \PrL$ consisting of those $D$ such that $D^*$ is a motivic pullback formalism, and such that every smooth representable map is $D$-suave, and every proper representable map is $D$-prim, as desired.
	\end{proof}
\end{thm}

Before giving the proofs of \Cref{thm:alg PPF,thm:alg V6FF}, we will need to go over some general results about algebraic stacks.

\begin{lem} \label{lem:qc imm}
	Let $f : X \to Y$ be a quasi-compact immersion of (classical) algebraic stacks. Then $f$ factors as a quasi-compact open immersion followed by a closed immersion.
	\begin{proof}
		By \stackscite{0CPU}, $f$ has a scheme-theoretic image, giving us a factorization $X \to Z \to Y$. Let $\tilde Y \to Y$ be a surjective smooth map where $\tilde Y$ is a scheme. We obtain Cartesian squares
		\[
			\begin{tikzcd}
				\tilde X \ar[d] \ar[r] & \tilde Z \ar[d] \ar[r] & \tilde Y \ar[d] \\
				X \ar[r] & Z \ar[r] & Y
			\end{tikzcd}
		,\]
		and since $f : X \to Y$ is quasi-compact, \stackscite{0CMK} says that $\tilde Z \to \tilde Y$ is the scheme-theoretic image of the base change $\tilde f : \tilde X \to \tilde Y$ of $f$.

		By the definition of scheme-theoretic image, we know that $Z \to Y$ is a closed immersion. Note that since the diagonal of $Z \to Y$ is quasi-compact, and $X \to Y$ is quasi-compact, we have that $X \to Z$ is quasi-compact, so it only remains to show that it is an open immersion. By \stackscite{0503}, it suffices to show that $\tilde X \to \tilde Z$ is an open immersion. This follows from \stackscite{01RG}.
	\end{proof}
\end{lem}

\begin{lem} \label{lem:local struct of alg stk}
	Every object of $\mathcal C^\alg$ has a Nisnevich cover consisting of quasi-admissible maps from derived algebraic stacks of the form $S/G$, where $G$ is a linearly reductive embeddable group scheme, and $S$ is a derived $G$-scheme such that $S_\clcl$ is $G$-quasi-projective.\footnote{The definition is given in \cite[Definition 2.5]{sixopsequiv}, which we recall here: a $G$-equivariant map $X \to Y$ between $G$-schemes $X,Y$ is $G$-(quasi-)projective if there is a $G$-equivariant vector bundle $V$ on $Y$, and a $G$-equivariant closed (\resp{} quasi-compact) immersion $X \to \proj(V)$ over $Y$.}
	\begin{proof}
		When $\mathcal C^\alg \subseteq \AlgStk^\nice$ and the quasi-admissible maps are the representable smooth morphisms, we may use \cite[Theorem 2.12(ii)]{SixAlgSt} to see that every object of $\mathcal C^\alg$ has a quasi-admissible Nisnevich cover by derived algebraic stacks of the form $S/G$, where $G$ is a linearly reductive embeddable group scheme over an affine scheme, and $S$ is a quasi-affine derived $G$-scheme, so we conclude by \cite[Lemma 2.12]{sixopsequiv} and \cite[Example 2.19]{SixAlgSt}.

		Otherwise, we have that every object of $\mathcal C^\alg$ has a quasi-admissible Nisnevich cover by algebraic stacks admitting quasi-projective morphisms to stacks of the form $\B G$, where $G$ is a linearly reductive embeddable group scheme over an affine scheme. By \cite[Lemma A.0.9]{Fundamentals}, any such algebraic stack is of the form $S/G$ for some $G$-quasi-projective $G$-scheme $S$.
	\end{proof}
\end{lem}

\begin{lem} \label{lem:source-locally qproj}
	Let $f : X \to Y$ be a representable separated finite type morphism in $\mathcal C^\alg$ between classical stacks (so $f = f_\clcl$). Then there is a projective cdh cover $X' \to X$ in $\mathcal C^\alg$ such that $X' \to Y$ is quasi-projective, and if $f$ is proper, then $X' \to Y$ is projective.
	\begin{proof}
		% Note that the map $f_\clcl : X_\clcl \to Y_\clcl$ is a composite of the closed immersion $X_\clcl \simeq (X \times_Y Y_\clcl)_\clcl \to X \times_Y Y_\clcl$ with the representable separated finite type morphism $f \times_Y Y_\clcl$, so $f_\clcl$ is also representable separated and of finite type, so we it suffices to show this result after replacing $f$ with $f_\clcl$.

		By \Cref{lem:local struct of alg stk}, \cite[Lemma A.0.9]{Fundamentals}, and \cite[Remark A.2]{SixAlgSt}, we have that every object $S$ of $\mathcal C^\alg$ satisfies that $S_\clcl$ has an \'{e}tale cover by algebraic stacks that have the resolution property. By \cite[Remark 2.2]{Noeth-approx}, it follows that every classical stack in $\mathcal C^\alg$ is of global type in the sense of \cite[Definition 2.1]{Noeth-approx}. Thus, as in \cite[A.4]{SixAlgSt}, the argument of \cite[Theorem 6.11]{SixAlgSt} holds for classical stacks in $\mathcal C^\alg$.

		This (along with \cite[Remark 6.3]{SixAlgSt}) shows that there is a projective cdh cover $X' \to X$ such that $X' \to Y$ is quasi-projective. Since $X'$ admits a projective map to a classical stack $X$ in $\mathcal C^\alg$, we have that $X' \in \mathcal C^\alg$.

		Finally, if $f : X \to Y$ is proper, it follows that $X' \to X' \to Y$ is a composite of proper maps, so it is a proper quasi-projective morphism. We conclude by \cite[Theorem 8.6(ii)]{Rydh_2015}, which states that proper quasi-projective morphisms between qcqs algebraic stacks are projective.
	\end{proof}
\end{lem}

\begin{proof}[Proof of \Cref{thm:alg V6FF}]
	Since any map $f \in E$ satisfies that $f_\clcl \in E$, it follows that if we write $E^\clcl$ for the collection of maps in $E$ between objects of $\AlgStk^\clcl$, then $E^\clcl$ is stable under composition, base change, and taking diagonals, and that $(-)_\clcl$ defines a morphism of geometric setups $(\mathcal C^\alg, E) \to (\mathcal C^{\alg,\clcl}, E^\clcl)$. Furthermore, $E^\clcl$ also satisfies that any map in $E^\clcl$ is cdh locally on the target of finite type, separated, and representable. Thus, by \Cref{lem:reduce to clcl}, it suffices to consider the case that $\mathcal C^\alg \subseteq \AlgStk^\clcl$.

	Let $I$ be the collection of (quasi-compact) open immersions in $\mathcal C^\alg$, let $P$ be the collection of projective morphisms in $\mathcal C^\alg$, and let $P \circ I$ be the collection of composites of maps in $I \cup P$. By \cite[Lemma A.0.3]{Fundamentals} and \Cref{lem:qc imm}, since any algebraic stack admitting a quasi-projective map to an object of $\mathcal C^\alg$ is in $\mathcal C^\alg$, $P \circ I$ is the collection of quasi-projective morphisms in $\mathcal C^\alg$, and any map in $P \circ I$ is of the form $p \circ j$, where $p \in P$, and $j \in I$. In fact, since every quasi-projective map to an object of $\mathcal C^\alg$ is in $\mathcal C^\alg$, it follows that the inclusion $\mathcal C^\alg \to \AlgStk^\clcl$ preserves base changes along quasi-projective maps, so that $I,P, P \circ I$ are closed under base change, diagonals, and composition.

	By \Cref{lem:source-locally qproj}, we have that for every map $X \to Y$ in $E$, there is a small cdh cover of $Y$ by maps $Y' \to Y$ such that there is a cdh cover of $X \times_Y Y'$ consisting of maps $X' \to X \times_Y Y'$ in $P \circ I$ such that $X' \to Y'$ is also in $P \circ I$.

	Since the inclusion $\mathcal C^\alg \subseteq \AlgStk^\clcl$ preserves base changes of quasi-projective maps, we have that every map in $P \circ I$ is truncated in $\mathcal C^\alg$. Thus, we may apply \Cref{rmk:6FF ext functor} to obtain the following commutative diagram of restrictions
	\[
		\begin{tikzcd}
			\VsixFF(\mathcal C^\alg, P \circ I)_{P, I} \ar[d, "\sim"'] \ar[r, "\sim"] & \VsixFF(\mathcal C^\alg, P \circ I) \ar[d] & \ar[l, "\sim"'] \VsixFF(\mathcal C^\alg, E) \ar[d] \\
			\PPF(\mathcal C^\alg) \ar[r, equals] & \PPF(\mathcal C^\alg) & \ar[l, equals] \PPF(\mathcal C^\alg)
		\end{tikzcd}
	.\]
	In particular, the functor $D \mapsto D^*$ defines an equivalence
	\[
		\VsixFF(\mathcal C^\alg, E) \to \PPF(\mathcal C^\alg)
	.\]
	% Hence, we conclude by \Cref{rmk:SHalg init}.
\end{proof}

\begin{proof}[Proof of \Cref{thm:alg PPF}]
	The pullback formalism $\SH^\alg$ is given by \cite[Theorm 5.1.11]{Fundamentals}, and \cite[Remark 5.1.12]{Fundamentals} shows that this coincides with the pullback formalisms constructed in \cite[\S4 and \S A.3]{SixAlgSt}, as well as \cite[\S6]{sixopsequiv} whenever this makes sense.
	\begin{description}
		
		\item[The structure of the proof] We will show that every separated quasi-admissible map is tangentially $\SH^\alg$-stable, and $\SH^\alg$ is a strongly projective pullback formalism on $\mathcal C^\alg$. It will then follow that by \Cref{thm:PPF}, the functor
			\[
				\PPF(\mathcal C^\alg)_{{\SH^\alg}/} \to \PF^\cstr_\bullet(\mathcal C^\alg)_{{\SH^\alg}/}
			\]
			is an equivalence, and any $D$ admitting a map from $\SH^\alg$ is also a strongly projective pullback formalism. Furthermore, by \Cref{prp:monoidal twists}(\ref{itm:monoidal twists/morph}), we also have that every separated quasi-admissible map is tangentially $D$-stable if $D \in \PF^\cstr_\bullet(\mathcal C^\alg)$ admits a map from $\SH^\alg$. To conclude, we show that the functor
			\[
				\PF^\cstr_\bullet(\mathcal C^\alg)_{{\SH^\alg}/} \to \PF(\mathcal C^\alg)
			\]
			is fully faithful, and its essential image is given by $\PF^\mot(\mathcal C^\alg)$.

		\item[$\SH^\alg$ is a constructible pullback formalism] Since $\SH^\alg$ is reduced and takes values in pointed categories, it suffices to show that every closed immersion is $\SH^\alg$-closed. This follows from \cite{Gluing}, but we also present the following argument using other results in the literature.

			In the case that $\mathcal C^\alg \subseteq \AlgStk^\nice$, and the quasi-admissible maps are the representable smooth morphisms, this is given by \cite[Theorem 4.10(ii)(d)]{SixAlgSt} and \Cref{prp:loc for stab PF}. Otherwise, by \Cref{lem:local struct of alg stk}, we know that every object of $\mathcal C^\alg$ must have a quasi-admissible Nisnevich cover by global quotients $S/G$ for $G$ a linearly reductive group scheme over an affine scheme, and $S$ a $G$-quasi-projective $G$-scheme. Since $\SH^\alg$ has descent for quasi-admissible Nisnevich covers, \Cref{lem:locality of closedness,prp:loc for stab PF} say that to check that all closed immersions are $D$-closed, it suffices to show that if $G$ is a linearly reductive group scheme over an affine scheme, and $i : Z \to S$ is a $G$-equivariant closed immersion of $G$-quasi-projective $G$-schemes, then $i/G : Z/G \to S/G$ is $\SH^\alg$-closed. This follows from \cite[Theorem 4.18]{sixopsequiv} and \cite[Corollary 4.19]{sixopsequiv}.

		\item[Describing the functor $\PF^\cstr_\bullet(\mathcal C^\alg)_{{\SH^\alg}/} \to \PF(\mathcal C^\alg)$] Note that any pointed reduced pullback formalism on $\mathcal C^\alg$ that satisfies the Thom stability and homotopy invariance properties from \Cref{defn:mot alg PF} takes values in stable categories by \Cref{lem:auto stable}. Thus, it follows that motivic pullback formalisms takes values in stable categories, so they are constructible pullback formalisms by \Cref{prp:loc for stab PF}. \cite[Theorem 5.1.11]{Fundamentals} shows that the functor is fully faithful with essential image given by the motivic pullback formalisms that have quasi-admissible Nisnevich excision. In fact, since motivic pullback formalisms take values in stable categories, this excision property is automatic by \Cref{prp:gluing implies exc} or \Cref{prp:elementary cdh descent}.

		\item[Every separated quasi-admissible map is tangentially $\SH^\alg$-stable]
			Let $f : X \to Y$ be a separated quasi-admissible map. Since $\mathcal C^\alg \to \AlgStk$ preserves base changes along quasi-admissible maps, we have that the diagonal of $f$ in $\mathcal C^\alg$ is a closed immersion, so \Cref{prp:monoidal twists}(\ref{itm:monoidal twists/desc}) shows that $\Sigma^f \simeq - \otimes \cls{X \times_Y X}/\cls{X \times_Y X \setminus X}$.

			If $Y$ is of the form $S/G$, for $G$ a linearly reductive embeddable group scheme over an affine scheme, and $S$ is a $G$-quasi-projective $G$-scheme, then \cite[Theorem 3.23]{sixopsequiv} shows that
			\[
				\cls{X \times_Y X}/\cls{X \times_Y X \setminus X} \simeq \cls{Tf}/\cls{Tf \setminus 0}
			,\]
			where $Tf$ is the tangent bundle of $f$. Thus, we have that 
			\[
				\Sigma^f \simeq - \otimes \cls{Tf}/\cls{Tf \setminus 0}
			.\]
			Since $\cls{Tf}/\cls{Tf \setminus 0}$ is an $\otimes$-invertible object, it follows that $\Sigma^f$ is invertible.

			For the case of general $Y$, we first note that by \Cref{rmk:alg nil invariance}, it suffices to assume $Y = Y_\clcl$, so by \Cref{lem:local struct of alg stk}, $Y$ admits a $\SH^\alg$-pseudocover by quasi-admissible maps from stacks of the form $S/G$ as above. Thus, we may conclude by \Cref{prp:monoidal twists}(\ref{itm:monoidal twists/bc}).

		\item[$\SH^\alg$ is a strongly projective pullback formalism]
			Let $P$ be a projectively $\SH^\alg$-saturated collection of maps in $\mathcal C^\alg$ (see \Cref{defn:proj sat}). We must show that every proper representable map is in $P$. For any $Y \in \mathcal C^\alg$, we have that $Y_\clcl \to Y$ is a $\SH^\alg$-acyclic closed immersion by \Cref{rmk:alg nil invariance}, so since $P$ contains all closed immersions, it suffices to show that proper representable maps to classical stacks are in $P$. In fact, for any $X \in \mathcal C^\alg$, $X_\clcl \to X$ is a map in $P$ that is invertible away from the (same) exceptionally closed map $X_\clcl \to X$, so the source-locality property of $P$ allow us to reduce to showing that representable proper maps between classical stacks are in $P$.

			By \Cref{lem:source-locally qproj}, for any such map $X \to Y$, there is a projective cdh cover $X' \to X$ in $\mathcal C^\alg$ such that $X' \to Y$ is projective, so it suffices to show that every projective map to a classical stack in $\mathcal C^\alg$ is in $P$.

			By \Cref{lem:local struct of alg stk}, every classical stack in $\mathcal C^\alg$ has a $\SH^\alg$-pseudocover by quasi-admissible maps from stacks of the form $S/G$, where $G$ is a linearly reductive embeddable group scheme over an affine scheme, and $S$ is a $G$-quasi-projective $G$-scheme, and \cite[Lemma A.0.9]{Fundamentals} shows that any projective map to $S/G$ is the composite of a closed immersion with the projection of a projective bundle. We know that every closed immersion is in $P$, and projections of projective bundles are separated, so their diagonals are in $P$, and they are stably $\SH^\alg$-ambidextrous by \cite[Theorem 6.9]{sixopsequiv} or \cite[Lemma 6.9]{SixAlgSt}, thus they are also in $P$. We conclude since $P$ is stable under composition.
	\end{description}
\end{proof}

\subsection{The six operations for complex analytic motivic homotopy theory} \label{S:hol}

Now we turn our attention to some applications for complex analytic stacks. Namely, we will produce stacky analytic and Betti realizations that are compatible with the six operations in \Cref{thm:analytification,rmk:Betti} and we produce a universal 6-functor formalism on complex analytic stacks in \Cref{thm:hol 6FF}.

All of the results of this section rely on the result, proven in \cite{Gluing}, that shows that embeddings of complex analytic stacks are $\SH^\hol$-closed, where $\SH^\hol$ is the pullback formalism of stable motivic homotopy on complex analytic stacks defined in \cite[Theorem 5.3.10]{Fundamentals}.

We give a brief review of the relevant notions for complex analytic stacks. See \cite[\S5.3]{Fundamentals} for a more detailed overview.

\begin{defn} \label{defn:HolStk}
	Let $\An_\comps$ be the site of complex spaces and open covers, as considered in \cite[\S5.3]{Fundamentals}. Since this is a subcanonical site, we will view $\An_\comps$ as a full subcategory of $\Shv(\An_\comps)$, and refer to objects of the essential image of the inclusion as \emph{representable}.

	% A map $X \to Y$ in $\Shv(\An_\comps)$ is said to be \emph{representable} if for any representable sheaf $Y' \in \An_\comps$, and map $Y' \to Y$, the pullback $X \times_Y Y'$ is representable.
	The notions of \emph{open substacks, embeddings, representable local biholomorphisms, and representable submersions} in $\Shv(\An_\comps)$ are given by maps that are representable by open subspaces, embeddings, local biholomorphisms, or submersions, respectively.

	We define the \emph{analytic Nisnevich topology} on $\Shv(\An_\comps)$ by declaring that representable open covers are analytic Nisnevich covers, and that if $Z \to X$ is an embedding with complement $U \to X$, and $\tilde X \to X$ is a representable local biholomorphism that is an equivalence over $Z$, then $\{U, \tilde X \to X\}$ is an analytic Nisnevich cover of $X$.

	We define $\HolStk^{\fin,\red}$ to be the full subcategory of $\Shv(\An_\comps)$ consisting of objects that admit analytic Nisnevich covers by global quotients $X/G$, where $G$ is a finite group acting on a reduced complex space $X$.

	Finally, we equip $\HolStk^{\fin,\red}$ with the quasi-admissibility structure of representable submersions.
\end{defn}

Now we recall the pullback formalism $\SH^\hol$ on $\HolStk^{\fin,\red}$ constructed in \cite[Theorem 5.3.10]{Fundamentals}. When $G$ is a finite group acting on a reduced complex space $X$, we can describe $\SH^\hol(X/G)$ as follows:
\begin{itemize}

	\item First we consider the category $\HolStk_{X/G}$, which is the category of $G$-equivariant submersions $X' \to X$, and $G$-equivariant maps over $X$ between these.

	\item The category $\HolStk_{X/G}$ inherits an analytic Nisnevich topology given by $G$-equivariant open covers, and families $\{\tilde X, U \to X\}$, where $\tilde X \to X$ is a $G$-equivariant local biholomorphism, and $U \to X$ is a $G$-invariant open subspace that is a complement of an embedding $Z \to X$ over which $\tilde X \to X$ is an isomorphism.

	\item We may then consider the category $H^\hol(X/G)$ consisting of sheaves $F$ on $\HolStk_{X/G}$ that are ``$\comps$-invariant'' in the sense that for any $X' \in \HolStk_{X/G}$, the map $F(X') \to F(X' \times \comps)$ is an equivalence, where $\comps$ has the trivial $G$-action. The inclusion $H^\hol(X/G) \to \Psh(\HolStk_{X/G})$ admits a left adjoint $L_\hol$.

	\item The category $H^\hol(X/G)$ is equipped with the Cartesian symmetric monoidal structure, and the smash product, $\wedge$, equips the category $H^\hol_\bullet(X/G)$ of pointed objects in $H^\hol(X/G)$ with a symmetric monoidal structure.

	\item Finally, $\SH^\hol(X/G)$ is given by formally adjoining $\wedge$-inverses of pointed $\comps$-invariant sheaves of the form $L_\hol V/L_\hol(V \setminus 0)$, where $V \to X$ is a $G$-equivariant vector bundle.

\end{itemize}

\subsubsection{Betti realization and analytification} \label{S:realization}

In this section, we will see how to give $\SH^\hol$ the structure of a well-behaved 6-functor formalism on complex algebraic stacks, and that in this case we also have a well-behaved morphism of 6-functor formalisms $\SH^\alg \to \SH^\hol$.

\begin{defn}
	Let $\AlgStk^{\fin, \clcl}_\comps$ be the category of qcqs algebraic stacks in $\AlgStk^\nice_{/\Spec \comps}$ that have a representable Nisnevich cover by global quotients of the form $X/G$, where $G$ is a constant finite group, and $X$ is a qcqs complex $G$-scheme such that $X$ is of finite type over $\comps$.
	% NOTE: finite type so that analytification exists

	We give $\AlgStk^{\fin, \clcl}_\comps$ the structure of a pullback context by setting the quasi-admissible maps to be the (qcqs) representable smooth maps.
\end{defn}

\begin{rmk} \label{rmk:global quot of ft algsp}
	% Using \cite[Proposition 5.7.6]{crit-plat} or \cite[Theorem 3.4.2.1]{SAG}, we find that every qcqs derived algebraic space admits a Nisnevich cover by separated quasi-compact derived schemes (so also by derived affines). Thus, if $X$ is a qcqs derived algebraic space of finite type over $\Spec \comps$, and $G$ is a finite group acting on $X$, then $X/G \in \AlgStk^\fin_\comps$.

	Let $X$ be a qcqs derived algebraic space of finite type over $\Spec \comps$, and let $G$ be a constant finite group acting on $X$. Then \cite[Theorem 2.14(i)]{SixAlgSt} shows that there is a $G$-equivariant Nisnevich cover $U \to X$ where $U$ is an affine derived scheme over $\Spec \comps$. In particular, since $U \to X$ is of finite type, we have that $U$ is of finite type over $\Spec \comps$, so $X_\clcl/G \in \AlgStk^{\fin, \clcl}_\comps$.
\end{rmk}

\begin{rmk}
	$\AlgStk^{\fin, \clcl}_\comps$ is a full anodyne pullback subcontext of $\AlgStk^\nice_{/\Spec \comps}$, and in fact, any algebraic stack that admits a finite type representable map to an object of $\AlgStk^{\fin, \clcl}_\comps$ must itself be in $\AlgStk^{\fin, \clcl}_\comps$.
\end{rmk}

\begin{rmk}
	Every map in $\AlgStk^{\fin, \clcl}_\comps$ is of finite type.
	\begin{proof}
		Since every map in $\AlgStk^{\fin,\clcl}_\comps$ is qcqs, it suffices to show that every map in $\AlgStk^{\fin,\clcl}_\comps$ is locally of finite type. By \stackscite{06U9}, it suffices to show that every algebraic stack in $\AlgStk^{\fin,\clcl}_\comps$ is of finite type over $\Spec \comps$. This follows from \stackscite{06U8}.
	\end{proof}
\end{rmk}

Let $E$ be the collection of representable morphisms in $\AlgStk^{\fin,\clcl}_\comps$. Since every map in $\AlgStk^{\fin, \clcl}_\comps$ is of finite type, and any finite type representable morphism in $\AlgStk^\clcl$ to an object of $\AlgStk^{\fin,\clcl}_\comps$ is in $\AlgStk^{\fin, \clcl}_\comps$, \Cref{thm:nice mot 6FF} says that $\SH^\alg$ can be seen as a 6-functor formalism $\Span(\AlgStk^{\fin,\clcl}_\comps, E) \to \PrL$.

Note that the full subcategory inclusion $\AlgStk_\comps^{\fin,\clcl,\red} \to \AlgStk^{\fin,\clcl}_\comps$ of classical reduced stacks admits a right adjoint $(-)_\red$, and composing this with analytification $(-)^\an : \AlgStk_\comps^{\fin,\clcl,\red} \to \HolStk^{\fin,\red}$ defines a morphism of pullback contexts $(-)^\an_\red : \AlgStk^{\fin,\clcl}_\comps \to \HolStk^{\fin, \red}$.

\begin{thm}[Analytic realization] \label{thm:analytification}
	In order to simplify notation, in the following, we change the definition of $\SH^\hol$ to denote the presheaf $((-)^\an_\red)^* \SH^\hol$ on $\AlgStk^{\fin,\clcl}_\comps$, given by $X \mapsto \SH^\hol(X^\an_\red)$.

	Then $\SH^\hol$ extends uniquely to a 6-functor formalism $\Span(\AlgStk^{\fin,\clcl}_\comps,E) \to \PrL$ satisfying the following:
	\begin{enumerate}

		\item $(\SH^\hol)^*$ has cdh descent, and $(\SH^\hol)^!$ has cdh descent.

		% \item For any $f \in E$,
		% 	\[
		% 		\begin{array}{cl}
		% 			f^! \simeq f^* & \text{if $f$ is \'{e}tale} \\
		% 			f_! \simeq f_* & \text{if $f$ is proper}
		% 		\end{array}
		% 	.\]

		\item Every smooth representable map is $\SH^\hol$-suave, and every proper representable map is $\SH^\hol$-prim.

		\item There is a unique morphism of 6-functor formalisms $\alpha : \SH^\alg \to \SH^\hol$ such that for any representable map $f$,
			% NOTE: in both cases, $f$ is automatically of finite type
			\begin{description}
				
				\item[if $f$ is smooth] then $\alpha^*$ is left adjointable at $f$ and $\alpha_!$ is right adjointable at $f$, and
					
				\item[if $f$ is proper] then $\alpha^*$ is right adjointable at $f$, and $\alpha_!$ is left adjointable at $f$.

			\end{description}

			% PERF: describe shriek and duality?

	\end{enumerate}
	
	\begin{proof}
		By \Cref{thm:nice mot 6FF}, it suffices to show that $\SH^\hol \in \PF^\mot(\AlgStk^{\fin, \clcl}_\comps)$. All of the conditions from \Cref{defn:mot alg PF} except for localization follow immediately from \cite[Theorem 5.3.10]{Fundamentals}, and the localization axiom follows from \cite{Gluing}, since $(-)^\an_\red$ sends closed immersions to embeddings.
	\end{proof}
\end{thm}

We now remark that $\SH^\hol$ can be seen as a refinement of a 6-functor formalism sending $X/G$ to the category of $G$-equivariant sheaves of spectra on $X^\an$. This allows us to use \Cref{thm:analytification} to obtain a version of Betti realization.
\begin{rmk}[Betti realization] \label{rmk:Betti}
	% NOTE: disk is stronger than plane
	% plane is contractible colimit of disks, so contracting disk also contracts plane
	% we can't use multiplication as a plane-indexed null-homotopy of the disk, since it does not land in the disk (multiplying the disk by a complex number of modulus > 1 does not land in the same disk)
	One of the properties characterizing the pullback formalism $\SH^\hol$ is a version of homotopy invariance: for any $S \in \HolStk^{\fin,\red}$, there is an equivalence $\cls{S \times \comps} \simeq \cls{S}$ in $\SH^\hol(S)$. This corresponds to a notion of homotopy that uses $\comps$ as an interval. Instead of using $\comps$, we can use the open unit disk $\mathbb{D} \subseteq \comps$ as our interval, which leads to a stronger version of homotopy invariance. Using \cite[Proposition 3.2.5]{Fundamentals}, we can consider a pullback formalism $\SH^\Betti$ given by localizing $\SH^\hol$ along the maps $M \otimes \cls{S \times \mathbb{D}} \to M$ in $\SH^\hol(S)$ for all $M \in \SH^\hol(S)$. Note that for any complex space $S$, the analytic Nisnevich topology on the category $\HolStk_S$ of representable submersions over $S$ is equivalent to the open covering topology, so by \cite[Theorem 1.8 and Remarque 1.9 or Lemme 1.10]{AyoubBetti}, we get a natural identification of $\SH^\Betti(S)$ with the category of sheaves of spectra on $S$:
	\[
		\SH^\Betti(S) \simeq \Shv_{\Sp}(S)
	.\]

	It is possible to use \Cref{thm:closed}(\ref{itm:closed/morphism}) or results from \cite{Gluing} to show that embeddings are $\SH^\Betti$-closed, in which case it is easy to check (or apply \Cref{thm:alg PPF} to see) that $((-)^\an_\red)^* \SH^\Betti \in \PF^\mot(\AlgStk^{\fin, \clcl}_\comps)$, so that if we use the same abuse of notation as in \Cref{thm:analytification} to view pullback formalisms $D$ on $\HolStk^{\fin,\red}$ as pullback formalisms $((-)^\an_\red)^* D$ on $\AlgStk^{\fin,\clcl}_\comps$, then we get a morphism $\beta : \SH^\hol \to \SH^\Betti$ in $\PF^\mot(\AlgStk^{\fin, \clcl}_\comps)$, leading to morphisms in $\VsixFF(\AlgStk^{\fin, \clcl}_\comps, E)$:
	\[
		\SH^\alg \xrightarrow{\alpha} \SH^\hol \xrightarrow{\beta} \SH^\Betti
	,\]
	where $\beta$ or $\beta \circ \alpha$ can be seen as Betti realization when restricted to complex spaces in $\HolStk^{\fin,\red}$ or finite type algebraic spaces in $\AlgStk^{\fin,\clcl}_\comps$.

	We note that by \cite[Proposition 3.2.5]{Fundamentals} and \Cref{thm:D-topology}, we can also apply sheafification for the representable \'{e}tale topology to $\SH^\Betti$ to obtain a morphism of pullback formalisms $\SH^\Betti \to \SH^\Betti_\et$ such that for each $S \in \AlgStk^{\fin,\clcl}_\comps$, the functor $\SH^\Betti(S) \to \SH^\Betti_\et(S)$ is given by localizing along covering sieves for the representable \'{e}tale topology. Note that if $S \in \AlgStk^{\fin,\clcl}_\comps$ satisfies that $S^\an$ is a complex space, then the analytification of any \'{e}tale cover of $S$ gives a cover for the analytic Nisnevich topology (and even the open covering topology), so that
	\[
		\SH^\Betti(S) \to \SH^\Betti_\et(S)
	\]
	is an equivalence.

	We also know that for any constant finite group $G$ acting on an algebraic space $X$ over $\Spec \comps$, the quotient $X/G$ is in $\AlgStk^{\fin,\clcl}_\comps$ by \Cref{rmk:global quot of ft algsp}, and the natural map
	\[
		\SH^\Betti_\et(X/G) \to \SH^\Betti_\et(X)^G
	\]
	is an equivalence by representable \'{e}tale descent, so we have a natural identification of $\SH^\Betti_\et(X/G)$ with the category of $G$-equivariant sheaves of spectra on $X^\an$:
	\[
		\SH^\Betti_\et(X/G) \simeq \Shv_{\Sp}(X^\an)^G
	.\]

	Since every object of $\AlgStk^{\fin,\clcl}_\comps$ has a representable \'{e}tale cover by algebraic spaces, it is easy to show that closed immersions are $\SH^\Betti_\et$-closed by reducing to the case of algebraic spaces using \Cref{lem:locality of closedness}, and then checking on stalks. Therefore, we also have that $\SH^\Betti_\et$ is in $\PF^\mot(\AlgStk^{\fin,\clcl}_\comps)$, so $\SH^\hol \to \SH^\Betti_\et$ extends to a morphism in $\VsixFF(\AlgStk^{\fin,\clcl}_\comps,E)$ which can be seen as a stacky Betti realization.
\end{rmk}

\subsubsection{6-functor formalisms on complex analytic stacks}

In the previous section, we showed how to give $\SH^\hol$ the structure of a 6-functor formalism on \emph{algebraic} stacks. In this section we will study properties of $\SH^\hol$ as a functor of complex analytic stacks.

First we define some prerequisite notions for complex analytic stacks:
\begin{defn} \label{defn:hol}
	Say a map $X \to Y$ in $\Shv(\An_\comps)$ is separated if the diagonal $X \to X \times_Y X$ is an embedding.

	Say a map in $\Shv(\An_\comps)$ is \emph{projective} if it is a composite of maps $X \to Y$ such that there is an analytic Nisnevich cover of $Y$ by maps $Y' \to Y$ such that the base change $X \times_Y Y' \to Y'$ is a composite of embeddings, and maps of the form $\proj(\mathcal E) \to S$, where $S \in \Shv(\An_\comps)$ and $\mathcal E$ is a vector bundle on $S$.

	The \emph{projective cdh topology} on $\Shv(\An_\comps)$ is given by declaring that the empty sieve covers the empty stack, and that if $Z \to S$ is an embedding with complement $U \to S$, and $X \to S$ is a projective map that is an equivalence over $U$, then $\{X,Z \to S\}$ is a covering family.

	Say a map $X \to Y$ in $\Shv(\An_\comps)$ is \emph{proper} if there is an analytic Nisnevich cover of $Y$ by maps $Y' \to Y$ such that there is projective cdh covering family of $X \times_Y Y'$ consisting of finitely many projective maps $X' \to X \times_Y Y'$ such that $X' \to Y'$ is also projective.
\end{defn}

\begin{rmk}
	The projective maps are stable under composition by definition. They are also stable under base change and taking diagonals by 
	% \Cref{lem:comp of maps with diagonal} and
	\Cref{lem:target-local geom setup}. By \Cref{lem:source-local geom setup}, the proper maps are also stable under composition, base change, and taking diagonals.
\end{rmk}

Now, fix a full subcategory $\mathcal C^\hol \subseteq \HolStk^{\fin,\red}$ such that
\begin{ass}
	The point stack is in $\mathcal C^\hol$, $\mathcal C^\hol$ admits finite products, and every object of $\HolStk^{\fin,\red}$ admitting a representable submersion or projective map to an object of $\mathcal C^\hol$ is in $\mathcal C^\hol$.
\end{ass}
Also equip $\mathcal C^\hol$ with collections of exceptionally closed and exceptionally quasi-proper maps satisfying \Cref{ass:exc}, by setting the exceptionally closed maps to be the embeddings, and the exceptionally quasi-proper maps to be the proper maps.

\begin{defn} \label{defn:mot hol PF}
	Define the category $\PF^\mot(\mathcal C^\hol)$ of \emph{motivic pullback formalisms} on $\mathcal C^\hol$ to be the full subcategory of $\PF(\mathcal C^\hol)$ consisting of those pullback formalisms $D$ satisfying the following:
	\begin{description}
		
		\item[Pointed and reduced] $D$ is a pointed reduced pullback formalism.

		\item[Localization] For any embedding $i : Z \to S$ in $\mathcal C^\hol$ with complement $j : U \to S$,
			\[
				D(Z) \xrightarrow{i_*} D(S) \xrightarrow{j^*} D(U)
			\]
			is a fibre sequence of pointed categories.

		\item[Thom stability] For any finite group $G$ acting on a reduced complex space $X$ such that $X/G \in \mathcal C^\hol$, and $G$-equivariant vector bundle $V \to X$, the object $\cls{V/G}/\cls{V/G \setminus 0} \in D(X)$ is $\otimes$-invertible.

		\item[Homotopy invariance] For any $S \in \mathcal C^\hol$, we have that $\cls{\comps \times S} \simeq \cls{S}$ in $D(S)$.

	\end{description}
\end{defn}

We consider $\SH^\hol$ as a pullback formalism on $\mathcal C^\hol$, and we obtain the following analogue of \Cref{thm:alg PPF}:
\begin{thm} \label{thm:hol PPF}
	Every object $D \in \PF^\mot(\mathcal C^\hol)$ is a strongly projective pullback formalism. Furthermore, $\SH^\hol$ is a projective pullback formalism, so we can consider the functors
	\[
		\PPF(\mathcal C^\hol)_{{\SH^\hol}/} \to \PF^\cstr_\bullet(\mathcal C^\hol)_{{\SH^\hol}/} \to \PF(\mathcal C^\hol)
	.\]
	The first functor is an equivalence, and the second is fully faithful with essential image given by $\PF^\mot(\mathcal C^\hol)$. In particular, $\SH^\hol$ is initial in $\PF^\mot(\mathcal C^\hol)$.
\end{thm}

The following result is key in the proof of \Cref{thm:hol PPF}:
\begin{lem} \label{lem:hol ambidext}
	For any $S \in \mathcal C^\hol$, and vector bundle $\mathcal E$ on $S$, the map $\proj(\mathcal E) \to S$ is stably $\SH^\hol$-ambidextrous.
	\begin{proof}
		Since analytic Nisnevich covers are $\SH^\hol$-pseudocovers, for any such $\proj(\mathcal E) \to S$, there is an $\SH^\hol$-pseudocover of $S$ consisting of quasi-admissible maps from objects of the form $X/G$, where $G$ is a finite group acting on a reduced complex space $X$. By \cite[Lemma B.0.2]{Fundamentals}, we can further refine this cover so that it trivializes $\mathcal E$.

		Maps of the form $\proj(\mathcal E) \to S$ as above are separated, so by \cite{Gluing} their diagonals are $\SH^\hol$-closed. Thus, by \Cref{thm:duality}(\ref{itm:duality/locality}), it suffices to show that maps of the form $\proj^n_S \to S$ are stably $\SH^\hol$-ambidextrous. Any such map is a base change along $S \to \pt$ of the map $\proj_\comps^n \to \pt$, so by \Cref{thm:duality}(\ref{itm:duality/proper}), it suffices to show that $\proj_\comps^n \to \pt$ is stably $\SH^\hol$-ambidextrous.

		Now, note that we have a morphism of constructible pullback formalisms $\SH^\alg \to ((-)_{\clcl, \red}^\an)^* \SH^\hol$, either using the universal property of \cite[Theorem 5.1.11]{Fundamentals}, or by taking the map of presheaves $\alpha^*$ associated to the morphism $\alpha$ of 6-functor formalisms given by \Cref{thm:analytification}. Therefore, \Cref{thm:duality}(\ref{itm:duality/morph}) shows that for any representable smooth separated stably $\SH^\alg$-ambidextrous map $f : X \to Y$, we have that $f_{\clcl, \red}^\an$ is stably $\SH^\hol$-ambidextrous.

		Since $\proj^n_\comps \to \Spec \comps$ is a smooth projective map, it is separated, so it is tangentially $\SH^\alg$-stable by \Cref{thm:alg PPF}, and it is quasi-admissible and exceptionally quasi-proper, so it is stably $\SH^\alg$-ambidextrous. Thus, since $\proj^n_\comps \to \pt$ is $(\proj^n_\comps \to \Spec \comps)_{\clcl,\red}^\an$, it follows that it is stably $\SH^\hol$-ambidextrous.
	\end{proof}
\end{lem}

\begin{proof}[Proof of \Cref{thm:hol PPF}]
	\hfill
	\begin{description}

		\item[The structure of the proof] We will show that $\SH^\hol$ is a strongly projective pullback formalism on $\mathcal C^\hol$. It will then follow that by \Cref{thm:PPF}, the functor
			\[
				\PPF(\mathcal C^\hol)_{{\SH^\hol}/} \to \PF^\cstr_\bullet(\mathcal C^\hol)_{{\SH^\hol}/}
			\]
			is an equivalence, and any $D$ admitting a map from $\SH^\hol$ is also a strongly projective pullback formalism. To conclude, we show that the functor
			\[
				\PF^\cstr_\bullet(\mathcal C^\hol)_{{\SH^\hol}/} \to \PF(\mathcal C^\hol)
			\]
			is fully faithful, and its essential image is given by $\PF^\mot(\mathcal C^\hol)$.

		\item[$\SH^\hol$ is a constructible pullback formalism] Since $\SH^\hol$ is reduced and takes values in pointed categories, it suffices to show that every embedding is $\SH^\hol$-closed, but this follows from \cite{Gluing}.

		\item[Describing the functor $\PF^\cstr_\bullet(\mathcal C^\hol)_{{\SH^\hol}/} \to \PF(\mathcal C^\hol)$] Note that any pointed reduced pullback formalism on $\mathcal C^\hol$ that satisfies the Thom stability and homotopy invariance properties from \Cref{defn:mot hol PF} take values in stable categories by \Cref{lem:auto stable}. Thus, it follows that motivic pullback formalisms takes values in stable categories, so they are constructible pullback formalisms by \Cref{prp:loc for stab PF}. \cite[Theorem 5.3.10]{Fundamentals} shows that the functor is fully faithful with essential image given by the motivic pullback formalisms that have quasi-admissible Nisnevich excision. In fact, since motivic pullback formalisms take values in stable categories, this excision property is automatic by \Cref{prp:gluing implies exc} or \Cref{prp:elementary cdh descent}.

		\item[$\SH^\hol$ is a strongly projective pullback formalism]
			Let $P$ be a projectively $\SH^\hol$-saturated collection of maps in $\mathcal C^\hol$ (see \Cref{defn:proj sat}). We must show that every exceptionally quasi-proper map is in $P$. Since every exceptionally quasi-proper map is proper, and by the definition of proper maps and the fact that any projective map to an object of $\mathcal C^\hol$ is a map in $\mathcal C^\hol$, it suffices to show that every projective map is in $P$. Since analytic Nisnevich covers are $\SH^\hol$-pseudocovers, it follows from the definition of projective maps that we can reduce to showing that embeddings are in $P$, and maps of the form $\proj(\mathcal E) \to S$ are in $P$, where $S \in \Shv(\An_\comps)$ and $\mathcal E$ is a vector bundle on $S$.

			The fact that embeddings are in $P$ follows from the fact that they are exceptionally closed. Since maps of the form $\proj(\mathcal E) \to S$ as above are separated, their diagonals are in $P$, so it suffices to show that they are stably $\SH^\hol$-ambidextrous, which is done in \Cref{lem:hol ambidext}.

	\end{description}
\end{proof}

We now present one way to give $\SH^\hol$ the structure of a 6-functor formalism. In the following, the analytic cdh topology is given as the union of the analytic Nisnevich topology, and the projective cdh topology.
\begin{thm} \label{thm:hol 6FF}
	% NOTE: need to say truncated
	% because don't know if proper maps are truncated
	% even if asked them to be representable, don't think that automatically get truncated
	% because analytic Nisnevich covers might not be eff epi in $\Shv(\An_\comps)$? In alg case, topology on $\Sch$ was etale or even fppf, so Nis covers are probably eff epi.
	Let $E$ be a collection of maps in $\mathcal C^\hol$ such that $(\mathcal C^\hol, E)$ is a geometric setup in the sense of \cite[Convention 2.1.3]{HM6FF}, and such that for every map $X \to Y$ in $E$, there is a small analytic cdh cover of $Y$ consisting of maps $Y' \to Y$ such that there is a small analytic cdh cover of $X \times_Y Y'$ consisting of truncated proper maps $X' \to X \times_Y Y'$ such that $X' \to Y'$ is also proper and truncated

	Then $\SH^\hol$ extends to a 6-functor formalism $\Span(\mathcal C^\hol, E) \to \PrL$, which is an initial object in the full subcategory of $\VsixFF(\mathcal C^\hol, E)$ consisting of those $D$ such that $D^*$ is a motivic pullback formalism. In fact, we have an equivalence
	\[
		\PF^\mot(\mathcal C^\hol) \to \VsixFF(\mathcal C^\hol, E)_{{\SH^\hol}/}
	.\]
\end{thm}

It should be possible to give many other versions of this result, but we have chosen this one for the sake of brevity, and as a way to demonstrate the use of our general results, leaving stronger statements for later works.

\begin{proof}[Proof of \Cref{thm:hol 6FF}]
	Let $I$ be the collection of equivalences, and $P$ be the collection of truncated proper maps in $\mathcal C^\hol$. Then \Cref{rmk:6FF ext functor} gives the following commutative diagram of restrictions
	\[
		\begin{tikzcd}
			\VsixFF(\mathcal C^\hol, P)_{P, I} \ar[d, "\sim"'] \ar[r, "\sim"] & \VsixFF(\mathcal C^\hol, P) \ar[d] & \ar[l, "\sim"'] \VsixFF(\mathcal C^\hol, E) \ar[d] \\
			\PPF(\mathcal C^\hol) \ar[r, equals] & \PPF(\mathcal C^\hol) & \ar[l, equals] \PPF(\mathcal C^\hol)
		\end{tikzcd}
	,\]
	so we find that the map
	\[
		\VsixFF(\mathcal C^\hol, E) \to \PPF(\mathcal C^\hol)
	\]
	is an equivalence. Thus, we conclude by \Cref{thm:hol PPF} and the argument of \Cref{rmk:SHalg init}.
\end{proof}

\appendix
\section{Miscellaneous Categorical Results}

In this section we record some unsorted abstract results about categories.

% NOTE: I commented out references to this lemma
% \begin{lem} \label{lem:comp of maps with diagonal}
% 	Let $\mathcal C$ be a category, and let $\mathcal K$ be a collection of maps in $\mathcal C$ that is stable under composition and base change. Then the composite of maps with diagonal in $\mathcal K$ has diagonal in $\mathcal K$.
% 	\begin{proof}
% 		Let
% 		\[
% 			X \xrightarrow{f} Y \xrightarrow{g} Z
% 		\]
% 		be maps in $\mathcal C$ such that $X \to X \times_Y X$ and $Y \to Y \times_Z Y$ are in $\mathcal K$.
%
% 		Consider the diagram
% 		\[
% 			\begin{tikzcd}
% 				X \ar[r] & X \times_Y X \ar[d] \ar[r] & X \times_Z X \ar[d] \\
% 				{} & Y \ar[r] & Y \times_Z Y
% 			\end{tikzcd}
% 		.\]
% 		The square is Cartesian, so since the bottom map is in $\mathcal K$, so is the top right map. Thus, $X \to X \times_Z X$ is a composite of maps in $\mathcal K$.
% 	\end{proof}
% \end{lem}

\begin{lem} \label{lem:source-locality of colim-pres}
	Let $I$ be a simplicial set, let $\mathcal D : I \to \Cat$ be a diagram, and let $\mathcal D' \subseteq \varprojlim \mathcal D$ be a full subcategory that admits $I$-indexed colimits. Suppose that for each $i \in I_0$, the functor $\mathcal D' \to \mathcal D(i)$ admits a left adjoint $L_i$. For any functor $F : \mathcal D' \to \mathcal C$ that preserves $I$-indexed colimits, we have that $F$ is colimit-preserving if and only if $F \circ L_i$ is colimit-preserving for all $i \in I_0$.
	\begin{proof}
		Since colimit-preserving functors are stable under composition, we only need to address the ``if'' direction.

		Write $\tilde{\mathcal D} \to I$ for a coCartesian fibration corresponding to $\mathcal D$. Note that by \cite[Corollary 3.3.3.2]{htt} or \kerodoncite{02TK}, we can identify $\varprojlim D$ with the category $\Fun_I^\coCart(I, \tilde{\mathcal D})$ of coCartesian sections of $\tilde{\mathcal D} \to I$. Since $\mathcal D'$ admits $I$-indexed colimits, \cite[Lemma D.4.7]{HM6FF} says that the inclusion $\mathcal D' \to \varprojlim D$ admits a left adjoint $L$ given as the composite
		\[
			\Fun_I^\coCart(I, \tilde{\mathcal D}) \to \Fun_I(I, \mathcal D' \times I) \simeq \Fun(I, \mathcal D') \xrightarrow{\varinjlim_I} \mathcal D'
		,\]
		where the first arrow is induced by a left adjoint relative $I$ (see \cite[Definition 7.3.2.2]{ha}) of the functor $\mathcal D' \times I \to \Fun_I^\coCart(I, \tilde{\mathcal D})$ corresponding to the inclusion $\mathcal D' \to \varprojlim \mathcal D$.

		Since $L$ has a fully faithful right adjoint, to show that $F$ is colimit-preserving, it suffices to show that $F \circ L$ is colimit-preserving. Now, since $F$ commutes with $I$-indexed colimits, it suffices to show that the following composite is colimit-preserving:
		\[
			\Fun_I^\coCart(I, \tilde{\mathcal D}) \to \Fun_I(I, \mathcal D' \times I) \to \Fun_I(I, \mathcal C \times I)
		.\]
		By considering restriction along the inclusion of the $0$-skeleton of $I$, we recognize this as the top row in the following diagram
		\[
			\begin{tikzcd}
				\Fun_I^\coCart(I, \tilde{\mathcal D}) \ar[d] \ar[r] & \Fun_I(I, \mathcal D' \times I) \ar[d] \ar[r] & \Fun_I(I, \mathcal C \times I) \ar[d] \\
				\prod_{i \in I_0} \mathcal D(i) \ar[r] & \prod_{i \in I_0} \mathcal D' \ar[r, "\prod_{i \in I_0} F"] & \prod_{i \in I_0} \mathcal C
			\end{tikzcd}
		.\]
		Furthermore, by \cite[Proposition 7.3.2.5]{ha}, we identify the bottom left horizontal arrow with $\prod_{i \in I_0} L_i$. Note that all of the vertical arrows are conservative and colimit-preserving, so it suffices to show that the composite of the bottom row is colimit-preserving, but this follows from the fact that for each $i \in I_0$, the composite $F \circ L_i$ is colimit-preserving.
	\end{proof}
\end{lem}

\begin{lem} \label{lem:radj cons iff ladj gens}
	Let $\{f_i : \mathcal C \to \mathcal D_i\}_i$ be a small collection of right adjoint functors of presentable categories, and for each $i$, let $f_i^*$ be a left adjoint of $f_i$. Then the family $\{f_i\}_i$ is jointly conservative if and only if the union of the images of the functors $\{f_i^*\}_i$ generate $\mathcal C$ under small colimits.
	\begin{proof}
		For each $i$, since $\mathcal D_i$ is presentable, we have a small set $S_i$ of objects of $\mathcal D_i$ that generate $\mathcal D_i$ under small colimits. Since $f_i^*$ preserves small colimits for all $i$, we have that the union of the images of $\{f_i^*\}$ generates $\mathcal C$ under colimits if and only if
		\[
			S \coloneqq \bigcup_i f_i^*(S_i)
		\]
		generates $\mathcal C$ under small colimits. Note that $S$ is a small union of small collections, so $S$ is small. Therefore, \cite[Corollary 2.5]{MonTow} says that $S$ generates $\mathcal C$ under colimits if and only if the functors $\{\mathcal C(f_i^*(Y), -)\}_{i, Y \in S_i}$ are jointly conservative, and since $f_i^* \dashv f_i$, this is equivalent to the functors $\{\mathcal D_i(Y, f_i(-))\}_{i, Y \in S_i}$ being jointly conservative, or equivalently, that the composite
		\[
			\mathcal C \to \prod_i \mathcal D_i \xrightarrow{\prod_i (\mathcal D_i(Y,-))_{Y \in S_i}} \prod_{i, Y \in S_i} \spaces
		\]
		is conservative.

		By \cite[Corollary 2.5]{MonTow}, we know already that the second functor is conservative, so the first functor is conservative if and only if the composite is conservative, as desired.
	\end{proof}
\end{lem}

\begin{lem} \label{lem:crit for linear adj by unit}
	Let $f^* : \mathcal D \to \mathcal C$ be a symmetric monoidal functor, and let $f_* : \mathcal C \to \mathcal D$ be a $\mathcal D$-linear functor. Given a transformation
	\[
		\epsilon : f^* f_* \to \id
	,\]
	we have that if $\epsilon$ is the counit of an adjunction $f^* \dashv f_*$, then there is a map $u : 1 \to f_* 1$ such that the composite
	\[
		1 \xrightarrow{f^* u} f^* f_* 1 \xrightarrow{\epsilon(1)} 1
	\]
	is equivalent to the identity, and the converse holds if $f_*$ and $\epsilon$ are $\mathcal D$-linear.
	\begin{proof}
		If $\epsilon$ is the counit of an adjunction $f^* \dashv f_*$, then we have a unit $\eta : \id \to f_* f^*$ such that
		\[
			f^* \xrightarrow{f^* \eta} f^* f_* f^* \xrightarrow{\epsilon f^*} f^*
		\]
		is equivalent to the identity. Thus, by setting $u = \eta 1$, we find that the composite
		\[
			1 \to f^* f_* 1 \to 1
		\]
		is equivalent to the identity.

		Conversely, given $u : 1 \to f_* 1$, since $f^*, f_*$ are $\mathcal D$-linear, we have that $- \otimes u$ defines a map $\eta : \id \to f_* f^*$, and since $\epsilon$ is $\mathcal D$-linear,
		\[
			f^* \xrightarrow{f^* \eta} f^* f_* f^* \xrightarrow{\epsilon f^*} f^*
		\]
		is equivalent to
		\[
			f^*(-) \otimes (1 \xrightarrow{f^* u} f^* f_* 1 \xrightarrow{\epsilon(1)} 1)
		,\]
		which we have assumed is equivalent to the identity.

		Furthermore, since $f_*$ is $\mathcal D$-linear, we have that
		\[
			\eta f_* \simeq f_*(-) \otimes u \simeq f_*(- \otimes f^* u)
		,\]
		and since $\epsilon$ is $\mathcal D$-linear, we have that
		\[
			\epsilon \simeq - \otimes \epsilon 1 \simeq -
		,\]
		so the composite
		\[
			f_* \xrightarrow{\eta f_*} f_* f^* f_* \xrightarrow{f_* \epsilon} f_*
		\]
		is equivalent to
		\[
			f_* \epsilon \circ \eta f_* \simeq f_*(- \otimes \epsilon 1) \circ f_*(- \otimes f^* u) \simeq f_*(- \otimes (\epsilon 1 \circ f^* u))
		,\]
		which is equivalent to the identity since $\epsilon 1 \circ f^* u$ is equivalent to the identity.
	\end{proof}
\end{lem}

\begin{lem} \label{lem:ladj ret to radg sec}
	Let
	\[
		\mathcal C \xrightarrow{s_*} \mathcal D \xrightarrow{r_*} \mathcal C
	\]
	be functors that have left adjoints $s^* \dashv s_*$ and $r^* \dashv r_*$, and fix a map $\alpha_* : \id \to r_* s_*$, with adjunct $\alpha^* : s^* r^* \to \id$. Then we have a commutative square
	\[
		\begin{tikzcd}
			r^* \ar[d] \ar[r, "r^* \alpha_*" ] & r^* r_* s_* \ar[d] \\
			s_* s^* r^* \ar[r, "s_* \alpha^*"'] & s_*
		\end{tikzcd}
	,\]
	where the left arrow is the unit of $s^* \dashv s_*$, and the right arrow is the counit of $r^* \dashv r_*$.
	\begin{proof}
		Note that one of the composites in the square is the map $r^* \to s_*$ adjunct to $\alpha_* : \id \to r_* s_*$, and the other is adjunct to $\alpha^* : s^* r^* \to \id$.
	\end{proof}
\end{lem}

\begin{lem} \label{lem:char ex seq of st cats}
	Let
	\[
		\mathcal A \xrightarrow{i_*} \mathcal T \xrightarrow{j^*} \mathcal B
	\]
	be functors of stable categories such that $j^* i_* \simeq 0$. Let $i^*, j_\sharp$ be left adjoints of $i_*,j^*$ respectively, and suppose $j_\sharp$ is fully faithful. Then the following are equivalent:
	\begin{enumerate}

		\item $i^*, j^*$ are jointly conservative, and $i^* \to i^* i_* i^*$ is an equivalence.

		\item $j_\sharp j^* \to \id \to i_* i^*$ is an exact triangle, and $i_*$ is conservative.

		\item $i_*$ is fully faithful, and for any object $T \in \mathcal T$, we have that $j^* T$ is a zero object if and only if $T$ is in the essential image of $i_*$.

	\end{enumerate}
	\begin{proof}
		Note that since $j^* i_* \simeq 0$, the left adjoint $i^* j_\sharp$ is also equivalent to $0$.

		Suppose that $i^*, j^*$ are jointly conservative. Note that $j^*(j_\sharp j^* \to \id \to i_* i^*)$ is $j^* j_\sharp j^* \to j^* \to 0$, which is exact, since $j^* j_\sharp j^* \to j^*$ is an equivalence, and $i^*(j_\sharp j^* \to \id \to i_* i^*)$ is $0 \to i^* \to i^* i_* i^*$, which is exact if $i^* \to i^* i_* i^*$ is an equivalence. Thus, the first condition implies the second.

		Now, suppose that $j_\sharp j^* \to \id \to i_* i^*$ is an exact triangle, and $i_*$ is conservative. For any $A \in \mathcal A$, since $j^* i_* A \simeq 0$, we have an exact triangle $0 \to i_* A \to i_* i^* i_* A$, that is, $\id \to i_* i^*$ is an equivalence at $i_* A$, so by the triangle identities for $i^* \dashv i_*$, we have that $i_* (i^* i_* A \to A)$ is an equivalence. Since $i_*$ is conservative, it follows that the counit $i^* i_* A \to A$ is an equivalence. Since this holds for all $A \in \mathcal A$, we find that the counit of $i^* \dashv i_*$ is an equivalence, so $i_*$ is fully faithful.

		Now, let $T \in \mathcal T$ such that $j^* T \simeq 0$. Then $j_\sharp j^* T \to T \to i_* i^* T$ is $0 \to T \to i_* i^* T$, so since this triangle is exact, $T \simeq i_* i^* T$, so $T$ is in the essential image of $T$.

		This concludes the proof that the second condition implies the third.

		Finally, suppose that $i_*$ is fully faithful, and the kernel of $j^*$ is the essential image of $i_*$, and let $f : X \to Y$ be a map in $\mathcal T$ such that $i^* f$ and $j^* f$ are equivalences. Then
		\[
			j^* \ker f \simeq \ker j^* f \simeq 0
		,\]
		so $\ker f$ is in the essential image of $i_*$, so $\ker f \simeq i_* i^* \ker f$ since $i_*$ is fully faithful, but
		\[
			i^* \ker f \simeq \ker i^* f \simeq 0
		,\]
		so
		\[
			\ker f \simeq i_* i^* \ker f \simeq i_* 0 \simeq 0
		,\]
		so $f$ is an equivalence (since $\mathcal T$ is stable).
	\end{proof}
\end{lem}

\section{6-functor formalisms} \label{S:app 6FF}

In this section, we will prove some auxiliary results useful for our study of 6-functor formalisms. Throughout this section, we will say ``geometric setup'' to refer to the notion given in \cite[Definition A.5.1]{Mann6FF}, which is less restrictive than the notion considered in \cite[\S2]{HM6FF}, and refers to a pair $(\mathcal C,E)$ of a category $\mathcal C$, and a collection of morphisms $E$ in $\mathcal C$ that contains all equivalences, and is stable under base change and composition. Note that we will sometimes conflate $E$ with the corresponding wide subcategory of $\mathcal C$.

If $(\mathcal C,E)$ is a geometric setup, we will write $\Span(\mathcal C,E)$ for the category referred to as $\Corr(\mathcal C)_{E, \all}$ in \cite{Mann6FF} and $\Corr(\mathcal C,E)$ in \cite{HM6FF}. This is also denoted by $\Span(\mathcal C, E)$ in \cite[Construction 3.18]{CLL6FF}.

In fact, this category enhances to an operad, as given in \cite[Definition A.5.4]{Mann6FF}. Following \cite[Definition A.5.7 and A.5.6]{Mann6FF}, we have that the category of 3-functor formalism is the category $\Alg_{\Span(\mathcal C,E)}(\widehat{\Cat})$ of lax symmetric monoidal functors $\Span(\mathcal C,E) \to \widehat{\Cat}$.
\begin{nota} \label{nota:6FF*!}
	As in \cite[Definition A.5.6 and A.5.7]{Mann6FF}, for any operad $\mathscr V$, and morphism of operads $D : \Span(\mathcal C,E) \to \mathscr V$, we have the following induced functors:
	\begin{itemize}

		\item By restricting along $\mathcal C^\op \to \Span(\mathcal C,E)$, and using \cite[Theorem 2.4.3.18]{ha}, we obtain a functor $D^* : \mathcal C^\op \to \CAlg(\PrL)$, and for any map $f$ in $\mathcal C$, we write $f^* \coloneqq D^*(f)$, and $f_*$ for its right adjoint when it exists.

		\item By restricting along $E \to \Span(\mathcal C,E)$, we obtain the functor $D_! : E \to \PrL$, where for any map $f$ in $\mathcal C$, we write $f_! \coloneqq D_!(f)$. If $f_!$ admits a right adjoint for all $f \in E$, we write $D^! : E^\op \to \PrR$ for the functor obtained by taking right adjoints, so for any map $f \in \mathcal C$, we have $D^!(f) \simeq f^!$.

	\end{itemize}
	In fact, if $\mathscr V$ is just a category, and $D$ is any functor $\Span(\mathcal C,E) \to \mathscr V$, we can still define $D^* : \mathcal C^\op \to \mathscr V$, and $D_! : E \to \mathscr V$.
\end{nota}

\begin{lem} \label{lem:equiv of 3FF}
	Let $(\mathcal C,E)$ be a geometric setup, let $\mathscr V$ be an operad, and let $\phi : D \to D'$ be a morphism in $\Alg_{\Span(\mathcal C,E)} \mathscr V$. Then $\phi$ is an equivalence if for all $S \in \mathcal C$, the morphism $D(S) \to D'(S)$ is an equivalence.
	\begin{proof}
		Note that since $\CAlg(\mathscr V) \to \mathscr V$ is conservative, it suffices to show that the composite functor
		\[
			\Alg_{\Span(\mathcal C,E)} \mathscr V \to \Alg_{\mathcal C^\op} \mathscr V \simeq \Fun(\mathcal C^\op, \CAlg(\mathscr V))
		\]
		is conservative, where $\mathcal C^\op$ is given the coCartesian structure of \cite[\S2.4.3]{ha}, and the last map is the equivalence given by \cite[Theorem 2.4.3.18]{ha}. In particular, it suffices to show that the first functor is conservative, but this functor fits into a commutative square
		\[
			\begin{tikzcd}
				\Alg_{\Span(\mathcal C,E)} \mathscr V \ar[d] \ar[r] & \Alg_{\mathcal C^\op} \mathscr V \ar[d] \\
				\Fun(\Span(\mathcal C, E)^\otimes, \mathscr V^\otimes) \ar[r] & \Fun(((\mathcal C^\op)^{\coprod})^\op, \mathscr V^\otimes)
			\end{tikzcd}
		\]
		where the remaining arrows are conservative.
	\end{proof}
\end{lem}

\subsection{Extending 6-functor formalisms} \label{S:6FF ext}

Fix a geometric setup $(\mathcal C,E)$. In this section we present refined versions of some of the extension results from \cite[\S A.5]{Mann6FF} and \cite[\S3.4]{HM6FF} about extending 6-functor formalisms on $(\mathcal C,E)$ to larger geometric setups.

First we consider the following refinement of \cite[Lemma A.5.11]{Mann6FF} and \cite[Proposition 3.4.8(ii)]{HM6FF}, suggested by Bastiaan Cnossen.
\begin{lem} \label{lem:lext 6FF}
	Let $\tilde E \supseteq E$ be a collection of maps such that $(\mathcal C, \tilde E)$ is a geometric setup, and let $\tau$ be a Grothendieck topology on $\mathcal C$ such that for every $f : X \to Y$ in $\tilde E$, there is a small $\tau$-covering family of $X$ consisting of maps $X' \to X$ in $E$ such that $X' \to Y$ is in $E$.

	For every 6-functor formalism $D : \Span(\mathcal C,E) \to \PrL$, if $D^!$ is a $\tau$-sheaf, then $D$ extends to a 6-functor formalism $\tilde D : \Span(\mathcal C, \tilde E) \to \PrL$ on $(\mathcal C, \tilde E)$, such that for any 6-functor formalism $D'$ on $(\mathcal C, \tilde E)$, the map of spaces
	\[
		\Alg_{\Span(\mathcal C, \tilde E)}(\PrL)(\tilde D, D') \to \Alg_{\Span(\mathcal C, E)}(D, D'|_{\Span(\mathcal C, E)})
	\]
	is an equivalence.
	% NOTE: don't get automatic suave prim unless assume that $E$ is closed under diagonals
	% need that in order to get diagram all of whose edges are in $E$
	\begin{proof}
		This follows from the arguments of \cite[Proposition 3.4.8(ii)]{HM6FF} and \cite[Lemma A.5.11]{Mann6FF}. We will sketch the key points here.

		As in the proof of \cite[Proposition 3.4.8.(ii)]{HM6FF}, since $\tau$ is a Grothendieck topology and $D^!$ is a $\tau$-sheaf, we can apply the argument of \cite[Lemma A.5.11]{Mann6FF}, which shows that we can apply \cite[Proposition 3.1.3.3]{ha} to produce a 6-functor formalism $\tilde D : \Span(\mathcal C, \tilde E)^\otimes \to \PrL^\otimes$, and an equivalence $D \to \tilde D|_{\Span(\mathcal C,E)^\otimes}$ exhibiting $\tilde D$ as a free $\Span(\mathcal C, \tilde E)$-algebra generated by $D$.
		% In particular, by \cite[Definition 3.1.3.1]{ha}, the argument of \cite[Lemma A.5.11]{Mann6FF}, and \cite[Example 3.1.1.17]{ha}, we have that the map $D \to D'|_{\Span(\mathcal C)_{E,\all}^\otimes}$ is an equivalence.
		% TODO: clarify

		Thus, we conclude by \cite[Proposition 3.1.3.2]{ha}.
	\end{proof}
\end{lem}

The following result is given by adapting and combining \cite[Proposition A.5.16]{Mann6FF} and \cite[Proposition 3.4.8(i)]{HM6FF}. In this case, we let $(\tilde{\mathcal C}, \tilde E)$ be another geometric setup such that there is a fully faithful inclusion $\mathcal C \subseteq \tilde{\mathcal C}$ sending $E$ to $\tilde E$. We will be interested in extending 6-functor formalisms on $(\mathcal C, E)$ to $(\tilde{\mathcal C}, \tilde E)$.
\begin{lem} \label{lem:rext 6FF}
	% NOTE: the proof of 3.4.8(i) really depends on right-cancellativeness, so do need to ask for HM version of geom setup
	Suppose that $E$ and $\tilde E$ are stable under taking diagonals, and that the inclusion $\mathcal C \to \tilde{\mathcal C}$ preserves base changes of maps in $E$. 
	% Let $(\mathcal C,E)$ and $(\tilde{\mathcal C}, \tilde E)$ be geometric setups in the sense of \cite[Convention 2.1.3]{HM6FF}, and let $\mathcal C \subseteq \tilde{\mathcal C}$ be a fully faithful inclusion of categories that preserves base changes of maps in $E$, and such that $E \subseteq \tilde E$.

	% NOTE: if tau is triv top, then this just means that for any $Y \to \tilde Y$, the base change is in $E$
	Let $\tau$ be a Grothendieck topology on $\tilde{\mathcal C}$ such that for any map $\tilde X \to \tilde Y$ in $\tilde E$, there is a small $\tau$-covering sieve $\mathcal U$ of $\tilde Y$ such that for any $Y \to \tilde Y$ in $\mathcal U$, if $Y \in \mathcal C$, then $\tilde X \times_{\tilde Y} Y \to Y$ is in $E$.
	% Let $(\tilde{\mathcal C}, \tilde E)$ be a geometric setup such that every map in $\tilde E$ to an object of $\mathcal C$ is in $E$.

	Let $D : \Span(\mathcal C, E) \to \widehat{\Cat}$ be a 3-functor formalism such that the right Kan extension $\tilde D^*$ of $D^*$ along $\mathcal C^\op \to \tilde{\mathcal C}^\op$ is a $\tau$-sheaf. Then $\tilde D^*$ extends to a 3-functor formalism $\tilde D : \Span(\tilde{\mathcal C}, \tilde E) \to \widehat{\Cat}$ such that $\tilde D$ extends $D$, and for any other 3-functor formalism $D' \in \Alg_{\Span(\tilde{\mathcal C}, \tilde E)}(\widehat{\Cat})$, the map of spaces
	\begin{equation} \label{eqn:rext 6FF hom}
		\Alg_{\Span(\tilde{\mathcal C}, \tilde E)}(\widehat{\Cat})(D', \tilde D) \to \Alg_{\Span(\mathcal C, E)}(D'|_{\Span(\mathcal C,E)}, D)
	\end{equation}
	is an equivalence.

	% NOTE: seems hard to get suave/prim statement unless assume that the covering sieves are gen by maps from objects of $\mathcal C$
	% and probably also need to assume that every map in $\tilde E$ to obj of $\mathcal C$ is in $E$

	% Furthermore, $\tilde D^*$ is the right Kan extension of $D^*$ along $\mathcal C^\op \to \tilde{\mathcal C}^\op$. In particular, for all $\tilde X \in \tilde{\mathcal C}$, the map
	% \[
	% 	\tilde D^*(\tilde X) \to \varprojlim_{\substack{X \to \tilde X \text{ in } (\tilde{\mathcal C}_{/\tilde X})^\op \\ \text{where } X \in \mathcal C}} D^*(X)
	% \]
	% is an equivalence.
\end{lem}

Before coming to the proof of \Cref{lem:rext 6FF}, we will present a result that combines it with \Cref{lem:lext 6FF}. For this, we will need to fix two Grothendieck topologies: $\tau^!$ is a Grothendieck topology on $\mathcal C$, and $\tau^*$ is a Grothendieck topology on $\tilde{\mathcal C}$ such that
\[ \tag{*} \label{ext 6FF crit}
	\parbox{0.9\textwidth}{For any map $\tilde X \to \tilde Y$ in $\tilde E$, there is a small $\tau^*$-covering sieve $\mathcal U$ of $\tilde Y$ such that for any $Y \to \tilde Y$ in $\mathcal U$, if $Y \in \mathcal C$, then $\tilde X \times_{\tilde Y} Y \in \mathcal C$, and there is a small $\tau^!$-covering family of $\tilde X \times_{\tilde Y} Y$ consisting of maps $X \to \tilde X \times_{\tilde Y} Y$ in $E$ such that $X \to Y$ is also in $E$.}
	% NOTE: this means that for any $Y \to \tilde Y$ in $\mathcal U$, the map satisfies that for all $Y' \to Y$ with $Y' \in \mathcal C$, the pullback $\tilde X_{\tilde Y} Y'$ is in $\mathcal C$
	% in particular,
		% if we define $\tilde E'$ to be the collection of $\tilde X \to \tilde Y$ in $\tilde E$ such that for all $Y \in \mathcal C$ and maps $Y \to \tilde Y$, the pullback $\tilde X_{\tilde Y} Y$ is in $\mathcal C$,
		% then for all maps $\tilde X \to \tilde Y$ in $\tilde E$, $\tilde Y$ has a $\tau^*$-covering family by maps $Y \to \tilde Y$ such that $\tilde_X \times_{\tilde Y} Y \to Y$ is in $\tilde E'$
\]

\begin{defn}
	For any geometric setup $(\mathcal C', E')$ where $\mathcal C \subseteq \mathcal C' \subseteq \tilde{\mathcal C}$ and $E \subseteq E'$, we write $\sixFF^\tau(\mathcal C',E')$ for the full subcategory of $\Alg_{\Span(\mathcal C', E')}(\PrL)$ consisting of 6-functor formalisms $D : \Span(\mathcal C',E') \to \PrL$ such that $D^!$ restricts to a $\tau^!$-sheaf on $E$, and the right Kan extension of $D^*|_{\mathcal C^\op}$ along $\mathcal C^\op \to \tilde{\mathcal C}^\op$ is a $\tau^*$-sheaf.
\end{defn}

We can now state our combined result as follows:
\begin{prp} \label{prp:ext 6FF}
	Suppose that $\mathcal C$ admits base changes along maps in $\tilde E$, and the inclusion $\mathcal C \to \tilde{\mathcal C}$ preserves these.
	% NOTE: need to ask for base changes along maps in $\tilde E$, not just $E$.
	% this is so that $\tilde E \cap \mathcal C$ is geom setup
	Then the restriction functor
	\[
		\sixFF^\tau(\tilde{\mathcal C}, \tilde E) \to \sixFF^\tau(\mathcal C, E)
	\]
	admits a fully faithful section $R$ whose essential image is given by those $D \in \sixFF^\tau(\tilde{\mathcal C}, \tilde E)$ such that $D^*$ is a right Kan extension of $D^*|_{\mathcal C^\op}$. Furthermore,
	\begin{itemize}

		\item if $\mathcal C \to \tilde{\mathcal C}$ is an equivalence and $\tau^*$ is the trivial topology, then $R$ is an equivalence, and
			%$R$ is a left adjoint, and

		\item if $\tau^!$ is the trivial topology,
			% for every map $\tilde X \to \tilde Y$ in $\tilde E$, there is a $\tau^*$-covering sieve $\mathcal U$ of $\tilde Y$ such that for any $Y \to \tilde Y$ in $\mathcal U$, if $Y \in \mathcal C$, then $\tilde X \times_{\tilde Y} Y \to Y$ is in $E$,
			then $R$ is a right adjoint.

	\end{itemize}
	
	% In general, we have that for any $D \in \sixFF^\tau(\mathcal C,E)$, the presheaf $(RD)^*$ is the right Kan extension of $D^*$ along $\mathcal C^\op \to \tilde{\mathcal C}^\op$.
\end{prp}

Before proving \Cref{prp:ext 6FF,lem:rext 6FF}, we present some results that are useful for producing the hypotheses of \Cref{lem:lext 6FF,lem:rext 6FF}.

\begin{lem} \label{lem:source-local geom setup}
	Let $E' \subseteq \bar E$ be collections of morphisms in a category $\mathcal C$. If $\tau$ is a Grothendieck topology on $\mathcal C$, let $\hat E$ be the collection of maps $f : X \to Y$ in $\bar E$ such that $X$ admits a small $\tau$-cover by maps $X' \to X$ in $E'$ such that $X' \to X \to Y$ is in $E'$.

	If $E'$ and $\bar E$ contain all equivalences, and are stable under composition, and base change, then the same is true of $\hat E$. Furthermore, if $E'$ and $\bar E$ are also stable under taking diagonals, then the same is true of $\hat E$.
	\begin{proof}
		It is clear that $\hat E$ contains all equivalences since $E' \subseteq \hat E$. To see that $\hat E$ is stable under base change, it suffices to note that $\tau$-covers are stable under base change, and so are maps in $E'$ and $\bar E$.

		To see that $\hat E$ is stable under composition, let
		\[
			X \xrightarrow{f} Y \xrightarrow{g} Z
		\]
		be maps in $\bar E$. Let $a : X' \to X$ be a map in $E'$ such that $f \circ a \in E'$, and let $b : Y' \to Y$ be a map in $E'$ such that $g \circ b \in E'$. Consider the commutative diagram
		\[
			\begin{tikzcd}
				X' \times_Y Y' \ar[d] \ar[r] & X \times_Y Y'\ar[d] \ar[r] & Y' \ar[d] & \\
				X' \ar[r] & X \ar[r] & Y \ar[r] & Z
			\end{tikzcd}
		.\]
		% Then $X' \times_Y Y' \to X$ is a composite of maps in $E'$, and $X' \times_Y Y' \to Y$ factors as
		% \[
		% 	X' \times_Y Y' \to X' \to X \to Y
		% ,\]
		% where $X' \times_Y Y' \to X'$ is in $E'$ since it is a base change of $b$, and $X' \to X \to Y$ is in $E'$ by assumption. Thus, $X' \times_Y Y' \to Y$ is in $E'$.
		Then $X' \times_Y Y' \to Z$ factors as
		\[
			X' \times_Y Y' \to Y' \to Y \to Z
		,\]
		where $X' \times_Y Y' \to Y'$ is in $E'$ since it is a base change of $X' \to Y$, and $Y' \to Y \to Z$ is in $E'$ by assumption, so $X' \times_Y Y' \to Z$ is in $E'$. Furthermore, $X' \times_Y Y' \to X$ is a composite of maps in $E'$, so it is in $E'$.

		If $f,g$ are in $\hat E$, we can choose $\tau$-covering families $\{X_i \to X\}_i$ and $\{Y_j \to Y\}_j$ consisting of families in $E'$ such that for each $i$, $X_i \to X \to Y$ is in $E'$, and for each $j$, $Y_j \to Y \to Z$ is in $E'$. Thus, $\{X_i \times_Y Y_j \to X\}_{i,j}$ is a covering family, and for each $i,j$, $X_i \times_Y Y_j \to X$ is in $E'$, and $X_i \times_Y Y_j \to Z$ is in $E'$. This shows that $g \circ f \in \hat E$.

		Finally, assume that $E'$ and $\bar E$ are stable under taking diagonals. Note that if $X' \to X \to Y$ are maps in $\bar E$, then we have a commutative diagram
		\[
			\begin{tikzcd}
				X' \ar[d] \ar[r] & X' \times_Y X' \ar[d] \\
				X \ar[r] & X \times_Y X
			\end{tikzcd}
		\]
		where all maps are in $\bar E$. If $X' \to Y$ is in $E'$, then the top map is in $E'$, and if $X' \to X$ is in $E'$, then the right map is in $E'$, so in this case, $X' \to X \times_Y X$ is in $E'$. If $X \to Y$ is in $\hat E$ then there is a small $\tau$-cover of $X$ by maps $X' \to X$ such that $X' \to X$ and $X' \to Y$ are in $E'$, so this argument shows that the same $\tau$-cover exhibits that $X \to X \times_Y X$ is in $\hat E$, as desired
	\end{proof}
\end{lem}

\begin{lem} \label{lem:target-local geom setup}
	Let $\mathcal C' \subseteq \bar{\mathcal C}$ be the inclusion of a full subcategory, and let $E',\bar E$ be collections of maps in $\mathcal C', \bar{\mathcal C}$ respectively that contain all equivalences, and are stable under base change and composition, and such that the inclusion $\mathcal C' \to \bar{\mathcal C}$ preserves base changes along $E'$.

	Define $\hat E$ to be the collection of maps $X \to Y$ in $\bar E$ such that for any map $Y' \to Y$, if $Y' \in \mathcal C'$, then the pullback $X \times_Y Y'$ is also in $\mathcal C'$, and $X \times_Y Y' \to Y'$ is in $E'$.

	Then $\hat E$ also contains all equivalences, and is stable under base change and composition. If $E'$ and $\bar E$ are also stable under taking diagonals, then so is $\hat E$.
	\begin{proof}
		It is clear that $\hat E$ contains all equivalences. The fact that $\hat E$ is stable under base change and composition follows from pasting Cartesian squares. 

		To show the property about diagonals, we note that by \cite[Lemma 2.1.5]{HM6FF}, $E'$ and $\bar E$ are right-cancellative, and it suffices to show that $\hat E$ is right-cancellative: if $X \to Y \to S$ are maps in $\bar{\mathcal C}$ such that $X \to S$ and $Y \to S$ are in $\hat E$, we must show that $X \to Y$ is in $\hat E$. We must show that for any $Y' \to Y$, if $Y' \in \mathcal C'$, the map $X \times_Y Y' \to Y'$ is in $E'$.

		First note that since $\bar E$ is right-cancellative, and $\hat E \subseteq \bar E$, we have that $X \to Y$ is in $\bar E$, so it admits all base changes. Using the fact that all of the maps in $X \to Y \to S$ are in $\bar E$, so admit all base changes, we produce the following diagram consisting of Cartesian squares:
		\[
			\begin{tikzcd}
				X \times_Y Y' \ar[d] \ar[r] & X \times_S Y' \ar[d] \ar[r] & X \ar[d] \\
				Y' \ar[r] & Y \times_S Y' \ar[d] \ar[r] & Y \ar[d] \\
				{} & Y' \ar[r] & S
			\end{tikzcd}
		.\]
		Since $Y \to S$ and $X \to S$ are in $E$, and $Y' \in \mathcal C'$, it follows that the middle vertical maps $Y \times_S Y' \to Y'$ and $X \times_S Y' \to Y'$ are in $E'$, so since $E'$ is right-cancellative, the top middle vertical map $X \times_S Y' \to Y \times_S Y'$ is in $E'$. Therefore, the base change $X \times_Y Y' \to Y'$ of this map is also in $E'$, as desired.

		% Indeed the bottom right Cartesian square exists because $Y \to S$ is in $\hat E$, and the right Cartesian rectangle exists because $X \to S$ is in $\hat E$. It follows that the top right square is Cartesian.
		%
		% Now, since $Y' \in \mathcal C'$, and both $X \to S$ and $Y \to S$ are in $\hat E$, we have that the vertical arrows
		% \[
		% 	X \times_S Y' \to Y' \quad\text{and}\quad Y \times_S Y' \to Y'
		% \]
		% are in $E'$. This implies that $X \times_S Y' \to Y \times_S Y'$ is in $E'$, so the base change along $Y' \to Y \times_S Y'$ exists, whence we can produce a Cartesian square in the top left corner. In particular, since the top right square is also Cartesian, the pullback
		% \[
		% 	X \times_Y Y' \simeq (X \times_S Y') \times_{Y \times_S Y'} Y'
		% \]
		% exists and is in $\mathcal C'$, and the base change $X \times_Y Y' \to Y'$ is in $E'$ since it is a base change of a map in $E'$ along a map from $Y'$.
	\end{proof}
\end{lem}

\begin{proof}[Proof of \Cref{lem:rext 6FF}]
	First note that since $E$ and $\tilde E$ are stable under taking diagonals, both $(\mathcal C,E)$ and $(\tilde{\mathcal C}, \tilde E)$ are geometric setups in the sense of \cite[Convention 2.1.3]{HM6FF}.

	Next note that since the inclusion $\mathcal C \to \tilde{\mathcal C}$ sends $E$ to $\tilde E$, and preserves base changes of maps in $E$, it follows that we have an induced map of operads $\Span(\mathcal C,E) \to \Span(\tilde{\mathcal C}, \tilde E)$.

	The map $\Span(\mathcal C,E)^\otimes \to \Span(\tilde{\mathcal C}, \tilde E)^\otimes$ defines a precomposition functor
	\begin{equation} \label{eqn:res geom setup fun}
		\Fun(\Span(\tilde{\mathcal C}, \tilde E)^\otimes, \widehat{\Cat}) \to \Fun(\Span(\mathcal C, E)^\otimes, \widehat{\Cat}),
	\end{equation}
	which preserves lax Cartesian structures, so it restricts to a functor
	\begin{equation} \label{eqn:res geom setup fun lax}
		\Fun^\lax(\Span(\tilde{\mathcal C}, \tilde E)^\otimes, \widehat{\Cat}) \to \Fun^\lax(\Span(\mathcal C, E), \widehat{\Cat}),
	\end{equation}
	where the notation $\Fun^\lax((-)^\otimes, \widehat{\Cat})$ is described in \cite[Proposition 2.4.1.7]{ha}, and refers to the full subfunctor of $\Fun((-)^\otimes, \widehat{\Cat})$ consisting of lax Cartesian structures. This result shows that we may identify \eqref{eqn:res geom setup fun lax} with the restriction functor
	\[
		\Alg_{\Span(\tilde{\mathcal C}, \tilde E)}(\widehat{\Cat}) \to \Alg_{\Span(\mathcal C, E)}(\widehat{\Cat})
	.\]

	Right Kan extension defines a right adjoint of \eqref{eqn:res geom setup fun}. When every map $\tilde X \to \tilde Y$ in $\tilde E$ satisfies that if $\tilde Y \in \mathcal C$, then $\tilde X \in \mathcal C$, the proof of \cite[Proposition A.5.16]{Mann6FF} shows that taking right Kan extensions defines a right adjoint of \eqref{eqn:res geom setup fun lax} that admits the desired description,
	% \footnote{The definition given in \kerodoncite{02Y9} is the formulation of right Kan extensions that makes this clearest.}
	and such that the right Kan extension $\tilde D$ of $D$ restricts to $D$, so by the dual of \cite[Lemma 3.3.1]{CARMELI2021107763}, the counit $\tilde D|_{\Span(\mathcal C,E)} \to D$ is an equivalence, whence \eqref{eqn:rext 6FF hom} is an equivalence for all $D'$.

	In general, let $\grave E \subseteq \tilde E$ be the subset consisting of maps $\tilde X \to \tilde Y$ such that for any $Y \in \mathcal C$, and map $Y \to \tilde Y$, the map $\tilde X \times_{\tilde Y} Y \to Y$ is in $E$. It follows from \Cref{lem:target-local geom setup} that $(\tilde{\mathcal C}, \grave E)$ is a geometric setup in the sense of \cite[Convention 2.1.3]{HM6FF}, and our argument above shows that $D$ extends to a 3-functor formalism $\grave D : (\tilde{\mathcal C}, \grave E) \to \widehat{\Cat}$ such that $\grave D^*$ is the right Kan extension of $D^*$. Thus, by our assumption, $\grave D^*$ is a $\tau$-sheaf, so the proof of \cite[Proposition 3.4.8(i)]{HM6FF} shows that in \eqref{eqn:res geom setup fun lax}, if $(\mathcal C,E)$ is replaced by $(\tilde{\mathcal C}, \grave E)$, then right Kan extension sends $\grave D$ to a 3-functor formalism $\tilde D$ extending $\grave D$. In particular, $\tilde D^* = \grave D^*$ is the right Kan extension of $D^*$, and \eqref{eqn:rext 6FF hom} is an equivalence for all $D'$.
\end{proof}

\begin{proof}[Proof of \Cref{prp:ext 6FF}]
	If $\mathcal C \to \tilde{\mathcal C}$ is an equivalence, and for every map $\tilde X \to \tilde Y$ in $\tilde E$ there is a small $\tau^!$-covering family of $\tilde X$ consisting of maps $X \to \tilde X$ in $E$ such that $X \to \tilde Y$ is in $E$, it follows from \Cref{lem:lext 6FF} and \kerodoncite{02FV} that the restriction functor has a left adjoint $R$, and that the unit of the adjunction is an equivalence, so this left adjoint is actually a fully faithful section. Now, since the unit is an equivalence, it follows from the triangle identities that for any $\tilde D \in \sixFF^\tau(\tilde{\mathcal C}, \tilde E)$, the counit $R(\tilde D|_{\Span(\mathcal C,E)}) \to \tilde D$ restricts to an equivalence on $\Span(\mathcal C,E)^\otimes \subseteq \Span(\tilde{\mathcal C}, \tilde E)^\otimes$. Therefore, this counit also restricts to an equivalence on $\mathcal C^\op \subseteq \Span(\mathcal C,E)$, so \Cref{lem:equiv of 3FF} shows that this counit is an equivalence, whence the restriction functor is an equivalence. Note that the hypotheses of this paragraph are satisfied if $\mathcal C \to \tilde{\mathcal C}$ is an equivalence, and $\tau^*$ is the trivial topology.
	% In particular, if $\tilde D \in \sixFF^\tau(\tilde{\mathcal C}, \tilde E)$ is the image of $D \in \sixFF^\tau(\mathcal C, E)$, we have that $\tilde D^* \simeq D^*$ is the right Kan extension of $D^*$ along the equivalence $\mathcal C^\op \to \tilde{\mathcal C}^\op$. Conversely, if $D \in \sixFF^\tau(\tilde{\mathcal C}, \tilde E)$

	On the other hand, if every map $\tilde X \to \tilde Y$ in $\tilde E$ satisfies that there is a $\tau^*$-covering sieve $\mathcal U$ of $\tilde Y$ such that for all $Y \to \tilde Y$ in $\mathcal U$, if $Y \in \mathcal C$, then $\tilde X \times_{\tilde Y} Y \to Y$ is in $E$, then \Cref{lem:rext 6FF} and \kerodoncite{02FV} show that the restriction functor has a fully faithful right adjoint $R$ such that for any $D \in \sixFF^\tau(\mathcal C,E)$, $(RD)^*$ is the right Kan extension of $D^*$. For any $\tilde D \in \sixFF^\tau(\tilde{\mathcal C}, \tilde E)$, the unit $\tilde D \to R \tilde D|_{\Span(\mathcal C,E)}$ restricts to the natural map from $\tilde D^*$ to the right Kan extension of $\tilde D^*|_{\mathcal C^\op}$, so by \Cref{lem:equiv of 3FF}, we have that $R$ has the desired essential image. Note that this condition holds if $\tau^!$ is the trivial topology.
	% NOTE: more complete argument:
	% since the unit map $\eta : \tilde D \to R \tilde D|_{\Span(\mathcal C,E)}$ is any map whose restriction along $\Span(\mathcal C,E)^\otimes \to \Span(\tilde{\mathcal C}, \tilde E)^\otimes$ is equivalent to the identity of $\tilde D|_{\Span(\mathcal C,E)}$, we have that $\eta^*|_{\mathcal C^\op}$ is equivalent to the identity of $\tilde D^*|_{\mathcal C^\op}$, so since $\mathcal C^\op \to \tilde{\mathcal C}^\op$ is fully faithful, $\eta^*$ is the natural map from $\tilde D^*$ to the right Kan extension of $\tilde D^*|_{\mathcal C^\op}$. Thus, by \Cref{lem:equiv of 3FF}, we have that $R$ has the desired essential image.

	In general, first define $E'$ to be the collection of maps $X \to Y$ in $\tilde E$ between objects of $\mathcal C$ such that there is a small $\tau^!$-covering family of $X$ consisting of maps $X' \to X$ in $E$ such that $X' \to Y$ is also in $E$. Since $\mathcal C$ admits base changes along maps in $\tilde E$, and the inclusion $\mathcal C \to \tilde{\mathcal C}$ preserves these, \Cref{lem:source-local geom setup} shows that $(\mathcal C, E')$ is a geometric setup in the sense of \cite[Convention 2.1.3]{HM6FF}.

	We have assumed that for any map $\tilde X \to \tilde Y$ in $\tilde E$, there is a small $\tau^*$-covering sieve $\mathcal U$ of $\tilde Y$ such that for any $Y \to \tilde Y$ in $\mathcal U$, if $Y \in \mathcal C$, then $\tilde X \times_{\tilde Y} Y \to Y$ is in $E'$.

	Thus, we obtain the result by composing the fully faithful adjoint sections of the following restriction functors:
	\[
		\sixFF^\tau(\tilde{\mathcal C}, \tilde E) \to \sixFF^\tau(\mathcal C, E') \to \sixFF^\tau(\mathcal C, E)
	.\]
\end{proof}

\subsection{Suave and prim maps} \label{S:suave prim}

The notions of suave and prim maps for 6-functor formalisms are studied in \cite[\S4.5]{HM6FF}. In this section, we will establish some additional results about these maps, as well as consider a generalization.

Before giving a definition of these notions, let us motivate them and indicate their importance by presenting the following result that summarizes some of their key properties (including results both from \cite[\S4.5]{HM6FF}, and this section):
% Given a 6-functor formalism $D$ on a geometric setup $(\mathcal C,E)$, the following result summarizes the main properties of $D$-suave and $D$-prim maps in $E$ proven in this section and in \cite[\S4.5]{HM6FF}:
\begin{thm} \label{thm:suave prim consequences}
	Let $D$ be a 6-functor formalism on a geometric setup $(\mathcal C,E)$ where $E$ is stable under taking diagonals.
	\begin{description}

		\item[Adjunctions] Let $f : X \to Y$ be a map in $E$.
			\begin{description}
				
				\item[If $f$ is $D$-suave] then the functors $f^*, f^!$ have adjoints
					\[
						f_\sharp \dashv f^* \quad\text{and}\quad f^! \dashv f_\flat
					.\]

				\item[If $f$ is $D$-prim] then the functors $f_*, f_!$ have adjoints
					\[
						f_* \dashv f^\sharp \quad\text{and}\quad f^\flat \dashv f_!
					.\]

			\end{description}
		
		\item[Base change]
			Let
			\[
				\begin{tikzcd}
					X' \ar[d, "p"'] \ar[r, "f'"] & Y' \ar[d, "q"] \\
					X \ar[r, "f"'] & Y
				\end{tikzcd}
			\]
			be a Cartesian square in $\mathcal C$ where $f,q \in E$.
			\begin{description}
				
				\item[If $f$ is $D$-suave] then so is every base change of $f$, and the natural maps
					\begin{equation} \label{eqn:suave bc}
						% f'_\sharp p^* \to q^* f_\sharp \qquad p_! f'^! \to f^! q_! \qquad p^* f^! \to f'^! q^* \qquad f_\sharp p_! \to q_! f'_\sharp.
						f^* q_* \to p_* f'^* \qquad p_! f'^! \to f^! q_! \qquad p^* f^! \to f'^! q^* \qquad f'^* q^! \to p^! f^*
					\end{equation}
					are equivalences.

				\item[If $f$ is $D$-prim] then so is every base change of $f$, and the natural maps
					\begin{equation} \label{eqn:prim bc}
						% q^* f_* \to f'_* p^* \qquad f^\flat q_! \to p_! f'^\flat \qquad f'^\flat q^* \to p^* f^\flat \qquad q_! f'_* \to f_* p_!.
						q^* f_* \to f'_* p^* \qquad f'_! p^! \to q^! f_! \qquad f_! p_* \to q_* f'_! \qquad q_! f'_* \to f_* p_!.
					\end{equation}
					are equivalences.

				\item[If $f$ is $D$-suave and $q$ is $D$-prim] then the natural maps
					\begin{equation} \label{eqn:suave prim bc}
						f_\sharp p_* \to q_* f'_\sharp \quad\text{and}\quad p^\flat f^! \to f'^! q^\flat
						% f'^* q^\sharp \to p^\sharp f^* \quad\text{and}\quad q_! f'_\flat \to f_\flat p_!
					\end{equation}
					are equivalences.

			\end{description}

		\item[Projection formula] Let $f : X \to Y$ be a map in $E$.
			\begin{description}
				
				\item[If $f$ is $D$-suave] the natural maps
					\begin{equation} \label{eqn:suave PF}
						f^! \otimes f^* \to f^!(- \otimes -) \quad\text{and}\quad f^* \underline\Hom(-, -) \to \underline\Hom(f^*, f^*)
					\end{equation}
					are equivalences.

				\item[If $f$ is $D$-prim] the natural maps
					\begin{equation} \label{eqn:prim PF}
						f_! \underline\Hom(f^*,-) \to \underline\Hom(-, f_!) \quad\text{and}\quad f_* \otimes - \to f_*(- \otimes f^*)
					\end{equation}
					are equivalences.

			\end{description}

			% TODO: actually can deduce duality from proj form?
			% I guess that given proj form, the duality statement actually becomes
			% just a description of the adjoints in the category of kernels
		\item[Duality] Let $f : X \to Y$ be a map in $E$. Write $\Delta : X \to X \times_Y X$ for the diagonal of $f$, and $\pi : X \times_Y X \to X$ for one of the projections.
			\begin{description}
				
				\item[If $f$ is $D$-suave] we may define the \emph{dualizing complex} of $f$ to be $\omega_f = \pi_\sharp \Delta_! 1 \in D(X)$. The natural maps
					\begin{equation} \label{eqn:suave dual}
						\omega_f \otimes f^* \to f^! \quad\text{and}\quad f^* \to \underline\Hom(\omega_f, f^!)
					\end{equation}
					are equivalences, so we also have natural equivalences
					\begin{equation} \label{eqn:suave adj dual}
						f_\flat \xrightarrow{\sim} f_* \underline\Hom(\omega_f, -) \quad\text{and}\quad f_!(\omega_f \otimes -) \xrightarrow{\sim} f_\sharp.
					\end{equation}

				\item[If $f$ is $D$-prim] we may define the \emph{codualizing complex} of $f$ to be $\delta_f = \pi_* \Delta_! 1 \in D(X)$. The natural maps
					\begin{equation} \label{eqn:prim dual}
						f_!(\delta_f \otimes -) \to f_* \quad\text{and}\quad f_! \to f_* \underline\Hom(\delta_f,-)
					\end{equation}
					are equivalences, so we also have natural equivalences
					\begin{equation} \label{eqn:prim adj dual}
						f^\sharp \xrightarrow{\sim} \underline\Hom(\delta_f, f^!) \quad\text{and}\quad \delta_f \otimes f^* \xrightarrow{\sim} f^\flat.
					\end{equation}

			\end{description}

		\item[Compatibility with morphisms] Let $\phi : D \to D'$ be a transformation of 6-functor formalisms, and let $f \in E$.
			\begin{description}
				
				\item[If $f$ is $D$-suave] then $f$ is also $D'$-suave, and the natural maps
					\begin{equation} \label{eqn:suave morph}
						f_\sharp \phi \to \phi f_\sharp \quad\text{and}\quad \phi f^! \to f^! \phi
					\end{equation}
					are equivalences.

				\item[If $f$ is $D$-prim] then $f$ is also $D'$-prim, and the natural maps
					\begin{equation} \label{eqn:prim morph}
						\phi f_* \to f_* \phi \quad\text{and}\quad f^\flat \phi \to \phi f^\flat
					\end{equation}
					are equivalences.

			\end{description}

		\item[Descent] Let $\{X_i \xrightarrow{f_i} S\}_i$ be a small family of maps in $E$, and suppose that $D$ takes values in categories that admit small limits and colimits.
			\begin{description}
				
				\item[If $f_i$ is $D$-suave for all $i$] and $\{f_i^*\}_i$ is jointly conservative, then $\{f_i^!\}_i$ is jointly conservative, and if $\{f_i^!\}_i$ is jointly conservative, then $D^!$ has descent along $\{f_i\}_i$.

				\item[If $f_i$ is $D$-prim for all $i$] and $\{f_i^\sharp\}_i$ is jointly conservative, then $\{f_i^!\}_i$ is jointly conservative, and if $\{f_i^!\}_i$ is jointly conservative, then $D^!$ has descent along $\{f_i\}_i$.

			\end{description}

	\end{description}
	\begin{proof}
		First note that since $E$ is stable under taking diagonals, we have that $(\mathcal C,E)$ is a geometric setup in the sense of \cite[Convention 2.1.3]{HM6FF}, which allows us to apply the results about 6-functor formalisms in \textit{loc. cit.}

		The fact that suave and prim maps are stable under base change follows from \cite[Lemma 4.5.9(i)]{HM6FF}.

		The descriptions of $\omega_f$ and $\delta_f$ come from \cite[Lemmas 4.5.6 and 4.5.5]{HM6FF}, and the equivalences \eqref{eqn:suave bc} and \eqref{eqn:prim bc} come from \cite[Lemma 4.5.11]{HM6FF}. By taking left and right adjoints of these equivalences, we deduce the existence of the additional adjoints $f_\flat, f_\sharp, f^\flat, f^\sharp$ and the equivalences \eqref{eqn:suave adj dual} and \eqref{eqn:prim adj dual}.

		The fact that \eqref{eqn:suave PF} and \eqref{eqn:prim PF} are equivalences follows from \Cref{cor:suave prim PF}.

		% PERF: improve argument
		The fact that the maps in \eqref{eqn:suave bc} and \eqref{eqn:prim bc} are equivalences follows immediately from \cite[Lemma 4.5.13]{HM6FF}. In fact, it is possible to adapt the proof of this result to show directly that the maps in \eqref{eqn:suave prim bc} are equivalences, but we can also argue as follows: in the proof of this result, it is shown that when $f$ is $D$-suave, the equivalence $f'_\sharp p^* \to q^* f_\sharp$ is the usual composite
		\[
			f'_!(\omega_{f'} \otimes p^*) \simeq f'_!(p^* \omega_f \otimes p^*) \simeq f'_! p^*(\omega_f \otimes -) \simeq  q^* f_!(\omega_f \otimes -)
		,\]
		and similarly, the equivalence $p_! f'^! \to f^! q_!$ is the usual composite
		\[
			p_!(\omega_{f'} \otimes f'^*) \simeq p_!(p^* \omega_f \otimes f'^*) \simeq \omega_f \otimes p_! f'^* \simeq \omega_f \otimes f^* q_!
		.\]
		Thus, the left and right mate squares
		\[
			\begin{tikzcd}
				D(X) \ar[d, "p^*"'] \ar[r, "f_\sharp"] & D(Y) \ar[d, "q^*"] \\
				D(X') \ar[r, "f'_\sharp"'] & D(Y')
			\end{tikzcd}
			\quad\text{and}\quad
			\begin{tikzcd}
				D(Y') \ar[d, "q_!"'] \ar[r, "f'^!"] & D(X') \ar[d, "p_!"] \\
				D(Y) \ar[r, "f^!"'] & D(X)
			\end{tikzcd}
		\]
		are equivalent to the outer rectangles in
		\[
			\begin{tikzcd}
				D(X) \ar[d, "p^*"'] \ar[r, "\otimes \omega_f"] & D(X) \ar[d, "p^*"'] \ar[r, "f_!"] & D(Y) \ar[d, "q^*"] \\
				D(X') \ar[r, "\otimes p^* \omega_f"'] & D(X') \ar[r, "f'_!"'] & D(Y')
			\end{tikzcd}
			\quad\text{and}\quad
			\begin{tikzcd}
				D(Y') \ar[d, "q_!"'] \ar[r, "f'^*"] & D(X') \ar[d, "p_!"] \ar[r, "\otimes p^* \omega_f"] & D(X') \ar[d, "p_!"] \\
				D(Y) \ar[r, "f^*"'] & D(X) \ar[r, "\otimes \omega_f"'] & D(X)
			\end{tikzcd}
		.\]
		Now, suppose that $q$ is $D$-prim. Then $p$ is also $D$-prim, so \Cref{cor:suave prim PF} shows that the leftmost and rightmost squares are vertically right and left adjointable respectively. Since $q$ is $D$-prim, \cite[Lemma 4.5.13]{HM6FF} shows that the remaining squares are vertically right and left adjointable respectively, so by \cite[Lemma F.6(2) and (3)]{TwAmb}, we find that the maps in \eqref{eqn:suave prim bc} are equivalences.

		The statements about compatibility with morphisms are shown in \Cref{prp:suave prim morph}.

		The descent statements are shown in \Cref{lem:suave prim descent}.
	\end{proof}
\end{thm}

Rather than recalling the definition of suave and prim maps given in \cite[\S4.5]{HM6FF}, we now present our generalized definition which is a bit simpler to state:
\begin{defn} \label{defn:suave prim}
	% NOTE: don't ask for base change against *
	% makes it a bit more inconvenient to show extension results
	Let $(\mathcal C,E)$ be a geometric setup, let $\mathscr V$ be a 2-category, and let $D : \Span(\mathcal C,E) \to \mathscr V$ be a functor. A map $f : X \to Y$ in $\mathcal C$ is said to be \emph{$D$-suave (\resp{} prim) against} $q : Y' \to Y$ in $E$ if
	\[
		\begin{tikzcd}[column sep=huge]
			D(Y') \ar[d, "q_!"'] \ar[r, "(f \times_Y Y')^*"] & D(X \times_Y Y') \ar[d, "(X \times_Y q)_!"] \\
			D(Y) \ar[r, "f^*"'] & D(X)
		\end{tikzcd}
	\]
	is horizontally left (\resp{} right) adjointable.

	If $f : X \to Y$ is $D$-suave (\resp{} prim) against all maps in $E$ to $Y$, we say simply that $f$ is \emph{$D$-suave (\resp{} prim)}.
\end{defn}

\Cref{defn:suave prim} is different from the definition given in \cite[Definition 4.5.1]{HM6FF}, which is stated in terms of the category of kernels for a 3-functor formalism $D$. In fact, by \Cref{lem:suave prim}, our notion is simply a generalization of that definition that does not make reference to monoidal structures, and for which $f$ need not be in $E$, and $\mathscr V$ can be any 2-category. The caveat is that this does not recover the more refined notions of suave and prim \emph{objects} given in \cite[\S4.4]{HM6FF}.
\begin{lem} \label{lem:suave prim}
	Let $D$ be a 3-functor formalism on a geometric setup $(\mathcal C,E)$ in the sense of \cite[Convention 2.1.3]{HM6FF}, and let $f : X \to Y$ be a map in $E$. Then $f$ is $D$-suave (prim) in the sense of \Cref{defn:suave prim} if and only if it is in the sense of \cite[Definition 4.5.1]{HM6FF}.
	\begin{proof}
		We have that \Cref{defn:suave prim} implies \cite[Definition 4.5.1]{HM6FF} by \cite[Lemmas 4.5.6 and 4.5.5]{HM6FF}, and the converse holds by \cite[Lemma 4.5.13]{HM6FF} (and its generalization given in \cite[Remark 4.5.15(i)]{HM6FF}).
	\end{proof}
\end{lem}

The notions given in \Cref{defn:suave prim} enjoy the following extension properties:
\begin{lem} \label{lem:suave prim source-local}
	Let $(\mathcal C,E)$ be a geometric setup, and let $Y : I \star \Delta^1 \to \mathcal C$ be a small diagram sending all edges to $E$. Let $f : X \to Y$ be a Cartesian transformation, and let $D : \Span(\mathcal C,E) \to \PrL$ be a functor such that $D X|_{I \star \{0\}}$ and $D Y|_{I \star \{0\}}$ are colimiting.
	\begin{itemize}
			
		\item Suppose that for all maps $i \to j$ in $I \star \Delta^1$, $f(j)$ is $D$-suave against $Y(i \to j)$ if $i \in I$. Then $f(1)$ is $D$-suave against $Y(0 \to 1)$.

		\item Suppose that for all maps $i \to j$ in $I \star \Delta^1$, $f(j)_*$ admits a right adjoint,\footnotemark{} and $f(j)$ is $D$-prim against $Y(i \to j)$ if $i \in I$. Then $f(1)$ is $D$-prim against $Y(0 \to 1)$.
			\footnotetext{Note that \Cref{rmk:prim upper sharp} shows that the existence of this right adjoint is often automatic.}

	\end{itemize}
	\begin{proof}
		% The ``if'' direction is clear, so we will only address the ``only if'' direction.
		The suave case follows immediately from \cite[Lemma D.0.2]{Fundamentals}. For the prim case, note that it suffices to check right adjointability of the square in \Cref{defn:suave prim} after taking right adjoints everywhere (which we can do since $D$ takes values in $\PrL$). The result then follows from \cite[Corollary 4.7.4.18]{ha}.
	\end{proof}
\end{lem}

\begin{lem} \label{lem:suave prim target-local}
	Let $(\mathcal C,E)$ be a geometric setup, let $\mathscr V$ be a 2-category, and let $D : \Span(\mathcal C,E) \to \mathscr V$ be a functor.

	Let $f : X \to Y$ be a map in $\mathcal C$, and let $q : Y' \to Y$ be a map in $E$. Suppose that there is a $D^*$-pseudocover $\{Y_i \to Y\}_i$ such that for each $i$,
	\begin{itemize}

		\item $D^*$ has left (\resp{} right) base change for $f$ against $Y_i \to Y$, and for $Y' \times_Y f$ against $Y' \times_Y Y_i \to Y'$, and

		\item $f \times_Y Y_i$ is $D$-suave (\resp{} prim) against $q \times_Y Y_i$.

	\end{itemize}
	Then $f$ is $D$-suave (\resp{} prim) against $q$.
	\begin{proof}
		We will only address the suave case, as the argument for the prim case is completely analogous.

		To fix notation, let
		\[
			\begin{tikzcd}
				X' \ar[d, "p"'] \ar[r, "f'"] & Y' \ar[d, "q"] \\
				X \ar[r, "f"'] & Y
			\end{tikzcd}
		\]
		be a Cartesian square in $\mathcal C$, and for each $i$, let
		\[
			\begin{tikzcd}
				X'_i \ar[d, "p_i"'] \ar[r, "f'_i"] & Y'_i \ar[d, "q_i"] \\
				X_i \ar[r, "f_i"'] & Y_i
			\end{tikzcd}
		\]
		be its base change along $Y_i \to Y$.

		Since $D^*$ has left base change for $f'$ against $Y'_i \to Y'$, and $f_i$ is $D$-suave against $q_i$, we know that the top and bottom squares in
		\[
			\begin{tikzcd}
				D(Y') \ar[d, "(Y'_i \to Y')^*"'] \ar[r, "f'^*"] & D(X') \ar[d, "(X'_i \to X')^*"] \\
				D(Y'_i) \ar[d, "(q_i)_!"'] \ar[r, "f_i'^*"] & D(X'_i) \ar[d, "(p_i)_!"] \\
				D(Y_i) \ar[r, "f_i^*"'] & D(X_i)
			\end{tikzcd}
		\]
		are horizontally left adjointable, so the outer rectangle is horizontally left adjointable by \cite[Lemma F.6(3)]{TwAmb}.

		This is equivalent to the outer rectangle in the following diagram
		\[
			\begin{tikzcd}
				D(Y') \ar[d, "q_!"'] \ar[r, "f'^*"] & D(X') \ar[d, "p_!"] \\
				D(Y) \ar[d, "(Y_i \to Y)^*"'] \ar[r, "f^*"] & D(X) \ar[d, "(X_i \to X)^*"] \\
				D(Y_i) \ar[r, "f_i^*"'] & D(X_i)
			\end{tikzcd}
		.\]
		Since $D^*$ has left base change for $f$ against $Y_i \to Y$, we know that the bottom square is horizontally left adjointable, so by \cite[Lemma F.6(3)]{TwAmb}, we know that $(Y_i \to Y)^*$ sends the horizontal left mate of the top square to an equivalence. Since $\{Y_i \to Y\}_i$ is a $D^*$-pseudocover, we conclude that the top square is horizontally left adjointable, as desired.
	\end{proof}
\end{lem}

The rest of the section will only be concerned with suave and prim maps that are covered by \cite[Definition 4.5.1]{HM6FF}. We fix a 3-functor formalism $D$ on a geometric setup $(\mathcal C,E)$, where $(\mathcal C,E)$ is a geometric setup in the sense of \cite[Convention 2.1.3]{HM6FF}, \ie, $E$ is a collection of maps in the category $\mathcal C$ that contains all equivalences, and is stable under composition, base change, and diagonals.

\begin{rmk} \label{rmk:prim upper sharp}
	Suppose that $D$ is a 6-functor formalism. Then for any $D$-prim map $f \in E$, \cite[Corollary 4.5.11(ii)]{HM6FF} shows that $f_*$ admits a right adjoint $f^\sharp$ given by $\underline\Hom(\delta_f, f^!)$. By \cite[Lemma 4.5.9(i)]{HM6FF}, we have that the $D$-prim maps in $E$ are stable under composition, so we can define a wide subcategory $E^\sharp_D$ of $\mathcal C$ whose morphisms are the $D$-prim maps in $E$. Thus, there is a presheaf $D^\sharp : (E^\sharp_D)^\op \to \PrR$ sending a $D$-prim map $f \in E$ to the right adjoint $f^\sharp$ of $f_*$.

	This is also mentioned in \Cref{thm:suave prim consequences}.
\end{rmk}

\begin{lem} \label{lem:suave prim descent}
	Suppose that $D$ is a 6-functor formalism.

	\begin{description}

		\item[Suave descent] If $D$ takes values in categories that admit small limits, then $D^!$ has descent along any small $D^!$-pseudocover consisting of $D$-suave maps in $E$. Furthermore, every $D^*$-pseudocover is a $D^!$-pseudocover if it consists of $D$-suave maps in $E$.

		\item[Prim descent] If $D$ takes values in categories that admit small colimits, then $D^!$ has descent along any small $D^!$-pseudocover consisting of $D$-prim maps in $E$. Furthermore, any $D^\sharp$-pseudocover is a $D^!$-pseudocover, where $D^\sharp$ is the presheaf described in \Cref{rmk:prim upper sharp}.
		
	\end{description}
	\begin{proof}
		Note that by \cite[Corollary 4.5.11(i)]{HM6FF}, we have that for any $D$-suave map $f \in E$, the functor $f^!$ has a right adjoint given by $f_* \underline\Hom(\omega_f, -)$. Thus, \cite[Lemma 4.5.13(i)]{HM6FF} shows that $D^!$ has right base change for any $D$-suave map. We also know that $D$-suave maps in $E$ are stable under base change by \cite[Lemma 4.5.9(i)]{HM6FF}, so by \Cref{thm:bc descent}, $D^!$ has descent for any $D^!$-pseudocover consisting of $D$-suave maps in $E$.

		Furthermore, if $f \in E$ is $D$-suave, then \cite[Corollary 4.5.11(i)]{HM6FF} also shows that $f^*$ extends along $f^!$, so any $D^*$-pseudocover consisting of $D$-suave maps is also a $D^!$-pseudocover.

		The statement for prim descent is proved similarly. In particular, we have that $D^!$ has left base change for $D$-prim maps by \cite[Lemma 4.5.13(ii)]{HM6FF}, and that $D$-prim maps in $E$ are stable under base change by \cite[Lemma 4.5.9(i)]{HM6FF}, so \Cref{thm:bc descent} shows that $D^!$ has descent along any $D^!$-pseudocover consisting of $D$-prim maps in $E$. As before, we also have that for any $D$-prim map $f \in E$, the functor $f^\sharp$ extends along $f^!$, so any $D^\sharp$-pseudocover is a $D^!$-pseudocover.
	\end{proof}
\end{lem}

The following result and its proof are adapted from \cite[Corollary 4.6 and Theorem 4.8]{LCA6FF}.
\begin{prp} \label{prp:suave prim morph}
	Let $\sigma : D \to D'$ be a transformation of 3-functor formalisms $D,D' : \Span(\mathcal C, E) \to \widehat{\Cat}$. Let $f : X \to Y$ be a map in $E$.
	\begin{enumerate}

		\item The functor $\sigma : D(X) \to D'(X)$ preserves $f$-suave and $f$-prim objects, as well as $f$-suave duals and $f$-prim duals.

		\item\label{itm:suave prim morph}
			If $f$ is $D$-suave (\resp{} $D$-prim), then it is also $D'$-suave (\resp{} $D'$-prim), the square
			\begin{equation} \label{eqn:suave prim ! morph}
				\begin{tikzcd}
					D(X) \ar[d, "\sigma"'] \ar[r, "f_!"] & D(Y) \ar[d, "\sigma"] \\
					D'(X) \ar[r, "f_!"'] & D'(Y)
				\end{tikzcd}
			\end{equation}
			is horizontally right (\resp{} left) adjointable, and the square
			\begin{equation} \label{eqn:suave prim * morph}
				\begin{tikzcd}
					D(Y) \ar[d, "\sigma"'] \ar[r, "f^*"] & D(X) \ar[d, "\sigma"] \\
					D'(Y) \ar[r, "f^*"'] & D'(X)
				\end{tikzcd}.
			\end{equation}
			is horizontally left (\resp{} right) adjointable. Furthermore, the horizontal right and left (\resp{} left and right) mate squares of these are the canonical commutative squares
			\[\begin{tabular}{@{} llcrr @{}}
				& \begin{tikzcd}
					D(Y) \ar[d, "\sigma"'] \ar[r, "\omega_f \otimes f^*"] & D(X) \ar[d, "\sigma"] \\
					D'(Y) \ar[r, "\omega_f \otimes f^*"'] & D'(X)
				\end{tikzcd} & \text{and} &
				\begin{tikzcd}
					D(X) \ar[d, "\sigma"'] \ar[r, "f_!(\omega_f \otimes -)"] & D(Y) \ar[d, "\sigma"] \\
					D'(X) \ar[r, "f_!(\omega_f \otimes -)"'] & D'(Y)
				\end{tikzcd} & \\
				\bigg(\resp{} & \begin{tikzcd}
					D(Y) \ar[d, "\sigma"'] \ar[r, "\delta_f \otimes f^*"] & D(X) \ar[d, "\sigma"] \\
					D'(Y) \ar[r, "\delta_f \otimes f^*"'] & D'(X)
				\end{tikzcd} & \text{and} &
				\begin{tikzcd}
					D(X) \ar[d, "\sigma"'] \ar[r, "f_!(\delta_f \otimes -)"] & D(Y) \ar[d, "\sigma"] \\
					D'(X) \ar[r, "f_!(\delta_f \otimes -)"'] & D'(Y)
			\end{tikzcd} & \bigg)
			\end{tabular}\]

	\end{enumerate}
\end{prp}

Before addressing the proof of \Cref{prp:suave prim morph}, we note the following corollary:
\begin{cor}[Projection formula] \label{cor:suave prim PF}
	Let $f : X \to Y$ be a map in $E$, and let $M \in D(Y)$.

	If $f$ is $D$-suave (\resp{} $D$-prim), then the square
	\begin{equation} \label{eqn:suave prim ! PF}
		\begin{tikzcd}
			D(X) \ar[d, "\otimes f^* M"'] \ar[r, "f_!"] & D(Y) \ar[d, "\otimes M"] \\
			D(X) \ar[r, "f_!"'] & D(Y)
		\end{tikzcd}
	\end{equation}
	is horizontally right (\resp{} left) adjointable, and the square
	\begin{equation} \label{eqn:suave prim * PF}
		\begin{tikzcd}
			D(Y) \ar[d, "\otimes M"'] \ar[r, "f^*"] & D(X) \ar[d, "\otimes f^* M"] \\
			D(Y) \ar[r, "f^*"'] & D(X)
		\end{tikzcd}.
	\end{equation}
	is horizontally left (\resp{} right) adjointable. Furthermore, the horizontal right and left (\resp{} left and right) mate squares of these are the canonical commutative squares
	\[\begin{tabular}{@{} llcrr @{}}
		& \begin{tikzcd}
			D(Y) \ar[d, "\otimes M"'] \ar[r, "\omega_f \otimes f^*"] & D(X) \ar[d, "\otimes f^* M"] \\
			D(Y) \ar[r, "\omega_f \otimes f^*"'] & D(X)
		\end{tikzcd} & \text{and} &
		\begin{tikzcd}
			D(X) \ar[d, "\otimes f^* M"'] \ar[r, "f_!(\omega_f \otimes -)"] & D(Y) \ar[d, "\otimes M"] \\
			D'(X) \ar[r, "f_!(\omega_f \otimes -)"'] & D'(Y)
		\end{tikzcd} & \\
		\bigg(\resp{} & \begin{tikzcd}
				D(Y) \ar[d, "\otimes M"'] \ar[r, "\delta_f \otimes f^*"] & D(X) \ar[d, "\otimes f^* M"] \\
				D'(Y) \ar[r, "\delta_f \otimes f^*"'] & D'(X)
			\end{tikzcd} & \text{and} &
			\begin{tikzcd}
				D(X) \ar[d, "\otimes f^* M"'] \ar[r, "f_!(\delta_f \otimes -)"] & D(Y) \ar[d, "\otimes M"] \\
				D'(X) \ar[r, "f_!(\delta_f \otimes -)"'] & D'(Y)
		\end{tikzcd} & \bigg)
	\end{tabular}\]
	\begin{proof}
		By \cite[Lemma 3.1.9]{HM6FF}, the precomposition $D_Y$ of $D$ along $\Span(E_{/Y}) \to \Span(\mathcal C,E)$ enhances to a lax symmetric monoidal functor $\Span(E_{/Y}) \to \LMod_{D(Y)}(\widehat{\Cat})$, so that for any $M \in D(Y)$, (using the fact that $D(Y)$ is actually \emph{symmetric} monoidal), the operation $\otimes M$ defines a endomorphism of the Cartesian monoidal forgetful functor $\LMod_{D(Y)} \widehat{\Cat} \to \widehat{\Cat}$ which induces an endomorphism $\otimes M : D_Y \to D_Y$ of the 3-functor formalism $D_Y$. The squares \eqref{eqn:suave prim ! PF} and \eqref{eqn:suave prim * PF} are the naturality squares of $\otimes M$ for the map $f : X \to Y$ in $E_{/Y}$. The result then follows from \Cref{prp:suave prim morph}.
	\end{proof}
\end{cor}

The following auxiliary result will be necessary for the proof of \Cref{prp:suave prim morph}:
\begin{lem} \label{lem:kernel morphism}
	Let $\alpha : D \to D'$ be a morphism of 3-functor formalisms, and for $S \in \mathcal C$ write $\mathcal K_{D,S}$ for the category of kernels constructed in \cite[Definition 4.1.3]{HM6FF}. Then there is a 2-functor $\Psi_{\alpha, D} : \mathcal K_{D,S} \to \ho_2 \Fun_2(\Delta^1, \widehat{\Cat})$, where $\ho_2$ is the operation of taking homotopy $(2,2)$-categories. This 2-functor sends $M \in \mathcal K_{D,S}(X,Y) = D(X \times_S Y)^\simeq$ to the commutative square
	\begin{equation} \label{eqn:kernel square}
		\begin{tikzcd}[column sep=huge]
			D(X) \ar[d, "\alpha(X)"'] \ar[r, "(\pi_Y)_!(M \otimes \pi_X^*)"] & D(Y) \ar[d, "\alpha(Y)"] \\
			D'(X) \ar[r, "(\pi_Y)_!(M \otimes \pi_X^*)"'] & D'(Y)
		\end{tikzcd},
	\end{equation}
	where $\pi_X, \pi_Y$ are the projections from $X \times_S Y$ to $X$ and $Y$.
	% TODO: summarize key points of construction given in Li He's paper
	\begin{proof}
		We recall the following description of the homotopy $(2,2)$-category $K_{D,S}$ of $\mathcal K_{D,S}$, described after \cite[Definition 4.1.3]{HM6FF} and given in \cite[Definition 2.2.3]{zavyalov6FF}, \cite[Remark 4.1]{LCA6FF}, and \cite[\S5]{Scholze6FF}:
		\begin{enumerate}

			\item The objects are the objects of $E_{/S}$.

			\item For every pair of objects $X,Y \in E_{/S}$, the category of morphisms $X \to Y$ in $K_{D,S}$ is given by $D(X \times_S Y)$.

			\item Given objects $X_1, X_2, X_3 \in E_{/S}$, the composition functor $D(X_2 \times_S X_3) \times D(X_1 \times_S X_2) \to D(X_1 \times_S X_3)$ is given by
				\[
					(B,A) \mapsto (\pi_{13})_!(\pi_{12}^* A \otimes \pi_{23}^* B)
				,\]
				where for each $i,j \in \{1,2,3\}$, $\pi_{ij} : X_1 \times_S X_2 \times_S X_3 \to X_i \times_S X_j$ is the projection.

			\item For any $X \in E_{/S}$, the identity of $X$ is $\Delta_!(1)$, where $\Delta : X \to X \times_S X$ is the diagonal.

		\end{enumerate}
		Recall the definitions of $\Psi_{D,S} : \mathcal K_{D,S} \to \widehat{\Cat}$ and $\Psi_{D',S} : \mathcal K_{D',S} \to \widehat{\Cat}$ from \cite[Proposition 4.1.5]{HM6FF}, and of $\phi_\alpha : \mathcal K_{D,S} \to \mathcal K_{D',S}$ from \cite[Proposition 4.2.1(i)]{HM6FF}.

		The argument of \cite[Theorem 4.8]{LCA6FF} constructs a morphism $\ho_2 \Psi_{D,S} \to \ho_2(\Psi_{D',S} \circ \phi_\alpha)$ in $\Fun_2(K_{D,S}, \ho_2 \widehat{\Cat})$ such that for any $X,Y \in E_{/S}$, and $M \in D(X \times_S Y)$, the naturality square
		\[
			\begin{tikzcd}[column sep=huge]
				\Psi_{D,S}(X) \ar[d] \ar[r, "\Psi_{D,S}(M)"] & \Psi_{D,S}(Y) \ar[d] \\
				\Psi_{D',S}(\phi_\alpha X) \ar[r, "\Psi_{D',S}(\phi_\alpha(M))"'] & \Psi_{D',S}(\phi_\alpha(Y))
			\end{tikzcd}
		\]
		in $\ho_2 \widehat{\Cat}$ is equivalent to the usual commutative square \eqref{eqn:kernel square}.

		This morphism $\ho_2 \Psi_{D,S} \to \ho_2(\Psi_{D',S} \circ \phi_\alpha)$ corresponds to a 2-functor $K_{D,S} \times \Delta^1 \to \ho_2 \widehat{\Cat}$, which corresponds to a 2-functor $K_{D,S} \to \Fun_2(\Delta^1, \ho_2 \widehat{\Cat})$, so we obtain $\Psi_{\alpha, S}$ as the composite
		\[
			\mathcal K_{D,S} \to \ho_2(\mathcal K_{D,S}) = K_{D,S} \to \Fun_2(\Delta^1, \ho_2 \widehat{\Cat}) \simeq \ho_2 \Fun_2(\Delta^1, \widehat{\Cat})
		.\]
	\end{proof}
\end{lem}

\begin{rmk}
	It should be possible to give a stronger $(\infty,2)$-categorical (instead of $(2,2)$-categorical) version of \Cref{lem:kernel morphism}, but since the version we have given is enough for the purposes of this paper, we have chosen to settle for this version. Nevertheless, we will now explain the proposed enhancement.

	Suppose for simplicity that $E$ contains all maps in $\mathcal C$. For any closed monoidal category $\mathscr V$, and diagram $\mathscr D : I \to \Alg_{\Span(\mathcal C)}(\mathscr V)$ of $\mathscr V$-valued 3-functor formalisms, we can view $\mathscr D$ as a lax symmetric monoidal functor $\tilde{\mathscr D} : \Span(\mathcal C) \to \Fun(I, \mathscr V)$. Using the closed monoidal structure on $\Span(\mathcal C)$ given \cite[Proposition 2.4.1]{HM6FF}, we may then use \cite[Proposition C.3.9]{HM6FF} to find that $\tilde{\mathscr D}$ decomposes as a composite
	\[
		\Span(\mathcal C) \xrightarrow{F_{\tilde{\mathscr D}}} \underline{\tau_{\tilde{\mathscr D}} \Span(\mathcal C)} \xrightarrow{\underline{G_{\tilde{\mathscr D}}}} \underline{\Fun(I, \mathscr V)}
	,\]
	where $\underline{(-)}$ denotes the operation of taking underlying categories of enriched categories, $\tau_{-}$ is the transfer of enrichment functor of \cite[Definition C.3.1]{HM6FF}, and in the second and third positions, we use the closed monoidal structures to view $\Span(\mathcal C)$ and $\Fun(I, \mathscr V)$ as self-enriched categories. The functor $F_{\tilde{\mathscr D}}$ comes from \cite[Lemma C.3.6]{HM6FF}, and the $\Fun(I, \mathscr V)$-enriched functor $G_{\tilde{\mathscr D}}$ comes from \cite[Lemma C.3.7]{HM6FF}.

	If the limit functor $\varprojlim_I : \Fun(I, \mathscr V) \to \mathscr V$ is symmetric monoidal, such as if the monoidal structure on $\mathscr V$ is Cartesian, then using \cite[Lemma C.3.5]{HM6FF}, we have that for $\Fun(I, \mathscr V)$-enriched categories, the operation of taking underlying categories is given by first applying the transfer of enrichment $\tau_{\varprojlim_I}$ along $\varprojlim_I$, and then taking the underlying category of the resulting $\mathscr V$-enriched category.

	Thus, since
	\[
		\varprojlim_I \mathscr D \simeq \varprojlim_I \circ \tilde{\mathscr D}
	,\]
	we find that $\tilde{\mathscr D}$ is equivalent to the following composite:
	\[
		\Span(\mathcal C) \xrightarrow{\Phi_{\mathscr D}} \underline{\mathcal K_{\varprojlim_{\mathscr D}}} \xrightarrow{\underline{\Psi_{\mathscr D}}} \underline{\tau_{\varprojlim_I} \Fun(I, \mathscr V)}
	,\]
	where
	\[
		\mathcal K_{\varprojlim_{\mathscr D}} = \tau_{\varprojlim \mathscr D} \Span(\mathcal C)
	,\]
	and $\Phi_{\varprojlim \mathscr D}$ admits a description like the one given in \cite[Proposition 4.1.5(i)]{HM6FF}, and $\Psi_{\mathscr D}$ is a $\mathscr V$-enriched functor analogous to the one given in \cite[Proposition 4.1.5(ii)]{HM6FF}, so for $M \in \mathcal K_{\varprojlim \mathscr D}(X,Y) \simeq \underline{(\varprojlim \mathscr D)(X \times Y)}$ and $i \in I$, we have
	\[
		\Psi_{\mathscr D}(M)(i) = (\pi_1)_!(M_i \otimes \pi_2^*) : \mathscr D(i)(X) \to \mathscr D(i)(Y)
	,\]
	where $\pi_1, \pi_2$ are the projections on $X \times Y$, and $M_i$ is the image of $M$ in $\underline{\mathscr D(i)(X \times Y)}$.

	Applying this in the case $I = \Delta^1$, and $\mathscr V = \widehat{\Cat}$, we recover the $(\infty,2)$-categorical version of \Cref{lem:kernel morphism}, up to identifying the $\widehat{\Cat}$-enriched category $\tau_{\varprojlim_I} \Fun(I, \widehat{\Cat})$ with the usual 2-category of functors $\Fun_2(I, \widehat{\Cat})$. This follows from \cite[Corollary 3.71(i)]{bienriched}, which says that $\tau_{\varprojlim_I} \Fun(I, \widehat{\Cat})$ is equivalent to the $\widehat{\Cat}$-enriched structure induced from the $\widehat{\Cat}$-linear structure given by restricting scalars along the left adjoint $\widehat{\Cat} \to \Fun(I, \widehat{\Cat})$ of $\varprojlim_I$.
\end{rmk}

\begin{proof}[Proof of \Cref{prp:suave prim morph}]
	The first statement follows immediately from the existence of the 2-functor $\mathcal K_{D,Y} \to \mathcal K_{D',Y}$ of \cite[Proposition 4.2.1(i)]{HM6FF} that acts as the identity on objects, and on the morphism categories $\Fun_{\mathcal K_{D,Y}}(U,V) \to \Fun_{\mathcal K_{D',Y}}(U,V)$ as $\sigma : D(U \times_X V) \to D'(U \times_X V)$.

	In particular, since $\sigma : D(X) \to D'(X)$ preserves monoidal units, it follows that if $f$ is $D$-suave (\resp{} $D$-prim), then it is also $D'$-suave (\resp{} $D'$-prim).

	Now, let $\Psi_{D,Y} : \mathcal K_{D,Y} \to \ho_2 \Fun_2(\Delta^1, \widehat{\Cat})$ be the 2-functor given in \Cref{lem:kernel morphism}. By viewing $1 \in D(X)$ as a morphism $X \to Y$ or $Y \to X$ in $\mathcal K_{D,Y}$, we obtain the commutative squares \eqref{eqn:suave prim ! morph} and \eqref{eqn:suave prim * morph}.

	If $f$ is $D$-suave, then as a morphism $X \to Y$, $1 \in D(X) \simeq D(X \times_Y Y)$ is a left adjoint, so since $\Psi_{D,Y}$ is a 2-functor, the square \eqref{eqn:suave prim ! morph} is a left adjoint in the $(2,2)$-category $\ho_2 \Fun_2(\Delta^1, \widehat{\Cat})$, so it is also a left adjoint in the 2-category $\Fun_2(\Delta^1, \widehat{\Cat}$. By the description of adjunctions in $\Fun(\Delta^1, \mathscr V)$ given in the proof of \cite[Proposition 2.1.5]{bispans}, we find that since $\omega_f : Y \to X$ is the right adjoint of $1 : X \to Y$, this square must be horizontally right adjointable, and we obtain the desired description of the horizontal right mate using \Cref{lem:kernel morphism} again (since $\Psi_{D,Y}$ preserves adjunctions).

	Using \cite[Proposition 4.1.4]{HM6FF}, we have that as a morphism $Y \to X$, $1 \in D(X) \simeq D(Y \times_Y X)$ is a right adjoint (when $f$ is $D$-suave), so the above argument shows that the square \eqref{eqn:suave prim * morph} is horizontally left adjointable, and the horizontal left mate admits the desired description.

	Similarly, when $f$ is $D$-prim, the same arguments show that \eqref{eqn:suave prim ! morph} is horizontally left adjointable, and \eqref{eqn:suave prim * morph} is horizontally right adjointable, and that the mate squares admit the desired descriptions.
\end{proof}

\subsection{Nagata 6-Functor Formalisms} \label{S:nagata}

Fix collections $I,P,E$ of maps in $\mathcal C$.
\begin{ass}
	% NOTE: see CLL 4.12
	% need $E$ to also be stable under diagonals for the corresponding wide subcategory to be closed under pullbacks
	The collections $I,P,E$ are stable under base change, composition, and taking diagonals, and $I,P$ are contained in $E$.
\end{ass}

% We also view $I,P,E$ as wide subcategories of $\mathcal C$.

One of the key results about 6-functor formalisms was given in \cite[Proposition A.5.10]{Mann6FF}, which gives criteria for \emph{constructing} 6-functor formalisms out of more accessible data. This result was later refined in \cite{6FFUnique} under some additional hypotheses, clarifying the role of the suave and prim properties in this construction. Another refinement is given in \cite{CLL6FF} without additional hypotheses, but with a new 2-categorical version of 6-functor formalisms, which are given by lax symmetric monoidal \emph{2-functors} from $\Span_2(\mathcal C,E)_{P,I}$ (constructed in \cite[Construction 4.12]{CLL6FF}), which is a 2-categorical enhancement of $\Span(\mathcal C,E)$ by \cite[Lemma 4.13]{CLL6FF}.

In this section we will be interested in studying this 2-categorical notion of 6-functor formalisms.

Rather than describing the 2-category $\Span_2(\mathcal C,E)_{P,I}$ (for which we refer the reader to \cite[Construction 4.12]{CLL6FF}), we recall some of its key properties. We have already mentioned that the underlying category of $\Span_2(\mathcal C,E)_{P,I}$ is $\Span(\mathcal C,E)$. Furthermore, by \cite[Proposition 4.14]{CLL6FF}, the composite functor
\[
	\mathcal C^\op \to \Span(\mathcal C,E) \to \Span_2(\mathcal C,E)_{P,I}
\]
is left adjointable at maps in $I$, right adjointable at maps in $P$, and if
\[
	\begin{tikzcd}
		X' \ar[d] \ar[r] & Y' \ar[d, "q"] \\
		X \ar[r, "f"'] & Y
	\end{tikzcd}
\]
is a Cartesian square in $\mathcal C$ such that $f \in P$ and $q \in I$, then this square is sent to a horizontally right-left adjointable square in $\Span_2(\mathcal C,E)_{P,I}$. Functors from $\mathcal C^\op$ to 2-categories satisfying these properties are called \emph{$(I,P)$-biadjointable}.

In fact, \cite[Theorems A and B]{CLL6FF} show that under certain hypotheses, $\Span_2(\mathcal C,E)_{P,I}$ is universal among 2-categories equipped with an $(I,P)$-biadjointable functor from $\mathcal C^\op$.

Using the fact that the underlying category of $\Span_2(\mathcal C,E)_{P,I}$ is given by $\Span(\mathcal C,E)$, we have a description of the 0-cells and 1-cells in $\Span_2(\mathcal C,E)_{P,I}$. The 2-cells are also described, for example, in \cite[Construction 4.12]{CLL6FF}: given $X,Y \in \mathcal C$, and spans $X \gets Z \to Y$ to $X \gets Z' \to Y$ in $\mathcal C$ corresponding to $1$-morphisms $\alpha, \beta : X \to Y$ in $\Span(\mathcal C,E)$, a 2-morphism $\alpha \to \beta$ in $\Span_2(\mathcal C,E)_{P,I}$ is given by a commutative diagram
\[
	\begin{tikzcd}
		& \ar[dl] Z \ar[dr] & \\
		X & \ar[u] Z'' \ar[d] & Y \\
		& \ar[ul] Z' \ar[ur] &
	\end{tikzcd}
,\]
where the map $Z'' \to Z$ is in $P$, and the map $Z'' \to Z'$ is in $I$.
\begin{rmk} \label{rmk:Span2 BC}
	Given a commutative square
	\[
		\begin{tikzcd}
			X' \ar[d, "p"'] \ar[r, "f'"] & Y' \ar[d, "q"] \\
			X \ar[r, "f"'] & Y
		\end{tikzcd}
	\]
	in $\mathcal C$, if $X' \to X \times_Y X$ is in $I$, and $f \in E$, then the following diagram
	\[
		\begin{tikzcd}
			& \ar[dl, "p"'] X' \ar[dr, "f'"] & \\
			X & \ar[u, equals] X' \ar[d] & Y' \\
			& \ar[ul] X \times_Y Y' \ar[ur] &
		\end{tikzcd}
	\]
	defines a 2-morphism
	\[
		(\xleftarrow{p}= \circ =\xrightarrow{f'}) \Rightarrow (=\xrightarrow{f} \circ \xleftarrow{q}=)
	,\]
	which gives a colax square
	\[
		\begin{tikzcd}
			X \ar[d, "\xleftarrow{p} = "'] \ar[r, "=\xrightarrow{f}"] & Y \ar[d, "\xleftarrow{q} ="] \\
			X' \ar[r, "=\xrightarrow{f'}"'] & Y'
		\end{tikzcd}
	\]
	in $\Span_2(\mathcal C, E)_{P,I}$.

	In fact, this is given by the horizontal left mate of the square
	\[
		\begin{tikzcd}
			Y \ar[d, "\xleftarrow{q}="'] \ar[r, "\xleftarrow{f}="] & X \ar[d, "\xleftarrow{p}="] \\
			Y' \ar[r, "\xleftarrow{f'}="'] & X'
		\end{tikzcd}
	.\]
	\begin{proof}
		As in the proof of \cite[Proposition 4.14]{CLL6FF}, if $f \in I$, then $=\xrightarrow{f}$ has a right adjoint given by $\xleftarrow{f}=$, where the unit and counit are given by
		\[
			\begin{tikzcd}
				& \ar[dl, equals] X \ar[dr, equals] & \\
				X & \ar[u, equals] X \ar[d] & X \\
				& \ar[ul] X \times_Y X \ar[ur] &
			\end{tikzcd}
			\quad\text{and}\quad
			\begin{tikzcd}
				& \ar[dl] X \ar[dr] & \\
				Y & \ar[u, equals] X \ar[d] & Y \\
				& \ar[ul, equals] Y \ar[ur, equals] &
			\end{tikzcd}.
		\]
		Thus, the horizontal left mate is given by the following composite of 2-morphisms, to be read from top to bottom. Also see the proof of \cite[Proposition 4.14]{CLL6FF}.
		\[
			\begin{tikzcd}
				& \ar[dl] X' \ar[dr] & \\
				X & \ar[u, "\sim"] X \times_X X' \ar[d] & Y' \\
				& \ar[ul] X \times_Y X \times_X X' \ar[d, no head, "\sim"] \ar[ur] & \\
				& \ar[dl] X \times_Y Y' \times_{Y'} X' \ar[dr] & \\
				X & \ar[u, "\sim"] X \times_Y X' \ar[d] & Y \\
				& \ar[ul] X \times_Y Y' \ar[ur] &
			\end{tikzcd}
		,\]
		which is equivalent to the following 2-morphism:
		\[
			\begin{tikzcd}
				& \ar[dl, "p"'] X' \ar[dr, "f'"] & \\
				X & \ar[u, equals] X' \ar[d] & Y' \\
				& \ar[ul] X \times_Y Y' \ar[ur] &
			\end{tikzcd}
		.\]

	\end{proof}

\end{rmk}

\begin{prp} \label{prp:nat sq adj 6FF}
	Let $\mathscr V$ be any 2-category, and let $D$ be a 2-functor $\Span_2(\mathcal C, E)_{P,I} \to \mathscr V$.

	\begin{enumerate}

		\item Let $\phi : D \to D'$ be a transformation. For any $f : X \to Y$ in $E$, we have naturality squares
			\begin{equation} \label{eqn:nat sq morph}
				\begin{tikzcd}
					D(X) \ar[d] \ar[r, "f_!"] & D(Y) \ar[d] \\
					D'(X) \ar[r, "f_!"'] & D'(Y)
				\end{tikzcd} \qquad
				\begin{tikzcd}
					D(Y) \ar[d] \ar[r, "f^*"] & D(X) \ar[d] \\
					D'(Y) \ar[r, "f^*"'] & D'(X)
				\end{tikzcd}
			.\end{equation}
			If $f \in I$ (\resp{} $P$), then the first square is given by the horizontal left (\resp{} right) mate of the second square.

		\item Let
			\begin{equation} \label{eqn:cart sq bc 6FF}
				\begin{tikzcd}
					X' \ar[d, "a"'] \ar[r, "f'"] & Y' \ar[d, "b" ] \\
					X \ar[r, "f"'] & Y
				\end{tikzcd}
			\end{equation}
			be a Cartesian square in $\mathcal C$, so we have a commutative square
			\begin{equation} \label{eqn:* bc 6FF}
				\begin{tikzcd}
					D(Y) \ar[d, "a^*"'] \ar[r, "f^*"] & D(X) \ar[d, "b^*"] \\
					D(Y') \ar[r, "f'^*"'] & D(X')
				\end{tikzcd}
			.\end{equation}
			If $f \in I$ (\resp{} $b \in P$), then the horizontal left (\resp{} vertical right) mate square of this square is given by the canonical square
			\[
				\begin{tikzcd}
					D(X) \ar[d, "a^*"'] \ar[r, "f_!"] & D(Y) \ar[d, "b^*"] \\
					D(X') \ar[r, "f'_!"'] & D(Y')
				\end{tikzcd}
				\left(\resp{}
				\begin{tikzcd}
					D(Y') \ar[d, "a_!"'] \ar[r, "f'^*"] & D(X') \ar[d, "b_!"] \\
					D(Y) \ar[r, "f^*"'] & D(X)
				\end{tikzcd}\right)
			.\]
			If $f \in I$ and $b \in P$, then the horizontal left-right mate square (equivalently, the vertical right-left mate square) is given by the canonical square
			\[
				\begin{tikzcd}
					D(X') \ar[d, "a_!"'] \ar[r, "f'_!"] & D(Y') \ar[d, "b_!"] \\
					D(X) \ar[r, "f_!"'] & D(Y)
				\end{tikzcd}
			.\]

		\item\label{itm:nat sq adj/suave prim}
			Let \eqref{eqn:cart sq bc 6FF} be a Cartesian square in $\mathcal C$ as in the previous point, and assume that $f \in E$, so we have a canonical commutative square
			\begin{equation} \label{eqn:! bc 6FF}
				\begin{tikzcd}
					D(X) \ar[d, "a^*"'] \ar[r, "f_!"] & D(Y) \ar[d, "b^*"] \\
					D(X') \ar[r, "f'_!"'] & D(Y')
				\end{tikzcd}.
			\end{equation}
			\begin{enumerate}

				\item Let $P^\natural \subseteq P$ be a collection of maps that is stable under base change, and such that $D^*$ has right-left base change for maps in $P^\natural$ against base changes of $b$. Then \eqref{eqn:! bc 6FF} is vertically right adjointable if $f$ is a composite of maps in $I \cup P^\natural$, and $b'^*$ has a left adjoint for any base change $b'$ of $b$ along such maps.

					In particular, every map in $I$ is $D|_{\Span(\mathcal C,E)}$-suave against composites of maps in $I \cup P$.

				\item Similarly, let $I^\natural \subseteq I$ be a collection of maps that is stable under base change, and such that $D^*$ has left-right base change for maps in $I^\natural$ against base changes of $b$. Then \eqref{eqn:! bc 6FF} is vertically right adjointable if $f$ is a composite of maps in $I^\natural \cup P$, and $b'^*$ has a right adjoint for any base change $b'$ of $b$ along such maps.

					In particular, every map in $P$ is $D|_{\Span(\mathcal C,E)}$-prim against composites of maps in $I \cup P$.

			\end{enumerate}

	\end{enumerate}

	\begin{proof}
		First note that by the proof of \cite[Proposition 4.14]{CLL6FF}, the functor $E \to \Span_2(\mathcal C,E)_{P,I}$ sends $f \in I$ (\resp{} $P$) to the left (\resp{} right) adjoint of the image of $f$ under $\mathcal C^\op \to \Span_2(\mathcal C,E)_{P,I}$.

		It follows that for any 2-category $\mathscr U$, and 2-functor $\Span_2(\mathcal C,E)_{P,I} \to \mathscr U$, we have that if $f \in I$, then $f_! \dashv f^*$, and if $f \in P$, then $f^* \dashv f_!$.

		\begin{enumerate}

			\item The transformation $\phi$ can be seen as a 2-functor $\Span_2(\mathcal C,E)_{P,I} \to \Fun(\Delta^1, \mathscr V)$. Thus, the statement follows from the description of adjunctions in $\Fun(\Delta^1, \mathscr V)$ given in the proof of \cite[Proposition 2.1.5]{bispans}.

			\item Since \eqref{eqn:cart sq bc 6FF} is a Cartesian square in $\mathcal C$, the argument of \Cref{rmk:Span2 BC} shows that the square 
				\[
					\begin{tikzcd}
						Y \ar[d, "\xleftarrow{b}="'] \ar[r, "\xleftarrow{f}="] & X \ar[d, "\xleftarrow{a}="] \\
						Y' \ar[r, "\xleftarrow{f'}="'] & X'
					\end{tikzcd}
				\]
				in $\Span_2(\mathcal C,E)_{P,I}$ is horizontally left adjointable if $f \in I$, vertically right adjointable if $b \in P$, and horizontally left-right adjointable if $f \in I$ and $b \in P$, and that furthermore, the corresponding horizontal left mate, vertical right mate, and horizontal left-right mates are given by the canonical squares
				\[
					\begin{tikzcd}
						X \ar[d, "\xleftarrow{a}="'] \ar[r, "=\xrightarrow{f}"] & Y \ar[d, "\xleftarrow{b}="] \\
						X' \ar[r, "=\xrightarrow{f'}"'] & Y'
					\end{tikzcd}
					\qquad
					\begin{tikzcd}
						Y' \ar[d, "=\xrightarrow{b}"'] \ar[r, "\xleftarrow{f'}="] & X' \ar[d, "=\xrightarrow{a}"] \\
						Y \ar[r, "\xleftarrow{f}="'] & X
					\end{tikzcd}
					\qquad
					\begin{tikzcd}
						X' \ar[d, "=\xrightarrow{a}"'] \ar[r, "=\xrightarrow{f'}"] & Y' \ar[d, "=\xrightarrow{b}"] \\
						X \ar[r, "=\xrightarrow{f}"'] & Y
					\end{tikzcd}
				.\]
				Thus, we conclude from the fact that $D$ is a 2-functor.

			\item \begin{enumerate}

				\item By \cite[Lemma F.6(4)]{TwAmb}, it suffices to show this when $f \in I$ and when $f \in P^\natural$. Indeed, if $f \in I$, then this follows from the previous point since \eqref{eqn:! bc 6FF} is the horizontal left mate square of \eqref{eqn:* bc 6FF}. It follows that \eqref{eqn:! bc 6FF} is horizontally right adjointable, so it is vertically left adjointable since $b^*$ and $a^*$ admit left adjoints.

					If $f \in P^\natural$, then $D^*$ has right-left base change for $f$ against $b$, so we conclude using the previous point again. The last statement follows from \cite[Proposition 4.14]{CLL6FF}.

				\item We once again reduce to the cases $f \in I^\natural$ and $f \in P$. If $f \in P$, then as before, we use the previous point to find that \eqref{eqn:! bc 6FF} is horizontally left adjointable, so it is vertically right adjointable.

					If $f \in I^\natural$, then $D^*$ has left-right base change for $f$ against $b$, so we conclude using the previous point again. The last statement follows from \cite[Proposition 4.14]{CLL6FF}.

			\end{enumerate}

		\end{enumerate}
		
	\end{proof}
\end{prp}

\begin{prp} \label{prp:2-cat 6FF trunc}
	Assume that $\mathcal C$ admits finite products, so that we can consider restriction along the symmetric monoidal functor $\Span(\mathcal C,E) \to \Span_2(\mathcal C,E)_{P,I}$ to obtain a functor
	% PERF: explain why this functor is monoidal
	\begin{equation} \label{eqn:2-cat 3FF res}
		\Alg_{\Span_2(\mathcal C,E)_{P,I}}(\widehat{\Cat}) \to \Alg_{\Span(\mathcal C,E)}(\widehat{\Cat})
	\end{equation}
	from the category of lax symmetric monoidal 2-functors $\Span_2(\mathcal C,E)_{P,I} \to \widehat{\Cat}$ to the category of lax symmetric monoidal functors $\Span(\mathcal C,E) \to \widehat{\Cat}$.

	Suppose that every map in $E$ is of the form $p \circ j$ for $p \in P$ and $j \in I$, and every map in $I \cup P$ is truncated. Then this functor is fully faithful with essential image given by the 3-functor formalisms $D : \Span(\mathcal C,E) \to \widehat{\Cat}$ such that every map in $I$ is $D$-suave and every map in $P$ is $D$-prim.
	\begin{proof}
		Since every map in $I \cup P$ is truncated, and $I,P$ are stable under taking diagonals, we have that for any 3-functor formalism $D$ on $(\mathcal C,E)$, all maps in $I$ are $D$-suave if and only if they are all cohomologically \'{e}tale in the sense of \cite[Definition 6.12]{Scholze6FF}, and all maps in $P$ are $D$-prim if and only if they are all cohomologically proper in the sense of \cite[Definition 6.10]{Scholze6FF}.

		Since every map in $E$ is of the form $p \circ j$ for $p \in P$ and $j \in I$, and $I,P$ are stable under base change, it follows from \cite[Proposition 2.13]{6FFUnique} that a 3-functor formalism $D$ on $(\mathcal C,E)$ is Nagata in the sense of \cite[Definition 2.15]{6FFUnique} if and only if all maps in $I$ are $D$-suave, and all maps in $P$ are $D$-prim.

		Since every map in $E$ is a composite of maps in $I \cup P$, it follows from \Cref{prp:nat sq adj 6FF}(\ref{itm:nat sq adj/suave prim}) that if $D : \Span(\mathcal C,E) \to \widehat{\Cat}$ is the restriction of a lax symmetric monoidal 2-functor $\Span_2(\mathcal C,E)_{P,I} \to \widehat{\Cat}$, then every map in $I$ is $D$-suave, and every map in $P$ is $D$-prim. Thus, \eqref{eqn:2-cat 3FF res} lands in the full subcategory of Nagata 3-functor formalisms.

		% PERF: explain why this agrees with Definition 4.30 of CLL
		By \cite[Theorem 3.3]{6FFUnique}, restriction along $\mathcal C^\op \to \Span(\mathcal C,E)$ induces an equivalence from the category of Nagata 3-functor formalisms to the subcategory of $\Alg_{\mathcal C^\op} \widehat{\Cat}$ (where $\mathcal C$ has the coCartesian monoidal structure) consisting of $(I,P)$-biadjointable lax symmetric monoidal functors $\mathcal C^\op \to \widehat{\Cat}$ of \cite[Definition 4.30]{CLL6FF}, and natural transformations between them that are left adjointable at maps in $I$ and right adjointable at maps in $P$.

		On the other hand, since every map in $I \cap P$ is truncated, \cite[Theorem B]{CLL6FF} shows that restriction along the composite $\mathcal C^\op \to \Span(\mathcal C,E) \to \Span_2(\mathcal C, E)_{P,I}$ also induces an equivalence of the domain of \eqref{eqn:2-cat 3FF res} with this category. It follows that the functor \eqref{eqn:2-cat 3FF res} induces an equivalence of its domain with the category of Nagata 3-functor formalisms, as desired.
	\end{proof}
\end{prp}

% \clearpage
\phantomsection
% \addcontentsline{toc}{part}{References}
\bibliography{refs}

\providecommand{\bysame}{\leavevmode\hbox to3em{\hrulefill}\thinspace}
\providecommand{\MR}{\relax\ifhmode\unskip\space\fi MR }
\providecommand{\MRhref}[2]{%
  \href{http://www.ams.org/mathscinet-getitem?mr=#1}{#2}
}
\providecommand{\href}[2]{#2}
\begin{thebibliography}{{Man}22}

\bibitem[AHW17]{affrepI}
Aravind Asok, Marc Hoyois, and Matthias Wendt, \emph{Affine representability results in {$\mathbb A^1$}-homotopy theory, { I}: vector bundles}, Duke Math. J. \textbf{166} (2017), no.~10, 1923--1953. \MR{3679884}

\bibitem[AOV08]{tame-stacks}
Dan Abramovich, Martin Olsson, and Angelo Vistoli, \emph{Tame stacks in positive characteristic}, Annales de l'Institut Fourier \textbf{58} (2008), no.~4, 1057--1091 (en). \MR{2427954}

\bibitem[Ayo07a]{Ayoub6I}
Joseph Ayoub, \emph{Les six op\'erations de {G}rothendieck et le formalisme des cycles \' evanescents dans le monde motivique. {I}}, Ast\'erisque (2007), no.~314, x+466. \MR{2423375}

\bibitem[Ayo07b]{Ayoub6II}
\bysame, \emph{Les six op\'erations de {G}rothendieck et le formalisme des cycles \'evanescents dans le monde motivique. {II}}, Ast\'erisque (2007), no.~315, vi+364. \MR{2438151}

\bibitem[Ayo10]{AyoubBetti}
\bysame, \emph{Note sur les op\'erations de {G}rothendieck et la r\'ealisation de {B}etti}, J. Inst. Math. Jussieu \textbf{9} (2010), no.~2, 225--263. \MR{2602027}

\bibitem[Ayo14]{AyoubICM}
\bysame, \emph{A guide to (\'etale) motivic sheaves}, Proceedings of the {I}nternational {C}ongress of {M}athematicians---{S}eoul 2014. {V}ol. {II}, Kyung Moon Sa, Seoul, 2014, pp.~1101--1124. \MR{3728654}

\bibitem[BH21]{mot-norms}
Tom Bachmann and Marc Hoyois, \emph{Norms in motivic homotopy theory}, Ast\'erisque (2021), no.~425, ix+207. \MR{4288071}

\bibitem[BS20]{C2equiv}
Mark Behrens and Jay Shah, \emph{{$C_2$}-equivariant stable homotopy from real motivic stable homotopy}, Ann. K-Theory \textbf{5} (2020), no.~3, 411--464. \MR{4132743}

\bibitem[CD19]{tri-cat-mixed-motives}
Denis-Charles Cisinski and Frédéric Déglise, \emph{{Triangulated Categories of Mixed Motives}}, Springer International Publishing, 2019.

\bibitem[CLL25]{CLL6FF}
Bastiaan Cnossen, Tobias Lenz, and Sil Linskens, \emph{Universality of span 2-categories and the construction of 6-functor formalisms}, 2025, \url{https://arxiv.org/abs/2505.19192v1}.

\bibitem[Cno23]{TwAmb}
Bastiaan Cnossen, \emph{Twisted ambidexterity in equivariant homotopy theory: Two approaches}, Ph.D. thesis, Universit{\"a}ts-und Landesbibliothek Bonn, 2023.

\bibitem[CSY21]{CARMELI2021107763}
Shachar Carmeli, Tomer~M. Schlank, and Lior Yanovski, \emph{Ambidexterity and height}, Advances in Mathematics \textbf{385} (2021), 107763.

\bibitem[DG22]{UnivFF}
Brad Drew and Martin Gallauer, \emph{The universal six-functor formalism}, Ann. K-Theory \textbf{7} (2022), no.~4, 599--649. \MR{4560376}

\bibitem[DK24]{6FFUnique}
Adam Dauser and Josefien Kuijper, \emph{Uniqueness of six-functor formalisms}, 2024, \url{https://arxiv.org/abs/2412.15780v2}.

\bibitem[{Dre}18]{MotivicHodge}
Brad {Drew}, \emph{{Motivic {H}odge modules}}, January 2018, \url{https://arxiv.org/abs/1801.10129}.

\bibitem[EH23]{bispans}
Elden Elmanto and Rune Haugseng, \emph{On distributivity in higher algebra i: the universal property of bispans}, Compositio Mathematica \textbf{159} (2023), no.~11, 2326–2415.

\bibitem[Hau21]{laxtrans}
Rune Haugseng, \emph{On lax transformations, adjunctions, and monads in $(\infty,2)$-categories}, 2021, \url{https://arxiv.org/abs/2002.01037}.

\bibitem[He25]{LCA6FF}
Li~He, \emph{The universal continuous six functor formalism on light condensed anima}, arXiv e-prints (2025), \url{https://arxiv.org/abs/2511.17944v1}.

\bibitem[Hei25]{bienriched}
Hadrian Heine, \emph{On bi-enriched $\infty$-categories}, 2025, \url{{https://arxiv.org/abs/2406.09832}}.

\bibitem[HM24]{HM6FF}
Claudius {Heyer} and Lucas {Mann}, \emph{{6-Functor Formalisms and Smooth Representations}}, October 2024, \url{https://arxiv.org/abs/2410.13038v1}.

\bibitem[Hoy14]{quadratic-refinement-GLV-trace}
Marc Hoyois, \emph{A quadratic refinement of the {G}rothendieck–{L}efschetz–{V}erdier trace formula}, Algebraic \& Geometric Topology \textbf{14} (2014), no.~6, 3603–3658.

\bibitem[Hoy17]{sixopsequiv}
\bysame, \emph{The six operations in equivariant motivic homotopy theory}, Adv. Math. \textbf{305} (2017), 197--279, Some corrections have been made to the arXiv version as recently as 2024, and references in the text should actually be interpreted as referring to the version of this paper found at \url{ arxiv.org/abs/1509.02145v5}. \MR{3570135}

\bibitem[Kha19]{locspalg}
Adeel~A. Khan, \emph{The {M}orel-{V}oevodsky localization theorem in spectral algebraic geometry}, Geom. Topol. \textbf{23} (2019), no.~7, 3647--3685. \MR{4046969}

\bibitem[Kha21]{SixAlgSp}
Adeel~A. Khan, \emph{{V}oevodsky’s criterion for constructible categories of coefficients}, \url{https://www.preschema.com/papers/six.pdf}, 2021.

\bibitem[KR24]{SixAlgSt}
Adeel~A. Khan and Charanya Ravi, \emph{Generalized cohomology theories for algebraic stacks}, Advances in Mathematics \textbf{458} (2024), 109975.

\bibitem[Lur09]{htt}
Jacob Lurie, \emph{Higher topos theory}, Annals of Mathematics Studies, vol. 170, Princeton University Press, Princeton, NJ, 2009. \MR{2522659}

\bibitem[Lur17]{ha}
\bysame, \emph{Higher algebra}, \url{http://www.math.harvard.edu/~lurie/papers/HA.pdf}, September 2017.

\bibitem[Lur25]{kerodon}
Jacob Lurie, \emph{Kerodon}, \url{https://kerodon.net}, 2025.

\bibitem[LZ17]{LZ6FF}
Yifeng {Liu} and Weizhe {Zheng}, \emph{{Enhanced six operations and base change theorem for higher {A}rtin stacks}}, September 2017, \url{https://arxiv.org/abs/1211.5948v3}.

\bibitem[Mag25]{Fundamentals}
Roy Magen, \emph{Universal properties and constructions of pullback formalisms in terms of invariance and stability}, 2025, \url{https://arxiv.org/abs/2510.17702}.

\bibitem[Magon]{Gluing}
Roy Magen, \emph{The gluing property of pullback formalisms}, in preparation.

\bibitem[{Man}22]{Mann6FF}
Lucas {Mann}, \emph{{A $p$-Adic 6-Functor Formalism in Rigid-Analytic Geometry}}, June 2022, \url{https://arxiv.org/abs/2206.02022v1}.

\bibitem[Rob15]{robalo}
Marco Robalo, \emph{{$K$}-theory and the bridge from motives to noncommutative motives}, Adv. Math. \textbf{269} (2015), 399--550. \MR{3281141}

\bibitem[Ryd11]{qDM_compactification}
David Rydh, \emph{{C}ompactification of {T}ame {D}eligne–{M}umford {S}tacks}, \url{https://people.kth.se/~dary/tamecompactification20110517.pdf}, 2011.

\bibitem[Ryd15a]{Rydh_2015}
\bysame, \emph{{A}pproximation of {S}heaves on {A}lgebraic {S}tacks}, International Mathematics Research Notices \textbf{2016} (2015), no.~3, 717–737.

\bibitem[Ryd15b]{Noeth-approx}
David Rydh, \emph{Noetherian approximation of algebraic spaces and stacks}, Journal of Algebra \textbf{422} (2015), 105--147.

\bibitem[Sch25]{Scholze6FF}
Peter Scholze, \emph{{Six-Functor Formalisms}}, \url{https://people.mpim-bonn.mpg.de/scholze/SixFunctors.pdf}, 2025.

\bibitem[{Sta}25]{stacks-project}
The {Stacks Project Authors}, \emph{\textit{Stacks Project}}, \url{https://stacks.math.columbia.edu}, 2025.

\bibitem[Tub25a]{HodgeStacks}
Swann Tubach, \emph{{Mixed Hodge modules on stacks}}, Forum of Mathematics, Sigma \textbf{13} (2025), e175.

\bibitem[Tub25b]{TubachHodge}
\bysame, \emph{{On the Nori and Hodge realisations of Voevodsky motives}}, Compositio Mathematica \textbf{161} (2025), no.~9, 2155–2201.

\bibitem[Voe01]{voe-crit}
Valdimir Voevodsky, \emph{{V}oevodsky’s lectures on cross functors}, \url{https://www.math.ias.edu/vladimir/node/94}, 2001.

\bibitem[Voe10]{voe-cd}
Vladimir Voevodsky, \emph{Homotopy theory of simplicial sheaves in completely decomposable topologies}, Journal of Pure and Applied Algebra \textbf{214} (2010), no.~8, 1384--1398.

\bibitem[Yan22]{MonTow}
Lior Yanovski, \emph{The monadic tower for {$\infty$}-categories}, Journal of Pure and Applied Algebra \textbf{226} (2022), no.~6, 106975.

\bibitem[Zav23]{zavyalov6FF}
Bogdan Zavyalov, \emph{{P}oincar\'e {D}uality in abstract 6-functor formalisms}, 2023, \url{https://arxiv.org/abs/2301.03821}, p.~arXiv:2301.03821.

\end{thebibliography}
\bibliographystyle{amsalpha}

\end{document}